\documentclass[reqno,a4paper]{amsart}
\usepackage{color}
\usepackage[colorlinks=true,allcolors=blue,backref=page]{hyperref}
\usepackage{amsmath, amssymb, amsthm}
\usepackage{mathrsfs}
\usepackage{mathtools}
\usepackage[noabbrev,capitalize,nameinlink]{cleveref}
\crefname{equation}{}{}
\usepackage{fullpage}
\usepackage[noadjust]{cite}
\usepackage{graphics}
\usepackage{pifont}
\usepackage{tikz}
\usepackage{bbm}
\usepackage[T1]{fontenc}

\usetikzlibrary{arrows.meta}

\usepackage{environ}
\usepackage{framed}
\usepackage{url}
\usepackage[linesnumbered,ruled,vlined]{algorithm2e}
\usepackage[noend]{algpseudocode}
\usepackage[labelfont=bf]{caption}
\usepackage{cite}
\usepackage{framed}
\usepackage[framemethod=tikz]{mdframed}
\usepackage{appendix}
\usepackage{graphicx}
\usepackage[textsize=tiny]{todonotes}
\usepackage{tcolorbox}
\usepackage{enumerate}
\allowdisplaybreaks[4]
\usepackage{enumerate}
\usepackage{stmaryrd}

\usepackage[shortlabels]{enumitem}
\crefformat{enumi}{#2#1#3}
\crefrangeformat{enumi}{#3#1#4 to~#5#2#6}
\crefmultiformat{enumi}{#2#1#3}
{ and~#2#1#3}{, #2#1#3}{ and~#2#1#3}

\DeclareSymbolFont{symbolsC}{U}{pxsyc}{m}{n}
\SetSymbolFont{symbolsC}{bold}{U}{pxsyc}{bx}{n}
\DeclareFontSubstitution{U}{pxsyc}{m}{n}
\DeclareMathSymbol{\medcircle}{\mathbin}{symbolsC}{7}

\usepackage{thmtools}
\usepackage{thm-restate}

\crefname{algocf}{Algorithm}{Algorithms}

\crefname{equation}{}{} 
\AtBeginEnvironment{appendices}{\crefalias{section}{appendix}} 

\usepackage[color,final]{showkeys} 

\colorlet{refkey}{orange!20}
\colorlet{labelkey}{blue!30}

\crefname{algocf}{Algorithm}{Algorithms}

\numberwithin{equation}{section}
\newtheorem{theorem}{Theorem}[section]
\newtheorem{proposition}[theorem]{Proposition}
\newtheorem{lemma}[theorem]{Lemma}
\newtheorem{claim}[theorem]{Claim}
\newtheorem*{claim*}{Claim}

\crefname{claim}{Claim}{Claims}

\newtheorem*{question*}{Question}

\theoremstyle{definition}
\newtheorem{definition}[theorem]{Definition}

\newtheorem*{definition*}{Definition}

\theoremstyle{remark}
\newtheorem{remark}[theorem]{Remark}
\newtheorem*{remark*}{Remark}

\newcommand{\norm}[1]{\bigg\lVert#1\bigg\rVert}
\newcommand{\snorm}[1]{\lVert#1\rVert}

\newcommand{\eps}{\varepsilon}
\renewcommand{\Re}{\on{Re}}

\newcommand{\su}{\subseteq}
\newcommand{\sm}{\setminus}

\newcommand{\mb}{\mathbb}

\newcommand{\mbm}{\mathbbm}
\newcommand{\mc}{\mathcal}

\newcommand{\mr}{\mathrm}

\newcommand{\on}{\operatorname}
\newcommand{\wh}{\widehat}
\newcommand{\wt}{\widetilde}

\let\originalleft\left
\let\originalright\right
\renewcommand{\left}{\mathopen{}\mathclose\bgroup\originalleft}
\renewcommand{\right}{\aftergroup\egroup\originalright}

\newcommand{\noop}[1]{}

\title{Anticoncentration in Ramsey graphs and a proof of the Erd\H{o}s--McKay conjecture}

\author[Kwan]{Matthew Kwan}
\address{Institute of Science and Technology (IST) Austria}
\email{matthew.kwan@ist.ac.at}

\author[A2]{Ashwin Sah}
\author[A3]{Lisa Sauermann}
\author[A4]{Mehtaab Sawhney}
\address{Department of Mathematics, Massachusetts Institute of Technology, Cambridge, MA 02139, USA}
\email{\{asah,lsauerma,msawhney\}@mit.edu}

\thanks{
Kwan was supported for part of this work by ERC Starting Grant ``RANDSTRUCT'' No.~101076777.
Sah and Sawhney were supported by NSF Graduate Research Fellowship Program DGE-2141064. Sah was supported by the PD Soros Fellowship. Sauermann was supported by NSF Award DMS-2100157, and for part of this work by a Sloan Research Fellowship.
}

\begin{document}

\global\long\def\mk#1{\textcolor{red}{\textbf{[MK comments:} #1\textbf{]}}}
\global\long\def\ls#1{\textcolor{blue}{\textbf{[LS comments:} #1\textbf{]}}}

\maketitle
\begin{abstract}
An $n$-vertex graph is called \emph{$C$-Ramsey} if it has no clique
or independent set of size $C\log_2 n$ (i.e., if it has near-optimal
Ramsey behavior). In this paper, we study edge-statistics in Ramsey graphs, in particular obtaining very precise control of the distribution of the number of edges in a random vertex subset of a $C$-Ramsey graph.
This brings together two ongoing lines of research: the study of ``random-like'' properties of Ramsey graphs and the study of small-ball probability for low-degree polynomials of independent random variables.

The proof proceeds via an ``additive structure'' dichotomy on the degree sequence, and involves a wide range of different tools from
Fourier analysis, random matrix theory, the theory of Boolean functions, probabilistic combinatorics, and low-rank approximation. In particular, a key ingredient is a new sharpened version of the quadratic Carbery--Wright theorem on small-ball probability for polynomials of Gaussians, which we believe is of independent interest.
One of the consequences of our result is the resolution of an old conjecture of Erd\H{o}s and McKay, for which Erd\H{o}s reiterated in several of his open problem collections, and for which he offered one of his notorious monetary prizes.

\end{abstract}

\section{Introduction}\label{sec:introduction}
An induced subgraph of a graph is called \emph{homogeneous} if
it is a clique or independent set (i.e., all possible edges are present,
or none are). One of the most fundamental results in Ramsey theory,
proved in 1935 by Erd\H{o}s and Szekeres~\cite{ES35}, states that every $n$-vertex
graph contains a homogeneous subgraph with at least $\frac{1}{2}\log_{2}n$
vertices\footnote{Since the original submission of the present paper, this bound was improved to $(\frac12+\varepsilon)\log_2$ for an absolute constant $\varepsilon>0$ in breakthrough work by Campos, Griffiths, Morris, and Sahasrabudhe~\cite{CGMS}.}. On the other hand, Erd\H{o}s~\cite{Erd47} famously used the probabilistic
method to prove that, for all $n\ge 3$, 
there is an $n$-vertex graph
with no homogeneous subgraph on $2\log_{2}n$ vertices. Despite significant
effort (see for example \cite{BRSW12,Li18,CZ16,FW81,Coh16,Chu81,Nag72,Abb72,Fra77,Gop14}), there are no known non-probabilistic constructions
of graphs with comparably small homogeneous sets, and in fact the
problem of explicitly constructing such graphs is intimately related
to \emph{randomness extraction} in theoretical computer science (see for example \cite{Sha11} for an introduction to the topic).

For some $C>0$, an $n$-vertex graph is called \emph{$C$-Ramsey} if it has
no homogeneous subgraph of size $C\log_{2}n$. We think of $C$ as
being a constant (not varying with $n$), so $C$-Ramsey graphs are
those graphs with near-optimal Ramsey behavior. It is widely believed that
$C$-Ramsey graphs must in some sense resemble random graphs (which would provide some explanation for why it is so hard to find explicit constructions), and
this belief has been supported by a number of theorems showing that
certain structural or statistical properties characteristic of random graphs
hold for all $C$-Ramsey graphs. The first result of this type was
due to Erd\H{o}s and Szemer\'edi~\cite{ES72}, who showed that every $C$-Ramsey
graph $G$ has edge-density bounded away from zero and one (formally,
for any $C>0$ there is $\eps_{C}>0$ such that for sufficiently large $n$, the number of edges
in any $C$-Ramsey graph with $n$ vertices lies between $\eps_{C}\binom{n}{2}$
and $(1-\eps_{C})\binom{n}{2}$).
Note that this implies fairly
strong information about the edge distribution on induced subgraphs
of $G$, because any induced subgraph of $G$ with at least $n^{\alpha}$ vertices is itself $(C/\alpha)$-Ramsey.

This basic result was the foundation for a large amount of further
research on Ramsey graphs; over the years many conjectures have been
proposed and many theorems proved (see for example \cite{AH91,ABKS09,AB89,AK09,AKS03,BS07,EH77,Erd92,JKLY20,KS19,KS20,NST16,PR99,She98,AO95,LPRT17}). Particular attention has focused on
a sequence of conjectures made by Erd\H{o}s and his collaborators,
exploring the theme that Ramsey graphs must have diverse induced subgraphs. For example, for a $C$-Ramsey
graph $G$ with $n$ vertices, it was proved by Pr\"omel and R\"odl~\cite{PR99} (answering
a conjecture of Erd\H{o}s and Hajnal) that $G$ contains every possible
induced subgraph on $\delta_{C}\log n$ vertices; by Shelah~\cite{She98} (answering
a conjecture of Erd\H{o}s and R\'enyi) that $G$ contains $2^{\delta_{C}n}$
non-isomorphic induced subgraphs; by the first author and Sudakov~\cite{KS19}
(answering a conjecture of Erd\H{o}s, Faudree, and S\'os) that $G$
contains $\delta_{C}n^{5/2}$ subgraphs that can be distinguished
by looking at their edge and vertex numbers; and by Jenssen, Keevash, Long, and Yepremyan~\cite{JKLY20} (improving on a conjecture of Erd\H{o}s, Faudree, and S\'os proved by Bukh and Sudakov~\cite{BS07}) that $G$ contains an induced subgraph with $\delta_C n^{2/3}$ distinct degrees (all for some $\delta_{C}>0$
depending on $C$).

Only one of Erd\H os' conjectures (on properties of $C$-Ramsey graphs) from this period has
remained open until now: Erd\H{o}s and McKay (see \cite{Erd92}) made the ambitious conjecture that for essentially any ``sensible'' integer
$x$, every $C$-Ramsey graph must necessarily contain an induced
subgraph with exactly $x$ edges. To be precise, they conjectured
that there is $\delta_{C}>0$ depending on $C$ such that for  any $C$-Ramsey graph
$G$ with $n$ vertices and any integer
$0\le x\le\delta_{C}n^{2}$, there is an induced subgraph of $G$ with exactly $x$ edges.
Erd\H{o}s reiterated this problem in several collections of his favorite
open problems in combinatorics~\cite{Erd92,Erd95} (also in \cite{Erd97}), and offered one of his notorious
monetary prizes (\$100) for its solution (see \cite{Erd95,CG98,Chu97}).

Progress on the Erd\H{o}s--McKay conjecture has come from four different
directions. First, the canonical example of a Ramsey graph is (a typical outcome of) an Erd\H{o}s--R\'enyi random
graph. It was proved by Calkin, Frieze and McKay~\cite{CFM92} (answering questions raised by Erd\H{o}s and McKay) that for any constants
$p\in(0,1)$ and $\eta>0$, a random
graph $\mathbb{G}(n,p)$ typically contains induced subgraphs with
all numbers of edges up to $(1-\eta)p\binom{n}{2}$. Second,
improving on initial bounds of Erd\H{o}s and McKay \cite{Erd92}, it was proved by Alon, Krivelevich, and Sudakov~\cite{AKS03} that there is $\alpha_{C}>0$ such that in a $C$-Ramsey graph on $n$ vertices, one can always find an induced subgraph with
any given number of edges up to $n^{\alpha_{C}}$. Third, improving on a result
of Narayanan, Sahasrabudhe, and Tomon~\cite{NST16}, the first author and Sudakov~\cite{KS20}
proved that there is $\delta_{C}>0$ such that in any $C$-Ramsey
graph on $n$ vertices contains induced subgraphs with $\delta_{C}n^{2}$ different numbers of edges (though without making any guarantee on
what those numbers of edges are). Finally, Long and Ploscaru~\cite{LP} recently
proved a \emph{bipartite} analog of the Erd\H{o}s--McKay conjecture.

As our first result, we prove a substantial strengthening of the
Erd\H{o}s--McKay conjecture\footnote{To see that this implies the Erd\H os--McKay conjecture, first note that we can assume $n$ is sufficiently large in terms of $C$ (specifically, we can assume $n\ge n_C$ for any $n_C\in \mb{N}$ by taking $\delta_C$ small enough that $\delta_Cn_C^2<1$). Now, by the above-mentioned result of Erd\H os and Szemer\'edi \cite{ES72}, there is $\varepsilon_C>0$ such that for every $C$-Ramsey graph $G$ on $n$ vertices we have $e(G)\ge \varepsilon_C\binom n2\ge \varepsilon_C n^2/4$. So, taking $\delta_C\le \varepsilon_C/8$, the Erd\H os--McKay conjecture follows from the $\eta=1/2$ case of \cref{thm:erdos-mckay}.}. Let $e(G)$ be the number of edges in a graph $G$.

\begin{theorem}\label{thm:erdos-mckay}
Fix $C > 0$ and $\eta>0$, and let $G$ be a $C$-Ramsey graph on $n$ vertices, where $n$ is sufficiently large with respect to $C$ and $\eta$. Then for any integer $x$ with $0\leq x\leq (1-\eta)e(G)$, there is a subset $U\subseteq V(G)$ inducing exactly $x$ edges.
\end{theorem}

Given prior results due to Alon, Krivelevich and Sudakov~\cite{AKS03}, \cref{thm:erdos-mckay} is actually a simple corollary of a much deeper result (\cref{thm:ramsey-sample}) on edge-statistics in Ramsey graphs, which we discuss in the next subsection.

\subsection{Edge-statistics and low-degree polynomials}

For an $n$-vertex graph $G$, observe that the number of edges $e(G[U])$
in an induced subgraph $G[U]$ can be viewed as an evaluation of a
quadratic polynomial associated with $G$. Indeed, identifying the vertex set of $G$ with $\{1,\ldots,n\}$ and writing $E$ for
the edge set of $G$, consider the $n$-variable quadratic polynomial
$f(\xi_{1},\ldots,\xi_{n})=\sum_{ij\in E}\xi_{i}\xi_{j}$. Then, for
any vertex set $U$, let $\vec{\xi}\,^{(U)}$ be the characteristic
vector of $U$ (with $\vec{\xi}\,^{(U)}_v=1$ if $v\in U$, and $\vec{\xi}\,^{(U)}_v=0$
if $v\notin U$). It is easy to check that the number of edges $e(G[U])$
induced by $U$ is precisely equal to $f(\vec{\xi}\,^{(U)})$. That
is, to say, the statement that $G$ has an induced subgraph with exactly
$x$ edges is precisely equivalent to the statement that there is
a binary vector $\vec{\xi}\in \{0,1\}^n$ with $f(\vec{\xi})=x$.

There are many combinatorial quantities of interest that can
be interpreted as low-degree polynomials of binary vectors. For example,
the number of triangles in a graph, or the number of 3-term arithmetic
progressions in a set of integers, can both be naturally interpreted as evaluations
of certain \emph{cubic} polynomials. More generally, the study of
\emph{Boolean functions} is the study of functions of the form $f\colon\{0,1\}^{n}\to\mathbb{R}$;
every such function can be written (uniquely) as
a multilinear polynomial, and the degree of this polynomial
is a fundamental measure of the ``complexity'' of the Boolean function.

One of the most important discoveries from the analysis of Boolean functions
is that it is fruitful to study the behavior of (low-degree) Boolean functions evaluated on a \emph{random} binary vector $\vec{\xi}\in\{0,1\}^n$. This is the perspective we take in this paper: as our main result, for any Ramsey graph $G$ and a random vertex subset $U$, we obtain very precise control over the distribution of $e(G[U])$.

\begin{theorem}\label{thm:ramsey-sample}
Fix $C,\lambda > 0$, let $G$ be a $C$-Ramsey graph on $n$ vertices and let $\lambda\le p\le 1-\lambda$. Then if $U$ is a random subset of $V(G)$ obtained by independently including each vertex with probability $p$, we have
\[\sup_{x\in\mb{Z}}\Pr[e(G[U]) = x] \le K_{C,\lambda} n^{-3/2}\]
for some $K_{C,\lambda}>0$ depending only on $C,\lambda$. Furthermore, for every fixed $A>0$, we have
\[\qquad\inf_{\substack{x\in\mb{Z}\\|x-p^2e(G)|\le An^{3/2}}}\Pr[e(G[U])=x] \ge \kappa_{C,A,\lambda}n^{-3/2}\]
for some $\kappa_{C,A,\lambda}>0$ depending only on $C,A,\lambda$, if $n$ is sufficiently large in terms of $C,\lambda$ and $A$.
\end{theorem}

It is not hard to show that for any $C$-Ramsey graph $G$, the standard deviation $\sigma$ of $e(G[U])$ is of order $n^{3/2}$. So, \cref{thm:ramsey-sample} says (roughly speaking) that in the ``bulk'' of the distribution of $e(G[U])$ (i.e., within roughly standard-deviation-range of the mean), the point probabilities are all of order $1/\sigma$. In \cref{sec:deductions} we will give the short deduction of \cref{thm:erdos-mckay} from \cref{thm:ramsey-sample} and the aforementioned theorem of Alon, Krivelevich, and Sudakov.

\begin{remark}\label{rem:expander-works}
Our proof of \cref{thm:ramsey-sample} can be adapted to handle slightly more general types of graphs than Ramsey graphs. For example, we can obtain the same conclusions in the case where $G$ is a $d$-regular graph with $0.01n\le d\le 0.99n$, such that the eigenvalues $\lambda_1\ge\cdots\ge\lambda_n$ of the adjacency matrix of $G$ satisfy $\max\{\lambda_2,-\lambda_n\}\le n^{1/2+0.01}$ (i.e., the case where $G$ is a dense graph with near-optimal spectral expansion). See \cref{rem:expander-rich,rem:expander-dense} for some discussion of the necessary adaptations. Notably, this class of graphs includes \emph{Paley graphs}, which are ``random-like'' graphs with an explicit number-theoretic definition (see for example \cite{KS06}). These graphs are currently one of the most promising candidates for explicit constructions of Ramsey graphs, though precisely studying the Ramsey properties of these graphs seems to be outside the reach of current techniques in number theory (see \cite{HP21,DSW21} for recent developments).
\end{remark}

\begin{remark}
If $p=1/2$, then the random set $U$ in \cref{thm:ramsey-sample} is simply a uniformly random subset of vertices. So, for $x$ close to $e(G)/4$, \cref{thm:ramsey-sample} tells us that the \emph{number} of induced subgraphs with $x$ edges is of order $2^n/n^{3/2}$. It would be interesting to investigate the number of $x$-edge induced subgraphs for general $x$ (not close to $e(G)/4$). From \cref{thm:ramsey-sample} one can deduce a \emph{lower bound} on this number approximately matching the behavior of an appropriate Erd\H os--R\'enyi random graph (i.e., for any constant $\eta>0$, and $\eta n^2\le x\le (1-\eta)e(G)$, there are at least $\exp(H(\sqrt{x/e(G)})n + o(n))$ subgraphs with $x$ edges, where $H$ denotes the base-$e$ entropy function). However, a corresponding upper bound does not in general hold: to characterize the number of $x$-edge induced subgraphs up to any sub-exponential error term, one must incorporate more detailed information about the Ramsey graph $G$ than just its number of edges. (To see this, consider a union of two disjoint independent Erd\H os--R\'enyi random graphs $\mb G(n/2,0.01)\sqcup \mb G(n/2,0.99)$, and count subgraphs with $0.001n^2$ edges.)
\end{remark}

There has actually been quite some recent interest (see for example \cite{KST19,FKS21,MMNT19,FS20,AHKT20}) studying random variables of the form $e(G[U])$ for a graph $G$ and a random vertex set $U$, largely due to a sequence of conjectures by Alon, Hefetz, Krivelevich,
and Tyomkyn~\cite{AHKT20} motivated by the classical topic of \emph{graph inducibility}. Specifically, these works studied the \emph{anticoncentration} behavior of $e(G[U])$ (generally speaking, anticoncentration inequalities provide upper bounds on the probability that a random variable falls in some small ball or is equal to some particular value). As discussed above, $e(G[U])$ can be naturally interpreted as a quadratic polynomial, so this study falls within the scope of the so-called \emph{polynomial Littlewood--Offord problem} (which concerns anticoncentration of general low-degree polynomials of various types of random variables). There has been a lot of work from several different directions (see for example \cite{Cos13,Kan17,KS20b,TV09,NV11,Hal77,TV10b,RV08,Ngu12,NV13}) on the extent to which anticoncentration in the (polynomial) Littlewood--Offord problem is controlled by algebraic or arithmetic structure, and the upper bound in \cref{thm:ramsey-sample} can be viewed in this context: Ramsey graphs yield quadratic polynomials that are highly unstructured in a certain \emph{combinatorial} sense, and we see that such polynomials have strong anticoncentration behavior.

The first author, Sudakov and Tran~\cite{KST19} previously suggested to study anticoncentration of $e(G[W])$ for a Ramsey graph $G$ and a random vertex subset $W$ \emph{of a given size}. In particular, they asked whether for
a $C$-Ramsey graph $G$ with $n$ vertices, and a uniformly random subset $W$ of exactly
$n/2$ vertices, we have $\sup_{x\in\mathbb{Z}}\Pr[e(G[W])=x]\le K_{C}/n$
for some $K_{C}>0$ depending only on $C$. Some progress was made
on this question by the first and third authors~\cite{KS20b}; as a simple corollary of \cref{thm:ramsey-sample}, we answer this question in the affirmative.

\begin{theorem}\label{thm:KST}
For $C>0$ and $0<\lambda<1$, there is $K=K(C,\lambda)$
such that the following holds. Let $G$ be a $C$-Ramsey graph on $n$ vertices and
let $W\su V(G)$ be a random subset of exactly $k$ vertices, for some given $k$ with $\lambda n\le k\le(1-\lambda)n$.
Then
\[\sup_{x\in\mathbb{Z}}\Pr[e(G[W])=x]\le\frac{K}{n}.\]
\end{theorem}

It is not hard to show that the upper bound in \cref{thm:KST} is best-possible (indeed, this can be seen by taking $G$ to be a typical outcome of an Erd\H os--R\'enyi random graph $\mb G(n,1/2)$). However, in contrast to the setting of \cref{thm:ramsey-sample}, in \cref{thm:KST} one cannot hope for a matching lower bound when $x$ is close to $\mb{E}[e(G[W])]$ (as can be seen by considering the case where $G$ is a typical outcome of the union of two disjoint independent Erd\H{o}s--R\'enyi random graphs $\mb G(n,1/4)\sqcup \mb G(n,3/4)$).

\subsection{Proof ingredients and ideas}\label{sub:auxiliary}
We outline the proof of \cref{thm:ramsey-sample} in more detail in \cref{sec:outline}, but here we take the opportunity to highlight some of the most important ingredients and ideas.

\subsubsection{An approximate local limit theorem}\label{sub:local-limit}
A starting point is that, in the setting of \cref{thm:ramsey-sample}, standard techniques show that $e(G[U])$ satisfies a central limit theorem: we have $\Pr[e(G[U])\le x]=\Phi((x-\mu)/\sigma)+o(1/\sigma)$ for all $x\in\mathbb{R}$, where $\Phi$ is the standard Gaussian cumulative distribution function, and $\mu,\sigma$ are the mean and standard deviation of $e(G[U])$. It is natural to wonder (as suggested in \cite{KS20b} as a potential path towards the Erd\H{o}s--McKay conjecture) whether this can be strengthened to a \emph{local} central limit theorem: could it be that for all $x\in\mb R$ we have $\Pr[e(G[U])=x]=\Phi'((x-\mu)/\sigma)/\sigma+o(1/\sigma)$ (where $\Phi'$ is the standard Gaussian density function)? In fact, the statement of \cref{thm:ramsey-sample} can be interpreted as a local central limit theorem ``up to constant factors''. This perspective also suggests a strategy for the proof of \cref{thm:ramsey-sample}: perhaps we can leverage Fourier-analytic techniques previously developed for local central limit theorems (e.g.~\cite{Gne48,Var15,Ber16,Ber18,GK16,KLP17}), obtaining our desired result as a consequence of estimates on the characteristic function (i.e., Fourier transform) of our random variable $e(G[U])$.

However, it turns out that a local central limit theorem actually does not hold in general: while the coarse-scale distribution of $e(G[U])$ is always Gaussian, in general $e(G[U])$ may have a rather nontrivial ``two-scale'' behavior, depending on the additive structure of the degree sequence of $G$ (see \cref{fig:theta}). Roughly speaking, this translates to a certain ``spike'' in the magnitude of the characteristic function of $e(G[U])$, which rules out na\"ive Fourier-analytic approaches. To overcome this issue, we need to capture the ``reason'' for the two-scale behavior: It turns out that this ``spike'' can only happen if the degree sequence of $G$ is in a certain sense ``additively structured'', implying that there is a partition of the vertex set into ``buckets'' such that vertices in the same bucket have almost the same degree. Then, if we reveal the size of the intersection of $U$ with each bucket, the \emph{conditional} characteristic function of $e(G[U])$ is suitably bounded. We deduce conditional bounds on the point probabilities of $e(G[U])$, and average these over possible outcomes of the revealed intersection sizes of $U$ with the buckets.

\captionsetup{width=\linewidth}
\begin{figure}
    \centering
    \includegraphics[scale=0.3,trim={0 10px 0 15px},clip]{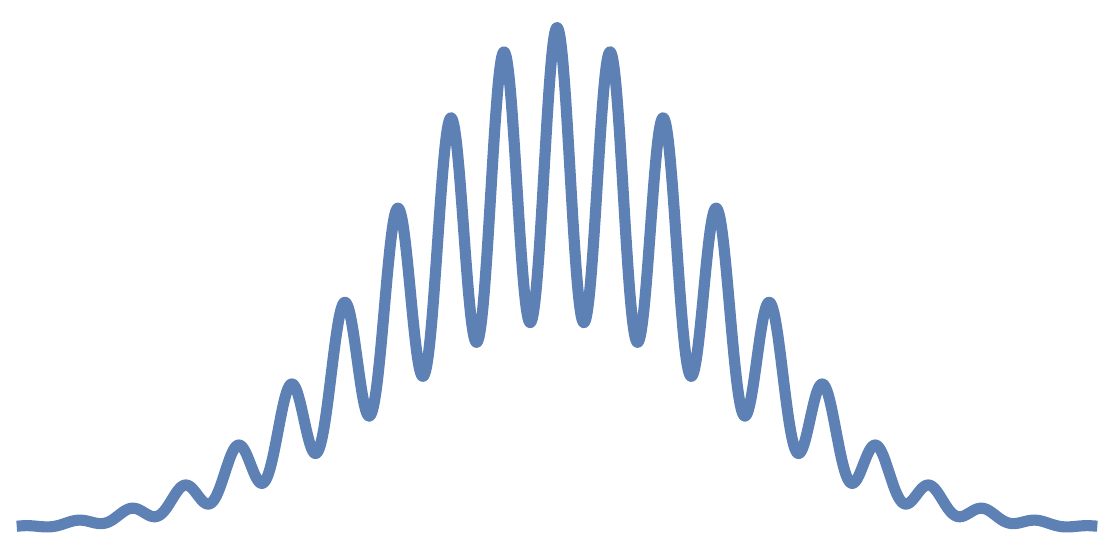}$\quad$
    \includegraphics[scale=0.2,trim={200px 20px 200px 20px},clip]{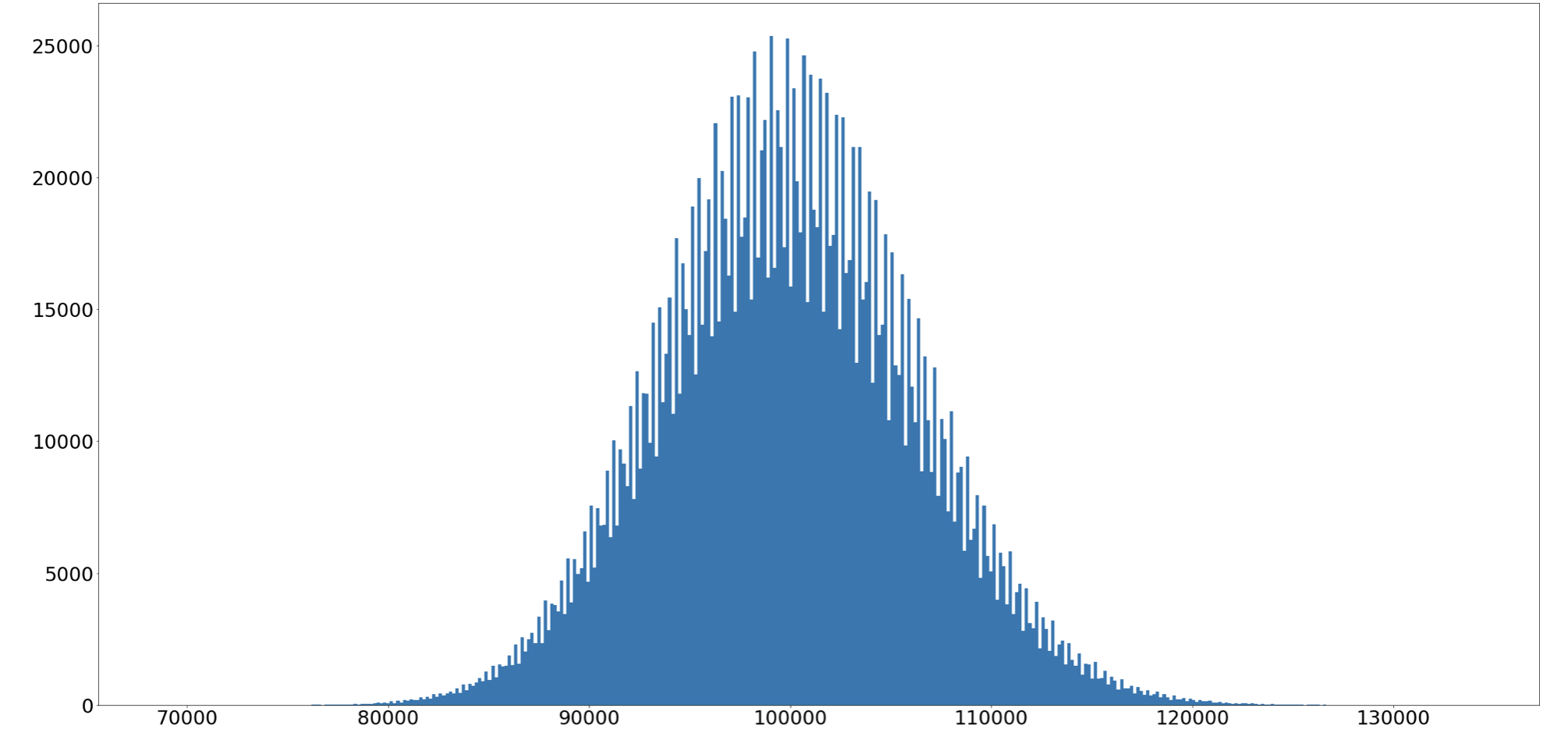}$\quad$
    \includegraphics[scale=0.2,trim={150px 20px 150px 20px},clip]{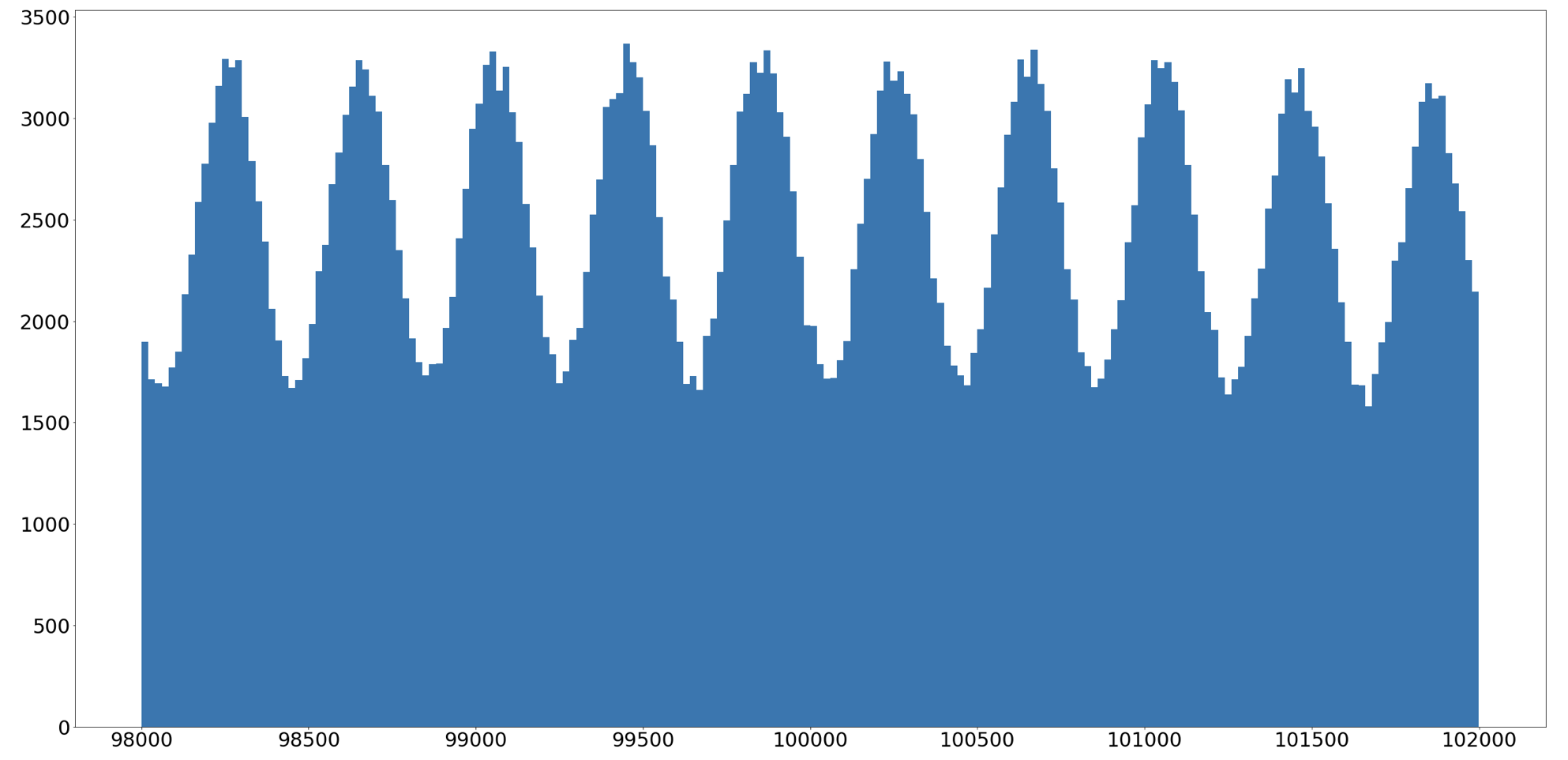}
    \caption{On the left is a cartoon of (one possibility for) the probability mass function of $e(G[U])$ for a Ramsey graph $G$ and a uniformly random vertex subset $U$: the large-scale behavior is Gaussian, but on a small scale we see many smaller Gaussian-like curves. The two images on the right are two different histograms at different scales, obtained from real data (namely, from two million independent samples of a uniformly random vertex subset in a graph $G$ obtained as an outcome of the Erd\H os--R\'enyi random graph $\mb G(1000,0.8)$).}
    \label{fig:theta}
\end{figure}

We remark that one interpretation of our proof strategy is that we are decomposing our random variable into ``components'' in physical space, in such a way that each component is well-behaved in Fourier space. This is at least superficially reminiscent of certain techniques in harmonic analysis; see for example \cite{Gut22}. Looking beyond the particular statement of \cref{thm:ramsey-sample}, we hope that the Fourier-analytic techniques in its proof will be useful for the general study of small-ball probability for low-degree polynomials of independent variables, especially in settings where Gaussian behavior may break down.

\subsubsection{Small-ball probability for quadratic Gaussian chaos}\label{subsub:gaussian-small-ball}

The general study of low-degree polynomials of independent random variables (sometimes called \emph{chaoses}) has a long and rich history. Some highlights include \emph{Kim--Vu polynomial concentration}~\cite{KV00}, the \emph{Hanson--Wright inequality}~\cite{HW71}, the \emph{Bonami--Beckner hypercontractive inequality} (see \cite{O14}), and \emph{polynomial chaos expansion} (see \cite{GS91}), which are fundamental tools in probabilistic combinatorics, high-dimensional statistics, the analysis of Boolean functions and mathematical modelling.

Much of this study has focused on low-degree polynomials of \emph{Gaussian} random variables, which enjoy certain symmetry properties that make them easier to study. While this direction may not seem obviously relevant to \cref{thm:ramsey-sample}, in part of the proof
we are able to apply the celebrated \emph{Gaussian invariance principle} of Mossel, O'Donnell, and Oleszkiewicz \cite{MOO10}, to compare our random variables of interest with certain ``Gaussian analogs''. Therefore, a key step in the proof of \cref{thm:ramsey-sample} is to study small-ball probability for quadratic polynomials of Gaussian random variables.

The fundamental theorem
in this area is the Carbery--Wright theorem~\cite{CW01}, which (specialized
to the quadratic case) says that for $0<\varepsilon<1$ and any real
quadratic polynomial $f=f(Z_{1},\ldots,Z_{n})$ of independent standard
Gaussian random variables $Z_{1},\ldots,Z_{n}\sim\mathcal{N}(0,1)$,
we have 
\[
\sup_{x\in\mathbb{R}}\Pr[|f-x|\le\varepsilon]=O\left(\sqrt{\varepsilon/\sigma(f)}\right).
\]
This is best-possible in general (for example, $\Pr[|Z_{1}^{2}|\le\varepsilon]$
scales like $\sqrt{\varepsilon}$ as $\varepsilon\to 0$). However, we are able to prove (in \cref{sec:anti-Gauss}) an
optimal bound of the form $O(\varepsilon/\sigma(f))$ in the case
where the degree-2 part of $f$ robustly has rank at least 3, in the sense of low-rank approximation (i.e.
in the case where the degree-2 part of $f$ is not close, in Frobenius\footnote{The Frobenius (or Hilbert-Schmidt) norm $\|M\|_{\mathrm{F}}$ of a matrix $M$ is the square
root of the sum of the squares of its entries.} norm, to a quadratic form of rank at most $2$).
\newcommand{\gaussianformula}{\[\sup_{x\in \mathbb R}\Pr[|f(\vec Z)-x|\le\eps] \le C_\eta\cdot \frac{\eps}{\sigma(f(\vec Z))}\] for some $C_\eta$ depending on $\eta$.}
\begin{restatable}{theorem}{gaussian}
\label{thm:gaussian-anticoncentration-upper-bound}
Let $\vec Z = (Z_1,\ldots,Z_n)\sim\mc{N}(0,1)^{\otimes n}$ be a vector of independent standard Gaussian random variables. Consider a real quadratic polynomial $f(\vec Z)$ of $\vec Z$, which we may write as 
\[f(\vec Z)=\vec Z^\intercal F \vec Z+\vec f\cdot \vec Z+f_0\]
for some nonzero symmetric matrix $F\in\mb{R}^{n\times n}$, some vector $\vec f\in\mb{R}^n$, and some $f_0\in\mb{R}$.
Suppose that for some $\eta>0$ we have
\[\min_{\substack{\wt F\in\mb{R}^{n\times n}\\\on{rank}(\wt F)\le 2}}\frac{\|F-\wt F\|^2_{\mr F}}{\|F\|^2_\mr{F}}\ge \eta.\]
Then for any $\eps > 0$ we have
\gaussianformula
\end{restatable}

We remark that our robust-rank-3 assumption is best possible, in the sense that this stronger bound may fail for quadratic forms
with robust rank $2$; for example $Z_{1}^{2}-Z_{2}^{2}$ has standard deviation $2$, and one can compute that $\Pr[|Z_{1}^{2}-Z_{2}^{2}|\le\varepsilon]$
scales like $\varepsilon\log(1/\varepsilon)$ as $\eps\to 0$.

We also remark that \cref{thm:gaussian-anticoncentration-upper-bound} can be interpreted as a kind of \emph{inverse theorem} or \emph{structure theorem}:
the only way for $f(\vec Z)$ to exhibit atypical small-ball
behavior is for $f$ to be close to a low-rank quadratic form (c.f.~inverse theorems for the \emph{Littlewood--Offord problem}~\cite{TV09,NV11,TV10b,RV08,Ngu12,NV13,KS20b}). It is also worth mentioning a different structure theorem due to Kane~\cite{Kan17}, showing that \emph{all} bounded-degree polynomials of Gaussian random variables can be, in a certain sense, ``decomposed'' into a small number of parts with typical small-ball behavior.

Finally, we remark that it would be interesting to investigate extensions of \cref{thm:gaussian-anticoncentration-upper-bound} to higher-degree polynomials. Our proof uses diagonalization of quadratic forms in a crucial way, and new ideas would therefore be required (the ideas in the aforementioned paper of Kane~\cite{Kan17} may be relevant).

\subsubsection{Rank of Ramsey graphs}\label{subsub:ramsey-rank}
In order to actually apply \cref{thm:gaussian-anticoncentration-upper-bound}, we need to use the fact that Ramsey
graphs have adjacency matrices which robustly have high rank. A version
of this fact was first observed by the first and third authors~\cite{KS20b},
but we will need a much stronger version involving a partition into submatrices (\cref{lem:ramsey-low-rank}). We believe that
the connection between rank and homogeneous sets is of very general
interest: for example, the celebrated log-rank conjecture in communication
complexity has an equivalent formulation (due to Nisan and Wigderson~\cite{NW94}) stating that a zero-one matrix with no large ``homogeneous
rectangle'' must have high rank. As part of our study of the rank of Ramsey graphs, we prove (\cref{prop:rank-close}) that binary matrices which are close to a low-rank real
matrix are also close to a low-rank \emph{binary} matrix. This may be of independent interest. 

\subsubsection{Switchings via moments}
It turns out that in the setting of \cref{thm:ramsey-sample}, Fourier-analytic estimates (in combination with the previously mentioned ideas) can only take us so far: for a $C$-Ramsey graph we can roughly estimate the probability that $e(G[U])$ falls in a given short interval (whose length depends only on $C$), but not the probability that $e(G[U])$ is equal to a particular value. To obtain such precise control, we make use of the \emph{switching} method, studying small perturbations to our random set $U$.

Roughly speaking, the switching method works as follows. To estimate the relative probabilities of events $\mc{A}$ and $\mc{B}$, one designs an appropriate ``switching'' operation that takes outcomes satisfying $\mc{A}$ to outcomes satisfying $\mc{B}$. One then obtains the desired estimate via upper and lower bounds on the number of ways to switch from an outcome satisfying $\mc{A}$, and the number of ways to switch to an outcome satisfying $\mc{B}$. This deceptively simple-sounding method has been enormously influential in combinatorial enumeration and the study of discrete random structures, and a variety of more sophisticated variations (considering more than two events) have been considered; see \cite{HM10,FM07} and the references therein.

In our particular situation (where we are switching between different possibilities of the set $U$), it does not seem to be possible to define a simple switching operation which has a controllable effect on $e(G[U])$ and for which we can obtain uniform upper and lower bounds on the number of ways to perform a switch. Instead, we introduce
an \emph{averaged} version of the switching method. Roughly speaking,
we define random variables that measure the number of ways to switch
between two classes, and study certain moments of these random variables. We believe this idea may have other applications.

\subsection{Notation}\label{sub:notation}
We use standard asymptotic notation throughout, as follows. For functions $f=f(n)$ and $g=g(n)$, we write $f=O(g)$ or $f \lesssim g$ to mean that there is a constant $C$ such that $|f(n)|\le C|g(n)|$ for sufficiently large $n$. Similarly, we write $f=\Omega(g)$ or $f \gtrsim g$ to mean that there is a constant $c>0$ such that $f(n)\ge c|g(n)|$ for sufficiently large $n$. Finally, we write $f\asymp g$ or $f=\Theta(g)$ to mean that $f\lesssim g$ and $g\lesssim f$, and we write $f=o(g)$ or $g=\omega(f)$ to mean that $f(n)/g(n)\to0$ as $n\to\infty$. Subscripts on asymptotic notation indicate quantities that should be treated as constants.

We also use standard graph-theoretic notation. In particular, $V(G)$ and $E(G)$ denote the vertex set of a graph $G$, and $e(G)=|E(G)|$ denotes the numbers of vertices and edges. We write $G[U]$ to denote the subgraph \emph{induced} by a set of vertices $U\su V(G)$. For a vertex $v\in V(G)$, its neighborhood (i.e., the set of vertices adjacent to $v$) is denoted by $N_G(v)$, and its degree is denoted $\deg_G(v)=|N_G(v)|$ (the subscript $G$ will be omitted when it is clear from context). 
We also write $N_U(v)=U\cap N(v)$ and $\deg_U(v)=|N_U(v)|$ to denote the degree of $v$ into a vertex set $U$.

Regarding probabilistic notation, we write $\mc N(\mu,\sigma^2)$ for the Gaussian distribution with mean $\mu$ and variance $\sigma^2$. As usual, we call a random variable with distribution $\mc N(0,1)$ a \emph{standard Gaussian} and we write $\mc N(0,1)^{\otimes n}$ for the distribution of a sequence of $n$ independent standard Gaussian variables. For a real random variable $X$, we write $\varphi_X\colon t\mapsto\mb{E}e^{itX}$ for the characteristic function of $X$. Though less standard, it is also convenient to write $\sigma(X) = \sqrt{\on{Var}X}$ for the standard deviation of $X$.

We also collect some miscellaneous bits of notation. We use notation like $\vec{x}$ to denote (column) vectors, and write $\vec x_I$ for the restriction of a vector $\vec x$ to the set $I$. We also write $M[I\!\times\!J]$ to denote the $I\times J$ submatrix of a matrix $M$. For $r\in\mb{R}$, we write $\snorm{r}_{\mb{R}/\mb{Z}}$ to denote the distance of $r$ to the closest integer, and for an integer $n\in\mb{N}$, we write $[n]=\{1,\ldots,n\}$. 
All logarithms in this paper without an explicit base are to base $e$, and the set of natural numbers $\mb N$ includes zero.

\subsection{Acknowledgments}
We thank Jacob Fox for comments motivating the inclusion of \cref{rem:expander-works}, and Zach Hunter for pointing out several minor corrections to the manuscript. We also thank the two anonymous referees for carefully reading the manuscript and many helpful comments.

\section{Short deductions}\label{sec:deductions}

We now present the short deductions of \cref{thm:erdos-mckay,thm:KST} from \cref{thm:ramsey-sample}.

\begin{proof}[Proof of \cref{thm:erdos-mckay} assuming \cref{thm:ramsey-sample}]
As mentioned in the introduction, Alon, Krivelevich, and Sudakov~\cite[Theorem~1.1]{AKS03} proved that there is some $\alpha=\alpha(C)>0$ such that the conclusion of \cref{thm:erdos-mckay} holds for all $0\leq x\leq n^{\alpha}$.

Fix $0<\lambda<1/2$ with $(1-\lambda)^2\ge 1-\eta$ and let $p=1-\lambda$. It now suffices to prove the desired statement for $n^{\alpha}\leq x\leq p^2 e(G)$, so consider such an integer $x$. Let us identify the vertex set of $G$ with $\{1,\ldots,n\}$. We can find some $m\in\{1,\ldots,n\}$ such that $e(G[\{1,\ldots,m\}])\geq x/p^2\geq e(G[\{1,\ldots,m-1\}])$. Let $G'$ denote the induced subgraph of $G$ on the vertex set $\{1,\ldots,m\}$ and note that
\[e(G')\geq x/p^2\geq e(G[\{1,\ldots,m-1\}])\geq e(G')-m.\]
Hence $|x-p^2e(G')|\leq p^2m\leq m^{3/2}$. As $m^2\geq e(G')\geq x/p^2\geq n^{\alpha}$, we have $m\geq n^{\alpha/2}$ and therefore $G'$ is a $(2C/\alpha)$-Ramsey graph. Thus, for a random subset $U$ of $V(G')=\{1,\ldots,m\}$ that includes each vertex of $G'$ with probability $p$, by \cref{thm:ramsey-sample} (with $A=1$) we have $e(G[U])=e(G'[U])=x$ with probability $\Omega_{C,\lambda}(m^{-3/2})$. In particular, if $n$ and therefore $m$ is sufficiently large with respect to $C,\lambda$, then there exists a subset $U\su V(G')\su V(G)$ with $e(G[U])=e(G'[U])=x$.
\end{proof}

\begin{proof}[Proof of \cref{thm:KST} assuming \cref{thm:ramsey-sample}]
We may assume that $n$ is sufficiently large with respect to $C$ and $\lambda$ (noting that the statement is trivially true for $n\le K$). Let $U$ be a random subset of $V(G)$ obtained by including each vertex with probability $k/n$ independently (recalling that \cref{thm:KST} concerns a random set $W$ of exactly $k$ vertices). A direct computation using Stirling's formula shows that $\Pr[|U|=k]\gtrsim_\lambda 1/\sqrt n$, so for each $x\in\mb Z$, \cref{thm:ramsey-sample} yields
\[\Pr[e(G[W])=x]=\Pr\Big[e(G[U])=x\Big||U|=k\Big]\le \frac{\Pr[e(G[U])=x]}{\Pr[|U|=k]}\lesssim_{C,\lambda} \frac1n.\qedhere\]
\end{proof}

It turns out that in order to prove \cref{thm:ramsey-sample}, it essentially suffices to consider the case $p=1/2$, as long as we permit some ``linear terms''. Specifically, instead of considering  random variable $e(G[U])$ we need to consider a random variable of the form $X = e(G[U]) + \sum_{v\in U}e_v+e_0$, as in the following theorem\footnote{As suggested by one of the anonymous referees, it could also be of interest to consider the case where $e_v$ is allowed to be negative (say $|e_v|\le Hn$). In this generality, we can no longer hope for upper bounds of order $n^{-3/2}$, but it should be possible to adjust the methods in this paper to prove a variation of \cref{thm:point-control}.}.

\begin{theorem}\label{thm:point-control}
Fix $C,H> 0$. Let $G$ be a $C$-Ramsey graph with $n$ vertices, and consider $e_0\in\mb Z$ and a vector $\vec{e}\in\mb{Z}^{V(G)}$ with $0\le e_v\le Hn$ for all $v\in V(G)$. Let $U\subseteq V(G)$ be a random vertex subset obtained by including each vertex with probability $1/2$ independently, and let $X = e(G[U]) + \sum_{v\in U}e_v+e_0$. Then 
\[\sup_{x\in\mb{Z}}\Pr[X = x] \lesssim_{C,H} n^{-3/2}\]
and for every fixed $A > 0$,
\[\inf_{\substack{x\in\mb{Z}\\|x-\mb E X|\le An^{3/2}}}\Pr[X=x]\gtrsim_{C,H,A}n^{-3/2}.\]
\end{theorem}

This theorem implies \cref{thm:ramsey-sample} (which also allows for a sampling probability $p\ne 1/2$), as we show next. The rest of the paper will be devoted to proving \cref{thm:point-control}.

\begin{proof}[Proof of \cref{thm:ramsey-sample} assuming \cref{thm:point-control}]
We may assume that $n$ is sufficiently large with respect to $C$ and $\lambda$. We proceed slightly differently depending on whether $p\le 1/2$ or $p>1/2$.

\medskip
\noindent\textit{Case 1: $p\le 1/2$. } In this case, we can realize the distribution of $U$ by first taking a random subset $U_0$ in which every vertex is present with probability $2p$, and then considering a random subset $U\subseteq U_0$ in which every vertex in $U_0$ is present with probability $1/2$. By a Chernoff bound, we have $|U_0|\ge pn\ge \lambda n$ with probability $1-o_\lambda(n^{-3/2})$, in which case $G[U_0]$ is a $(2C)$-Ramsey graph. We may thus condition on such an outcome of $U_0$. By \cref{thm:point-control}, the conditional probability of the event $X=x$ is at most $O_{C}(|U_0|^{-3/2})\lesssim_{C,\lambda}n^{-3/2}$, proving the desired upper bound.

For the lower bound, first note that $e(G[U_0])$ has expectation $(2p)^2e(G)$ and variance $\sigma(e(G[U_0]))^2= \sum_{uv,wz\in E(G)}\mb{E}[(\mbm{1}_{u,v\in U_0}-(2p)^2)(\mbm{1}_{w,z\in U_0}-(2p)^2)] \le n^{3}$ (note that there are at most $n^3$ non-zero summands, since the summands for distinct $u,v,w,z$ are zero). Hence by Chebyshev's inequality and a Chernoff bound, with probability at least $1/2$ the outcome of $U_0$ satisfies $|e(G[U_0])-(2p)^2e(G)|\le 2n^{3/2}$ and $|U_0|\ge \lambda n$. Conditioning on such an outcome of $U_0$, the lower bound in \cref{thm:ramsey-sample} follows from the lower bound in \cref{thm:point-control} applied to $G[U_0]$ (noting that $x\in\mb{Z}$ with $|x-p^2e(G)|\le An^{3/2}$ differs from $\mb{E}[e(G[U])|U_0] = e(G[U_0])/4$ by at most $(A+1)n^{3/2}\le (A+1)/\lambda^3\cdot |U_0|^{3/2}$).

\medskip
\noindent\textit{Case 2: $p> 1/2$. } In this case, we can realize the distribution of $U$ by first taking a random subset $U_0$ in which every vertex is present with probability $2p-1$ and then considering a random superset $U\supseteq U_0$ in which every vertex outside $U_0$ is present with probability $1/2$.

By a Chernoff bound, we have $|V(G)\setminus U_0|\ge (1-p)n\ge \lambda n$ with probability $1-o_\lambda(n^{-3/2})$, in which case $G[V(G)\setminus U_0]$ is a $(2C)$-Ramsey graph. Conditioning on such an outcome of $U_0$, the upper bound in \cref{thm:ramsey-sample} follows from the upper bound in \cref{thm:point-control} applied to $G[V(G)\setminus U_0]$ (where now we take $e_0=e(G[U_0])$ and $e_v=\deg_{U_0}(v)$ for each $v\in V(G)\setminus U_0$ and $H=1/\lambda$).

For the lower bound, observe that $\mb{E}[e(G[U])|U_0]=e(G[U_0])+e(V(G)\setminus U_0,U_0)/2+e(G[V(G)\setminus U_0])/4$ has expectation $\mb{E}e(G[U])=p^2e(G)$ and variance at most $n^3$ (by a similar calculation as in Case 1). Thus, by Chebyshev's inequality and a Chernoff bound with probability at least $1/2$ the outcome of $U_0$ satisfies $|\mb{E}[e(G[U])|U_0]-p^2e(G)|\le 2n^{3/2}$ and $|V(G)\setminus U_0|\ge \lambda n$. Conditioning on such an outcome of $U_0$, the lower bound in \cref{thm:ramsey-sample} follows from the lower bound in \cref{thm:point-control} applied to $G[V(G)\setminus U_0]$ (again taking $e_0=e(G[U_0])$ and $e_v=\deg_{U_0}(v)$ for each $v\in V(G)\setminus U_0$ and $H=1/\lambda$ and observing that $|x-\mb{E}[e(G[U])|U_0]|\le (A+2)/\lambda^3\cdot |V(G)\setminus U_0|^{3/2}$).
\end{proof}

\section{Proof discussion and outline}\label{sec:outline}
In the previous section, we saw how all of our results stated in the introduction follow from \cref{thm:point-control}. Here we discuss the high-level ideas of the proof of \cref{thm:point-control}, and the obstacles that must be overcome. Afterwards, we will outline the organization of the rest of the paper.

\subsection{Central limit theorems at multiple scales}\label{subsec:multiple-scales}
As mentioned in the introduction, our starting point is the possibility that a \emph{local central limit theorem} might hold for the random variable
$X=e(G[U])+\sum_{v\in U} e_v+e_0$ in \cref{thm:point-control}. However, some further thought reveals that such a theorem cannot hold in general. To appreciate this, it is illuminating
to rewrite $X$ in the so-called \emph{Fourier--Walsh} basis: define $\vec{x}\in\{-1,1\}^{V(G)}$ by taking $x_{v}=1$ if
$v\in U$, and $x_{v}=-1$ if $v\notin U$. Then, we have
\begin{equation}\label{eq:fourier-walsh}
X = \mb{E}X + \frac{1}{2}\sum_{v\in V(G)}\left(e_{v}+\frac{1}{2}\deg_{G}(v)\right)x_{v} + \frac{1}{4}\sum_{uv\in E(G)}x_{u}x_{v}.
\end{equation}
Writing
$L=\frac{1}{2}\sum_{v\in V(G)}\left(e_{v}+\frac{1}{2}\deg_{G}(v)\right)x_{v}$ and $Q=\frac{1}{4}\sum_{uv\in E(G)}x_{u}x_{v}$, we have $X=\mb{E}X+L+Q$. Essentially, we have isolated the ``linear part'' $L$ and the ``quadratic
part'' $Q$ of the random variable $X$, in such a way that the covariance between $L$ and $Q$ is
zero. It turns out that $L$ typically dominates the large-scale behavior
of $X$: the variance of $L$ is always of order $n^{3}$, whereas
the variance of $Q$ is only of order $n^{2}$. It is easy to show
that $L$ satisfies a central limit theorem (being a sum of independent
random variables). However, this central limit theorem may break down
at small scales: for example, it is possible that in $G$, every vertex has degree exactly $n/2$, in which case (for $\vec{e}=\vec 0$) the linear part $L$ only takes values in the lattice $(n/8)\mathbb{Z}$.

In this $(n/2)$-regular case (with $\vec{e}=\vec 0$), we might hope to prove \cref{thm:point-control} in two stages: having shown
that $L$ satisfies a central limit theorem, we might hope to show
that $Q$ satisfies a local central limit theorem after conditioning
on an outcome of $L$ (in this case, revealing $L$ only reveals the number of vertices in our
random set $U$, so there is still plenty of randomness remaining
for $Q$).

If this strategy were to succeed, it would reveal that in this case the true distribution
of $X$ is Gaussian on two different scales: when ``zoomed out'',
we see a bell curve with standard deviation about $n^{3/2}$, but
``zooming in'' reveals a superposition of many smaller bell curves each
with standard deviation about $n$ (see \cref{fig:theta}). This kind of
behavior can be described in terms of a so-called \emph{Jacobi theta function}, and has
been observed in combinatorial settings before (by the second and fourth authors~\cite{SS20}).


\subsection{An additive structure dichotomy}\label{sub:additive-structure-dichotomy}

There are a few problems with the above plan. When $G$ is regular,
we have the very special property that revealing $L$ only reveals
the number of vertices in $U$ (after which $U$ is a uniformly random
vertex set of this revealed size). There are many available tools
to study random sets of fixed size (this setting is often called the
``Boolean slice''). However, in general, revealing $L$ may result
in a very complicated conditional distribution.

We handle this issue via an additive structure
dichotomy, using the notion of \emph{regularized least common denominator
}(RLCD) introduced by Vershynin~\cite{Ver14} in the context of random matrix theory (a ``robust version'' of the notion of \emph{essential LCD} previously introduced by Rudelson and Vershynin~\cite{RV08}). Roughly speaking, we consider the RLCD of the degree sequence
of $G$. If this RLCD is small, then the degree sequence is ``additively
structured'' (as in our $(n/2)$-regular example), which (as we prove in \cref{lem:decomposition-abstract}) has the consequence
that the vertices of $G$ can be divided into a small number of ``buckets'' of vertices which have roughly the same coefficient in $L$ (i.e.~the values of $e_v+\deg_G(v)/2$ are roughly the same).
This means that conditioning on the number of vertices of $U$ inside
each bucket is tantamount to conditioning on the approximate value of $L$ (crucially, this conditioning dramatically reduces the variance), while
the resulting conditional distribution
is tractable to analyze.

On the other hand, if the RLCD is large, then the degree sequence
is ``additively unstructured'', and the linear part $L$ is well-mixing
(satisfying a central limit theorem at scales polynomially smaller than $n$).
In this case, it essentially is possible\footnote{Strictly speaking, we do not quite obtain an estimate for point probabilities,
but only for probabilities that $X$ falls in very short intervals (the length of the interval we can control depends on the distance
from the mean and the desired multiplicative error). Throughout this outline, we use the term ``local limit theorem'' in a rather imprecise way.}
to prove a local central limit theorem for $X$ (this is the easier
of the two cases of the additive structure dichotomy). Concretely, an example of this case is when $G$
is a typical outcome of an inhomogeneous random graph on the vertex
set $\{m/4,\ldots,3m/4\}$, where each edge $ij$ is present with probability
$i\cdot j/m^{2}$ independently.

\subsection{Breakdown of Gaussian behavior}
Recall from the previous subsection that in the ``additively structured'' case, we study the distribution of $e(G[U])$ after conditioning on the sizes of the intersections of $U$ with our ``buckets'' of vertices (which, morally speaking, corresponds to ``conditioning on the approximate value of $L$''). It turns out that even after this conditioning, a local central limit theorem may \emph{still} fail to hold, in quite a dramatic way: it can happen that, conditionally, no central limit theorem holds at all (meaning
that when we ``zoom in'' we do not see bell curves but some completely
different shapes).
For example, if $G$ is a typical outcome of two independent disjoint
copies of the Erd\H{o}s--R\'enyi random graph $\mathbb{G}(n/2,1/2)$, then one may think of all vertices being in the same bucket, and one can show
that the limiting distribution of
$e(G[U])$ conditioned on the event $|U|=n/2$ (up to translation and scaling) is\footnote{Heuristically, the $Z_1^2$ term can be explained as follows. Conditioning on  $|U|=n/2$, the number $s$ of vertices of $U$ on the left side (i.e.~in the left copy of $\mathbb{G}(n/2,1/2)$) is hypergeometrically distributed, and approaches a limiting distribution of $n/4+(\sqrt{n}/4)Z_1$. The number of pairs of vertices in $U$ on the same side of $G$ is roughly $(s^2 +(n/2-s)^2)/2=n^2/4+(s-n/4)^2$, and so it is distributed like $n^2/4+(n/16)Z_1^2$. The linear term involving $Z_2$ comes from the random distribution of the edges in the two copies of $\mathbb{G}(n/2,1/2)$.} that of $Z_{1}^{2}+2\sqrt{3}Z_{2}$, where $Z_{1},Z_{2}$ are
independent standard Gaussian random variables (see \cref{fig:wrong-shape}). 

\captionsetup{width=\linewidth}

\begin{figure}
    \centering
    \includegraphics[scale=0.15,trim={100px 20px 100px 20px},clip]{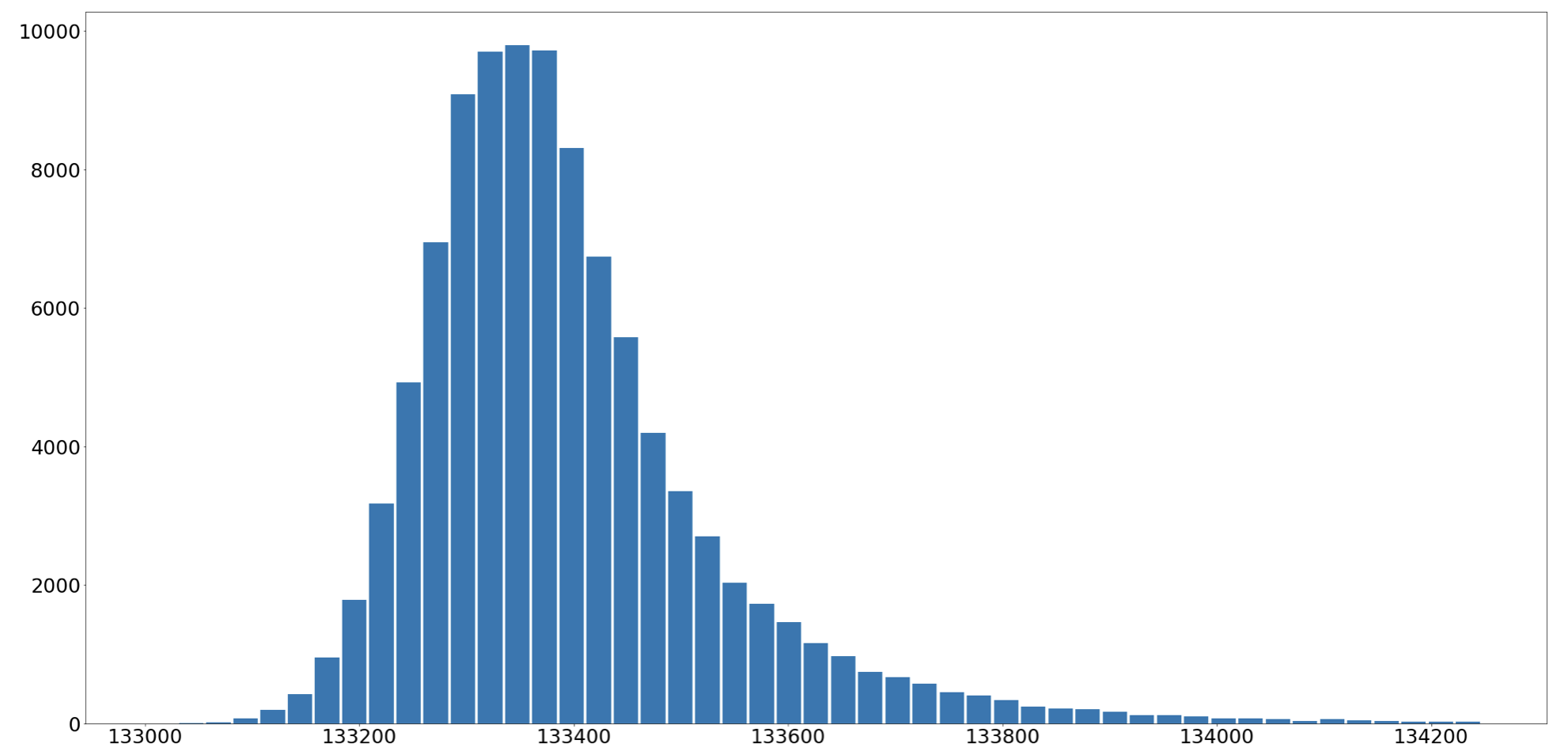}$\qquad$
    \includegraphics[scale=0.25,trim={0 10px 0 0},clip]{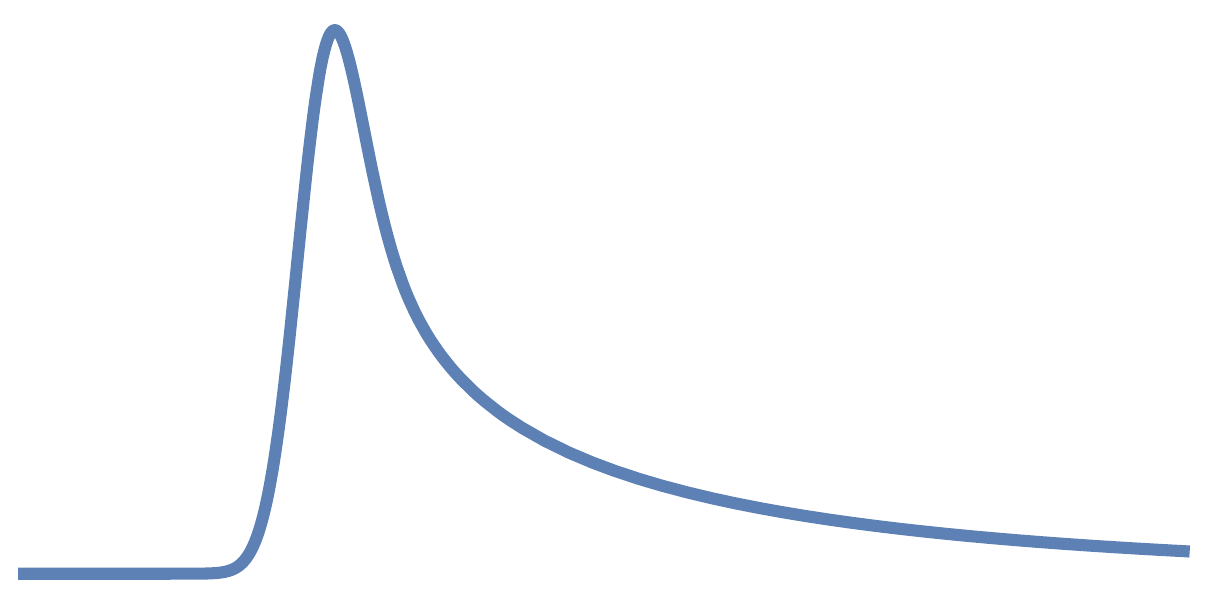}
    \caption{On the left, we obtain $G$ as a disjoint union of two independent Erd\H os--R\'enyi random graphs $\mb G(800,0.96)$, and we consider 500000 independent samples of a uniformly random vertex subsets $U$ \emph{with exactly 800 vertices}. The resulting histogram for $e(G[U])$ may look approximately Gaussian, but closer inspection reveals asymmetry in the tails. This is not just an artifact of small numbers: the limiting distribution comes from a nontrivial quadratic polynomial of Gaussian random variables. Actually, it is possible for the skew to be much more exaggerated (the curve on the right shows one possibility for the limiting probability mass function of $e(G[U])$), but this is difficult to observe computationally, as this shape only really becomes visible for enormous graphs $G$.}
    \label{fig:wrong-shape}
\end{figure}

In general, one can use a \emph{Gaussian invariance principle}~\cite{MOO10,FM19,FKMW18} to show that the asymptotic conditional distribution
of $e(G[U])$ always corresponds to some quadratic polynomial of Gaussian
random variables (see also \cite{BMM,BDMM}); instead of proving a local central limit theorem, we need to prove some type of local
limit theorem for convergence to that distribution.

In order to prove a local limit theorem of this type, it is
necessary to ensure that the limiting distribution (some quadratic
polynomial of Gaussian random variables) is ``well-behaved''. This
is where the tools discussed in \cref{subsub:gaussian-small-ball,subsub:ramsey-rank} come in: we prove that adjacency
matrices of Ramsey graphs robustly have high rank,
then apply certain variations of \cref{thm:gaussian-anticoncentration-upper-bound}.

\subsection{Controlling the characteristic function}
We are now left with the task of actually proving the necessary local
limit theorems. For this, we work in Fourier space, studying the characteristic
functions $\varphi_{Y}\colon\tau\mapsto\mathbb{E}e^{i\tau Y}$ of certain
random variables $Y$ (namely, we need to consider both the random
variable $X=e(G[U])+\sum_{v\in U} e_v+e_0$ and certain conditional random variables arising
in the additively structured case). Our aim is to compare $Y$
to an approximating random variable $Z$ (where $Z$ is either a Gaussian
random variable or some quadratic polynomial of Gaussian random variables).
This amounts to proving a suitable upper bound on $|\varphi_{Y}(\tau)-\varphi_{Z}(\tau)|$,
for as broad a range of $\tau$ as possible (if one wants to precisely
estimate point probabilities $\Pr[Y=x]$, it turns out that one needs to handle all $\tau$ in the range $[-\pi,\pi]$). We use different techniques for different ranges of $\tau\in\mathbb{R}$.

In the regime where $\tau$ is very small (e.g., when $|\tau|\le n^{0.01}/\sigma(Y)$),
$\varphi_{Y}(\tau)$ controls the large-scale distribution of $Y$,
so depending on the setting we either employ standard techniques
for proving central limit theorems, or a Gaussian invariance principle.

For larger $\tau$, it will be easy to show that our approximating characteristic function $\varphi_{Z}(\tau)$
is exponentially small in absolute value, so estimating $|\varphi_{Y}(\tau)-\varphi_{Z}(\tau)|$
amounts to proving an upper bound on $|\varphi_{Y}(\tau)|$, exploiting
cancellation in $\mathbb{E}e^{i\tau Y}$ as $e^{i\tau Y}$ varies
around the unit circle. Depending on the value of $\tau$, we are
able to exploit cancellation from either the ``linear'' or the ``quadratic''
part of $Y$.

To exploit cancellation from the linear part, we adapt a decorrelation technique first
introduced by Berkowitz~\cite{Ber18} to study clique counts in random graphs (see also \cite{SS20}), involving a subsampling argument and a Taylor
expansion. While all previous applications of this technique exploited the particular symmetries and combinatorial structure of a specific polynomial of interest, here we instead take advantage of the robustness inherent in the definition of RLCD. We hope that these types of ideas will be applicable to the study of even more general types of polynomials.

To exploit cancellation from the quadratic part, we use the method of \emph{decoupling},
building on arguments of the first and third authors~\cite{KS20b}. Our improvements involve taking advantage of Fourier cancellation ``on multiple scales'', which requires a sharpening of arguments of the first author and Sudakov~\cite{KS20} (building on work
of Bukh and Sudakov~\cite{BS07}) concerning ``richness'' of Ramsey graphs.

The relevant ideas for all the Fourier-analytic estimates discussed in this subsection will be discussed in more detail in the appropriate
sections of the paper (\cref{sec:initial-fourier,sec:high-fourier}).

\subsection{Pointwise control via switching}\label{subsec:switching}

Unfortunately, it seems to be extremely difficult to study the cancellations
in $\varphi_{X}(\tau)$ for very large $\tau$, and we are only able
to control the range where $|\tau|\le\nu$ for some small constant $\nu=\nu(C)$
(recalling that $G$ is $C$-Ramsey). As a consequence, the above
ideas only prove the following weakening of \cref{thm:point-control} (where we control the probability of $X$ lying in a constant-length interval instead of being equal to a particular value).

\begin{theorem}\label{thm:short-interval}
Fix $C > 0$. There is $B = B(C) > 0$ so the following holds for any fixed $H>0$. Let $G$ be an $C$-Ramsey graph with $n$ vertices, and consider $e_0\in\mb{R}$ and a vector $\vec{e}\in\mb{R}^{V(G)}$ with $0\le e_v\le Hn$ for all $v\in V(G)$. Let $U\subseteq V(G)$ be a random vertex subset obtained by including each vertex with probability $1/2$ independently, and let $X = e(G[U]) + \sum_{v\in U}e_v +e_0$. Then 
\[\sup_{x\in\mb Z}\Pr[|X-x|\le B]\lesssim_{C,H}n^{-3/2},\]
and for every fixed $A>0$,
\[\inf_{\substack{x\in\mb{Z}\\|x-\mb E X|\le An^{3/2}}}\Pr[|X-x|\le B]\gtrsim_{C,H,A}n^{-3/2}.\]
\end{theorem}

\cref{thm:short-interval} already implies the upper bound in   \cref{thm:point-control}, but not the lower bound. In  \cref{sec:switch}, we deduce the desired lower bound on point probabilities from \cref{thm:short-interval} (interestingly, this deduction requires both the lower and the upper bound in \cref{thm:short-interval}). As mentioned in the introduction, for this deduction, we introduce an ``averaged'' version of the so-called \emph{switching method}. In particular, for $\ell\in\{-B,\ldots,B\}$, we consider the pairs of vertices $(y,z)$ with  $y\in U$ and $z\notin U$ such that modifying $U$ by removing $y$ and adding
$z$ (a ``switch'') increases $e(G[U])$ by exactly
$\ell$. We define random variables that measure the number of ways to perform such switches, and deduce \cref{thm:point-control} by studying certain moments of these random variables. Here we again need to use some arguments involving ``richness'' of Ramsey graphs, and we also make use of the technique of \emph{dependent random choice}.


\subsection{Technical issues}

The above subsections describe the high-level ideas of the proof, but there are various technical issues that arise, some of which have a substantial impact on the complexity of the proof. Most importantly, in the additively structured case, we outlined how to prove a conditional local limit theorem for the quadratic part $Q$, but we completely swept under the rug how to then ``integrate'' this over outcomes of the conditioning. Explicitly, if we encode the bucket intersection sizes in a vector $\vec \Delta$, we have outlined how to use Fourier-analytic techniques to prove certain estimates on conditional probabilities of the form $\Pr[|X-x|\le B|\vec \Delta]$, but we then need to average over the randomness of $\vec \Delta$ to obtain $\Pr\left[|X-x|\le B\right]=\mb E[\Pr[|X-x|\le B| \vec{\Delta}]]$ (taking into account that certain outcomes of $\vec \Delta$ give a much larger contribution than others).

There are certain relatively simple arguments with which we can accomplish this averaging while losing logarithmic factors in the final probability bound (namely, using a concentration inequality for $Q$ conditioned on $\vec{\Delta}$, we can restrict to only a certain range of outcomes $\vec{\Delta}$ which give a significant contribution to the overall probability $\Pr[|X-x|\le B]$).
To avoid logarithmic losses, we need to make sure that our conditional probability bounds ``decay away from the mean'', which requires a non-uniform version of \cref{thm:gaussian-anticoncentration-upper-bound} (with a decay term), and some specialized tools for converting control of $|\varphi_{Y}(\tau)-\varphi_{Z}(\tau)|$ into bounds on small-ball probabilities for $Y$. Also, we need some delicate moment estimates comparing dependent random variables of ``linear'' and ``quadratic'' types, to quantify the dependence between certain fluctuations in the conditional mean and variance as we vary $\vec \Delta$.

Furthermore, for the switching argument described in the previous subsection, it is important (for technical reasons discussed in \cref{rem:dependence-B}) that in the setting of \cref{thm:short-interval}, $B$ does not depend on $A$ and $H$. To achieve this, we develop Fourier-analytic tools that take into account ``local smoothness'' properties of the approximating random variable $Z$.

\subsection{Organization of the paper}
In \cref{sec:preliminaries} we collect a variety of (mostly known) tools which will be used throughout the paper. Then, in \cref{sec:anti-Gauss} we prove \cref{thm:gaussian-anticoncentration-upper-bound} (our sharp small-ball probability estimate for quadratic polynomials of Gaussians), and in \cref{sec:fourier-analysis} we prove some general ``relative'' Esseen-type inequalities deducing bounds on small-ball probabilities from Fourier control.

In \cref{sec:initial-fourier,sec:high-fourier} we obtain bounds on the characteristic function $\varphi_X(\tau)$ for various ranges of $\tau$ (specifically, bounds due to ``cancellation of the linear part'' appear in \cref{sec:initial-fourier}, and bounds due to ``cancellation of the quadratic part'' appear in \cref{sec:high-fourier}). This is already enough to handle the additively unstructured case of \cref{thm:short-interval}, which we do in \cref{sec:short-interval-unstructured}.

Most of the rest of the paper is then devoted to the additively structured case of \cref{thm:short-interval}. In \cref{sec:ramsey-robust-rank} we study the ``robust rank'' of Ramsey graphs, and in \cref{sec:coupling-lemmas} we prove some lemmas about quadratic polynomials on products of Boolean slices. All the ingredients collected so far come together in \cref{sec:short-interval-structured}, where the additively structured case of \cref{thm:short-interval} is proved.

Finally, in \cref{sec:switch} we use a switching argument to deduce \cref{thm:point-control} from \cref{thm:short-interval}.

\section{Preliminaries}\label{sec:preliminaries}
In this section we collect some basic definitions and tools that will be used throughout the paper.

\subsection{Basic facts about Ramsey graphs}\label{sub:ramsey-graphs}
First, as mentioned in the introduction, the following classical result about Ramsey graphs is due to Erd{\H{o}}s and Szemer\'{e}di \cite{ES72}.

\begin{theorem}\label{thm:dense-ramsey}
For any $C$, there exists $\eps=\eps(C)>0$ such that for every sufficiently large $n$, every $C$-Ramsey graph $G$ on $n$ vertices satisfies $\eps \binom n 2\leq e(G)\leq (1-\eps)\binom n 2$.
\end{theorem}

\begin{remark}\label{rem:expander-dense}
In the setting of \cref{rem:expander-works}, where $G$ has near-optimal spectral expansion, the expander mixing lemma (see for example \cite[Corollary~9.2.5]{AS16}) implies that (for sufficiently large $n$) all subsets of $G$ with at least $n^{1/2+0.02}$ vertices have density very close to the overall density of $G$. It is possible to use this fact in lieu of \cref{thm:dense-ramsey} in our proof of \cref{thm:point-control}.
\end{remark}

More recently, building on work of Bukh and Sudakov~\cite{BS07}, the first author and Sudakov~\cite{KS20} proved that every Ramsey graph contains an induced subgraph in which the collection of vertex-neighborhoods is ``rich''. Intuitively speaking, the richness condition here means that for all linear-size vertex subsets $W$, there are only very few vertex-neighborhoods that have an unusually large or unusually small intersection with $W$.

\begin{definition}\label{def:rich}
Consider $\delta,\rho,\alpha>0$. We say that an $m$-vertex graph $G$ is \emph{$(\delta,\rho,\alpha)$-rich} if for every subset $W\su V(G)$ of size $|W|\geq \delta m$, there are at most $m^\alpha$ vertices $v\in V(G)$ with the property that $|N(v)\cap W|\leq \rho |W|$ or $|W\setminus N(v)|\leq \rho |W|$.
\end{definition}

When the parameter $\alpha$ is omitted, it is assumed to take the value $1/5$. That is to say, we write ``$(\delta,\rho)$-rich'' to mean ``$(\delta,\rho,1/5)$-rich''.

The following lemma is a slight generalization of \cite[Lemma 4]{KS20} (and is proved in the same way).

\begin{lemma}\label{lem:rich-subset}
For any fixed $C,\alpha>0$, there exists $\rho=\rho(C,\alpha)$ with $0<\rho<1$ such that the following holds. For $n$ sufficiently large in terms of $C$ and $\alpha$, for any $m\in\mb{R}$ with $\sqrt n\le m\leq \rho n$, and any $C$-Ramsey graph $G$ on $n$ vertices, there is an induced subgraph of $G$ on at least $m$ vertices which is $((m/n)^\rho,\rho,\alpha)$-rich.
\end{lemma}

For two disjoint vertex sets $U,W$ in a graph $G$, we write $e(U,W)$ for the number of edges between $U,W$ and write $d(U,W)=e(U,W)/(|U||W|)$ for the density between $U,W$. We furthermore write $d(U)=e(U)/\binom{|U|}{2}$ for the density inside the set $U$.

\begin{proof}
We introduce an additional parameter $K$, which will be chosen to be large in terms of $C$ and $\alpha$. We will then choose $\rho=\rho(C,\alpha)$ with $0<\rho<1$ to be small in terms of $K$, $C$, and $\alpha$. We do not specify the values of $K$ and $\rho$ ahead of time, but rather assume they are sufficiently large or small to satisfy certain inequalities that arise in the proof.

Let $\delta=(m/n)^\rho$ and suppose for the purpose of contradiction that every set of at least $m$ vertices fails to induce a $(\delta,\rho,\alpha)$-rich
subgraph. We will inductively construct a sequence of induced
subgraphs $G=G[U_{0}]\supseteq G[U_{1}]\supseteq\cdots\supseteq G[U_{K}]$
and disjoint vertex sets $S_{1},\ldots,S_{K}$ of size $|S_1|=\cdots=|S_K|=\lceil m^{\alpha}/2\rceil$ such that for each $i=1,\ldots,K$, we have
$|U_{i}|\ge(\delta/4)|U_{i-1}|$ and $S_{i}\subseteq U_{i-1}\setminus U_i$,
as well as
\[\big[e(S_{i},\{u\})\leq 4\rho\cdot |S_i|\text{ for all }u\in U_i\big]\text{ or }\big[e(S_{i},\{u\})\geq (1-4\rho)\cdot |S_i|\text{ for all }u\in U_i\big].\label{eq:density-cases}\]
This will suffice, as follows. First note that for each $i=1,\ldots,K$, we have
\[\big[d(S_{i},S_{j})\leq 4\rho\text{ for all }j\in \{i+1,\ldots,K\}\big]\text{ or }\big[d(S_{i},S_{j})\geq 1-4\rho\text{ for all }j\in \{i+1,\ldots,K\}\big].\]
Without loss of generality suppose that the first case holds for at least
half of the indices $i=1,\ldots,K$, and let $S$ be the union of the corresponding sets $S_{i}$. Then one can compute $d(S)\leq 4\rho+1/K$. On the other hand $|S|\geq (K/2)\cdot m^{\alpha}/2\geq m^{\alpha}\geq n^{\alpha/2}$ and therefore $G[S]$ is a $(2C/\alpha)$-Ramsey graph. However, now the density bound $d(S)\leq 4\rho+1/K$ contradicts \cref{thm:dense-ramsey} if $\rho$ is sufficiently small and $K$ is sufficiently large (in terms of $C$ and $\alpha$).

Let $U_{0}=V(G)$. For $i=1,\ldots,K$ we will construct
the vertex sets $U_{i}$ and $S_{i}$, assuming that $U_{0},\ldots,U_{i-1}$ and $S_{1},\ldots,S_{i-1}$
have already been constructed. Note that we have $|U_{i-1}|\ge (\delta/4)^{i-1}n\ge  (\delta/4)^Kn=(m/n)^{\rho K}4^{-K}n\ge  m$, using that $\rho K\leq 1/3$ and $m/n\leq\rho\leq 8^{-K}$ for $\rho$ being sufficiently small with respect to $K$. Therefore, by our assumption, $U_{i-1}$ must contain a set $W$
of at least $\delta|U_{i-1}|$ vertices and a set $Y$
of more than $|U_{i-1}|^{\alpha}\geq m^{\alpha}$ vertices contradicting $(\delta,\rho,\alpha)$-richness.
Suppose without loss of generality that $|N(v)\cap W|\le\rho|W|$
for at least half of the vertices $v\in Y$, and let $S_{i}\su Y\su U_{i-1}$ be a set of precisely $\lceil m^{\alpha}/2\rceil$ such vertices $v\in Y$. Then, let $U=W\setminus S_{i}\su U_{i-1}\setminus S_i$ and note that we have  $|U|\ge|W|/2$ since $|W|\geq \delta|U_{i-1}|\geq 4 \cdot (\delta/4)^Kn\geq 4m\geq 2|S_{i}|$. Furthermore, let $U_{i}\subseteq U$ be the set of vertices $u\in U$ with
$e(S_{i}, \{ u\} )\le4\rho\cdot |S_i|$. Now, we
just need to show $|U_{i}|\ge(\delta/4)|U_{i-1}|$.
To this end, note that for all $v\in S_{i}$ we have $e(\{ v\} ,U)=|N(v)\cap U|\le |N(v)\cap W|\leq  \rho|W|\leq 2\rho |U|$.
Hence,
\[
|U\setminus U_{i}|\cdot 4\rho\cdot |S_i|\leq \sum_{w\in U\setminus U_{i}}e(S_{i},\{ w\})=e(S_{i},U\setminus U_{i} )\le e(S_{i},U)=\sum_{v\in S_{i}}e(\{ v\} ,U)\le |S_i|\cdot 2\rho|U|,
\]
implying that $|U\setminus U_i|\leq |U|/2$ and hence $|U_{i}|\geq |U|/2\ge |W|/4\geq (\delta/4)|U_{i-1}|$,
as desired.
\end{proof}

\begin{remark}\label{rem:expander-rich}
In the setting of \cref{rem:expander-works}, where $G$ is dense and has near-optimal spectral expansion (and $n$ is sufficiently large), the expander mixing lemma can be used to prove that every induced subgraph of $G$ on at least $n^{0.9}$ vertices is $(n^{-0.05},0.005,\alpha)$-rich (and therefore \cref{lem:rich-subset} holds) for $\alpha \ge 0.2$. It is possible to use this in lieu of \cref{lem:rich-subset} in our proof of \cref{thm:point-control}.
\end{remark}

\subsection{Characteristic functions and anticoncentration}\label{sub:characteristic-functions}
For a real random variable $X$, recall that the \emph{characteristic function} $\varphi_X\colon\mb{R}\to \mb C$ is defined by $\varphi_X(t)=\mb E [e^{itX}]$. Note that we have $|\varphi_X(t)|\le 1$ for all $t\in \mb{R}$. If $\varphi_X(t)$ is absolutely integrable, then $X$ has a continuous density $p_X$, which can be obtained by the \emph{inversion formula}
\begin{equation}\label{eq:inversion}
p_{X}(u)=\frac{1}{2\pi}\int_{-\infty}^\infty e^{-itu}\varphi_{X}(t)\,dt.
\end{equation}
Next, recall the \emph{L\'evy concentration function}, which measures the maximum small-ball probability.
\begin{definition}\label{def:levy-concentration}
For a real random variable $X$ and $\varepsilon\ge 0$, we define $\mc{L}(X,\eps) = \sup_{x\in\mb{R}}\Pr[|X-x|\le\eps]$.
\end{definition}

If $X$ has a density $p_X$, then we trivially have $\mc{L}(X,\eps)\le \eps\max_{x\in\mb{R}} p_X(x)$. We can also control small-ball probabilities using only a certain range of values of the characteristic function, via \emph{Esseen's inequality} (see for example \cite[Lemma~6.4]{Rud14}):
\begin{theorem}\label{thm:esseen}
There is $C_{\ref{thm:esseen}} > 0$ so that for any real random variable $X$ and any $\eps > 0$, we have
\[\mc L(X,\eps)\le C_{\ref{thm:esseen}}\cdot \eps \int_{-2/\varepsilon}^{2/\varepsilon}|\varphi_X(t)|\,dt.\]
\end{theorem}
In \cref{sec:fourier-analysis} we will prove some more sophisticated ``relative'' variants of \cref{thm:esseen}.

\subsection{Distance-to-integer estimates, and regularized least common denominator}\label{sub:RLCD}
For $r\in\mb{R}$, let $\|r\|_{\mb{R}/\mb Z}$ denote the (Euclidean) distance of $r$ to the nearest integer. Recall that the Rademacher distribution is the uniform distribution on the set $\{-1,1\}$. If $x$ is Rademacher-distributed, then for any $r\in\mb{R}$ we have the well-known estimate
\begin{equation}\label{eq:cos}
|\mb{E}[\exp(irx)]| = |\!\cos(r)|\le 1-\snorm{r/\pi}_{\mb{R}/\mb{Z}}^2\le  \exp(-\snorm{r/\pi}_{\mb{R}/\mb{Z}}^2).
\end{equation}
If $\vec \xi\in \{0,1\}^n$ is a uniformly random length-$n$ binary vector, then for any $\vec a\in\mb{R}^n$ and any $b\in\mb{R}$, we can rewrite $\vec a\cdot \vec \xi+b$ as a weighted sum of independent Rademacher random variables. Specifically, we have  $\vec a\cdot \vec \xi+b=\vec{r}\cdot \vec{x}+\mb E [\vec a\cdot \vec \xi+b]$, where $\vec{r}=\vec{a}/2 \in\mb{R}^n$ and $\vec{x}\in \{-1,1\}^n$ is obtained from $\vec \xi\in \{0,1\}^n$ by replacing all zeroes by $-1$'s. Then $\vec{x}$ is uniformly random in $\{-1,1\}^n$, so  \cref{eq:cos} yields
\begin{equation}\label{eq:cos-0-1}
|\mb{E}[\exp(i(\vec{a}\cdot\vec{\xi}+b))]|=|\mb{E}[\exp(i(\vec{r}\cdot \vec{x}))]|=\prod_{j=1}^{n} |\mb{E}[\exp(ir_jx_j)]|\leq \exp\Bigg(-\sum_{j=1}^{n}\snorm{a_j/(2\pi)}_{\mb{R}/\mb{Z}}^2\Bigg).
\end{equation}
In the case where we want to study $\vec a\cdot \vec x$ where $\vec x\in \{0,1\}^n$ is a uniformly random binary vector \emph{with a given number of ones} (i.e., a random vector on a \emph{Boolean slice}), one has the following estimate. 
\begin{lemma}\label{lem:slice-estimate}
Fix $c>0$. Let $\vec{a}\in\mb{R}^{n}$, and suppose that for some $0< \delta\le 1/2$ there are disjoint pairs
$\{i_{1},j_{1}\},\ldots,\{i_{M},j_{M}\}\subseteq [n]$ such that  $\|(a_{i_{k}}-a_{j_{k}})/(2\pi)\|_{\mb{R}/\mb Z}\ge\delta$ for each $k=1,\ldots,M$. Let $s$ be an integer with $cn\le s\le (1-c)n$. Then for a
random zero-one vector $\vec{\xi}\in\left\{ 0,1\right\}^{n}$ with exactly $s$ ones, we have
\[
|\mb E[\exp(i(\vec{a}\cdot\vec{\xi}))]|\lesssim \exp\left(-\Omega_c(M\delta^2)\right).
\]
\end{lemma}

\cref{lem:slice-estimate} can be deduced from \cite[Theorem~1.1]{Roo}. For the reader's convenience we include an alternative proof, reducing it to \cref{eq:cos-0-1}.
\begin{proof}
We may assume that $M\leq cn/4$ (indeed, noting that $M\leq n/2$ we can otherwise just replace $M$ by $\lfloor cn/4\rfloor$). The random vector $\vec{\xi}$ corresponds to a uniformly random subset $U\su [n]$ of size $s$. Let us first expose the intersection sizes $|U\cap \{i_1,j_1\}|, \ldots, |U\cap \{i_M,j_M\}|$, one at a time. For each $k=1,\ldots,M$ we have $|U\cap \{i_k,j_k\}|=1$ with probability at least $c(1-c)/4$ even when conditioning on any outcomes for the previously exposed intersection sizes $|U\cap \{i_1,j_1\}|, \ldots, |U\cap \{i_{k-1},j_{k-1}\}|$. Hence the number of indices $k\in [M]$ with $|U\cap \{i_k,j_k\}|=1$ stochastically dominates a binomial random variable with distribution $\mr{Bin}(M,c(1-c)/4)$. Thus, by a Chernoff bound (see e.g.~\cref{lem:chernoff}), with probability at least $1-\exp(-\Omega_c(M))$ there is a set $K\su[M]$ of at least $c(1-c)M/8$ indices $k$ with $|U\cap \{i_k,j_k\}|=1$. Let us expose and condition on all coordinates of $\vec{\xi}\in \{0,1\}^n$ except those in $\bigcup_{k\in K}\{i_k,j_k\}$. The only remaining randomness of the vector $\vec{\xi}\in \{0,1\}^n$ is that for each $k\in K$ we have either $\xi_{i_k}=1$ or $\xi_{j_k}=1$ (each with probability $1/2$, independently for all $k\in K$). Thus, after all of this conditioning, we have $\vec{a}\cdot \vec{\xi}=\sum_{k\in K} (a_{i_k}-a_{j_k})\xi_{i_k}+b$ for some $b\in\mb{R}$, where $(\xi_{i_k})_{k\in K}\in \{0,1\}^K$ is uniformly random. Thus, \cref{eq:cos-0-1} implies $|\mb E[\exp(i(\vec{a}\cdot\vec{\xi}))]|\leq \exp(-\sum_{k\in K} \|(a_{i_{k}}-a_{j_{k}})/(2\pi)\|_{\mb{R}/\mb Z}^2)\leq \exp(-\Omega_c(M\delta^2))$, as desired.
\end{proof}

The above estimates motivate the notion of the \emph{essential least common denominator} (LCD) of a vector $\vec v \in\mb{S}^{n-1}\subseteq \mb{R}^n$ (where $\mb{S}^{n-1}$ is the unit sphere in $\mb{R}^n$). The following formulation of this notion was proposed by Rudelson (see \cite[(1.17)]{Ver14} and the remarks preceding), in the context of random matrix theory.

\begin{definition}[LCD]\label{def:LCD}
For $t>0$, let $\log_+t = \max\{0,\log t\}$. For $L\ge 1$ and $\vec v \in\mb{S}^{n-1}\subseteq \mb{R}^n$, the (essential) least common denominator\footnote{To briefly explain the name ``LCD'', recall that the ordinary least common denominator of the entries of a rational vector $\vec v\in\mb{S}^{n-1}\cap \mb Q^n$ is $\inf\{\theta>0\colon\on{dist}(\theta \vec v, \mb{Z}^{n})=0\}$.} $D_L(\vec v)$ is defined as
\[D_L(\vec v) = \on{inf}\left\{\theta> 0: \on{dist}(\theta \vec v, \mb{Z}^{n}) < L\sqrt{\log_{+}(\theta /L)}\right\}.\]
(Here $\on{dist}(\theta \vec v, \mb{Z}^{n})=\sqrt{\sum_{i=1}^n \|\theta v_i\|_{\mb{R}/\mb Z}^2}$ denotes the Euclidean distance from $\theta \vec v$ to the nearest point in the integer lattice $\mb Z^n$.)
\end{definition}

The following lemma gives a lower bound on the LCD of a unit vector $\vec v$ in terms of $\snorm{\vec v}_\infty$. 
\begin{lemma}[{\cite[Lemma~6.2]{Ver14}}]\label{lem:LCD-init}
If $\vec v\in\mb{S}^{n-1}$ and $L\ge 1$, then
\[D_L(\vec v)\ge 1/(2\snorm{\vec v}_\infty).\]
\end{lemma}
\begin{proof}
Note that for $\theta \le 1/(2\snorm{\vec v}_{\infty})$ we have that $\snorm{\theta \vec{v}}_\infty\le 1/2$. Thus we have that 
\[\on{dist}(\theta \vec{v},\mb{Z}^{n}) = \on{dist}(\theta \vec{v},\vec{0}) = \theta >L\sqrt{\log_{+}(\theta/L)}\]
where we have used that $x>\sqrt{\log_{+}(x)}$ for $x > 0$. The result follows by the definition of LCD.
\end{proof}

If $D_L(\vec v)$ is large, then we can obtain strong control over the characteristic function of random variables of the form $\vec v\cdot \vec x$, for an i.i.d.~Rademacher vector $\vec x$ (specifically, we are able to compare such characteristic functions to the characteristic function $\varphi_Z(t)=e^{-t^2/2}$ of a standard Gaussian $Z\sim \mc N(0,1)$). However, if $D_L(\vec v)$ is small, then in a certain sense $\vec v$ is ``additively structured'', and we can deduce certain combinatorial consequences. Actually, to obtain the consequences we need, we will use the following more robust notion known as \emph{regularized LCD}, introduced by Vershynin~\cite{Ver14}.

\begin{definition}[regularized LCD]\label{def:RLCD}
Fix $L\ge 1$ and $0<\gamma<1$. For a vector $\vec v\in\mb{R}^n$ with fewer than $n^{1-\gamma}$ zero coordinates, the \emph{regularized least common denominator (RLCD)} $\wh{D}_{L,\gamma}(\vec v)$, is defined as
\[\wh{D}_{L,\gamma}(\vec v) = \max\{D_L(\vec v_I/\|\vec v_I\|_2)\colon |I|=\lceil n^{1-\gamma}\rceil\},\]
where $\vec{v}_I\in\mb{R}^I$ denotes the restriction of $\vec v$ to the indices in $I$.
\end{definition}

If a vector $\vec{d}$ is ``additively structured'' in the sense of having small RLCD, we can partition its index set into a small number of ``buckets'' such that the values of $d_i$ are similar inside each bucket. This is closely related to $\varepsilon$-net arguments using LCD assumptions that have previously appeared in the random matrix theory literature (see for example \cite[Lemma~7.2]{Rud14}).

\begin{lemma}\label{lem:decomposition-abstract}
Fix $H > 0$ and $0<\gamma<1/4$ and $L\ge 1$. Let $\vec{d}\in\mb{R}_{\ge 0}^{n}$ be a vector such that $\snorm{\vec{d}}_\infty\le Hn$ and $\snorm{\vec{d}_S}_2\geq n^{3/2-2\gamma}$ for every subset $S\su [n]$ of size $|S|=\lceil n^{1-\gamma}\rceil$, and assume that $n$ is sufficiently large with respect to $H$, $\gamma$ and $L$.

If $\wh{D}_{L,\gamma}(\vec{d}) \le n^{1/2}$, then there exists a partition $[n] = R\cup (I_1\cup\cdots\cup I_m)$ and real numbers $\kappa_1,\ldots,\kappa_m\ge 0$ with $|R|\le  n^{1-\gamma}$ and $|I_1|=\cdots=|I_m|=\lceil n^{1-2\gamma}\rceil$ such that for all $j=1,\ldots,m$ and $i\in I_j$ we have $|d_i - \kappa_j|\le n^{1/2+4\gamma}$.
\end{lemma}

\begin{proof}
Choose a partition $[n] = R\cup(I_1\cup\cdots\cup I_m)$ and $\kappa_j\geq 0$ for $j=1,\ldots,m$ with $|I_1|=\cdots=|I_m| = \lceil n^{1-2\gamma}\rceil$ such that $|d_i-\kappa_j|\le n^{1/2+4\gamma}$ for all $1\le j\le m$ and $i\in I_j$, such that $m$ is as large as possible. It then suffices to prove that $|R|\leq n^{1-\gamma}$.

So let us assume for contradiction that $|R|>n^{1-\gamma}$, and fix a subset $S\su R$ of size $|S|=\lceil n^{1-\gamma}\rceil$. Note that $D_L(\vec{d}_S/\snorm{\vec{d}_S}_2)\le \wh{D}_{L,\gamma}(\vec{d})\le n^{1/2}$ by \cref{def:RLCD}. Furthermore, since  $\snorm{\vec{d}_S/\snorm{\vec{d}_S}_2}_{\infty}\le Hn/n^{3/2-2\gamma}= Hn^{-1/2+2\gamma}$, \cref{lem:LCD-init} implies $D_L(\vec{d}_S/\snorm{\vec{d}_S}_2)\ge (H^{-1}/2) n^{1/2-2\gamma}$. Thus, by \cref{def:LCD}, there is some $\theta\in[(H^{-1}/2)n^{1/2-2\gamma},2n^{1/2}]$ such that
\begin{equation}\label{eq:lcd-discrepancy}
\snorm{(\theta/\snorm{\vec{d}_S}_2)\vec{d}_S-\vec{w}}_2\le L\sqrt{\log_+ (\theta/L)}\le L\sqrt{\log n}
\end{equation}
for some $\vec{w}\in\mb{Z}^{S}$. By choosing $\vec{w}$ to minimize the left-hand side, we may assume that $\vec{w}$ has nonnegative entries (recall that $\vec{d}$ has nonnegative entries).

Now, the number of indices $i\in S$ with $|(\theta/\snorm{\vec{d}_S}_2)d_i-w_i|> n^{-1/2+2\gamma}$ is at most
\[\frac{\snorm{(\theta/\snorm{\vec{d}_S}_2)\vec{d}_S-\vec{w}}_2^2}{n^{-1+4\gamma}}\leq \frac{L^2 \log n}{n^{-1+4\gamma}}\le n^{1-3\gamma}.\]
Furthermore, note that $\theta\leq 2n^{1/2}$ and \cref{eq:lcd-discrepancy} imply $\snorm{\vec{w}}_2\le 3n^{1/2}$, and hence the number of indices $i\in S$ with $w_i\geq n^{2\gamma/3}$ is at most $9n^{1-4\gamma/3}$. Thus, as $|S|=\lceil n^{1-\gamma}\rceil$, there must be at least $|S|/2\geq n^{1-\gamma}/2$ indices $i\in S$ with $|(\theta/\snorm{\vec{d}_S}_2)d_i-w_i|\leq n^{-1/2+2\gamma}$ and $w_i\in [0,n^{2\gamma/3}]\cap\mb{Z}$. Hence by the pigeonhole principle there is some $\kappa\geq 0$ and a subset $I_{m+1}\su S\su R$ of size $|I_{m+1}| = \lceil n^{1-2\gamma}\rceil $ such that for all $i\in I_{m+1}$ we have $w_i=\kappa$ and
\[|(\theta/\snorm{\vec{d}_S}_2)d_i-\kappa |=|(\theta/\snorm{\vec{d}_S}_2)d_i-w_i|\leq  n^{-1/2+2\gamma}= \frac{n^{1/2-2\gamma}}{n^{(1-\gamma)/2}n}\cdot n^{1/2+(7/2)\gamma}\lesssim_H \frac{\theta}{\snorm{\vec{d}_S}_2}\cdot n^{1/2+(7/2)\gamma}.\]
Defining $\kappa_{m+1}=(\snorm{\vec{d}_S}_2/\theta)\kappa\geq 0$, this implies $|d_i-\kappa_{m+1}|\leq n^{1/2+4\gamma}$ for all $i\in I_{m+1}$. But now the partition $V(G) = (R\setminus I_{m+1})\cup(I_1\cup\cdots\cup I_{m+1})$ contradicts the maximality of $m$.
\end{proof}

\subsection{Low-rank approximation}\label{sub:low-rank-approximation}
Recall the definition of the \emph{Frobenius norm} (also called the \emph{Hilbert--Schmidt norm}): for a matrix $M\in\mb{R}^{n\times n}$, we have
\[\|M\|_\mr{F}=\Big(\sum_{i,j=1}^n M_{ij}^2\Big)^{1/2}=\sqrt{\on{trace}(M^\intercal M)}.\] If $M$ is symmetric, then $\|M\|_{\mr F}^2$ is the sum of squares of the eigenvalues of $M$ (with multiplicity).

Famously, Eckart and Young~\cite{EY36} proved that for any real matrix $M$, the degree to which $M$ can be approximated by a low-rank matrix $\widetilde M$ can be described in terms of the spectrum of $M$. 
The following statement is specialized to the setting of real symmetric matrices.

\begin{theorem}\label{thm:eckart-young}
Consider a symmetric matrix $M\in\mb{R}^{n\times n}$, and let $\lambda_1,\ldots,\lambda_n$ be its eigenvalues. Then for any $r=0,\ldots,n$ we have
\[\min_{\substack{\widetilde M\in\mb{R}^{n\times n}\\\on{rank}(\widetilde M)\le r}}\|M-\widetilde M\|^2_{\mr F}=\min_{\substack{I\su [n]\\|I|=n-r}}\sum_{i\in I}\lambda_{i}^2,\]
where the minimum is over all (not necessarily symmetric\footnote{It is easy to show that there is always a symmetric matrix $\wt M$ which attains this minimum, though this will not be necessary for us.}) matrices $\widetilde M\in\mb{R}^{n\times n}$ with rank at most $r$.
\end{theorem}

\subsection{Analysis of Boolean functions}\label{sub:hypercontractivity}
In this subsection we collect some tools from the theory of Boolean functions. A thorough introduction to the subject can be found in \cite{O14}.

Consider a multilinear polynomial $f(x_1,\ldots,x_n) = \sum_{S\subseteq [n]} a_S\prod_{i\in S}x_i$. An easy computation shows that if $\vec x$ is a sequence of independent Rademacher or independent standard Gaussian random variables, then $\mb{E}[f(\vec x)]=a_\emptyset$ and
\begin{equation}\label{eq:boolean-variance}
\on{Var}[f(\vec x)]=\sum_{\emptyset\ne S\subseteq [n]} a_S^2.
\end{equation}
Thus, in the case $\deg f=2$, we can consider the contributions to the variance $\on{Var}[f(\vec x)]$ coming from the ``linear'' part and the ``quadratic'' part.  This will be important in our proof of \cref{thm:point-control}.

We will need the following bound on moments of low-degree polynomials of Rademacher or standard Gaussian random variables (which is a special case of a phenomenon called \emph{hypercontractivity}).
\begin{theorem}[{\cite[Theorem~9.21]{O14}}]\label{thm:gauss-moment}
Let $f$ be a polynomial in $n$ variables of degree at most $d$. Let $\vec x=(x_1,\ldots,x_n)$ either be a vector of independent Rademacher random variables or a vector of independent standard Gaussian random variables. Then for any real number $q\geq 2$, we have
\[\mb{E}\big[|f(\vec x)|^q\big]^{1/q}\le \big(\sqrt{q-1}\big)^d\mb{E}\big[f(\vec x)^2\big]^{1/2}.\]
\end{theorem}
We emphasize that we do not require $f(\vec x)$ to have mean zero, so in the general setting of \cref{thm:gauss-moment} one does not necessarily have $\mb{E}[f(\vec x)^2]^{1/2}=\sigma(f(\vec x))$ (though in our proof of \cref{thm:point-control} we will only apply \cref{thm:gauss-moment} in the case where $\mb{E}[f(\vec x)]=0$).

Note that \cite[Theorem~9.21]{O14} is stated only for Rademacher random variables; the Gaussian case of \cref{thm:gauss-moment} follows by approximating Gaussian random variables with sums of Rademacher random variables, using the central limit theorem.

Next, one can use \cref{thm:gauss-moment} to obtain the following concentration inequality. The Rademacher case is stated as \cite[Theorem~9.23]{O14}, and the Gaussian case may be proved in the same way.
\begin{theorem}\label{thm:concentration-hypercontractivity}
Let $f$ be a polynomial in $n$ variables of degree at most $d$. Let $\vec x=(x_1,\ldots,x_n)$ either be a vector of independent Rademacher random variables or a vector of independent standard Gaussian random variables. Then for any $t\ge (2e)^{d/2}$,
\[\Pr\left[|f(\vec x)|\ge t(\mb{E}[f(x)^2])^{1/2}\right]\le \exp\left(-\frac{d}{2e} t^{2/d}\right).\]
\end{theorem}

\subsection{Basic concentration inequalities}\label{sub:concentration-inequalities}
We will frequently need the Chernoff bound for binomial and hypergeometric distributions (see for example \cite[Theorems~2.1 and~2.10]{JLR00}). Recall that the hypergeometric distribution $\mr{Hyp}(N,K,n)$ is the distribution of $|Z\cap U|$, for fixed sets $U\subseteq S$ with $|S|=N$ and $|U|=K$ and a uniformly random size-$n$ subset $Z\subseteq S$.

\begin{lemma}[Chernoff bound]\label{lem:chernoff}
Let $X$ be either:
\begin{itemize}
    \item a sum of independent random variables, each of which take values in $\{0,1\}$, or
    \item hypergeometrically distributed (with any parameters).
\end{itemize}
Then for any $\delta>0$ we have
\[\Pr[X\le (1-\delta)\mb{E}X]\le\exp(-\delta^2\mb{E}X/2),\qquad\Pr[X\ge (1+\delta)\mb{E}X]\le\exp(-\delta^2\mb{E}X/(2+\delta)).\]
\end{lemma}

We will also need the following concentration inequality, which is a simple consequence of the Azuma--Hoeffding martingale concentration inequality (a special case appears in \cite[Corollary~2.2]{GKM17}, and the general case follows from the same proof). 

\begin{lemma}\label{lem:slice-concentration}
Consider a partition $[n]=I_1\cup\cdots\cup I_m$, and sequences $(\ell_1,\ldots,\ell_m), (\ell_1',\ldots,\ell_m')\in\mb N^m$ with $\ell_k+\ell'_k\leq |I_k|$ for $k=1,\ldots,m$ (and $\ell_1+\cdots+\ell_m+\ell_1'+\cdots+\ell_m'>0$). Let $S\su \{-1,0,1\}^{n}$ be the set of vectors $\vec{x}\in \{-1,0,1\}^{n}$ such that $\vec{x}_{I_k}$ has exactly $\ell_k$ entries being $1$ and exactly $\ell_k'$ entries being $-1$ for each $k=1,\ldots,m$. Let $a>0$ and suppose that $f\colon S\to\mb{R}$ is a function such that we have $|f(\vec{x})-f(\vec{x}')|\le a$ for any two vectors $\vec{x}, \vec{x}'\in S$ which differ in precisely two coordinates (i.e., which are obtained from each other by switching two entries inside some set $I_k$). Then for a uniformly random vector $\vec{x}\in S$ and any $t\geq 0$ we have
\[\Pr[|f(\vec{x})-\mb E f(\vec{x})|\ge t]\le2\exp\left(-\frac{t^{2}}{2\cdot (\ell_1+\cdots+\ell_m+\ell_1'+\cdots+\ell_m')\cdot a^2}\right).\]
\end{lemma}
\begin{proof}
We sample a uniformly random vector $\vec{x}\in S$ in $\ell:=\ell_1+\cdots+\ell_m+\ell_1'+\cdots+\ell_m'$ steps, as follows. In the first $\ell_1$ steps, we pick the $\ell_1$ indices $i\in I_1$ such that $x_i=1$ (at each step, pick an index $i\in I_1$ uniformly at random among the indices where $x_i$ is not yet defined, and define $x_i=1$). In the next $\ell_2$ steps we pick the $\ell_2$ indices $i\in I_2$ such that $x_i=1$, and so on. After $\ell_1+\cdots+\ell_m$ steps we have defined all the $1$-entries of $\vec{x}$. Now, we repeat the procedure (for $\ell_1'+\cdots+\ell_m'$ steps) for the $-1$-entries.

For $t=0,\ldots,\ell$, define $X_t$ to be the expectation of $f(\vec{x})$ conditioned on the coordinates of $\vec{x}$ defined up to step $t$. Then $X_0,\ldots,X_t$ is the Doob martingale associated to our process of sampling $\vec{x}$. Note that $X_0=\mb E f(\vec{x})$ and $X_\ell=f(\vec{x})$.

We claim that we always have $|X_t-X_{t-1}|\le a$ for $t=1,\ldots,\ell$. Indeed, let us condition on any outcomes of the first $t-1$ steps of our process of sampling $\vec{x}$. Now, for any two possible indices $i$ and $i'$ chosen the $t$-th step, we can couple the possible outcomes of $\vec{x}$ if $i$ is chosen in the $t$-th step with the possible outcomes of $\vec{x}$ if $i'$ is chosen in the $t$-th step, simply by switching the $i$-th and the $i'$-th coordinate. Using our assumption on $f$, this shows that for any two possible outcomes in the $t$-th step the corresponding conditional expectations differ by at most $a$. This implies $|X_t-X_{t-1}|\le a$, as claimed.

Now the inequality in the lemma follows from the Azuma--Hoeffding inequality for martingales (see for example \cite[Theorem~2.25]{JLR00}).
\end{proof}

\section{Small-ball probability for quadratic polynomials of Gaussians}\label{sec:anti-Gauss}

In this section we prove \cref{thm:gaussian-anticoncentration-upper-bound}, which we reproduce for the reader's convenience. For the sake of convenience in the proofs and statements, in this section the notation $a\lesssim b$ simply means that $a\le C b$ for some constant $C$ (i.e., there is no stipulation that $n$, the number of variables, be large).

\renewcommand{\gaussianformula}{\[\mc{L}(f(\vec Z),\eps) \lesssim_\eta \frac{\eps}{\sigma(f(\vec{Z}))}.\]}
\gaussian*

\begin{remark}\label{rem:robust-rank-eigenvalue-condition}
By \cref{thm:eckart-young}, the robust rank assumption in \cref{thm:gaussian-anticoncentration-upper-bound} is equivalent to the assumption that every subset $I\su [n]$ of size $|I|=n-2$ satisfies $\sum_{i\in I} \lambda_i^2\geq \eta(\lambda_1^2+\cdots+\lambda_n^2)$, where $\lambda_1,\ldots,\lambda_n$ denote the eigenvalues of $F$.
\end{remark}

We remark that for \emph{any} real random variable $X$, one can use Chebyshev's inequality to show that $\mc{L}(X,\eps) = \Omega(\eps/\sigma(X))$, so the bound in \cref{thm:gaussian-anticoncentration-upper-bound} is best-possible.

In the proof of \cref{thm:point-control}, we will actually need a slightly more technical \emph{non-uniform} version of \cref{thm:gaussian-anticoncentration-upper-bound} that decays away from the mean (at a high level, this is proved by combining \cref{thm:gaussian-anticoncentration-upper-bound} with the hypercontractive tail bound in \cref{thm:concentration-hypercontractivity}, via a ``splitting'' technique; for this splitting technique we need our rank assumption to be slightly stronger than in \cref{thm:gaussian-anticoncentration-upper-bound}). We will also need a \emph{lower} bound on the probability that $f(\vec Z)$ falls in a given interval of length $\varepsilon$, as long as this interval is relatively close to $\mb E f(\vec Z)$, and lies on ``the correct side'' of $\mb E f(\vec Z)$ (this lower bound requires no rank assumption).

\begin{theorem}\label{thm:gaussian-anticoncentration-technical}
Let $\vec Z = (Z_1,\ldots,Z_n)\sim\mc{N}(0,1)^{\otimes n}$ be a vector of independent standard Gaussian random variables. Consider a non-constant real quadratic polynomial $f(\vec Z)$ of $\vec Z$, which we may write as 
\[f(\vec Z)=\vec Z^\intercal F \vec Z+\vec f\cdot \vec Z+f_0\]
for some symmetric matrix $F\in\mb{R}^{n\times n}$, some vector $\vec f\in\mb{R}^n$ and some $f_0\in\mb{R}$.
\begin{enumerate}
    \item Suppose that $F$ is nonzero and \[\min_{\substack{\wt F\in\mb{R}^{n\times n}\\\on{rank}(\wt F)\le 3}}\frac{\|F-\wt F\|^2_{\mr F}}{\|F\|^2_\mr{F}}\ge \eta.\] Then for any $x\in\mb{R}$ and any $0\le \eps\le \sigma(f)$, we have
\[\Pr[f-\mb E f \in [x, x+ \eps]]\lesssim_{\eta}\frac{\eps}{\sigma(f)}\exp\left(-\Omega\left(\frac{|x|}{\sigma(f)}\right)\right).\]
    \item Let $\lambda_1, \ldots, \lambda_n$ be the eigenvalues of $F$. Suppose that $|\lambda_i|\leq \lambda_1$ for $i=1,\ldots,n$. Then for any $A>0$ and $0\le \eps\le \sigma(f)$, we have
\[\inf_{0\le x\le A\sigma(f)}\Pr[f-\mb E f\in[x,x+\eps]] \gtrsim_A \frac{\eps}{\sigma(f)}.\]
\end{enumerate}
\end{theorem}
\begin{remark}
Note that the infimum in (2) is only over nonnegative $x$ (this nonnegativity assumption corresponds to our implicit assumption that $\lambda_1\ge 0$). A two-sided bound is not possible in general, as the polynomial $f(\vec Z)=Z_1^2$ shows. 
Also, while the rank assumption in \cref{thm:gaussian-anticoncentration-upper-bound} (robustly having rank at least 3) was best-possible, we believe that the rank assumption in \cref{thm:gaussian-anticoncentration-technical}(1) (robustly having rank at least 4) can be improved; it would be interesting to investigate this further (e.g., one might try to prove \cref{thm:gaussian-anticoncentration-technical}(1) directly rather than deducing it from \cref{thm:gaussian-anticoncentration-upper-bound} via our splitting technique).

In addition, in \cref{thm:gaussian-anticoncentration-technical}(2), the quantitative bound for implicit constant hidden by ``$\gtrsim_A$'' is rather poor; our proof provides a dependence of the form $\exp(-\exp(O(A^2)))$. We believe that the correct dependence is $\exp(-O(A^2))$, and it may be interesting to prove this.
\end{remark}

By orthogonal diagonalization of $F$ and the invariance of the distribution of $\vec Z$ under orthonormal transformations, in the proofs of \cref{thm:gaussian-anticoncentration-upper-bound,thm:gaussian-anticoncentration-technical} we can reduce to the case where $f(\vec Z)=a_0+\sum_{i=1}^n (a_i Z_i+\lambda_iZ_i^2)$ for some $a_0,\ldots,a_n\in\mb{R}$. This is a sum of independent random variables, so we can proceed using Fourier-analytic techniques.

The rest of this section proceeds as follows. First, in \cref{sub:gaussian-fourier-estimates}, we prove \cref{lem:gaussian-tricks}, which encapsulates certain Fourier-analytic estimates that are effective when no individual term $a_i Z_i+\lambda_iZ_i^2$ contributes too much to the variance of $f(\vec Z)$ (essentially, these are the estimates one needs for a central limit theorem).

Second, in \cref{sub:upper-small-ball} we prove the uniform upper bound in \cref{thm:gaussian-anticoncentration-upper-bound}. In the case where no individual term contributes too much to the variance of $f(\vec Z)$ we use \cref{lem:gaussian-tricks}, and otherwise we need some more specialized Fourier-analytic computations.

Third, in \cref{sub:lower-small-ball} we prove the lower bound in \cref{thm:gaussian-anticoncentration-technical}(2). Again, we use \cref{lem:gaussian-tricks} in the case where no individual term contributes too much to the variance of $f(\vec Z)$, while in the case where one of the terms is especially influential we perform an explicit (non-Fourier-analytic) computation.

Then, in \cref{sub:upper-non-uniform} we deduce the non-uniform upper bound in \cref{thm:gaussian-anticoncentration-technical}(1) from \cref{thm:gaussian-anticoncentration-upper-bound}, using a ``splitting'' technique.

Finally, in \cref{sub:gaussian-quadratic-fourier} we prove an auxiliary technical estimate on characteristic functions of quadratic polynomials of Gaussian random variables, in terms of the ``rank robustness'' of the quadratic polynomial (which we will need in the proof of \cref{thm:short-interval}).

\subsection{Gaussian Fourier-analytic estimates}\label{sub:gaussian-fourier-estimates}

In this subsection we prove several Fourier-analytic estimates. First, we state a formula for the absolute value of the characteristic function of a univariate quadratic polynomial of a Gaussian random variable. One can prove this by direct computation, but we instead give a quick deduction from the formula for the characteristic function of a non-central chi-squared distribution (i.e., of a random variable $Z^2$ where $Z\sim\mc{N}(\mu,\sigma^2)$; see for example \cite{Pat49}).

\begin{lemma}\label{lem:char-gauss}
Let $W\sim \mc N(0,1)$ and let $X=a W+\lambda W^2$ for some $a,\lambda\in\mb{R}$. We have
\[|\varphi_X(t)| = \frac{\exp(-a^2t^2/(2+8\lambda^2t^2))}{(1+4\lambda^2t^2)^{1/4}}.\]
\end{lemma}
\begin{proof}
If $\lambda=0$, then $\varphi_{X}(t)=\varphi_{aW}(t)=\varphi_W(at)=\exp(-a^2t^2/2)$, as desired. So let us assume $\lambda\neq 0$. Note that $X=a W+\lambda W^2 = \lambda(W + a/(2\lambda))^2-a^2/(4\lambda)$ and thus
\[
|\varphi_X(t)| = |\varphi_{\lambda(W + a/(2\lambda))^2}(t)| =  |\varphi_{(W + a/(2\lambda))^2}(\lambda t)|. 
\]
Using the formula for the characteristic function of a non-central chi-squared distribution with $1$ degree of freedom and non-centrality parameter $(a/(2\lambda))^2$, we obtain
\[|\varphi_{(W + a/(2\lambda))^2}(\lambda t)| =\frac{\left|\exp\left(\frac{i\cdot a^2/(4\lambda^2)\cdot \lambda t}{1-2i\lambda t}\right)\right|}{|1-2i\lambda t|^{1/2}}
=\frac{\left|\exp\left(\frac{i\cdot a^2/(4\lambda^2)\cdot \lambda t\cdot (1+2i\lambda t)}{1+4\lambda^2 t^2}\right)\right|}{(1+4\lambda^2 t^2)^{1/4}}
=\frac{\exp\left(\frac{-a^2t^2}{2(1+4\lambda^2t^2)}\right)}{(1+4\lambda^2t^2)^{1/4}}.\qedhere\]
\end{proof}

The crucial estimates in this subsection are encapuslated in the following lemma.
\begin{lemma}\label{lem:gaussian-tricks}
There are constants $C_{\ref{lem:gaussian-tricks}},C_{\ref{lem:gaussian-tricks}}'>0$ such that the following holds. Let $W_1,\ldots,W_n\sim \mc{N}(0,1)$ be independent standard Gaussian random variables, and fix sequences $\vec a,\vec \lambda\in\mb{R}^n$ not both zero. Define random variables $X_1,\ldots,X_n$ and $X$ as well as nonnegative $\sigma_1,\ldots,\sigma_n, \sigma, \Gamma\in\mb{R}$ by
\[X_i = a_iW_i + \lambda_i (W_i^2-1),\quad X=\sum_{i=1}^n X_i,\quad \sigma_i^2 = \sigma(X_i)^2= a_i^2 + 2\lambda_i^2, \quad \sigma^2 = \sum_{i=1}^n\sigma_i^2,\quad \Gamma=\frac{\sigma^{3}}{\sum_{i=1}^n\sigma_i^3}.\]
\begin{enumerate}
    \item[(a)] If $\int_{-\infty}^{\infty}\prod_{i=1}^{n}|\varphi_{X_i}(t)|\,dt<\infty$, then $X$ has a continuous density function $p_X\colon\mb{R}\to \mb{R}_{\ge 0}$ satisfying
    \[\sup_{u\in\mb{R}}\bigg|p_{X}(u)-\frac{e^{-u^2/(2\sigma^2)}}{\sigma\sqrt{2\pi}}\bigg|\le C_{\ref{lem:gaussian-tricks}}\bigg( \frac{1}{\Gamma\sigma} + \int_{|t|\ge \Gamma/(32\sigma)}\prod_{i=1}^{n}|\varphi_{X_i}(t)|\,dt\bigg).\]
    \item[(b)] If $\sigma_i^2\le \sigma^2/4$ for all $i=1,\ldots, n$, then for any $K>0$, we have 
    \[\int_{|t|\ge K/\sigma}\;\prod_{i=1}^{n}|\varphi_{X_i}(t)|\,dt\le \frac{C_{\ref{lem:gaussian-tricks}}'}{K\sigma}.\]
\end{enumerate}
\end{lemma}
\begin{remark}
Note that 
$\sigma^3 = \sum_{i=1}^{n}\sigma_i^2\cdot \sigma\ge \sum_{i=1}^{n}\sigma_i^3$ and therefore $\Gamma\ge 1$.
\end{remark}

The first part follows essentially immediately from the classical proof of the central limit theorem (see for example \cite{Pet75}). 
\begin{proof}[{Proof of \cref{lem:gaussian-tricks}(a)}]
First, note that we may assume that there are no indices $i$ with $\sigma_i=0$ (indeed, if $\sigma_i=0$, then $\lambda_i=a_i=0$ and we can just omit all such indices).  By rescaling, we may assume that $\sigma^2 = 1$. Note that $\varphi_{X}(t)=\prod_{i=1}^{n}\varphi_{X_i}(t)$, and hence $\int_{-\infty}^{\infty}|\varphi_{X}(t)|\,dt<\infty$. Also recall that the standard Gaussian distribution has density $u\mapsto e^{-u^2/2}/\sqrt{2\pi}$ and characteristic function $t\mapsto e^{-t^2/2}$. Thus, by the inversion formula~\cref{eq:inversion}, it suffices to show that
\begin{equation}\label{eq:integral-for-tricks}
\frac{1}{2\pi}\int_{-\infty}^{\infty}\bigg|\prod_{i=1}^{n}\varphi_{X_i}(t) - e^{-t^2/2}\bigg|\,dt\lesssim \frac1\Gamma + \int_{|t|\ge \Gamma/32}\;\prod_{i=1}^{n}|\varphi_{X_i}(t)|\,dt.
\end{equation}
Note that $\mb E [X_i]=0$ for $i=1,\ldots,n$, and let us write $L=(\sum_{i=1}^{n}\mb{E}[|X_i|^3])/(\sum_{i=1}^n \sigma_i^{2})^{3/2}=\sum_{i=1}^{n}\mb{E}[|X_i|^3]$. Then for $|t|\le 1/(4L)$, by \cite[Chapter~V,~Lemma~1]{Pet75} (which is a standard estimate in proofs of central limit theorems) we have
\[\bigg|\prod_{i=1}^{n}\varphi_{X_i}(t) - e^{-t^2/2}\bigg|=\left|\varphi_{X}(t) - e^{-t^2/2}\right|\le 16L\cdot |t|^3e^{-t^2/3}.\]
By H\"older's inequality and \cref{thm:gauss-moment} (hypercontractivity) we have $\sigma_i^{3}\le \mb{E}[|X_i|^{3}]\le 8\sigma_i^{3}$ for $i=1,\ldots,n$, so we obtain $1/\Gamma\leq L\leq 8/\Gamma$. Thus, the interval $|t|\le\Gamma/32$ contributes at most $\int_{-\Gamma/32}^{\Gamma/32}16L\cdot |t|^3e^{-t^2/3}\,dt\lesssim L\int_{-\infty}^{\infty}|t|^3e^{-t^2/3}\,dt \lesssim L\lesssim 1/\Gamma$ to the integral in \cref{eq:integral-for-tricks}. Therefore we obtain
\[\int_{|t|\ge {\Gamma}/{32}}\bigg|\prod_{i=1}^{n}\varphi_{X_i}(t) - e^{-t^2/2}\bigg|\,dt\le \int_{|t|\ge {\Gamma}/{32}}e^{-t^2/2} + \bigg|\prod_{i=1}^{n}\varphi_{X_i}(t)\bigg|\,dt\lesssim \frac{1}{\Gamma} + \int_{|t|\ge {\Gamma}/{32}} \bigg|\prod_{i=1}^{n}\varphi_{X_i}(t)\bigg|\,dt.\qedhere\]
\end{proof}

To prove \cref{lem:gaussian-tricks}(b), we use H\"older's inequality and \cref{lem:char-gauss}.

\begin{proof}[{Proof of \cref{lem:gaussian-tricks}(b)}]
As before we may assume that there are no indices $i$ with $\sigma_i=0$, and by rescaling we may assume that $\sigma^2 = 1$. Via \cref{lem:char-gauss}, we estimate 
\begin{align*}
\int_{|t|\ge K}\bigg|\prod_{i=1}^{n}\varphi_{X_i}(t)\bigg|\,dt
&\le \prod_{i=1}^{n}\bigg(\int_{|t|\ge K}|\varphi_{X_i}(t)|^{1/\sigma_i^2}\,dt\bigg)^{\sigma_i^2}=\prod_{i=1}^{n}\Bigg(\int_{|t|\ge K}\frac{\exp\Big(-{a_i^2t^2}/\left({(2+8\lambda_i^2t^2)\sigma_i^2}\right)\Big)}{(1+4\lambda_i^2t^2)^{1/(4\sigma_i^2)}}\,dt\Bigg)^{\sigma_i^2}\\
&\le\prod_{i=1}^{n}\Bigg(\int_{|t|\ge K}\frac{\exp\Big(-{a_i^2t^2}/\left({(2+8\lambda_i^2t^2)\sigma_i^2}\right)\Big)}{1+\lambda_i^2t^2/\sigma_i^2}\,dt\Bigg)^{\sigma_i^2}.
\end{align*}

In the first step we have used H\"older's inequality with weights $\sigma_1^2,\ldots,\sigma_n^2$ (which sum to 1) and in the final step we have used Bernoulli's inequality (which says that $(1+x)^r\ge 1+rx$ for $x\ge 0$ and $r\ge 1$; recall that we are assuming that  $4(a_i^2+2\lambda_i^2)=4\sigma_i^2\le 1$ for each $i$).

Since $\sum_{i=1}^n\sigma_i^2=1$, it now suffices to prove that for each $i=1,\ldots,n$ we have
\[\int_{|t|\ge K}\frac{\exp\Big(-{a_i^2t^2}/\left({(2+8\lambda_i^2t^2)\sigma_i^2}\right)\Big)}{1+\lambda_i^2t^2/\sigma_i^2}\,dt\lesssim \frac{1}{K}.\]

Fix some $i$. If $|\lambda_i|\ge |a_i|$, then $\lambda_i^2\geq \sigma_i^2/3$ and
\[\int_{|t|\ge K}\frac{\exp\Big(-{a_i^2t^2}/\left({(2+8\lambda_i^2t^2)\sigma_i^2}\right)\Big)}{1+\lambda_i^2t^2/\sigma_i^2}\,dt\le\int_{|t|\ge K}\frac{1}{1+t^2/3}\,dt\lesssim \frac{1}{K}.\]
Otherwise, if $|a_i|\ge |\lambda_i|$, we have $a_i^2\geq \sigma_i^2/3$, $\sigma_i^2\le 1$, and therefore 
\begin{align*}
\frac{\exp\Big(-{a_i^2t^2}/\left({(2+8\lambda_i^2t^2)\sigma_i^2}\right)\Big)}{1+\lambda_i^2t^2/\sigma_i^2}&\le \frac{\Big(1+{a_i^2t^2}/\left({(2+8\lambda_i^2t^2)\sigma_i^2}\right)\Big)^{-1}}{1+\lambda_i^2t^2/\sigma_i^2}\\
&\lesssim \frac{\big(1+{t^2}/({1+\lambda_i^2t^2})\big)^{-1}}{1+\lambda_i^2t^2/\sigma_i^2}\le \frac{\big(1+{t^2}/{(1+\lambda_i^2t^2)}\big)^{-1}}{1+\lambda_i^2t^2} = \frac{1}{1+(1+\lambda_i^2)t^2}.
\end{align*}
It follows that
\begin{align*}
&\int_{|t|\ge K}\frac{\exp\Big(-{a_i^2t^2}/\left({(2+8\lambda_i^2t^2)\sigma_i^2}\right)\Big)}{1+\lambda_i^2t^2/\sigma_i^2}\,dt \lesssim \int_{|t|\ge K}\frac{1}{1+(1+\lambda_i^2)t^2}\,dt \lesssim \frac{1}{K}. \qedhere
\end{align*}
\end{proof}

\subsection{Uniform anticoncentration}\label{sub:upper-small-ball}
In this subsection, we prove \cref{thm:gaussian-anticoncentration-upper-bound}. The crucial ingredient is the following Fourier-analytic estimate.

\begin{lemma}\label{lem:inductive-trick}
Recall the definitions and notation in the statement of \cref{lem:gaussian-tricks}, and fix a parameter $\eta>0$. Suppose that $n\ge 2$ and $\sum_{i\in I}\lambda_i^2\ge \eta \lambda_j^2$ for all $I\subseteq [n]$ with $|I|= n-2$ and all $j\in[n]$. Then 
\[\int_{|t|\ge 1/(32\sigma)}\prod_{i=1}^{n}|\varphi_{X_i}(t)|\,dt\lesssim_\eta \frac{1}{\sigma}.\]
\end{lemma}
\begin{proof}
We may assume without loss of generality that $|\lambda_1|\geq\cdots\geq |\lambda_n|$. By adding at most two terms with $a_i=\lambda_i=0$, we may assume $n$ is divisible by $3$. Note that if $\sigma_i^2\le \sigma^2/4$ for all $i\in [n]$, the result follows immediately from \cref{lem:gaussian-tricks}(b). Therefore it suffices to consider the case when there is an index $j$ such that $\sigma_j^2\ge \sigma^2/4$.

Note that the given condition implies $\sum_{k=1}^{n/3}\lambda_{3k}^2\ge \frac{1}{3}\sum_{k= 3}^{n}\lambda_k^2\ge \eta \lambda_j^2/3$. Now, \cref{lem:char-gauss} yields
\begin{align*}
\prod_{i=1}^{n}|\varphi_{X_i}(t)|&\le \exp\bigg(\frac{-a_j^2t^2}{2+8\lambda_j^2t^2}\bigg)\prod_{i=1}^{n}\frac{1}{(1+4\lambda_i^2t^2)^{1/4}}\le \exp\bigg(\frac{-a_j^2t^2}{2+8\lambda_j^2t^2}\bigg)\prod_{i=1}^{n/3}\frac{1}{(1+4\lambda_{3i}^2t^2)^{3/4}}\\
&\le \exp\bigg(\frac{-a_j^2t^2}{2+8\lambda_j^2t^2}\bigg)\bigg(1+4\sum_{i=1}^{n/3}\lambda_{3i}^2t^2\bigg)^{-3/4}\le \exp\bigg(\frac{-a_j^2t^2}{2+8\lambda_j^2t^2}\bigg)(1+\eta \lambda_j^2t^2)^{-3/4}\\
&\le \bigg(1 + \frac{a_j^2t^2}{2+8\lambda_j^2t^2}\bigg)^{-3/4}(1+\eta\lambda_j^2t^2)^{-3/4}\lesssim_{\eta} (\lambda_j^2t^2+a_j^2t^2)^{-3/4}\lesssim (\sigma_j |t|)^{-3/2}\lesssim (\sigma |t|)^{-3/2}.
\end{align*}

Thus we have 
\[\int_{|t|\ge 1/(32\sigma)}\prod_{i=1}^{n}|\varphi_{X_i}(t)|\,dt\lesssim_{\eta}\int_{|t|\ge 1/(32\sigma)}(\sigma |t|)^{-3/2}\,dt\lesssim 1/\sigma.\qedhere\]
\end{proof}

The proof of \cref{thm:gaussian-anticoncentration-upper-bound} is now immediate.
\begin{proof}[{Proof of \cref{thm:gaussian-anticoncentration-upper-bound}}]
By rescaling we may assume $\sigma(f)=1$. It suffices to show that the probability density function $p_{f-\mb E f}$ of $f-\mb Ef$ satisfies $p_{f-\mb E f}(u)\lesssim_{\eta} 1$ for all $u$.

Since $F$ is a real symmetric matrix, we can write $F = QD Q^{\intercal}$ where $D$ is a diagonal matrix with entries $\lambda_1,\ldots,\lambda_n$ and $Q$ is an orthogonal matrix. Let $\vec W=Q^{\intercal}\vec Z$, and note that $\vec W$ is also distributed as $\mc{N}(0,1)^{\otimes n}$ (since the distribution $\mc{N}(0,1)^{\otimes n}$ is invariant under orthogonal transformations). We have
\[f(\vec Z) = f_0+\vec{f}\cdot \vec{Z} + \vec{Z}^\intercal F\vec{Z} = f_0+\vec{f}\cdot (Q\vec{W}) + \vec{W}^\intercal Q^\intercal FQ\vec{W} = f_0+(Q^\intercal\vec{f})\cdot \vec{W} + \vec{W}^\intercal D \vec{W}.\]
Let $\vec{a} = (a_1,\ldots,a_n)=Q^\intercal\vec{f}$. We have
\[
f-\mb E f= \sum_{i=1}^n(a_iW_i+\lambda_i(W_i^2-1)).
\]
Let $\sigma_1,\ldots,\sigma_n\geq 0$ be such that $\sigma_i^2=a_i^2+2\lambda_i^2$, so $1=\sigma(f)^2 = \sigma_1^2+\cdots+\sigma_n^2$. Note that the assumption in the theorem statement implies $n\ge 3$, and combining the assumption with \cref{thm:eckart-young} yields
\[\eta\le \min_{\substack{\wt F\in\mb{R}^{n\times n}\\\on{rank}(\wt F)\le 2}}\frac{\|F-\wt F\|^2_{\mr F}}{\|F\|^2_\mr{F}}=\min_{\substack{I\su [n]\\|I|=n-2}} \frac{\sum_{i\in I}\lambda_i^2}{\lambda_1^2+\cdots+\lambda_n^2}.\]
Hence for any subset $I\subseteq[n]$ with $|I|= n-2$ and any $j\in [n]$ we obtain $\sum_{i\in I}\lambda_i^2\ge \eta(\lambda_1^2+\cdots+\lambda_n^2)\geq \eta\lambda_j^2$. Let $\Gamma$ be as in \cref{lem:gaussian-tricks} and recall that $\Gamma\ge 1$.

Now, by combining \cref{lem:gaussian-tricks}(a) and \cref{lem:inductive-trick}, we have that 
\begin{align*}
\sup_{u\in\mb{R}}p_{f}(u)=\sup_{u\in\mb{R}}p_{f-\mb E f}(u)\lesssim \frac{1}{\sqrt{2\pi}} + \frac{1}{\Gamma} + \int_{|t|\ge \Gamma/32}\prod_{i=1}^{n}|\varphi_{X_i}(t)|\,dt\lesssim_{\eta} 1. 
\end{align*}
By integrating over the desired interval, we obtain the bound in \cref{thm:gaussian-anticoncentration-upper-bound}.
\end{proof}

\subsection{Lower bounds on small-ball probabilities}\label{sub:lower-small-ball}
Let us now prove the lower bound in \cref{thm:gaussian-anticoncentration-technical}(2). Note that \cref{lem:gaussian-tricks}(b) does not apply when some $\sigma_i$ is especially influential; in that case we will use the following bare-hands estimate.

\begin{lemma}\label{lem:explicit-gaussian}
Fix $A'\geq 1$ and let $W\sim \mc N(0,1)$ and for some $a,\lambda\in \mathbb{R}$ (not both zero) let $X=aW+\lambda(W^2-1)$, so $\sigma(X)^2=a^2+2\lambda^2$. Suppose that
\begin{enumerate}
    \item $\lambda\ge 0$, or
    \item $\sigma(X)\ge 10A'\cdot |\lambda|$.
\end{enumerate}
Then for any $0\le u\le A'\sigma(X)$, we have $p_X(u)\gtrsim_{A'} 1/{\sigma(X)}$.
\end{lemma}
\begin{proof}
We may assume $a\ge 0$ (changing $a$ to $-a$ does not change the distribution of $X$). First note that the case $\lambda=0$ is easy, since then we have $\sigma(X)=a$ and $p_X(u)=e^{-(u/a)^2/2}/(\sqrt{2\pi}a)\gtrsim_{A'} 1/\sigma(X)$. So let us assume $\lambda\neq 0$ and define $g\colon\mathbb{R}\to\mathbb{R}$ by
\[g(t)=at+\lambda(t^2-1)=\lambda\cdot \left(t+\frac{a}{2\lambda}\right)^2-\lambda-\frac{a^2}{4\lambda}.\]
Then for all $t\in \mathbb{R}$ we have
\[(g'(t))^2=4\lambda^2\cdot  \left(t+\frac{a}{2\lambda}\right)^2=4\lambda\cdot g(t)+4\lambda^2+a^2\leq 4\sigma(X)\cdot |g(t)|+4\sigma(X)^2.\]
Hence, for any $t\in\mb{R}$ with $g(t)=u$, recalling $0\le u\le A'\sigma(X)$, we obtain $|g'(t)|\leq \sqrt{4(A'+1)\sigma(X)^2}\lesssim_{A'} \sigma(X)$.

We claim that we can find $t\in [-3A',3A']$ with $g(t)=u$. Indeed, in case (1), we have $g(0)=-\lambda\leq u$ and $g(2A'+1)\geq 2A'a+2A'\lambda\geq A'\sigma(X)\geq u$, and hence by the intermediate value theorem there exists $t\in [0,2A'+1]\su [-3A',3A']$ with $g(t)=u$. In case (2), observe that $a^2=\sigma(X)^2-2\lambda^2\geq 100 A'^2\cdot\lambda^2-2\lambda^2\geq 81A'^2\cdot\lambda^2$, so $a\geq 9A'\cdot|\lambda|$ and therefore $|\lambda(9A'^2-1)|\leq A'a$ and $\sigma(X)^2=a^2+2\lambda^2\leq 4a^2$. Hence $g(-3A')=-3A'a+\lambda(9A'^2-1)\leq -2A'a\leq 0\leq u$ and $g(3A')=3A'a+\lambda(9A'^2-1)\geq 2A'a\geq A'\sigma(X)\geq u$ and we can again conclude that there exists $t\in [-3A',3A']$ with $g(t)=u$.

Now, we have
\[p_X(u)=p_{g(W)}(g(t))\geq \frac{p_W(t)}{|g'(t)|}\gtrsim_{A'} \frac{e^{-(3A')^2/2}}{\sigma(X)}\gtrsim_{A'} \frac{1}{\sigma(X)}.\qedhere\]

\end{proof}

We need one more ingredient for the proof of \cref{thm:gaussian-anticoncentration-technical}(2): a variant of the Paley-Zygmund inequality which tells us that under a fourth-moment condition, random variables are reasonably likely to have small fluctuations in a given direction. We include a short proof; the result can also easily be deduced from \cite[Lemma~3.2(i)]{AGS04}.
\begin{lemma}\label{lem:paley}
Fix $B\ge 1$. If $X$ is a real random variable with $\mb{E}[X] = 0$ and $\sigma(X)>0$ satisfying  $\mb{E}[X^4]\le B\sigma(X)^4$, then
\[\Pr[-2\sqrt{B}\sigma(X) \le X\le 0]\ge 1/(5B).\]
\end{lemma}
\begin{proof}
By rescaling we may assume that $\sigma(X) = 1$. Note that then
\begin{align*}
9B^2\cdot \Pr[-2\sqrt{B}\le X\le 0] &=\mb{E}[9B^2\mbm{1}_{-2\sqrt{B}\le X\le 0}]\\
&\ge \mb{E}[-X(X+2\sqrt{B})(X-\sqrt{B})^2]\\
&=-\mb{E}[X^4]+3B\cdot \mb{E}[X^2]-2B^{3/2}\mb{E}[X]=-\mb{E}[X^4]+3B\ge 2B
\end{align*}
where we have used that $-x(x+2\sqrt{B})(x-\sqrt{B})^2=(B - (x+\sqrt{B})^2)(x-\sqrt{B})^2\le 9B^2\mbm{1}_{-2\sqrt{B}\le x\le 0}$ for all $x\in\mb{R}$. The result follows. 
\end{proof}

Now we prove \cref{thm:gaussian-anticoncentration-technical}(2).

\begin{proof}[Proof of \cref{thm:gaussian-anticoncentration-technical}(2)]
We may assume $\sigma(f)=1$. Borrowing the notation from the proof of \cref{thm:gaussian-anticoncentration-upper-bound}, we write
\[
f-\mb E f= \sum_{i=1}^n(a_iW_i+\lambda_i(W_i^2-1)),
\]
with $(W_1,\ldots,W_n)\sim \mathcal{N}(0,1)^{\otimes n}$, and $\sigma_i^2=a_i^2+2\lambda_i^2$ (then we have $1=\sigma^2=\sigma_1^2+\cdots+\sigma_n^2$). It now suffices to prove that for all $u\in [0,A+1]$ we have $p_{f-\mb E f}(u) \gtrsim_A 1$. Let $L$ be a large integer depending only on $A$ (such that $L\geq 2$ and $L\geq C_{\ref{lem:gaussian-tricks}}(1+32C_{\ref{lem:gaussian-tricks}}')\cdot 2\sqrt{2\pi}\cdot e^{(A+1)^2/2}$ for the constants $C_{\ref{lem:gaussian-tricks}}$ and $C_{\ref{lem:gaussian-tricks}}'$ in \cref{lem:gaussian-tricks}). We break into cases.

First, suppose $\max_i\sigma_i\le1/L$. In this case, we define $\Gamma=\sigma(f)^{3}/\sum_{i=1}^n\sigma_i^3=1/\sum_{i=1}^n\sigma_i^3$ and note that  $\sum_{i=1}^{n}\sigma_i^3\le (\max_i\sigma_i)(\sum_{i=1}^{n} \sigma_i^2)\le 1/L$, so $\Gamma\ge L$. We also have $\sigma_i^2\leq 1/L^2\leq 1/4$, so \cref{lem:gaussian-tricks}(b) applies. So by combining parts (a) and (b) of \cref{lem:gaussian-tricks}, for all $u\in [0,A+1]$ we obtain, as desired,
\[p_{f-\mb E f}(u)\ge \frac{e^{-u^2/2}}{\sqrt{2\pi}}-\frac{C_{\ref{lem:gaussian-tricks}}(1+32C_{\ref{lem:gaussian-tricks}}')}{\Gamma}\ge \frac{e^{-(A+1)^2/2}}{\sqrt{2\pi}}-\frac{C_{\ref{lem:gaussian-tricks}}(1+32C_{\ref{lem:gaussian-tricks}}')}{L}\ge \frac12\cdot \frac{e^{-(A+1)^2/2}}{\sqrt{2\pi}}\gtrsim_A 1.\]

Otherwise, there is $i^*\in [n]$ such that $\sigma_{i^*} \ge1/L$. We claim that then there is an index $j\in[n]$ satisfying at least one of the following two conditions:
\begin{itemize}
    \item[(1)] $\sigma_{j} \ge 1/(10(A+19)L^2)$ and $\lambda_{j}\geq 0$, or
    \item[(2)] $\sigma_{j} \ge 1/L$ and $10(A+19)L\cdot |\lambda_{j}|\le \sigma_{j}$.
\end{itemize}
Indeed, if $10(A+19)L\cdot |\lambda_{i^*}|\le \sigma_{i^*}$ we can simply take $j=i^*$ and (2) is satisfied. Otherwise we have $|\lambda_{i^*}|> \sigma_{i^*}/(10(A+19)L)\geq 1/(10(A+19)L^2)$ and the assumption in \cref{thm:gaussian-anticoncentration-technical}(2) yields $\lambda_1\geq |\lambda_{i^*}|\geq 1/(10(A+19)L^2)$. So in particular $\lambda_1\geq 0$ and $\sigma_1\geq \lambda_1\geq 1/(10(A+19)L^2)$ and we can take $j=1$ and (1) is satisfied.

Now, let $X_j=a_{j} W_{j}+\lambda_{j}(W_{j}^2-1)$ and let $X'=f-\mb E f-X_j=\sum_{i\neq j} (a_i W_i+\lambda_i(W_i^2-1))$ contain all terms of $f-\mb E f$ except the term $X_j$. By \cref{thm:gauss-moment} (hypercontractivity) we have $\mb E[(X')^4]\leq 81\sigma(X')^4$ and therefore \cref{lem:paley} shows that $-18\le -18\sigma(X')\le X'\le 0$ with probability at least $1/405$. 

We claim that we can apply \cref{lem:explicit-gaussian} to $X_j$ and $u\in [0,A+19]$, showing that $p_{X_j}(u)\gtrsim_A 1/\sigma_{j}\geq 1$. Indeed, in case (1) we have $0\leq u\leq 10(A+19)^2L^2\sigma_j$ and can apply case (1) of \cref{lem:explicit-gaussian} with $A'=10(A+19)^2L^2$, while in case (2) we have $0\leq u\leq (A+19)L\sigma_j$ and can apply case (2) of \cref{lem:explicit-gaussian} with $A'=(A+19)L$.

Therefore, for any $u\in [0,A+1]$ we obtain
\[p_{f-\mb E f}(u)=p_{X'+X_j}(u)\geq \int_{-18}^{0}  p_{X'}(y)p_{X_j}(u-y)\,dy\gtrsim_A \int_{-18}^{0}  p_{X'}(y)\,dy =\Pr[-18\leq X'\leq 0]\gtrsim 1.\qedhere\]
\end{proof}

\subsection{Non-uniform anticoncentration}\label{sub:upper-non-uniform}
In this subsection we prove \cref{thm:gaussian-anticoncentration-technical}(1), which is essentially a non-uniform version of \cref{thm:gaussian-anticoncentration-upper-bound}. We begin with a lemma giving non-uniform anticoncentration bounds for a quadratic polynomial of a single Gaussian variable, i.e., for one of the terms in our sum.

\begin{lemma}\label{lem:explicit-gaussian-2}
Let $W\sim \mc N(0,1)$ and for some $a,\lambda\in \mathbb{R}$ (not both zero) let $X=aW+\lambda(W^2-1)$, so $\sigma^2:=\sigma(X)^2=a^2+2\lambda^2$. Suppose we are given some $x\geq 10^3\sigma$ satisfying $|\lambda|\cdot x\leq a^2/10$. Then for each $u\in\mb{R}$ with $x/10\leq |u|\leq 2x$, we have \[p_X(u)\lesssim \frac{1}{|a|}\exp\left(-\frac{x}{\sigma}\right).\]
\end{lemma}
\begin{proof}
Define the function $g\colon\mathbb{R}\to\mathbb{R}$ by $g(t)=at+\lambda(t^2-1)$. As in the proof of \cref{lem:explicit-gaussian}, we can calculate $(g'(t))^2=4\lambda\cdot g(t)+4\lambda^2+a^2$ for all $t\in\mb{R}$. Now, consider some $u\in\mb{R}$ with $x/10\leq |u|\leq 2x$. There are at most two different $t\in\mb{R}$ with $g(t)=u$. For any such $t$, we have (using the assumption that $|\lambda|\cdot x\leq a^2/10$) \[(g'(t))^2\geq 4\lambda^2+a^2-4|\lambda|\cdot 2x\geq a^2/5\geq a^2/9.\]
We furthermore claim that any such $t$ must satisfy $|t|\geq x/(20|a|)$. Indeed, if $|t|< x/(20|a|)$, then (using that $x\geq 10^3\sigma\geq 10^3a$ and the assumption $|\lambda|\cdot x\leq a^2/10$)
\[|g(t)|=|at+\lambda(t^2-1)|\leq |a|\cdot \frac{x}{20|a|}+|\lambda|\cdot \max\left\{\frac{x^2}{400a^2},1\right\}\leq \frac{x}{20}+|\lambda|\cdot\frac{x^2}{400a^2}\leq \frac{x}{20}+\frac{x}{4000}<\frac{x}{10}.\]
As $|u|\geq x/10$, this contradicts $g(t)=u$. Thus any $t\in\mb{R}$ with $g(t)=u$ must indeed also satisfy $|t|\geq x/(20|a|)$. Now, we obtain (using again that $x\geq 10^3\sigma\geq 10^3a$)
\[p_X(u)=\sum_{\substack{t\in\mb{R}\\g(t)=u}} \frac{p_W(t)}{|g'(t)|}\leq 2\cdot \frac{1}{|a|/3}\cdot \exp\left(-\frac{x^2}{800a^2}\right)\lesssim \frac{1}{|a|}\exp\left(-\frac{x}{\sigma}\right).\qedhere\]
\end{proof}

Now, we prove \cref{thm:gaussian-anticoncentration-technical}(1). The main idea is to divide our random variable $f-\mb E f$ into independent parts, to take advantage of exponential tail bounds (by \cref{thm:concentration-hypercontractivity} or \cref{lem:explicit-gaussian-2}) for one of the parts, and anticoncentration bounds (by \cref{thm:gaussian-anticoncentration-upper-bound}) for the rest of the parts.

\begin{proof}[Proof of \cref{thm:gaussian-anticoncentration-technical}(1)]
By rescaling, we may assume $\sigma:=\sigma(f)=1$. If $|x|\leq 10^3=10^3\sigma(f)$, the desired bound follows from \cref{thm:gaussian-anticoncentration-upper-bound}. So we may assume that $|x|\geq 10^3\sigma(f)$. Also note that the assumption in \cref{thm:gaussian-anticoncentration-technical}(1) implies that $\eta\le 1$. Borrowing the notation from the proof of \cref{thm:gaussian-anticoncentration-upper-bound}, we write
\[
f-\mb E f= \sum_{i=1}^n(a_iW_i+\lambda_i(W_i^2-1)),
\]
with $(W_1,\ldots,W_n)\sim \mc N(0,1)^{\otimes n}$ and $\sigma_i^2=a_i^2+2\lambda_i^2$ (then we have $1=\sigma^2=\sigma_1^2+\cdots+\sigma_n^2$). We may assume that $|\lambda_1|\geq \cdots\geq|\lambda_n|$. Note that using \cref{thm:eckart-young} the assumption in \cref{thm:gaussian-anticoncentration-technical}(1) implies that for every subset $I\su [n]$ of size $|I|=n-3$ we have $\sum_{i\in I}\lambda_i^2\geq \eta(\lambda_1^2+\cdots+\lambda_n^2)$. In particular, $\sum_{i=4}^{n}\lambda_i^2\geq\eta(\lambda_1^2+\cdots+\lambda_n^2)$.

By adding at most three terms with $a_i=\lambda_i=0$, we may assume that $n\equiv 1 \pmod{4}$. For a subset $J\subseteq[n]$, let $X_{J}=\sum_{i\in J}(a_iW_i+\lambda_i(W_i^2-1))$
and $\sigma_{J}^{2}=\sum_{i\in J}\sigma_{i}^2=\sigma(X_{J})^2$. 

Let $i^*\in [n]$ be chosen such that $\sigma_{i^*}^2$ is maximal, and define $J_0=\{i^*\}$. We claim that we can find a partition of $[n]\setminus J_0=[n]\setminus \{i^*\}$ into four subsets $J_1,J_2,J_3,J_4$ satisfying the following conditions.
\begin{enumerate}
    \item[(a)] For $h=1,2,3,4$, we have $\sigma_{[n]\setminus J_h}^2\geq \eta/2$.
    \item[(b)] For any $h=0,\ldots,4$ and any subset $I\su [n]\setminus J_h$ of size $|I|=n-|J_h|-2$, we have $\sum_{i\in I}\lambda_i^2\geq (\eta/4)\cdot (\lambda_1^2+\cdots+\lambda_n^2)$.
\end{enumerate}
Indeed, we can build such a partition iteratively: let us divide $[n]\setminus \{i^*\}$ into $n/4$ quadruplets (starting with the four smallest indices, then the next four, and so on). Iteratively, for each quadruplet, distribute one element to each of $J_1,J_2,J_3,J_4$ in the following way. We assign the index $i$ in the quadruplet with the largest $\sigma_i^2$ to the set $J_h$ which had the smallest value of $\sigma_{J_h}^{2}$ at the end of the last step, we assign the index $i$ with the second-largest $\sigma_i^2$ to the set $J_h$ which had the second-smallest value of $\sigma_{J_h}^{2}$, and so on. One can check that this assignment process maintains the property that at the end of any step, the values $\sigma_{J_h}^{2}$ for $h=1,2,3,4$ differ by at most $\max_{i}\sigma_{i}^{2}=\sigma_{i^*}^2$. Hence $\sigma_{[n]\setminus J_1}^2\geq \sigma_{J_2}^2+\sigma_{i^*}^2\geq \sigma_{J_1}^2=1-\sigma_{[n]\setminus J_1}^2$, so $\sigma_{[n]\setminus J_1}^2\geq 1/2\geq\eta/2$. Analogously, one can show $\sigma_{[n]\setminus J_h}^2\geq \eta/2$ for $h=2,3,4$, so (a) is satisfied. To check (b), note that for each $h=0,\ldots,4$ the set $[n]\setminus J_h$ is missing either one element from each of the quadruplets considered during the construction (if $1\leq h\leq 4$) or is missing one element in total (if $h=0$). For a subset $I\su [n]\setminus J_h$ of size $|I|=n-|J_h|-2$, two additional elements are missing. Thus, for every $k=1,\ldots,n/4$ the set $I\su [n]$ is missing at most $k+2$ of the elements in $[4k]$. Thus, recalling that $|\lambda_1|\geq\cdots\geq|\lambda_n|$, we obtain
\[
\sum_{i\in I}\lambda_i^2 \ge \lambda_4^2+(\lambda_6^2+\lambda_7^2+\lambda_8^2)+(\lambda_{10}^2+\lambda_{11}^2+\lambda_{12}^2)+\cdots \geq \lambda_4^2+\lambda_8^2+\lambda_{12}^2+\cdots\ge \frac{1}{4}\sum_{i=4}^{n}\lambda_i^2\geq  (\eta/4)\cdot (\lambda_1^2+\cdots+\lambda_n^2).
\]
This establishes (b). Thus, the sets $J_1,\ldots,J_4$ indeed satisfy the desired conditions.

By our assumption $|x|\geq 10^3\sigma(f)$ and by $0\leq \eps\leq \sigma(f)$, we have $|y|\geq 0.999|x|\geq (5/6)\cdot |x|$ for all $y\in [x,x+\eps]$. Thus, whenever $f-\mb{E} f=\sum_{i=1}^n (a_iW_i+\lambda_i(W_i^2-1))=X_{J_0}+\cdots+X_{J_4}$ is contained in the interval $[x,x+\eps]$, we must have $|X_{J_h}|\geq |x|/6$ for at least one $h\in \{0,\ldots,4\}$. So, we have
\begin{equation}\label{eq:non-uniform-anticoncentration-first-bound}
\Pr[f-\mb{E} f\in [x,x+\eps]]
\leq \sum_{h=0}^{4}\Pr\Big[|X_{J_h}|\geq |x|/6\text{ and }X_{[n]\setminus J_h}\in [x-X_{J_h},x-X_{J_h}+\eps]\Big].
\end{equation}
For $h= 1,\ldots,4$, note that
\begin{align}
&\Pr\Big[|X_{J_h}|\geq |x|/6\text{ and }X_{[n]\setminus J_h}\in [x-X_{J_h},x-X_{J_h}+\eps]\Big]\notag\\
&\qquad\le \Pr[|X_{J_h}|\ge|x|/6]\cdot \mc L(X_{[n]\setminus J_h},\varepsilon)\lesssim_\eta \exp\left(-\frac{2}{2e}\cdot \frac{|x|}{6\sigma_{J_h}}\right)\cdot \frac{\varepsilon}{\sigma_{[n]\setminus J_h}} \lesssim_{\eta}\frac{\eps}{\sigma}\exp\left(-\Omega\left(\frac{|x|}{\sigma}\right)\right),\label{eq:non-uniform-anticoncentration-summand}
\end{align}
where in the second step we applied \cref{thm:concentration-hypercontractivity} to $X_{J_h}$ with $t=|x|/(6\sigma_{J_h})\geq |x|/(6\sigma)$ and  \cref{thm:gaussian-anticoncentration-upper-bound} to $X_{[n]\setminus J_h}$ (noting that the assumption of \cref{thm:gaussian-anticoncentration-upper-bound} is satisfied by condition (b), see also \cref{rem:robust-rank-eigenvalue-condition}), and in the last step we used that $\sigma_{[n]\setminus J_h}^2\geq \eta/2$ by condition (a).

We now distinguish two cases. First, let us assume that $\sigma_{[n]\setminus J_0}\ge \eta^2/(100|x|)$. In this case, similarly to \cref{eq:non-uniform-anticoncentration-summand}, we can bound (recalling that $\sigma=1$)
\begin{align*}
&\Pr\Big[|X_{J_0}|\geq |x|/6\text{ and }X_{[n]\setminus J_0}\in [x-X_{J_0},x-X_{J_0}+\eps]\Big]\\
&\qquad\lesssim_\eta \exp\left(-\frac{2}{2e}\cdot \frac{|x|}{6\sigma_{J_0}}\right)\cdot \frac{\varepsilon}{\sigma_{[n]\setminus J_0}} \lesssim_{\eta}\frac{\eps}{\sigma}\cdot \frac{|x|}{40\sigma}\cdot  \exp\left(-\frac{|x|}{20\sigma}\right)\le \frac{\eps}{\sigma}\exp\left(-\frac{|x|}{40\sigma}\right),
\end{align*}
where in the last step we used that $te^{-t}\le 1/e\le 1$ for all $t\in \mb{R}$ (specifically, we used this for $t=|x|/(40\sigma)$). Together with \cref{eq:non-uniform-anticoncentration-summand}, this enables us to bound all five summands on the right-hand side of \cref{eq:non-uniform-anticoncentration-first-bound}, implying the desired bound for $\Pr[f-\mb{E} f\in [x,x+\eps]]$.



It remains to consider the case that $\sigma_{[n]\setminus J_0}< \eta^2/(100|x|)$. Then we in particular have $\sigma_{i^*}^2=1-\sigma_{[n]\setminus J_0}^2\ge 1-\eta^4/(10^4|x|^2)\ge 1-\eta/2$. Furthermore, the assumption in \cref{thm:gaussian-anticoncentration-technical}(1) implies $\sum_{i\in [n]\sm \{i^*\}} \lambda_i^2\geq \eta (\lambda_1^2+\cdots+\lambda_n^2)$, and therefore $\lambda_{i^*}^2\leq (1-\eta)(\lambda_1^2+\cdots+\lambda_n^2)\leq (1-\eta)/2$ (recalling that $1=\sigma^2=\sum_{i=1}^{n}(a_i^2+2\lambda_i^2)$). Thus, we obtain $a_{i^*}^2=\sigma_{i^*}^2-2\lambda_{i^*}^2\geq (1-\eta/2)-(1-\eta)= \eta/2$. Our assumption also implies $\eta^4/(10^4|x|^2)> \sigma_{[n]\setminus J_0}^2 \ge\sum_{i\in [n]\sm \{i^*\}} \lambda_i^2 \ge \eta \lambda_{i*}^2$, meaning that $|\lambda_{i^*}|\cdot |x|\le \eta/100 \le a^2_{i^*}/10$.

Now, we observe
\begin{align*}
    \Pr[f-\mb{E} f\in [x,x+\eps]]&\leq \sum_{h=1}^{4}\Pr\Big[|X_{J_h}|\geq |x|/6\text{ and }X_{[n]\setminus J_h}\in [x-X_{J_h},x-X_{J_h}+\eps]\Big]\\
    &\qquad\qquad+\Pr\Big[|X_{[n]\setminus J_0}|\leq (4/6)|x|\text{ and }X_{J_0}\in [x-X_{[n]\setminus J_0},x-X_{[n]\setminus J_0}+\eps]\Big].
\end{align*}
Again, \cref{eq:non-uniform-anticoncentration-summand} gives an upper bound  for the summands for $h=1,\ldots,4$. To bound the last summand, let us fix any outcome of $X_{[n]\setminus J_0}$ with $|X_{[n]\setminus J_0}|\leq (4/6)|x|$. Then the probability that $X_{J_0}=a_{i^*}W_{i^*}+\lambda_{i^*}(W_{i^*}^2-1)$ lies in the interval $[x-X_{[n]\setminus J_0},x-X_{[n]\setminus J_0}+\eps]$ (which has length $\eps$ and is somewhere between $x/10$ and $2x$) is by \cref{lem:explicit-gaussian-2} bounded by
\[\Pr[X_{J_0}\in [x-X_{[n]\setminus J_0},x-X_{[n]\setminus J_0}+\eps]]\lesssim \frac{\eps}{|a_{i^*}|}\exp\left(-\frac{|x|}{\sigma_{J_0}}\right)\lesssim_\eta \frac{\eps}{\sigma}\exp\left(-\frac{|x|}{\sigma}\right),\]
where in the last step we used that $a_{i^*}^2\geq \eta/2$ (see above). Thus, we again obtain the desired bound for $\Pr[f-\mb{E} f\in [x,x+\eps]]$.
\end{proof}

\subsection{Control of Gaussian characteristic functions}\label{sub:gaussian-quadratic-fourier}
For later, we also record the fact that under a robust rank assumption, characteristic functions of certain ``quadratic'' functions of Gaussian random variables decay rapidly.

\begin{lemma}\label{lem:gaussian-high-t}
Fix a positive integer $r$. Let $\vec Z = (Z_1,\ldots,Z_n)\sim\mc{N}(0,1)^{\otimes n}$ be a vector of independent standard Gaussian random variables. Consider a real quadratic polynomial $f(\vec Z)$ of $\vec Z$, written as 
\[f(\vec Z)=\vec Z^\intercal F \vec Z+\vec f\cdot \vec Z+f_0\]
for some symmetric matrix $F\in\mb{R}^{n\times n}$, some vector $\vec f\in\mb{R}^n$ and some $f_0\in\mb{R}$. Let
\[s=\min_{\substack{\wt F\in\mb{R}^{n\times n}\\\on{rank} \wt F\le r}}\|F-\wt F\|_\mr F^2.\]
Then for any $\tau\in\mb{R}$, we have
\[|\varphi_{f(\vec Z)}(\tau)|=|\mb{E}[\exp(i\tau f(\vec{Z}))]|\lesssim_{r} \frac{1}{(1+\tau^2s)^{r/4}}.\]
\end{lemma}
\begin{proof}
Let $\lambda_1,\ldots,\lambda_n$ be the eigenvalues of $F$, ordered such that $|\lambda_1|\ge\cdots\ge|\lambda_n|$. By \cref{thm:eckart-young}, we have $s=\sum_{j=r+1}^n\lambda_j^2$.

As in the proof of \cref{thm:gaussian-anticoncentration-upper-bound}, we write $f(\vec Z)-\mb{E}[f(\vec Z)]=\sum_{j=1}^{n}(a_jW_j + \lambda_i(W_j^2-1))$, where $(W_1,\ldots,W_n)\sim \mc{N}(0,1)^{\otimes n}$ are independent standard Gaussians. From \cref{lem:char-gauss}, recall that
\[|\mb{E}\exp(i\tau (a_j W_j+\lambda_j(W_j^2-1)))|=|\mb{E}\exp(i\tau (a_j W_j+\lambda_jW_j^2))| \le \frac{1}{(1+4\lambda_j^2\tau^2)^{1/4}}.\]
for $j=1,\ldots,n$. We then deduce
\begin{align*}
|\mb{E}[\exp(i\tau f(\vec{Z}))]| &= \prod_{j=1}^{n}|\mb{E}[\exp(i\tau (a_jW + \lambda_j(W_j^2-1)))]|
\le \prod_{j=1}^{n}\frac{1}{(1+4\lambda_j^2\tau^2)^{1/4}}\\
&\le \prod_{j=1}^{r}\left(1+4\tau^2\sum_{t=0}^{\lfloor{(n-j)/r}\rfloor}\lambda_{j + rt}^2\right)^{-1/4}\le \left(1+4\tau^2\sum_{t=0}^{\lfloor{(n-r)/r}\rfloor}\lambda_{r+rt}^2\right)^{-r/4}\\
&\le \left(1+\frac{4\tau^2}r\sum_{j=r+1}^{n}\lambda_{j}^2\right)^{-r/4}\lesssim_{r} \frac{1}{(1+\tau^2s)^{r/4}}.\qedhere
\end{align*}
\end{proof}

\section{Small-ball probability via characteristic functions}\label{sec:fourier-analysis}
Recall that \emph{Esseen's inequality} (\cref{thm:esseen}) states that $\mc L(X,\eps)\lesssim \eps \int_{-2/\varepsilon}^{2/\varepsilon}|\varphi_X(t)|\,dt$ for any real random variable $X$. We will need a ``relative'' version of Esseen's inequality, as follows.

\begin{lemma}\label{lem:esseen-upper-crude}
Let $X,Y$ be real random variables. For any $\varepsilon>0$ we have
\[\mc L(X,\varepsilon)\lesssim \mc L(Y,\varepsilon)+\varepsilon\int_{-2/\varepsilon}^{2/\varepsilon} |\varphi_X(t)-\varphi_Y(t)|\,dt.\]
\end{lemma}

In the proof of \cref{lem:esseen-upper-crude} we use the Fourier transform: for a function $f\in L^1(\mb{R})$ we write
\[\hat{f}(\xi) = \int_{-\infty}^{\infty}e^{-it\xi}f(t)\,dt.\]

\begin{proof}[Proof of \cref{lem:esseen-upper-crude}]
By rescaling it suffices to prove the claim when $\eps = 1$. Let us abbreviate the second summand on the right-hand side of the desired inequality by $I := \int_{-2}^2|\varphi_X(t)-\varphi_Y(t)|\,dt$. Furthermore, let $\psi = \mbm{1}_{[-1,1]}\ast\mbm{1}_{[-1,1]}$ (where $\ast$ denotes convolution); note that $0\le \psi(t)\le 2$ for all $t$, and the support of $\psi$ is inside the interval $[-2,2]$. Let $f(t) =\hat{\psi}(t) = (\widehat{\mbm{1}_{[-1,1]}}(t))^2$; we compute
\[f(t) =\bigg(\int_{-1}^{1}e^{-itx}\,dx\bigg)^2 =\bigg(\frac{2\sin t}{t}\bigg)^2.\]
for $t\neq 0$ and $f(0)=2^2$. Note that for $|t|\le 1$ we have $f(t)\ge  1$, and for all $t\in\mb{R}$ we have $f(t)\le\min\{4,4/t^2\}\leq 8/(t^2+1)$. By the formula for the Fourier transform and the triangle inequality, for any $x\in\mb{R}$ we have
\begin{align*}
|\mb{E}[f(X-x)-f(Y-x)]| &= \bigg|\mb{E}\int_{-\infty}^{\infty}\psi(\theta)(e^{-i\theta(X-x)}-e^{-i\theta(Y-x)})\,d\theta\bigg|\\
&\le \int_{-\infty}^{\infty}\psi(\theta)\big|\mb{E}\big[e^{-i\theta(X-x)}-e^{-i\theta(Y-x)}\big]\big|\,d\theta\\
&=\int_{-\infty}^{\infty}\psi(-t)|\varphi_X(t)-\varphi_Y(t)|\,dt\le 2\int_{-2}^2|\varphi_X(t)-\varphi_Y(t)|\,dt=2I.  
\end{align*}

Now, note that for any $s\in\mb{R}$ we have
\[\Pr[|X-s|\le 1]=\mb E[\mbm 1_{|X-s|\le 1}]\le \mb{E}[f(X-s)]\le \mb{E}[f(Y-s)] + |\mb{E}[f(X-s)-f(Y-s)]|\le \mb{E}[f(Y-s)]+2I,\]
and therefore
\begin{align}
    \Pr[|X-s|\le 1]\le  \mb{E}[f(Y-s)]+2I&\le \sum_{j\in\mb{Z}}\frac{8}{j^2+1}\Pr[|Y-s-j|\le 1] + 2I\label{eq:esseen-first-step}\\
    &\le \mc L(Y,1)\sum_{j\in\mb Z}\frac8{j^2+1}+2I\le 40\cdot \mc L(Y,1)+2I.\notag
\end{align}
Thus, $\mc L(X,1)\le 40\cdot \mc L(Y,1)+2I\lesssim L(Y,1)+I$, as desired.
\end{proof}

Next, we will need a slightly more sophisticated exponentially decaying \emph{non-uniform} version of \cref{lem:esseen-upper-crude}.

\begin{lemma}\label{lem:esseen-upper}
Let $X,Y$ be real random variables. Suppose that for some $0<\eta<1$ and $0<\eps\le\sigma$ we have
\[\Pr[|Y-x|\le \eps]\le \frac{\eps}{\eta\sigma}\exp(-\eta|x|/\sigma)\]
for all $x\in\mb{R}$.
Then for all $x\in\mb{R}$,
\[\Pr[|X-x|\le \eps]\lesssim \frac{\eps^2}{x^2 + \sigma^2}+ \frac{\eps}{\eta\sigma}\exp(-\eta|x|/(2\sigma))+\eps\int_{-2/\eps}^{2/\eps}|\varphi_X(t)-\varphi_Y(t)|\,dt.\]
\end{lemma}
\begin{proof}
As in \cref{lem:esseen-upper-crude}, we may assume that $\eps = 1$, and let us again write $I := \int_{-2}^2|\varphi_X(t)-\varphi_Y(t)|\,dt$. Note that the assumption in the lemma statement implies $\mc L(Y,1)\le 1/(\eta\sigma) \le e\cdot 1/(\eta\sigma)\cdot \exp(-\eta/2)$. So if $|x|\le \sigma$, the desired bound follows from \cref{lem:esseen-upper-crude}. Otherwise, if $|x|\ge \sigma$, then \cref{eq:esseen-first-step} implies
\begin{align*}\Pr[|X-x|\le 1]&\lesssim \sum_{j\in\mb{Z}}\frac{\Pr[|Y-x-j|\le 1]}{j^2+1} + I\\
&=\sum_{\substack{j\in\mb{Z}\\|j+x|\ge|x|/2}}\frac{\Pr[|Y-x-j|\le 1]}{j^2+1} + \sum_{\substack{j\in\mb{Z}\\|j+x|< |x|/2}}\frac{\Pr[|Y-x-j|\le 1]}{j^2+1} + I\\
&\le \sup_{\substack{y\in\mb{R}\\|y|\ge |x|/2}}\Pr[|Y-y|\le 1]\cdot  \sum_{j\in\mb Z}\frac1{j^2+1}+\sum_{\substack{j\in\mb{Z}\\|j-(-x)|<|x|/2}}\frac{\Pr[|Y-x-j|\le 1]}{(x/2)^2+1} + I,\\
&\lesssim \frac{\eps}{\eta\sigma}\exp(-\eta|x|/(2\sigma))+\frac{1}{x^2+1}+I
\end{align*}
from which the desired result follows (using that $x^2+1\geq x^2\gtrsim x^2+\sigma^2$ since we assumed $|x|\ge\sigma$).
\end{proof}

It turns out that these ideas are not only useful for anticoncentration; we can also derive \emph{lower bounds} on the probability that $X$ is close to some point $x$, given local control over the behavior of $Y$ near $x$.

\begin{lemma}\label{lem:esseen-lower}
There is an absolute constant $C_{\ref{lem:esseen-lower}}$ such that the following holds. Let $X,Y$ be real random variables, and suppose $Y$ is continuous with a density function $p_Y$. Let $\varepsilon>0$ and $x\in\mb{R}$ and suppose that $K\ge 1$ and $R\geq 4$ are such that $p_Y(y_1)/p_Y(y_2)\le K$ for all $ y_1,y_2\in [x-R\eps,x+ R\eps]$. Then
\[\Pr[|X-x|\le 10^4K\eps]\ge \frac{1}{8}\Pr[|Y-x|\le \eps]-C_{\ref{lem:esseen-lower}}\bigg(R^{-1}\mc{L}(Y,\eps)+\eps\int_{-2/\eps}^{2/\eps}|\varphi_Y(t)-\varphi_X(t)|\,dt\bigg).\]
\end{lemma}

The reader may think of $K$ as a constant (in our applications of this lemma, we will take $K=2$). We remark that it would be possible to state a cruder version of this lemma with no assumption on the density $p_Y$. This would be sufficient to prove a version of \cref{thm:short-interval} where $B$ also depends on $A$ and $H$ (in addition to depending on $C$), but this would not be enough for the proof of \cref{thm:point-control} (for technical reasons discussed in \cref{rem:dependence-B}).

\begin{proof}
It again suffices to prove the claim when $\eps = 1$. Let the function $f$ and $I:= \int_{-2}^2|\varphi_X(t)-\varphi_Y(t)|\,dt$ be as in the proof of \cref{lem:esseen-upper-crude}, 
and recall that $\mbm 1_{[-1,1]}(t)\le f(t)\le \min\{4,4/t^2\}\le 8/(t^2+1)$ for all $t\in\mb{R}$ and furthermore $|\mb E[f(X-x)]-\mb E[f(Y-x)]|\le 2I$. We have
\begin{align}
\Pr[|X-x|\le 10^4K]
&\ge\frac{1}{4}\mb{E}[f(X-x)\mbm{1}_{|X-x|\le 10^4K}]=\frac{1}{4}\mb{E}[f(X-x)]-\frac{1}{4}\mb{E}[f(X-x)\mbm{1}_{|X-x|> 10^4K}]\notag\\
&\ge\frac{1}{4}\mb{E}[f(Y-x)] - \frac{I}{2} -\frac{1}{4}\mb{E}[f(X-x)\mbm{1}_{|X-x|> 10^4K}]\notag\\
&\ge \frac{1}{4}\Pr[|Y-x|\le 1]-\frac{I}{2}-\sum_{\substack{j\in\mb Z\\|j|\ge 9999K}}\frac{2}{j^2+1}\Pr[|X-x-j|\le 1].\label{eq:esseen-lower}
\end{align}
As in \cref{eq:esseen-first-step}, we have 
\[\Pr[|X-x-j|\le 1]\le \sum_{k\in\mb{Z}}\frac{8}{k^2+1}\Pr[|Y-x-j-k|\le 1] + 2I,\]
so
\begin{align*}
&\sum_{\substack{j\in\mb Z\\|j|\ge 9999K}}\frac{2}{j^2+1}\Pr[|X-x-j|\le 1]\le 16\sum_{\substack{j,k\in\mb Z\\|j|\ge 9999K}}\frac{\Pr[|Y-x-j-k|\le 1]}{(j^2+1)(k^2+1)}+2\left(\sum_{j\in\mb Z}\frac{2}{j^2+1}\right)I\\
&\qquad\qquad\le 16\sum_{\substack{j,k\in\mb Z\\9999K\le|j|\le (R-1)/2\\|k|\le (R-1)/2}}\frac{K\Pr[|Y-x|\le 1]}{(j^2+1)(k^2+1)} + 16\sum_{\substack{j,k\in \mathbb{Z}\\\max\{|j|,|k|\}> (R-1)/2}}\frac{\mc{L}(Y,1)}{(j^2+1)(k^2+1)} + 20I\\
&\qquad\qquad\le 16K\cdot \Pr[|Y-x|\le 1]\cdot 5\cdot \frac{2}{9999K-1} + 16\cdot 2\cdot 5\cdot \frac{2}{(R-3)/2}\cdot \mc{L}(Y,1) + 20I\\
&\qquad\qquad\le \frac{1}{8}\Pr[|Y-x|\le 1]+O(R^{-1})\cdot \mc{L}(Y,1)+O(I),
\end{align*}
where we used that $\sum_{j\in\mb Z} 1/(j^2+1)\leq 5$ and $\sum_{j\in\mb Z, |j|\geq T} 1/(j^2+1)\leq 2\sum_{j\in\mb Z, j\geq T} 1/(j(j-1)) \le  2/(T-1)$ for $T > 1$. Plugging this into \cref{eq:esseen-lower} gives the desired result.
\end{proof}

\section{Characteristic function estimates based on linear cancellation}\label{sec:initial-fourier}
Consider $X$ as in \cref{thm:short-interval}, and let $X^*=(X-\mb E X)/\sigma(X)$.
When $t$ is not too large, we can prove estimates on $\varphi_{X^*}(t)$ purely using the linear behavior of $X$ (treating the quadratic part as an ``error term''). In this section we prove two different results of this type.

First, when $t$ is very small, there is essentially no cancellation in $\varphi_{X^*}(t)$, and we have the following crude estimate. Roughly speaking, we use the simple observation (from \cref{subsec:multiple-scales}) that $X$ can be interpreted as a sum of independent random variables (a ``linear part''), plus a ``quadratic part'' with negligible variance. We can then use standard estimates for characteristic functions of sums of independent random variables.

\begin{lemma}\label{prop:lowest-t}
Fix $\eps, H > 0$. Let $G$ be an $n$-vertex graph with density at least $\varepsilon$, and consider $e_0\in\mb{R}$ and a vector $\vec{e}\in\mb{R}^{V(G)}$ with $0\le e_v\le Hn$ for all $v\in V(G)$. Let $U\subseteq V(G)$ be a random vertex subset obtained by including each vertex with probability $1/2$ independently, and let $X = e(G[U]) + \sum_{v\in U}e_v + e_0$. Let $X^*=(X-\mb E X)/\sigma(X)$, and let $Z\sim \mc N(0,1)$ be a standard normal random variable. Then, for all $t\in\mb{R}$, we have
\[|\varphi_{X^*}(t)-\varphi_Z(t)|\lesssim_{\eps,H}|t|n^{-1/2}.\]
\end{lemma}

We remark that on its own \cref{prop:lowest-t} implies a central limit theorem (stating that $X$ is asymptotically Gaussian) by L\'evy’s continuity theorem (see for example \cite[Theorem~3.3.17]{Dur19}).

\begin{proof}
Define the random vector $\vec x\in \{-1,1\}^{V(G)}$ by taking $x_v=1$ if $v\in U$, and $x_v=-1$ if $v\notin U$ (so $x_v$ for $v\in V(G)$ are independent Rademacher random variables). Then, we compute
\begin{align*}
X &= e_0+\frac{e(G)}{4}+\frac{1}{2}\sum_{v\in V(G)} e_v+\frac{1}{2}\sum_{v\in V(G)}\Big(e_v+\frac{1}{2}\deg_G(v)\Big)x_v + \frac{1}{4}\sum_{uv\in E(G)}x_ux_v\\
&=\mb E X+\frac{1}{2}\sum_{v\in V(G)}\left(e_{v}+\frac{1}{2}\deg_{G}(v)\right)x_{v}+\frac{1}{4}\sum_{uv\in E(G)}x_{u}x_{v},
\end{align*}
as in \cref{eq:fourier-walsh}. Defining $d_v = e_v+\deg_G(v)/2$ for $v\in V(G)$, we deduce that
\[X-\mb{E}X=\frac{1}{2}\vec{d}\cdot\vec{x}+\frac{1}{4}\sum_{uv\in E(G)}x_ux_v.\]
That is to say, $X-\mb E X$ has a ``linear part'' $\frac{1}{2}\vec d\cdot \vec x$ and a ``quadratic part'' $\frac{1}{4}\sum_{uv\in E(G)}x_ux_v$. Recalling \cref{eq:boolean-variance}, we have $\sigma(X)^2 = \frac{1}{4}\snorm{\vec{d}}_2^2 + \frac{1}{16}e(G)\geq\frac{1}{4}\snorm{\vec{d}}_2^2\ge \frac{1}{4}\snorm{\vec{d}}_1^2/n \gtrsim_\varepsilon n^3$ (here we are using our density assumption as well as the assumption that $e_v\geq 0$ for all $v\in V(G)$).

First, we compare $X^*=(X-\mb E X)/\sigma(X)$ to its linear part $(\vec d\cdot \vec x)/(2\sigma(X))$. For all $t\in \mathbb{R}$, we have  $|\exp(it)-1|\le|t|$ and therefore
\begin{align}
\Big|\varphi_{X^*}(t)&-\mb{E}[e^{it(\vec{d}\cdot\vec{x})/(2\sigma(X))}]\Big|\le \mb{E}\Big|\exp\Big(\frac{it}{4\sigma(X)}\sum_{uv\in E(G)}x_ux_v\Big)-1\Big|\le\frac{|t|}{4\sigma(X)}\mb{E}\Big|\sum_{uv\in E(G)}x_ux_v\Big|\notag\\
&\le\frac{|t|}{4\sigma(X)}\left(\mb{E}\Big[\Big(\sum_{uv\in E(G)}x_ux_v\Big)^2\Big]\right)^{1/2}=\frac{|t|}{4\sigma(X)}\cdot e(G)^{1/2}\leq \frac{|t|}{\Omega_{\eps}(n^{3/2})}\cdot n\lesssim_\eps |t|n^{-1/2}.\label{eq:low-t-linearization}
\end{align}

Next, the linear part can be handled as in a standard proof of a quantitative central limit theorem (c.f.~\cref{lem:gaussian-tricks}). Let $\sigma_1 = \sigma(\vec{d}\cdot\vec{x})=\snorm{\vec{d}}_2$ and $\Gamma = (\sum_{v\in V(G)}d_v^2)^{3/2}/\sum_{v\in V(G)}d_v^3\gtrsim_H \snorm{\vec{d}}_2^3/n^4\gtrsim_{\varepsilon,H} n^{1/2}$ (recalling that $\snorm{\vec{d}}_2^2 \gtrsim_\varepsilon n^3$), and note that $\varphi_Z(u) = e^{-u^2/2}$. For $|u|\le\Gamma/4$, we have
\[\Big|\mb{E}[e^{iu(\vec{d}\cdot\vec{x})/\sigma_1}]-\varphi_Z(u)\Big|\le 16\Gamma^{-1}|u|^3e^{-u^2/3}\]
by \cite[Chapter~V,~Lemma~1]{Pet75}. This yields
\[\Big|\mb{E}[e^{iu(\vec{d}\cdot\vec{x})/\sigma_1}]-\varphi_Z(u)\Big|\lesssim_{\eps,H}|u|n^{-1/2}\]
for all $u\in\mb{R}$ (this is trivial for $|u|\ge \Gamma/4 \gtrsim_{\eps,H} n^{1/2}$). Taking $u = t\sigma_1/(2\sigma(X))$ and using $\sigma_1/(2\sigma(X)) = \snorm{\vec{d}}_2/(\snorm{\vec{d}}_2^2 + \frac{1}{4}e(G))^{1/2}=1-O_\eps(n^{-1})$, we have
\begin{equation}\label{eq:low-t-berry-esseen}
\Big|\mb{E}[e^{it(\vec{d}\cdot\vec{x})/(2\sigma(X))}]-\varphi_Z(t)\Big|\le
\Big|\mb{E}[e^{iu(\vec{d}\cdot\vec{x})/\sigma_1}]-\varphi_Z(u)\Big| + |\varphi_Z(u)-\varphi_Z(t)|
\lesssim_{\eps,H}|t|n^{-1/2}.
\end{equation}
Here, we used that the function $\varphi_Z(u)=e^{-u^2/2}$ has bounded derivative, and therefore $|\varphi_Z(u)-\varphi_Z(t)|\lesssim |u-t|=|\sigma_1/(2\sigma(X))-1|\cdot |t|=O_\eps(n^{-1}|t|)$. The desired inequality now follows from \cref{eq:low-t-linearization,eq:low-t-berry-esseen}.
\end{proof}

As mentioned above, \cref{prop:lowest-t} will be used for very small $t$. When $t$ is somewhat larger we will need a stronger bound which takes into account the interaction between the linear and quadratic parts of our random variable. Specifically, writing $Z_1$ and $Z_2$ for the linear and quadratic parts of our normalized random variable $X^*$, we show that  $e^{itZ_2}$ does not ``correlate adversarially'' with $e^{itZ_1}$, using an argument due to Berkowitz~\cite{Ber18}. Roughly speaking, the idea is as follows. Considering $\vec x\in \{-1,1\}^{V(G)}$ as in the proof of \cref{prop:lowest-t}, we can apply Taylor's theorem to the exponential function to approximate $e^{itZ_2}$ by a polynomial in $Z_2$, thereby approximating $\varphi_{X^*}(t)$ by a sum of terms of the form $\mb E[\prod_{i\in S}x_Se^{itZ_1}]$ (where the sets $S$ are rather small). Then, we observe that it is impossible for terms of the form $\prod_{i\in S}x_S$ to correlate in a pathological way with $e^{itZ_1}$, because all but $|S|$ of the terms in the ``linear'' random variable $Z_1$ are independent from $\prod_{i\in S}x_S$. We can use this observation to prove very strong upper bounds on the magnitude of each of our terms $\mb E[\prod_{i\in S}x_Se^{itZ_1}]$ (we do not attempt to understand any potential cancellation between these terms, but the resulting loss is not severe as there are not many choices of $S$).

In some range of $t$, the above idea can be used to prove a much stronger bound than in \cref{prop:lowest-t} (where we obtained a bound of $|t|n^{-1/2}$). However, na\"ively, this idea is only suitable in the regime $|t|\lesssim \sqrt n$, for two reasons. The first reason is that (one can compute that) the typical order of magnitude of $Z_2$ is about $1/\sqrt n$, so a Taylor series approximation for $e^{itZ_2}$ becomes increasingly ineffective as $|t|$ increases past $\sqrt n$. The second reason is that depending on the structure of our graph $G$ it is possible that $|\varphi_{Z_1}(\Theta(\sqrt n))|\gtrsim 1$, meaning that consideration of the linear part of $X^*$ simply does not suffice to prove our desired bound on $\varphi_{X^*}(t)$ (for example, this occurs when $\vec e=\vec 0$ and $G$ is regular).

In order to overcome the first of these issues, we restrict our attention to a small vertex subset $I$, taking advantage of the different way that the linear and quadratic parts scale (related ideas appeared previously in \cite{Ber16}) . Specifically, we condition on an outcome of the vertices sampled outside $I$, leaving only the randomness within $I$ (corresponding to the sequence $\vec x_I\in \{-1,1\}^I$). We then redefine $Z_1$ and $Z_2$ to be the linear and quadratic parts of the \emph{conditional} random variable $X^*$ (as a quadratic polynomial in $\vec x_I$). Dropping to a subset in this way significantly reduces the variance of $Z_2$, but may have a much milder effect on $Z_1$, in which case the above Taylor expansion techniques described above are effective.

The second issue is more fundamental, and is essentially the reason for the case distinction in our proof of \cref{thm:short-interval} (recall \cref{sub:additive-structure-dichotomy}). Specifically, the range of $t$ which we are able to consider depends on a certain RLCD (recall the definitions in \cref{sub:RLCD}).

\begin{lemma}\label{prop:linear-RLCD}
Fix $C,H > 0$ and $0<\gamma<1/4$, and let $L = \lceil 100/\gamma\rceil$. Then there is $\alpha=\alpha(C,H,\gamma)>0$ such that the following holds. Let $G$ be a $C$-Ramsey graph with $n$ vertices, where $n$ is sufficiently large with respect to $C, H$, and $\gamma$, and consider $e_0\in\mb{R}$ and a vector $\vec{e}\in\mb{R}^{V(G)}$ with $0\le e_v\le Hn$ for all $v\in V(G)$. Let $\vec{d}\in\mb{R}^{V(G)}$ be given by $d_v = e_v + \deg_G(v)/2$ for all $v\in V(G)$. Next, let $U\subseteq V(G)$ be a random vertex subset obtained by including each vertex with probability $1/2$ independently, and define $X = e(G[U]) + \sum_{v\in U}e_v+e_0$. Let $X^*=(X-\mb E X)/\sigma(X)$. Then for any $t\in\mb{R}$ with
\[n^{2\gamma}\le |t|\le \alpha\cdot \min\{n^{\gamma/2}\wh{D}_{L,\gamma}(\vec{d}),\; n^{1/2+\gamma/8}\}.\]
we have
\[|\varphi_{X^*}(t)|\lesssim_{C,H,\gamma}n^{-5}.\]
\end{lemma}

Before proving \cref{prop:linear-RLCD}, we record a simple fact about the vector $\vec{d}$ in the lemma statement.

\begin{lemma}\label{lem:two-norm-subsets-vec-d}
Fix $C>0$ and let $G$ be a $C$-Ramsey graph with $n$ vertices, where $n$ is sufficiently large with respect to $C$. Consider a vector $\vec{e}\in\mb{R}_{\ge 0}^{V(G)}$ and define $\vec{d}\in\mb{R}^{V(G)}$ by $d_v = e_v + \deg_G(v)/2$ for all $v\in V(G)$. Then for any subset $I\su V(G)$ of size $|I|\geq \sqrt{n}$, we have $\snorm{\vec{d}_I}_2\gtrsim_C |I|^{3/2}$.
\end{lemma}
\begin{proof}
Note that $G[I]$ is a $(2C)$-Ramsey graph, so by \cref{thm:dense-ramsey} we have $e(G[I])\gtrsim_C |I|^2$. Thus,
\[\snorm{\vec{d}_I}_2^2= \sum_{v\in I}\left(e_v+\frac12 \deg_G(v)\right)^2\geq \sum_{v\in V}( \deg_{G[I]}(v)/2)^2\ge |I|\cdot \left(\frac{e(G[I])}{|I|}\right)^2 \gtrsim_C |I|^3.\qedhere\]
\end{proof}

Note that this lemma in particular implies that in the setting of \cref{prop:linear-RLCD} the vector $\vec{d}$ has fewer than $n^{1-\gamma}$ zero coordinates, meaning that $\wh{D}_{L,\gamma}(\vec{d})$ is well-defined (recall \cref{def:RLCD}).

In the proof of \cref{prop:linear-RLCD}, we will also use the following Taylor series approximation for the exponential function.

\begin{lemma}\label{lem:taylor}
For all $z\in\mb{C}$ and $K\in\mb{N}$, we have
\[\bigg|e^z - \sum_{j=0}^K\frac{z^j}{j!}\bigg| \le e^{\max\{0,\Re (z)\}}\frac{|z|^{K+1}}{K!}.\]
\end{lemma}
\begin{proof}
This follows from Taylor's theorem with the integral form for the remainder: note that
\[\bigg|\int_{0}^{z}e^{t}(z-t)^{K}\,dt\bigg|=|z|^{K+1}\bigg|\int_{0}^{1}e^{sz}(1-s)^{K}\,ds\bigg|\le e^{\max\{0,\Re (z)\}} |z|^{K+1}.\qedhere\]
\end{proof}

Now we prove \cref{prop:linear-RLCD}.

\begin{proof}[Proof of \cref{prop:linear-RLCD}]
Let us define $\vec x\in \{-1,1\}^{V(G)}$ by taking $x_v=1$ if $v\in U$, and $x_v=-1$ if $v\notin U$ (and note that then $\vec{x}$ is a vector of independent Rademacher random variables). As in the proof of \cref{prop:lowest-t}, we obtain $X-\mb{E}X=\frac{1}{2}\vec{d}\cdot\vec{x}+\frac{1}{4}\sum_{uv\in E(G)}x_ux_v$ and $\sigma(X) \gtrsim_C n^{3/2}$ (here, we used that by \cref{thm:dense-ramsey} the graph $G$ has density at least $\eps$ for some $\eps=\eps(C)>0$ only depending on $C$). We furthermore have $\sigma(X) = (\frac{1}{4}\snorm{\vec{d}}_2^2+\frac{1}{16}e(G))^{1/2}\lesssim_H n^{3/2}$.

By the definition of RLCD (\cref{def:RLCD}), there is a subset $I\subseteq V(G)$ of size $|I|=\lceil n^{1-\gamma}\rceil$ such that 
\[\wh{D}_{L,\gamma}(\vec{d})=D_L(\vec{d}_I/\snorm{\vec{d}_I}_2).\]

\medskip
\noindent\textit{Step 1: Reducing to the randomness of $\vec x_I$.}
The first step is to condition on a typical outcome of $\vec x_{V(G)\setminus I}\in \{-1,1\}^{V(G)\setminus I}$, so that we can work purely with the randomness of $\vec x_I\in \{-1,1\}^I$. Define the vector $\vec y\in \mathbb R^I$ by taking
\[y_v =\frac 14\sum_{\substack{u\in V(G)\setminus I\\uv\in E(G)}} x_u\]
for each $v\in I$. Also, let
\[Z_1=\Big(\frac{1}{2}\vec d_I+\vec y\Big)\cdot \vec x_I,\quad\quad Z_2=\frac{1}{4}\sum_{\substack{u,v\in I\\uv\in E(G)}}x_ux_v.\]
Note that $X-\mb E[X|\vec x_{V(G)\setminus I}]=Z_1+Z_2$. Using the fact that $|\mb E [e^{it (Y+c)}]|=|\mb E [e^{it Y}]|$ for any real random variable $Y$ and non-random $c\in\mb{R}$, we have
\begin{equation*}
|\varphi_{X^*}(t)|=|\mb{E}[e^{itX/\sigma(X)}]| \le \mb{E}|\mb{E}[e^{itX/\sigma(X)}|\vec{x}_{V(G)\setminus I}]| =\mb{E}\left|\mb{E}\left[\exp\left(\frac{it(Z_1+Z_2)}{\sigma(X)}\right)\middle|\vec{x}_{V(G)\setminus I}\right]\right|.
\end{equation*}

The inner expectation on the right-hand side always has magnitude at most 1. Since $\deg_G(v)\leq n$ for $v\in I$, with a Chernoff bound we see that with probability at least $1-\exp(-\Omega(n^{\gamma/4}))$ we have $|y_v|\le n^{1/2+\gamma/8}$ for all $v\in I$. Conditioning on a fixed outcome of $\vec x_{V(G)\setminus I}$ such that this is the case, it now suffices to show that
\begin{equation}\label{eq:to-show-linear-init}
\left|\mb{E}\left[\exp\left(\frac{it(Z_1+Z_2)}{\sigma(X)}\right)\right]\right|\lesssim_{C,H,\gamma}n^{-5}
\end{equation}
for all $t\in \mathbb{R}$ with $n^{2\gamma}\le |t|\le \alpha\cdot \min\{n^{\gamma/2}\wh{D}_{L,\gamma}(\vec{d}), n^{1/2+\gamma/8}\}$, where $\alpha=\alpha(C,H,\gamma)>0$ is chosen sufficiently small (in particular, we may assume $\alpha<1$).

\medskip
\noindent\textit{Step 2: Taylor expansion. }
Let $K = \lceil 10/\gamma\rceil$. By \cref{lem:taylor} we have
\begin{align}
\bigg|\mb{E}\bigg[\exp\!\bigg(\frac{it(Z_1+Z_2)}{\sigma(X)}\bigg)\bigg]\bigg| &= \bigg|\mb{E}\bigg[\exp\!\bigg(\frac{itZ_1}{\sigma(X)}\bigg)\exp\!\bigg(\frac{itZ_2}{\sigma(X)}\bigg)\bigg]\bigg|\notag\\
&\le \bigg|\mb{E}\bigg[\exp\!\bigg(\frac{itZ_1}{\sigma(X)}\bigg)\sum_{j=0}^{K}\frac{1}{j!}\bigg(\frac{itZ_2}{\sigma(X)}\bigg)^{j}\bigg]\bigg| + \mb{E}\bigg[\frac{1}{K!}\bigg(\frac{|tZ_2|}{\sigma(X)}\bigg)^{K+1}\bigg]\label{eq:sum-to-bound-linear-init}
\end{align}
Recalling that $|I|=\lceil n^{1-\gamma}\rceil$ and our assumption that $|t|\leq n^{1/2+\gamma/8}$, we have
\[\mb{E}[(tZ_2/\sigma(X))^2]=\frac{t^2}{\sigma(X)^2}\cdot \mb{E}[Z_2^2]\le \frac{t^2}{\sigma(X)^2}\cdot |I|^2\lesssim_{C} \frac{n^{1+\gamma/4}}{n^3}\cdot n^{2-2\gamma}= n^{-7\gamma/4}.\]
By \cref{thm:gauss-moment} (hypercontractivity), we deduce $\mb{E}[(|tZ_2|/\sigma(X))^{K+1}]\lesssim_{C,\gamma} n^{-7\gamma(K+1)/8}$. Thus, using that $(K+1)\gamma\geq 10$, we obtain
\begin{equation}\label{eq:error-term-bound}
\mb{E}\bigg[\frac{1}{K!}\bigg(\frac{|tZ_2|}{\sigma(X)}\bigg)^{K+1}\bigg]\lesssim_{C,\gamma}n^{-5}.
\end{equation}
Also, note that $\sum_{j=0}^{K}\frac{1}{j!}(itZ_2/\sigma(X))^{j}$ is a polynomial of degree $2K$ in $\vec x_I$. Noting that $x_v^2 = 1$ for all $v$, one can represent this polynomial as a linear combination of at most $|I|^{2K}<n^{2K}$ multilinear monomials $\prod_{v\in S}x_v$ with $|S|\leq 2K$. The coefficient of each such monomial has absolute value $O_{C,\gamma}(1)$, recalling that $|t|\le n^{1/2+\gamma/8}$ and $\sigma(X)=\Omega_C(n^{3/2})$ and $|I|=\lceil n^{1-\gamma}\rceil$ (and $K=\lceil 10/\gamma\rceil$). For the rest of the proof, our goal is now to show that for any set $S\subseteq I$ with $|S|\le 2K$ we have
\begin{equation}\label{eq:per-term-goal}
\bigg|\mb{E}\bigg[\exp\!\bigg(\frac{itZ_1}{\sigma(X)}\bigg)\prod_{v\in S}x_v\bigg]\bigg|\lesssim_{C,H,\gamma} n^{-5-2K}.
\end{equation}
The desired bound \cref{eq:to-show-linear-init} will then follow from \cref{eq:sum-to-bound-linear-init}, bounding the first summand by summing \cref{eq:per-term-goal} over all choices of $S$ and bounding the second summand via \cref{eq:error-term-bound}.

\medskip
\noindent\textit{Step 3: Relating to the LCD. } So let us fix some subset $S\subseteq I$ with $|S|\le 2K$. Let $\vec f=\frac{1}{2}\vec d_I+\vec y\in\mb{R}^I$, so $Z_1=\vec f\cdot \vec x_I$. Noting that $|x_v|\le 1$ for all $v\in I$, and using \cref{eq:cos}, we have
\begin{align}
\bigg|\mb{E}\bigg[\exp\!\bigg(\frac{itZ_1}{\sigma(X)}\bigg)\prod_{v\in S}x_v\bigg]\bigg|
&= \bigg|\mb{E}\bigg[\prod_{v\in I\setminus S}\exp\!\bigg(\frac{itf_vx_v}{2\sigma(X)}\bigg)\cdot \prod_{v\in  S}\exp\!\bigg(\frac{itf_vx_v}{2\sigma(X)}\bigg)x_v\bigg]\bigg|
\le \prod_{v\in I\setminus S} \bigg|\mb{E}\bigg[\exp\!\bigg(\frac{itf_vx_v}{2\sigma(X)}\bigg)\bigg]\bigg|\notag\\
&\le \exp\left(-\sum_{v\in I\setminus S}\norm{\frac{tf_v}{2\pi\sigma(X)}}_{\mb{R}/\mb{Z}}^2\right)\le\exp\left(|S|-\on{dist}\!\bigg(\frac{|t|\vec f}{2\pi\sigma(X)},\mb{Z}^{I}\bigg)^2\right).
\label{eq:summand-to-bound}
\end{align}
(Here we used that for any $\vec a\in\mb{R}^I$ we have $\sum_{v\in I\setminus S}\|a_v\|_{\mb{R}\setminus \mb Z}^2=\on{dist} (\vec a_{I\setminus S},\mb Z^{I\setminus S})^2\ge\on{dist}(\vec a_{I},\mb Z^{I})^2-|S|$.)

Since $|t|\le n^{1/2+\gamma/8}$ and $\sigma(X)=\Omega_C(n^{3/2})$ and we are conditioning on $\vec x_{V(G)\setminus I}$ such that $|y_v|\le n^{1/2+\gamma/8}$ for all $v\in I$, we have (using that $|I|=\lceil n^{1-\gamma}\rceil$)
\[\frac{|t|\snorm{\vec{y}}_2}{2\pi\sigma(X)}\lesssim_{C} \frac{n^{1/2+\gamma/8}\cdot (|I|^{1/2})\cdot n^{1/2+\gamma/8}}{n^{3/2}}\lesssim n^{-\gamma/4},\]
and therefore $|t|\snorm{\vec{y}}_2/(2\pi\sigma(X))\le 1$ for sufficiently large $n$. By our assumption $|t|\le \alpha n^{\gamma/2}\hat{D}_{L,\gamma}(\vec{d})=\alpha n^{\gamma/2}D_{L}(\vec{d}_I/\|\vec d_I\|_2)$, we have
\[\frac{|t|\snorm{\vec{d}_I}_2}{4\pi\sigma(X)}\lesssim_{C,H} \frac{\alpha n^{\gamma/2}D_{L}(\vec{d}_I/\|\vec d_I\|_2)\cdot |I|^{1/2}\cdot n}{n^{3/2}} \lesssim \alpha D_{L}(\vec{d}_I/\|\vec d_I\|_2).\]
Hence, by choosing $\alpha=\alpha(C,H,\gamma)>0$ to be sufficiently small in terms of $C$, $H$, and $\gamma$, for sufficiently large $n$ we obtain $|t|\snorm{\vec{d}_I}_2/(4\pi\sigma(X))<D_{L}(\vec{d}_I/\|\vec d_I\|_2)$ and therefore
\begin{align}
\on{dist}\!\bigg(\frac{|t|\vec f}{2\pi\sigma(X)},\mb{Z}^{I}\bigg)\ge \on{dist}\!\bigg(\frac{|t|(\vec d_I/2)}{2\pi\sigma(X)},\mb{Z}^{I}\bigg) - \frac{|t|\snorm{\vec{y}}_2}{2\pi\sigma(X)}&\ge\on{dist}\!\bigg(\frac{|t|\snorm{\vec{d}_I}_2}{4\pi\sigma(X)}\cdot \frac{\vec{d}_I}{\snorm{\vec{d}_I}}_2,\mb{Z}^{I}\bigg) - 1\notag\\
&\ge L\sqrt{\log_{+}\bigg(\frac{|t|\snorm{\vec{d}_I}_2}{4\pi L\sigma(X)}\bigg)} - 1\label{eq:lower-bound-from-LCD}
\end{align}
where we applied the definition of LCD (see \cref{def:LCD}). Now, $|t|\snorm{\vec{d}_I}_2/(4\pi L \sigma(X))\gtrsim_{C,H,\gamma} n^{\gamma/2}$, since $|t|\ge n^{2\gamma}$ and $\sigma(X)\lesssim_H n^{3/2}$ and $\snorm{\vec{d}_I}_2\gtrsim_C|I|^{3/2}\gtrsim n^{(3/2)-3\gamma/2}$ by \cref{lem:two-norm-subsets-vec-d}. Thus, for sufficiently large $n$, we have $|t|\snorm{\vec{d_I}}_2/(4\pi L \sigma(X))\ge n^{\gamma/4}$, and therefore the term \cref{eq:lower-bound-from-LCD} is at least $L\sqrt{\log_{+}(n^{\gamma/4})}-1\geq (L/2)\sqrt{\log_{+}(n^{\gamma/4})}$. Then, recalling that $L=\lceil100/\gamma\rceil$ and $K = \lceil 10/\gamma\rceil$ and $|S|\leq 2K$, it follows that
\[ \on{dist}\!\bigg(\frac{|t|\vec f}{2\pi\sigma(X)},\mb{Z}^{I}\bigg)^2 \ge \bigg(\frac{L}{2}\sqrt{\log_{+}(n^{\gamma/4})}\bigg)^2 \ge \frac{10^{4}}{4\gamma^2}\cdot \frac{\gamma}{4}\log n\ge (4K+5)\log n\ge |S|+(2K+5)\log n.\]
Combining this with \cref{eq:summand-to-bound}, we obtain the desired inequality \cref{eq:per-term-goal}.
\end{proof}

\section{Characteristic function estimates based on quadratic cancellation}\label{sec:high-fourier}
In \cref{sec:initial-fourier}, we proved some bounds on the characteristic function of a random variable $X$ of the form $X=e(G[U])+\sum_{v\in U}e_{v}+e_0$ purely using the linear part of $X$. In this section we prove a bound which purely uses the \emph{quadratic} part of $X$ (this will be useful for larger $t$).

In the setting and notation of \cref{sec:initial-fourier}, the regime where this result is effective corresponds to a range where $|t|$ is roughly between $n^{1/2+\Omega(1)}$ and $n^{3/2}$. However, the bounds in this section will need to be applied in two slightly different settings (recalling from \cref{sub:additive-structure-dichotomy} that the proof of \cref{thm:short-interval} bifurcates into two cases). To facilitate this, we consider random variables $X$ of a slightly different type than in \cref{sec:initial-fourier}: instead of studying the number of edges in a uniformly random vertex subset, we study the number of edges in a uniformly random vertex subset \emph{of a particular size}. We can interpret this as studying a conditional distribution, where we condition on an outcome of the number of vertices of our random subset (if desired, we can deduce bounds in the unconditioned setting simply by averaging over all possible outcomes).

We remark that in this setting where our random subset has a fixed size, it is no longer true that the standard deviation $\sigma(X)$ must have order of magnitude $n^{3/2}$. Indeed, the order of magnitude of $\sigma(X)$ depends on $\vec e$ and the degree sequence of $G$. Therefore, it is more  convenient to study the characteristic function of $X$ directly, instead of its normalized version $X^*=(X-\mb EX)/\sigma(X)$. To avoid confusion, we will use the variable name ``$\tau$'' instead of ``$t$'' when working with characteristic functions of random variables that have not been normalized (so, informally speaking, the translation is that $\tau=t/\sigma(X)$).

\begin{lemma}\label{thm:quadratic-decoupling}
Fix $C>0$ and $0<\eta<1/2$.
There is $\nu=\nu(C,\eta)>0$ such that the following holds. Let $G$
be a $C$-Ramsey graph with $n$ vertices, where $n$ is sufficiently large with respect to $C$ and $\eta$, and consider a vector $\vec{e}\in\mathbb{R}^{V(G)}$ and $e_0\in\mb{R}$. Consider $\ell\in\mb N$ with $\eta n\le\ell\le(1-\eta)n$, and let $U$ be a uniformly
random subset of $\ell$ vertices in $G$, and let $X=e(G[U])+\sum_{v\in U}e_{v}+e_0$.
Then for any $\tau\in\mb{R}$ with $n^{-1+\eta}\le|\tau|\le\nu$ we have
\[|\varphi_{X}(\tau)|\le n^{-5}.\]
\end{lemma}

The proof of \cref{thm:quadratic-decoupling} depends crucially on \emph{decoupling} techniques. Generally speaking, such techniques allow one to reduce from dependent situations to independent ones (see \cite{PV96} for a book-length treatment). 
In our context, decoupling allows us to reduce the study of ``quadratic'' random variables to the study of ``linear'' ones. Famously, a similar approach was taken by Costello, Tao, and Vu~\cite{CTV06} to study singularity of random symmetric matrices.

To illustrate the basic idea of decoupling, consider an $n$-variable quadratic polynomial $f$ and a sequence of random variables $\vec \xi\in\mb{R}^n$. If $[n]=I\cup J$ is a partition of the index set into two subsets, then we can break $\vec\xi=(\xi_1,\ldots,\xi_n)$ into two subsequences $\vec\xi_{I}\in\mb{R}^I$ and $\vec\xi_{J}\in\mb{R}^J$ (and write $f(\vec\xi)=f(\vec\xi_{I},\vec \xi_{J})$). Let us assume that the random vectors $\vec\xi_{I}$ and $\vec\xi_{J}$ are independent. Now, if $\vec\xi_{J}'$ is an independent copy of $\vec\xi_{J}$, then $Y:=f(\vec\xi_{I},\vec\xi_{J})-f(\vec\xi_{I},\vec\xi_{J}')$, is a \emph{linear} polynomial in $\vec \xi_I$, after conditioning on any outcomes of $\vec \xi_J,\vec \xi_J'$ (roughly speaking, this is because ``the quadratic part in $\vec\xi_{I}$ gets cancelled out''). Then, for any $\tau\in\mb{R}$, we can use the inequality
\begin{align}
|\varphi_{f(\vec \xi)}(\tau)|^2=\left|\mb E e^{i\tau f(\vec\xi_I,\vec\xi_J)}\right|^{2}  &\le \mb E\left[\left|\mb E[e^{i\tau f(\vec\xi_{I},\vec\xi_{J})}\mid \vec\xi_{I}]\right|^2\right] = \mb E\left[\mb E[e^{i\tau (f(\vec\xi_{I},\vec\xi_{J})-f(\vec\xi_{I},\vec\xi_{J}'))}\mid \vec\xi_{I}]\right]\notag \\
& = \mb E\left[\mb E[e^{i\tau (f(\vec\xi_{I},\vec\xi_{J})-f(\vec\xi_{I},\vec\xi_{J}'))}\mid \vec\xi_{J},\vec\xi_{J}']\right]\notag\\
&\le \mb E\left[\left|\mb E[e^{i\tau (f(\vec\xi_{I},\vec\xi_{J})-f(\vec\xi_{I},\vec\xi_{J}'))}\mid \vec\xi_{J},\vec\xi_{J}']\right|\right].\label{eq:decoupling}
\end{align}
(This inequality appears as \cite[Lemma~3.3]{KS20b}; similar inequalities appear in \cite{Ber18,Ngu12}.) Crucially, the expression $\mb E[e^{i\tau (f(\vec\xi_{I},\vec\xi_{J})-f(\vec\xi_{I},\vec\xi_{J}'))}\mid \vec\xi_{J},\vec\xi_{J}']$ can be interpreted as an evaluation of the characteristic function of a linear polynomial in $\vec \xi_I$, which is easy to understand.

In general, \cref{eq:decoupling} incurs some loss (one generally obtains bounds which are about the square root of the truth). However, under certain assumptions about the degree-2 part of $f$, this square-root loss ``in Fourier space'' does not seriously affect the final bounds one gets ``in physical space''. Specifically, the first and third authors~\cite{KS20b} observed that it suffices to assume that the degree-2 part of $f$ ``robustly has high rank'', and observed that quadratic forms associated with Ramsey graphs always satisfy this robust high rank assumption (we will prove a similar statement in \cref{lem:ramsey-low-rank}).

Our proof of \cref{thm:quadratic-decoupling} will be closely related to the proof of the main result in \cite{KS20b}, although our approach is slightly different, as we need to take more care with quantitative aspects. In particular, instead of working with a qualitative robust-high-rank assumption we will directly make use of the fact that in any Ramsey graph, there are many disjoint tuples of vertices with very different neighborhoods (this can be interpreted as a particular sense in which the adjacency matrix of $G$ robustly has high rank).

\begin{lemma}\label{lem:good-tuples}
For any $C,\beta>0$, there is $\zeta=\zeta(C,\beta)>0$ such that the following holds for all sufficiently large $n$. Let $G$ be a $C$-Ramsey graph with $n$ vertices, and let $q=\lfloor\zeta \log n\rfloor$. Then there is a partition $V(G)=I\cup J$ and a collection $\mathcal{V}\subseteq I^{q}$ of at least $n^{1-\beta}$ disjoint $q$-tuples of vertices in $I$, such that for all $(v_1,\ldots,v_q)\in \mathcal{V}$ we have
\begin{equation}\label{eq:condition-q-tuple}
|J\setminus (N(v_1)\cup \cdots\cup N(v_r))|\geq n^{1-\beta}\quad\text{and}\quad |(J\cap N(v_r))\setminus (N(v_1)\cup\cdots\cup N(v_{r-1}))|\geq n^{1-\beta}
\end{equation}
for all $r=1,\ldots,q$.
\end{lemma}

\begin{proof}
By \cref{lem:rich-subset} (applied with $m=n^{1-\beta/2}$ and $\alpha = 1/5$), for some $\rho=\rho(C)$ with $0<\rho<1$ we can find a vertex subset $R\su V(G)$ with $|R|\ge n^{1-\beta/2}$, such that the induced subgraph $G[R]$
is $(n^{-\rho\beta/2},\rho)$-rich. Let us now define $\zeta=\beta\rho/(2\log(1/\rho))>0$, and let $q=\lfloor\zeta \log n\rfloor$.

We claim that for any subset $U\subseteq R$ of at size at least $|U|>n^{1/5}$, we can iteratively construct a $q$-tuple $(v_{1},\ldots,v_{q})\in U^q$ with
\begin{equation}\label{eq:condition-q-tuple-stronger}
|R\setminus (N(v_1)\cup\cdots\cup N(v_r))|\geq \rho^r|R| \quad\text{and}\quad |(R\cap N(v_r))\setminus (N(v_1)\cup\cdots\cup N(v_{r-1}))|\geq \rho^r|R|
\end{equation}
for $r=1,\ldots,q$. Indeed, for any $0\le k<q$, consider a $k$-tuple
$(v_{1},\ldots,v_{k})\in U^k$ satisfying \cref{eq:condition-q-tuple-stronger} for $r=1,\ldots,k$. Since $\rho^k\ge\rho^q\geq \rho^{\zeta\log n}= n^{-\rho\beta/2}$, we can apply the definition of $G[R]$ being $(n^{-\rho\beta/2},\rho)$-rich (see \cref{def:rich}) to the set $W:=R\setminus (N(v_1)\cup\cdots\cup N(v_k))$ of size $|W|\geq \rho^k|R|$, and conclude that there are at most $|R|^{1/5}\leq n^{1/5}$ vertices
$v\in U$ satisfying $|(R\cap N(v))\setminus (N(v_1)\cup\cdots\cup N(v_k))|=|N(v)\cap W|\le\rho|W|$ or $|R\setminus (N(v_1)\cup\cdots\cup N(v_k)\cup N(v))|=|W\setminus N(v)|\le\rho|W|$. Hence, as $|U|>n^{1/5}$, there exists a vertex $v_{k+1}\in U$ with $|(R\cap N(v_{k+1}))\setminus (N(v_1)\cup \cdots\cup N(v_k))|>\rho|W|\geq\rho^{k+1}|R|$ and $|R\setminus (N(v_1)\cup \cdots\cup N(v_{k+1}))|> \rho|W|\geq\rho^{k+1}|R|$. So we can indeed construct a $q$-tuple $(v_{1},\ldots,v_{q})\in U^q$ satisfying \cref{eq:condition-q-tuple-stronger} for $r=1,\ldots,q$.

By repeatedly applying the above claim, we can now greedily construct a collection $\mathcal{V}\su R^q$ of $\lceil n^{1-\beta}\rceil$ disjoint $q$-tuples of vertices in $R$ such that each such $q$-tuple $(v_1,\ldots,v_q)\in \mathcal{V}$ satisfies \cref{eq:condition-q-tuple-stronger} for $r=1,\ldots,q$ (indeed, as long as our collection $\mathcal{V}$ has size $|\mathcal{V}|<n^{1-\beta}$, the number of vertices appearing in some $q$-tuple in $\mathcal{V}$ is at most $q\cdot n^{1-\beta}<(\zeta\log n)\cdot n^{1-\beta}<n^{1-\beta/2}/2\leq |R|/2$, and hence there are at least $|R|/2>n^{1/5}$ vertices in $R$ remaining). Now, define $I$ to be the set of the $q\cdot\lceil n^{1-\beta}\rceil\le(\zeta\log n)\cdot 2n^{1-\beta}\le n^{1-\beta(1+\rho)/2}/2$ vertices appearing in the $q$-tuples in $\mathcal{V}$, and let $J=V(G)\setminus I$. We claim that now for every $(v_1,\ldots,v_q)\in \mathcal{V}$ and every $r=1,\ldots,q$ the desired conditions in \cref{eq:condition-q-tuple} follows from \cref{eq:condition-q-tuple-stronger}. Indeed, by \cref{eq:condition-q-tuple-stronger} the sets appearing in \cref{eq:condition-q-tuple} have size at least $\rho^r|R|-|R\cap I|\geq \rho^q\cdot n^{1-\beta/2}-|I|\geq n^{-\beta\rho/2}\cdot n^{1-\beta/2}-n^{1-\beta(1+\rho)/2}/2=n^{1-\beta(1+\rho)/2}/2\geq n^{1-\beta}$ (using that $\rho<1$ and $n$ is sufficiently large).
\end{proof}

Roughly speaking, the condition in \cref{eq:condition-q-tuple} states that $(v_1,\ldots,v_q)$ have very different neighborhoods. This allows us to obtain strong joint probability bounds on degree statistics, as follows.
\begin{lemma}\label{lem:joint-probability}
Fix $\eta>0$. In an $n$-vertex graph $G$, let $(v_{1},\ldots,v_{q})$ be a tuple of vertices satisfying \cref{eq:condition-q-tuple} (for all $r=1,\ldots,q$) for some vertex subset $J\su V(G)$ and some $0<\beta<1$. For some $\ell\in\mb N$ with $\eta n\le\ell\le(1-\eta)n$, let $U$ be a random subset of $\ell$ vertices of $G$. Consider any $\tau\in\mb{R}\setminus \{0\}$, any $0<\delta\leq 1/2$, and $\vec{x}\in\mb{R}^{q}$. Then
\[\Pr\left[\left\Vert\tau\deg_{U\cap J}(v_{r})-\tau\deg_{U\cap J}(v_{1})+x_{r}\right\Vert_{\mb{R}/\mb Z}<\delta\text{ for  }r=2,\ldots,q\right]\le\left(O_\eta\left(\frac{(|\tau|+\delta) (|\tau|+n^{-(1-\beta)/2})}{|\tau|}\right)\right)^{q-1}.\]
\end{lemma}

To prove \cref{lem:joint-probability} we will need the following estimate for hypergeometric distributions.

\begin{lemma}\label{lem:circle}
Fix $\eta>0$. For some even positive integer $k$, let  $Z\sim\mr{Hyp}(k,k/2,\ell)$ with $\eta k\le\ell\le(1-\eta)k$. Then for any $\tau\in\mb{R}\setminus\{0\}$, any $0<\delta\leq 1/2$ and $x\in\mb{R}$, we have
\[\Pr\left[\left\Vert\tau Z+x\right\Vert_{\mb{R}/\mb Z}\le\delta\right]\lesssim_{\eta}\frac{(|\tau|+\delta) (|\tau|+1/\sqrt{k})}{|\tau|}.\]
\end{lemma}
\begin{proof}
We may assume that $x\in[-\tau\mb{E}Z, -\tau\mb{E}Z +1]$, which implies that $x/\tau$ differs from $-\mb{E}Z$ by at most $1/|\tau|$. Note that the standard deviation of $Z$ is $\Theta_\eta(\sqrt{k})$; by direct computation or a non-uniform quantitative central limit theorem for the hypergeometric distribution (for example \cite[Theorem~2.3]{LCM06}), for any $y\in\mb{R}$ we have 
\[
\Pr[Z-\mb E Z=y]\lesssim_\eta \frac{\exp\left(-\Omega_{\eta}(y^2/k)\right)}{\sqrt{k}}.
\]
It follows that
\begin{align*}
\Pr\left[\left\Vert\tau Z+x\right\Vert_{\mb{R}/\mb Z}\le\delta\right] &\le\sum_{i\in\mb Z}\Pr\left[\left|Z+\frac{x}{\tau}-\frac{i}{\tau}\right|
\le \frac{\delta}{|\tau|}\right] \lesssim_\eta \sum_{i\in\mb Z} \sum_{\substack{j\in\mb{Z}\\|j+x/\tau-i/\tau|\leq \delta/|\tau|}} \!\!\!\!\!\!\!\frac{\exp\left(-\Omega_{\eta}((j-\mb EZ)^2/k)\right)}{\sqrt{k}}\\
&\lesssim_\eta\sum_{i\in\mb Z}\left(1+2\frac{\delta}{|\tau|}\right)\frac{\exp\left(-\Omega_{\eta}\left((\max\{0,|i/\tau|-(1+\delta)/|\tau|\})^2/k\right)\right)}{\sqrt{k}}\\
&\le \left(1+2\frac{\delta}{|\tau|}\right)\left(\sum_{\substack{i\in\mb Z\\|i|>4}}\frac{\exp\left(-\Omega_{\eta}\left(i^2/(4\tau^2k)\right)\right)}{\sqrt{k}}+\sum_{\substack{i\in\mb Z\\|i|
\le 4}}\frac{1}{\sqrt{k}}\right)\\
&\lesssim_\eta \frac{|\tau|+\delta}{|\tau|}\cdot \left(\frac{|\tau|\sqrt{k}}{\sqrt{k}}+\frac{1}{\sqrt{k}}\right)=\frac{(|\tau|+\delta) (|\tau|+1/\sqrt{k})}{|\tau|},
\end{align*}
where in the third step we used that for any $i\in\mb Z$ there are at most $1+2\delta/|\tau|$ integers $j\in\mb Z$ satisfying $|j+x/\tau-i/\tau|\leq \delta/|\tau|$, and for every such integer we have $|j-\mb EZ|\geq |i|/\tau-1/|\tau|-\delta/|\tau|$ (since $x/\tau$ differs from $-\mb{E}Z$ by at most $1/|\tau|$).
\end{proof}

From this we deduce \cref{lem:joint-probability}.

\begin{proof}[Proof of \cref{lem:joint-probability}]
For $r=2,\ldots,q$, let $\mathcal{E}_{r}$ be the event that $\|\tau\deg_{U\cap J}(v_{i})-\tau\deg_{U\cap J}(v_{1})+x_{i}\|_{\mb{R}/\mb Z}<\delta$. We claim that
\[\Pr[\mathcal{E}_{r}\,|\,\mathcal{E}_{2}\cap\cdots\cap\mathcal{E}_{r-1}]\lesssim_{\eta}\frac{(|\tau|+\delta) (|\tau|+n^{-(1-\beta)/2})}{|\tau|}.\]
for every  $r=2,\ldots,q$. This will suffice, since the desired probability in the statement of \cref{lem:joint-probability} is
\[\Pr[\mathcal{E}_{2}\cap\cdots\cap\mathcal{E}_{q}]=\prod_{r=2}^{q}\Pr[\mathcal{E}_{r}\,|\,\mathcal{E}_{2}\cap\cdots\cap\mathcal{E}_{r-1}].\]
Now fix $r\in \{2,\ldots,q\}$. By assumption both of the sets appearing in condition \cref{eq:condition-q-tuple} have size at least $\lceil n^{1-\beta}\rceil$. Inside each of these two sets, we choose some subset of size exactly $\lceil n^{1-\beta}\rceil$ and we define $S\su J\setminus (N(v_1)\cup\cdots\cup N(v_{r-1}))$ to be the union of these two subsets. Then $|S|=2\lceil n^{1-\beta}\rceil$ and $|S\cap N(v_r)|=\lceil n^{1-\beta}\rceil$. For the random set $U\su V(G)$ of size $\ell$, let us now condition on an outcome of $|U\cap S|$ such that $(\eta/2)|S|\le|U\cap S|\le(1-\eta/2)|S|$
(by a Chernoff bound for hypergeometric random variables, as in \cref{lem:chernoff}, this happens with probability $1-n^{-\omega_\eta(1)}\geq 1- ((|\tau|+\delta)/|\tau|)\cdot n^{-(1-\beta)/2}$),
and condition on any outcome of $U\setminus S$ (as $S$ is disjoint from  $N(v_1)\cup\cdots\cup N(v_{r-1})$, this determines
the value of $\deg_{U\cap J}(v_{j})$ for $j=1,\ldots,r-1$ and in particular determines whether the events
$\mathcal{E}_{j}$ hold for $j=2,\ldots,r-1$). Now, conditionally, $\deg_{ U\cap S}(v_{r})=|U\cap S\cap N(v_r)|$
has a hypergeometric distribution $\mr{Hyp}(|S|,|S|/2,|U\cap S|)$,
so the claim follows from \cref{lem:circle} (taking $x=\tau\deg_{(U\cap J)\setminus S}(v_{1})-\tau\deg_{U\cap J}(v_{1})+x_{r}$), recalling that $|S|=2\lceil n^{1-\beta}\rceil$.
\end{proof}

We are now ready to prove \cref{thm:quadratic-decoupling}.

\begin{proof}[Proof of \cref{thm:quadratic-decoupling}]
We apply \cref{lem:good-tuples} with $\beta=\eta/3$, obtaining a partition $V(G)=I\cup J$ and a collection $\mathcal{V}\su I^q$ of at least $n^{1-\eta/3}$ disjoint $q$-tuples of vertices in $I$, where $q=\lfloor \zeta\log n\rfloor$ with $\zeta=\zeta(C,\eta/3)>0$, such that each $q$-tuple $(v_1,\ldots,v_q)\in \mathcal{V}$ satisfies \cref{eq:condition-q-tuple} for $r=1,\ldots,q$. Let $A$ denote the adjacency matrix of $G$ and let $\vec \xi\in \{0,1\}^n$ be the characteristic vector of the random set $U$ (meaning $\vec{\xi}_v=1$ if $v\in U$, and $\vec{\xi}_v=0$ if $v\notin U$), so $\vec \xi\in\{0,1\}^n$ is a uniformly random vector with precisely $\ell$ ones. We define
\[f(\vec{\xi}):=X=e(G[U])+\sum_{v\in U}e_v+e_0=\frac{1}{2}\vec{\xi}^{\intercal}A\vec{\xi}+\vec{e}\cdot\vec{\xi}+e_0.\]
For the rest of the proof we condition on an outcome of $|U\cap I|$
satisfying $(\eta/2)|I|\le|U\cap I|\le(1-\eta/2)|I|$. By a Chernoff bound for hypergeometric random variables, as in \cref{lem:chernoff}, this occurs with probability $1-n^{-\omega_\eta(1)}$ (as $\eta n\leq\ell\leq (1-\eta)n$ and $|I|\geq n^{1-\eta/3}$), so the characteristic function for the random variable $X$ under this conditioning differs from the original characteristic function $\varphi_X$ by at most $n^{-\omega_\eta(1)}$. Hence it suffices to prove that $|\varphi_X(\tau)|\le n^{-6}$ (for $n^{-1+\eta}\le|\tau|\le\nu$) for our conditional random variable $X$.

Let $\vec \xi_I$ and $\vec \xi_J$ be the restrictions of $\vec \xi$ to the index sets $I$ and $J$. Having conditioned on $|U\cap I|$, these vectors $\vec{\xi}_{I}$ and $\vec{\xi}_{J}$
are independent from each other. Let $\vec{\xi}_{J}'$ be an independent
copy of $\vec{\xi}_{J}$; by \cref{eq:decoupling} we have
\begin{equation}\label{eq:decoupling-applied}
|\varphi_{X}(\tau)|^2=|\varphi_{f(\vec{\xi})}(\tau)|^2=\left|\mb E e^{i\tau f(\vec{\xi}_{I},\vec{\xi}_{J})}\right|^{2}\le\mb E\left[\left|\mb E[e^{i\tau(f(\vec{\xi}_{I},\vec{\xi}_{J})-f(\vec{\xi}_{I},\vec{\xi}_{J}'))}\mid\vec{\xi}_{J},\vec{\xi}_{J}']\right|\right].
\end{equation}

Now, we can write $f(\vec{\xi}_{I},\vec{\xi}_{J})-f(\vec{\xi}_{I},\vec{\xi}_{J}')=\sum_{i\in I}a_{i}\xi_{i}+b$, where $a_{i}=\sum_{j\in J}A_{i,j}(\xi_{j}-\xi_{j}')$ for each $i\in I$ and $b$ only depends on $\vec{\xi}_{J}$ and $\vec{\xi}_{J}'$ (but not on $\vec{\xi}_{I}$). Let $\delta=n^{-1/2+\eta/3}$.
\begin{claim}\label{claim:quadratic-decoupling-claim-in-proof}
    With probability at least $1-n^{-12}/2$ the outcome of $(\vec{\xi}_{J},\vec{\xi}_{J}')$ is such that
    \[\left\Vert\tau a_{i}/(2\pi)-\tau a_{i'}/(2\pi)\right\Vert_{\mb{R}/\mb Z}\ge\delta\]
    for at least $\left|\mathcal{V}\right|/2\geq n^{1-\eta/3}/2$ disjoint pairs $(i,i')\in I^{2}$.
\end{claim}
Assuming \cref{claim:quadratic-decoupling-claim-in-proof}, it follows from \cref{lem:slice-estimate} that with probability at least $1-n^{-12}/2$, the outcome of $\vec{\xi}_{J}$ and $\vec{\xi}_{J}'$ is such that
\begin{align*}
\left|\mb E[e^{i\tau(f(\vec{\xi}_{I},\vec{\xi}_{J})-f(\vec{\xi}_{I},\vec{\xi}_{J}'))}\mid\vec{\xi}_{J},\vec{\xi}_{J}']\right|&=\left|\mb E[e^{i\tau\left(\sum_{i\in I} a_i\xi_i+b\right)}\mid\vec{\xi}_{J},\vec{\xi}_{J}']\right|=\left|\mb E[e^{i\sum_{i\in I} \tau a_i\xi_i}\mid\vec{\xi}_{J},\vec{\xi}_{J}']\right|\lesssim e^{-\Omega_{\eta}(n^{\eta/3})}.
\end{align*}
For sufficiently large $n$, the right-hand side is bounded by $n^{-12}/2$. Noting that the expectation on the left-hand side is bounded by $1$ for all outcomes of $\vec{\xi}_{J}$ and $\vec{\xi}_{J}'$, we can conclude that the right-hand side of \cref{eq:decoupling-applied} is bounded by $n^{-12}$ and therefore $|\varphi_{X}(\tau)|\leq n^{-6}$ for sufficiently large $n$, as desired. It remains to prove \cref{claim:quadratic-decoupling-claim-in-proof}.

\begin{proof}[Proof of \cref{claim:quadratic-decoupling-claim-in-proof}]
    Let us also condition on any outcome of $\vec{\xi}_{J}'$. Say that a $q$-tuple $(v_{1},\ldots,v_{q})\in\mathcal{V}$ is \emph{bad} if no pair $(v_r, v_1)\in I^2$ with $r\in\{2,\ldots,q\}$ has the property in the claim. In other words, $(v_{1},\ldots,v_{q})$ is bad if for all $r=2,\ldots,q$ we have $\left\Vert\tau a_{v_r}/(2\pi)-\tau a_{v_1}/(2\pi)\right\Vert_{\mb{R}/\mb Z}<\delta$. 

For any $q$-tuple $(v_{1},\ldots,v_{q})\in\mathcal{V}$ we can bound the probability that $(v_{1},\ldots,v_{q})$ is bad by applying \cref{lem:joint-probability} with $x_{r}=-(\tau/(2\pi))\sum_{j\in J}(A_{v_r,j}-A_{v_1,j})\xi_{j}'$ for $r=2,\ldots,q$ (recall that $(v_{1},\ldots,v_{q})$ satisfies \cref{eq:condition-q-tuple}), obtaining
\begin{align*}
\Pr[(v_{1},\ldots,v_{q})\text{ is bad}] & =\Pr\left[\left\Vert\tau a_{v_r}/(2\pi)-\tau a_{v_1}/(2\pi)\right\Vert_{\mb{R}/\mb Z}<\delta\text{ for }r=2,\ldots,q\right]\\
& =\Pr\left[\left\Vert(\tau/(2\pi))\deg_{U\cap J}(v_{r})-(\tau/(2\pi))\deg_{U\cap J}(v_{1})+x_{r}\right\Vert_{\mb{R}/\mb Z}<\delta\text{ for }r=2,\ldots,q\right]\\
 & \le\left(O_\eta\left(\frac{(|\tau/(2\pi)|+\delta) (|\tau/(2\pi)|+n^{-(1-\beta)/2})}{|\tau/(2\pi)|}\right)\right)^{q-1}\\
 &\le\left(O_\eta\left(\frac{(|\tau|+n^{-1/2+\eta/3}) (|\tau|+n^{-1/2+\eta/6})}{|\tau|}\right)\right)^{q-1}\le \left(O_\eta(\nu+n^{-\eta/2})\right)^{\lfloor\zeta\log n\rfloor-1},
\end{align*}
using that $n^{-1+\eta}\le|\tau|\le\nu$. Now, if $\nu$ is sufficiently small with respect to $C$ and $\eta$ (and consequently also sufficiently small with respect to $\zeta$), we deduce that $\Pr[(v_{1},\ldots,v_{q})\text{ is bad}]\leq 1/(4n^{12})$. Hence the expected number of bad tuples $(v_{1},\ldots,v_{q})\in\mathcal{V}$ is at most  $|\mathcal{V}|/(4n^{12})$. Thus, by Markov's inequality, with probability at least $1-n^{-12}/2$ there are at most $|\mathcal{V}|/2$ bad $q$-tuples in $\mathcal{V}$. When this is the case, among each of the at least $|\mathcal{V}|/2$ different $q$-tuples $(v_{1},\ldots,v_{q})\in\mathcal{V}$ that are not bad we can find a pair $(v_r, v_1)\in I^2$ with the desired property that $\left\Vert\tau a_{v_r}/(2\pi)-\tau a_{v_1}/(2\pi)\right\Vert_{\mb{R}/\mb Z}\geq\delta$. Since the $q$-tuples in $\mathcal{V}$ are all disjoint, this gives at least $|\mathcal{V}|/2$ disjoint pairs in $I^2$ with this property, thus proving the claim.\end{proof}
As we saw earlier, this finishes the proof of \cref{thm:quadratic-decoupling}.\end{proof}

\section{Short interval control in the additively unstructured case}\label{sec:short-interval-unstructured}
Now we can combine the characteristic function estimates in \cref{sec:initial-fourier,sec:high-fourier} to prove \cref{thm:short-interval} in the ``additively unstructured'' case (recall the outline in \cref{sub:additive-structure-dichotomy}), defined as follows. This definition is chosen so that the term $\wh{D}_{L,\gamma}(\vec{d})$ appearing in \cref{prop:linear-RLCD} is large, meaning that \cref{prop:linear-RLCD} can be applied to a wide range of $|t|$.

\begin{definition}\label{def:additively-structured}
Fix $0<\gamma<1/4$, consider a graph $G$ with $n$ vertices and a vector $\vec e\in\mb{R}_{\ge 0}^{V(G)}$, and let $d_v=e_v+\deg_G(v)/2$ for all $v \in V(G)$.
Say that $(G,\vec e)$ is \emph{$\gamma$-unstructured} if $\wh{D}_{L,\gamma}(\vec{d})\ge n^{1/2}$, for $L=\lceil 100/\gamma\rceil$. Otherwise, $(G,\vec e)$ is \emph{$\gamma$-structured}.
\end{definition}

From now on we fix $\gamma=10^{-4}$. For our proof of \cref{thm:short-interval}, we split into two cases, depending on whether $(G,\vec e)$ is $\gamma$-structured. In this section we will prove \cref{thm:short-interval} in the case where $(G,\vec e)$ is $\gamma$-unstructured. Eventually (in \cref{sec:short-interval-structured}) we will handle the case where $(G,\vec e)$ is $\gamma$-structured, i.e., where $\wh{D}_{L,\gamma}(\vec{d})< n^{1/2}$. While the arguments in this section work for any constant $0<\gamma<1/4$, the proof of the $\gamma$-structured case in \cref{sec:short-interval-structured} requires $\gamma$ to be sufficiently small (this is why we define $\gamma=10^{-4}$).

\begin{proof}[Proof of \cref{thm:short-interval} in the $\gamma$-unstructured case]
Fix $C,H>0$, let $G$ and $\vec{e}\in\mb{R}^{V(G)}$ and $e_0\in\mb{R}$ be as in \cref{thm:short-interval}, and assume that $(G,\vec e)$ is $\gamma$-unstructured and that $n$ is sufficiently large with respect to $C$ and $H$. Recall that $U$ is a uniformly random subset of $V(G)$ and $X=e(G[U])+\sum_{v\in U} e_v+e_0$, and also recall (e.g.~from the proof of \cref{prop:linear-RLCD}) that $\sigma(X)=\Theta_{C,H}(n^{3/2})$. Let $Z\sim \mc{N}(\mb E X,\sigma(X))$ be a Gaussian random variable with the same mean and variance as $X$.

First note that for any $\tau\in\mb{R}$, \cref{prop:lowest-t} implies
\[|\varphi_X(\tau)-\varphi_Z(\tau)|=\big|\varphi_{(X-\mb{E} X)/\sigma(X)}(\tau\sigma(X))-\varphi_{(Z-\mb{E} X)/\sigma(X)}(\tau\sigma(X))\big |\lesssim_{C,H}|\tau|\sigma(X)n^{-1/2}\lesssim_{C,H}|\tau|n\]
(noting that the graph $G$ has density at least $\Omega_C(1)$ by \cref{thm:dense-ramsey}). Then, note that since $|\varphi_Z(\tau)|=\exp(-\sigma(X)^2\tau^2/2)$, for $|\tau|\ge n^{2\gamma}/\sigma(X)$ we have $|\varphi_Z(\tau)|\le\exp(-n^{4\gamma}/2)$. Furthermore, in \cref{prop:linear-RLCD} we have $\wh{D}_{L,\gamma}(\vec{d})\ge n^{1/2}$ by our assumption that $(G,\vec e)$ is $\gamma$-unstructured. Hence for $\alpha=\alpha(C,H,\gamma)>0$ as in \cref{prop:linear-RLCD}, we obtain that $|\varphi_X(\tau)|=|\varphi_{(X-\mb{E} X)/\sigma(X)}(\tau\sigma(X))|\lesssim_{C,H,\gamma} n^{-5}$ for $n^{2\gamma}/\sigma(X)\le |\tau|\le \alpha n^{1/2+\gamma/8}/\sigma(X)$.

Let $\nu=\nu(C,\gamma/9)>0$ be as in \cref{thm:quadratic-decoupling}. Note that by a Chernoff bound we have $n/4\le |U|\le 3n/4$ with probability $1-e^{-\Omega(n)}$. If we condition on such an outcome of $|U|$, then for $n^{-1+\gamma/9}\le |\tau|\le \nu$, \cref{thm:quadratic-decoupling} shows that the conditional characteristic function of $X$ is bounded in absolute value by $n^{-5}$ (assuming that $n$ is sufficiently large). It follows that for this range of $|\tau|$ we have $|\varphi_X(\tau)|\lesssim_{C,H} n^{-5}+e^{-\Omega(n)}\lesssim n^{-5}$.

Recalling that $\sigma(X)=\Theta_{C,H}(n^{3/2})$ (and therefore $n^{-1+\gamma/9}\le \alpha n^{1/2+\gamma/8}/\sigma(X)$ for sufficiently large $n$), we can conclude that for $n^{2\gamma}/\sigma(X)\le |\tau|\le \nu$ we have $|\varphi_X(\tau)|\lesssim_{C,H} n^{-5}$ and $|\varphi_X(\tau)-\varphi_Z(\tau)|\lesssim_{C,H} n^{-5}+\exp(-n^{4\gamma}/2)\lesssim n^{-5}$. Hence, defining $\varepsilon=2/\nu>0$ (which only depends on $C$), we obtain
\[
\int_{-2/\varepsilon}^{2/\varepsilon}|\varphi_X(\tau)-\varphi_Z(\tau)|\,d\tau\lesssim_{C,H} \int_{-n^{2\gamma}/\sigma(X)}^{n^{2\gamma}/\sigma(X)}|\tau|n\,d\tau+2\nu\cdot n^{-5}\lesssim_{C,H}n^{4\gamma-2}.
\]
Let $B=B(C)=10^4\cdot 2\varepsilon$. For the upper bound in \cref{thm:short-interval}, note that by \cref{lem:esseen-upper-crude} for all $x\in\mb{R}$ we have (using that $\mathcal{L}(Z,\eps)\le 2\eps/\sigma(X)\lesssim_{C,H} n^{-3/2}$ as $p_Z(u)\le 1/\sigma(X)$ for all $u\in\mb{R}$)
\[\Pr[|X-x|\leq B]\le 2\cdot 10^4\cdot \mathcal{L}(X,\eps)\lesssim \mathcal{L}(Z,\eps)+\eps \int_{-2/\varepsilon}^{2/\varepsilon}|\varphi_X(\tau)-\varphi_Z(\tau)|\,d\tau\lesssim_{C,H} n^{-3/2}.\]
For the lower bound in \cref{thm:short-interval}, fix some $A>0$. We can apply \cref{lem:esseen-lower} with $K=2$ and any fixed $R\geq 4$ (which we will chose sufficiently large in terms of $C, H, \gamma$, and $A$). Indeed, note that for any fixed $A>0$ and $R\geq 4$, for $x\in\mb Z$ with $|x-\mb{E}X|\leq An^{3/2}$ and $y_1,y_2\in [x-R\eps,x+R\eps]$, we have that $p_Z(y_1)/p_Z(y_2)\le \exp(-((y_1-\mb{E}X)^2-(y_2-\mb{E}X)^2)/(2\sigma(X)^2))\le\exp(2R\eps\cdot 4An^{3/2}/\Theta_{C,H}(n^3))\le 2$ if $n$ is sufficiently large with respect to $C, H, A$, and $R$. Hence \cref{lem:esseen-lower} yields
\begin{align*}
\Pr[|X-x|\leq B]&\geq \frac{1}{8}\Pr[|Z-x|\le \eps]-C_{\ref{lem:esseen-lower}}\bigg(R^{-1}\mc{L}(Z,\eps)+\eps\int_{-2/\eps}^{2/\eps}|\varphi_Y(\tau)-\varphi_Z(\tau)|\,d\tau\bigg)\\
&\ge \eps\cdot \frac{\exp(-A^2n^3/(2\sigma(X)^2))}{8\sqrt{2\pi}\sigma(X)}-\frac{C_{\ref{lem:esseen-lower}}}{R}\cdot \frac{2\eps}{\sigma(X)}-C_{\ref{lem:esseen-lower}}\cdot O_{C,H}(n^{4\gamma-2})\gtrsim_{C,H,A} n^{-3/2},
\end{align*}
if $R$ is chosen to be large enough with respect to $C, H$, and $A$ (recall again that $\sigma(X)=\Theta_{C,H}(n^{3/2})$).
\end{proof}

\section{Robust rank of Ramsey graphs}\label{sec:ramsey-robust-rank}
In \cite{KS20b}, the first and third authors observed that the adjacency matrix of a Ramsey graph is far from any matrix with rank $O(1)$. We will need a much stronger version of this fact: the adjacency matrix of a Ramsey graph is far from all matrices built out of a small number of rank-$O(1)$ ``blocks'' (in the proof of \cref{thm:short-interval}, these blocks will correspond to the buckets of vertices with similar values of $d_v$). 
Recall that $\snorm{M}_\mr{F}^2$ is the sum of the squares of the entries of $M$.

\begin{lemma}\label{lem:ramsey-low-rank}
Fix $0<\delta<1$, $C>0$, $r\in\mb{N}$ and consider a $C$-Ramsey graph $G$ on $n$ vertices with adjacency matrix $A$. Suppose we are given a partition $V(G) = I_1\cup\cdots\cup I_{m}$, with $|I_1|=\cdots=|I_m|$ and $n^\delta/2\le m\le 2n^\delta$. Then, for any $B\in\mb{R}^{n\times n}$ with $\on{rank}(B[I_j\!\times\! I_k])\le r$ for all $j,k\in [m]$, we have $\snorm{A-B}_{\mr{F}}^2\gtrsim_{C,r,\delta} n^2$.
\end{lemma}

The proof of \cref{lem:ramsey-low-rank} has several ingredients, including the fact that if a binary matrix is close to a low-rank matrix, then it is actually close to a \emph{binary} low-rank matrix. Note that for binary matrices $A,Q$, the squared Frobenius norm $\snorm{A-Q}_{\on{F}}^2$ can be interpreted as the \emph{edit distance} between $A$ and $B$: the minimum number of entries that must be changed to obtain $B$ from $A$.

\begin{proposition}\label{prop:rank-close}
Fix $r\in\mb{N}$. Consider a binary matrix $A\in\{0,1\}^{n\times n}$ and a real matrix $B\in \mathbb{R}^{n\times n}$ such that $\on{rank} B\leq r$ and $\snorm{A-B}_{\on{F}}^2\leq \eps n^2$ for some $\eps>0$. Then there is a binary matrix $Q\in \{0,1\}^{n\times n}$ with $\on{rank} Q\leq r$ and $\snorm{A-Q}_{\on{F}}^2\le C_r \sqrt \eps n^2$, for some $C_r$ depending only on $r$.
\end{proposition}

We remark that it is possible to give a more direct proof of a version of \cref{prop:rank-close} with dramatically worse quantitative aspects (i.e., replacing $\sqrt \varepsilon$ by a function that decays extremely slowly as $\varepsilon\to 0$), using a bipartite version of the \emph{induced graph removal lemma} (see for example \cite[Theorem~3.2]{Con13}). For the application in this paper, quantitative aspects are not important, but we still believe our elementary proof and the strong bounds in \cref{prop:rank-close} are of independent interest (induced removal lemmas typically require the so-called \emph{strong regularity lemma}, which is notorious for its terrible quantitative aspects). Our proof of \cref{prop:rank-close} relies on the following lemma.

\begin{lemma}\label{lem:rank-init}
Fix $r\in\mb{N}$. Let $\eta>0$, and let $A\in\{0,1\}^{n\times n}$ be a binary matrix where every entry is colored either red or green, in such a way that fewer than $\eta^2/(10\cdot 2^r)^2\cdot n^2$ entries are red. Suppose that every $(r+1)\times (r+1)$ submatrix of $A$ consisting only of green entries is singular. Then there exists a binary matrix $Q\in\{0,1\}^{n\times n}$ with $\on{rank} Q\leq r$ which differs from $A$ in at most $\eta \cdot n^2$ entries.
\end{lemma}

\begin{proof}
For $\ell\in\mb N$, let us call an $\ell\times \ell$ submatrix of some matrix green if all its $\ell^2$ entries are green.

First, consider all rows and columns of $A$ that contain at least $\eta/(10\cdot 2^{2r})\cdot n$ red entries. There can be at most $(\eta/10)\cdot n$ such rows and at most $(\eta/10)\cdot n$ such columns. Let us define a new matrix $A_1\in\{0,1\}^{n\times n}$ where we replace each of these rows by an all-zero row and each of these columns by an all-zero column, and where we re-color all elements in these replaced rows and columns green. Note that then $A_1$ and $A$ differ in at most $(2\eta/10)\cdot n^2$ entries, and $A_1$ still has the property that each green $(r+1)\times (r+1)$ submatrix is singular. Furthermore, each row and column in $A_1$ contains at most $\eta/(10\cdot 2^{2r})\cdot n$ red entries.

Now choose $\ell$ maximal such that $A_1$ contains a non-singular green $\ell\times\ell$ submatrix. Clearly, $\ell\leq r$, and without loss of generality we assume that the $\ell\times \ell$ submatrix $A_1[\,[\ell]\!\times\! [\ell]\,]$ in the top-left corner of $A_1$ is non-singular and green. By the choice of $\ell$, every green $(\ell+1)\times(\ell+1)$ submatrix in $A_1$ is singular.

Now, in the first $\ell$ rows of $A_1$ there are at most $\ell\cdot \eta/(10\cdot 2^{2r})\cdot n\leq (\eta/10)n$ red entries. For each of these red entries in the first $\ell$ rows of $A_1$, let us replace its entire column by green zeroes (i.e., an all-zero column with all entries colored green). Similarly, in the first $\ell$ columns of $A_1$ there are at most $(\eta/10)n$ red entries, and for each of these red entries let us replace its entire row by green zeroes. We obtain a new matrix $A_2\in \{0,1\}^{n\times n}$ differing from $A_1$ in at most $(2\eta/10)\cdot n^2$ entries. In this matrix $A_2$ it is still true that each green $(\ell+1)\times(\ell+1)$ submatrix in $A_1$ is singular, but that $A_2[\,[\ell]\!\times\![\ell]\,]$ is non-singular. Furthermore, $A_2$ has no red entries anywhere in the first $\ell$ rows or first $\ell$ columns.

Next, consider the set of columns of $A_2\in\{0,1\}^{n\times n}$ with indices in $\{\ell+1,\ldots,n\}$. There is a partition $\{\ell+1,\ldots,n\}=I_1\cup\cdots\cup I_{2^r}$ such that for each $k=1,\ldots,2^r$, the columns of $A_2$ with indices in $I_k$ all agree in their first $\ell$ rows. For each $k=1,\ldots,2^r$ with $|I_k|\leq \eta/(10\cdot 2^r)\cdot n$, let us replace all columns with indices in $I_k$ by green all-zero columns. Similarly, there is a partition $\{\ell+1,\ldots,n\}=J_1\cup\cdots\cup J_{2^r}$ such that the rows with indices in the same set $J_k$ all agree in their first $\ell$ columns. For each $k=1,\ldots,2^r$ with $|J_k|\leq\eta/(10\cdot 2^r)\cdot n$, replace all rows with indices in $J_k$ with green all-zero rows. In this way, we obtain a new matrix $A_3\in\{0,1\}^{n\times n}$ differing from $A_2$ in at most $(2\eta/10)\cdot n^2$ entries. Still, all green $(\ell+1)\times(\ell+1)$ submatrices in $A_3$ are singular, $A_3[\,[\ell]\!\times\![\ell]\,]$ is non-singular, and all entries in the first $\ell$ rows and in the first $\ell$ columns of $A_3$ are green.

Finally, define the matrix $Q\in\{0,1\}^{n\times n}$ by replacing the red entries in $A_3$ as follows. For each red entry $(j,i)$ in $A_3$ we have $j\in J_k$ and $i\in I_{k'}$ for some $k$ and $k'$ such that $|J_k|, |I_{k'}|> \eta/(10\cdot 2^r)\cdot n$. So, the submatrix $A_3[J_k\!\times\! I_{k'}]$ of $A_3$ must contain at least one green entry (since $A_3$ has fewer than $\eta^2/(10\cdot 2^r)^2\cdot n^2$ red entries). Let us now replace the red $(j,i)$-entry in $A_3$ by some green entry in $A_3[J_k\!\times\! I_{k'}]$. Replacing all red entries in this way, we obtain a matrix $Q\in \{0,1\}^{n\times n}$ differing from $A_3$ in at most $\eta^2/(10\cdot 2^r)^2\cdot n^2\leq (\eta/10)\cdot n^2$ entries.

All in all, $Q$ differs from $A$ in at most $(7\eta/10)\cdot n^2\leq \eta\cdot n^2$ entries. The $\ell\times\ell$ submatrix $Q[\,[\ell]\!\times\![\ell]\,]$ is still non-singular. We claim that whenever we extend this $\ell\times \ell$ submatrix in $Q$ to an $(\ell+1)\times(\ell+1)$ submatrix by taking an additional row $j\in\{\ell+1,\ldots,n\}$ and an additional column $i\in\{\ell+1,\ldots,n\}$, the resulting $(\ell+1)\times(\ell+1)$ submatrix of $Q$ is singular. If the $(j,i)$-entry in $A_3$ is green, then this $(\ell+1)\times(\ell+1)$ submatrix of $Q$ agrees with the corresponding submatrix in $A_3$, which is green and therefore singular. If the $(j,i)$-entry in $A_3$ is red, then the $(j,i)$-entry in $Q$ agrees with some green $(j',i')$-entry in $A_3$ where $j,j'\in J_k$ and $i,i'\in I_{k'}$ for some $k,k'$. Hence the desired $(\ell+1)\times(\ell+1)$ submatrix of $Q$ agrees with the $(\ell+1)\times(\ell+1)$ submatrix  $A_3[\,([\ell]\cup \{i'\})\!\times\!([\ell]\cup \{j'\})\,]$ of $A_3$, which is green and therefore singular. Hence we have shown that all $(\ell+1)\times(\ell+1)$ submatrices of $Q$ that contain $Q[\,[\ell]\!\times\![\ell]\,]$ are singular. Since $Q[\,[\ell]\!\times\![\ell]\,]$ is non-singular, this implies that $\on{rank} Q = \ell\leq r$.
\end{proof}

Now we are ready to prove \cref{prop:rank-close}.

\begin{proof}[Proof of \cref{prop:rank-close}]
Choose some $0<c_r<1$ depending only on $r$ such that\footnote{For the sake of giving explicit bounds, note that we can take any $c_r<(2^{-r}/(r!\cdot r^2))^2$. Indeed, note that any matrix $S\in\{0,1\}^{(r+1)\times (r+1)}$ which is non-singular has $|\!\det(S)|\ge 1$. Suppose there is a matrix $T$ such that  $\det(T)= 0$ and $\snorm{S-T}_\infty<c_r^{1/2}$. This implies that $\snorm{T}_{\infty}\le 2$ and therefore switching entries of $S$ and $T$ one by one changes the determinant by at most $r!\cdot 2^r \cdot c_r^{1/2}<r^{-2}$. As we switch $r^2$ entries and $\det(S)\ge 1$ while $\det(T) = 0$, we obtain a contradiction.}
\[c_r<\inf\{\snorm{S-T}_\infty^2\colon S\in \{0,1\}^{(r+1)\times (r+1)} \text{ non-singular},~T\in \mathbb{R}^{(r+1)\times(r+1)} \text{ singular}\},\]
where $\snorm{S-T}_{\infty}$ denotes the maximum absolute value $|(S-T)_{i,j}|$ among the entries of $S-T$.

Let $A$ and $B$ be matrices as in the lemma statement. Let us color each entry $A_{i,j}$ of $A$ red if $|A_{i,j}-B_{i,j}|^2> c_r$, and green otherwise. Then, as $\snorm{A-B}_{\on{F}}^2\leq\eps n^2$, there are fewer than $\eps n^2/c_r$ red entries in $A$. Furthermore, as $\on{rank} B\leq r$, by the choice of $c_r$, every $(r+1)\times (r+1)$ submatrix of $A$ consisting only of green entries must be singular. Thus, taking $C_r=10\cdot 2^r/\sqrt{c_r}$ the desired statement follows from \cref{lem:rank-init} with $\eta=(10\cdot 2^r)\sqrt{\eps/c_r}$.
\end{proof}

We also need the simple fact that low-rank binary matrices can be partitioned into a small number of homogeneous parts. This essentially corresponds to a classical bound on the log-rank conjecture. 

\begin{lemma}\label{lem:rank-partition}
Fix $r\in\mb{N}$, and let $s=2^r$. For any binary matrix $Q\in\{0,1\}^{n\times n}$ with $\on{rank} Q\leq r$, we can find partitions $P_1\cup\cdots\cup P_{s}$ and $R_1\cup\cdots\cup R_{s}$ of $[n]$, such that for all $i,j\in [s]$, the submatrix $Q[P_i\!\times\! R_j]$ consists of only zeroes, or only ones.
\end{lemma}
\begin{proof}
First, we claim that the matrix $Q$ has most $2^r$ different row vectors: indeed, let $r'=\on{rank} Q\leq r$ and suppose without loss of generality that the submatrix $Q[\,[r']\!\times\! [r']\,]$ is non-singular. Then each row of $Q$ can be expressed as a linear combination of the first $r'$ rows, and any two rows of $Q$ which agree in the first $r'$ entries must be given by the same linear combination. Hence there can be at most $2^r=s$ different row vectors in the matrix $Q$, and we obtain a partition $[n]=P_1\cup\cdots\cup P_{s}$ such that any two rows with indices in the same set $P_i$ are identical.

Similarly, there is a partition $[n]=P_1\cup\cdots\cup P_{s}$ such that any two columns with indices in the same set $R_j$ are identical. Now, for all $i,j\in [s]$, all entries of the submatrix $Q[P_i\!\times\! R_j]$ must be identical to each other, i.e., must be either all zeroes or all ones.
\end{proof}

Apart from \cref{prop:rank-close,lem:rank-partition}, in our proof of \cref{lem:ramsey-low-rank} we will also use the fact that  every $n$-vertex graph has a clique or independent set of size at least $\frac12\log n$ (this is a quantitative version of Ramsey's theorem proved by Erd\H{o}s and Szekeres~\cite{ES35}, as mentioned in the introduction).

\begin{proof}[Proof of \cref{lem:ramsey-low-rank}]
By \cref{thm:dense-ramsey} there exists some $\alpha=\alpha(C,\delta)>0$ such that every $2C/(1-\delta)$-Ramsey graph on sufficiently many vertices has density at least $\alpha$ and at most $1-\alpha$. Fix a sufficiently large integer $D=D(C,\delta)$ such that $1/\log_2 D<\alpha/4$, and choose $\eps=\eps(C,r,\delta)>0$ small enough such that $\sqrt{\eps}<1/D^2$ and  $\eps^{1/4}<\alpha/(2^{2rD+1}C_r)$, where $C_r$ is the constant in \cref{prop:rank-close}. It suffices to prove that we have $\|A-B\|^2_\mr F\ge \varepsilon n^2$ if $n$ is sufficiently large with respect to $C,\delta$, and $r$. So let us assume for contradiction that $\|A-B\|^2_\mr F< \varepsilon n^2$.

Note that $\sum_{1\le k<j\le m}\big\|(A-B)[I_j\!\times\! I_k]\big\|_{\mr F}^2\le \|A-B\|_{\mr F}^2\le  \varepsilon n^2$, so there can be at most $\sqrt{\eps}m^2$ pairs $(j,k)$ with $1\le j<k\le m$ such that $\big\|(A-B)[I_j\!\times\! I_k]\big\|_{\mr F}^2\ge \sqrt{\eps}(n/m)^2$. Hence a uniformly random subset of $[m]$ of size $D$ contains such a pair $(j,k)$ with probability at most $\binom{D}{2}\cdot\sqrt{\eps}<1$. Thus, there exists a subset of $[m]$ of size $D$ not containing any such pair $(j,k)$, and we may assume without loss of generality that $[D]$ is such a subset. Then for any $1\le j<k\le D$ we have $\big\|(A-B)[I_j\!\times\! I_k]\big\|_{\mr F}^2< \sqrt{\eps}(n/m)^2=\sqrt{\eps}\cdot |I_j|\cdot |I_k|$.

For any $1\le j<k\le D$, by \cref{prop:rank-close} (recalling that $\operatorname{rank}(B[I_j\!\times\! I_k])\le r$) we can find a binary matrix $Q^{(j,k)}\in \{0,1\}^{I_j\!\times\! I_k}$ with $\operatorname{rank}(Q^{(j,k)})\le r$ and $\|A[I_j\!\times\! I_k]-Q^{(j,k)}\|_\mr F^2\le  C_r \eps^{1/4} (n/m)^2$. Now, by \cref{lem:rank-partition}, we can find partitions of $I_j$ and $I_k$ into $2^r$ parts each, such that the corresponding $(2^r)^2$ submatrices of $Q^{(j,k)}$ each consist either only of zeroes or only of ones. Let us choose such partitions for all pairs $(j,k)$ with $1\le j<k\le D$, and for each of the sets $I_1,\ldots,I_D$, let us take a common refinement of the $D-1$ partitions of that set. This way, for each of the sets $I_1,\ldots,I_D$ we obtain a partition into $2^{r(D-1)}$ parts in such a way that for all $1\le j<k\le D$ each of the submatrices of $Q^{(j,k)}$ induced by the partitions of $I_j$ and $I_k$ consist either only of zeroes or only of ones.

For each $j=1,\ldots,D$, inside one of the parts of this partition of $I_j$, we can now choose a subset $I_j'\su I_j$ of size $|I_j'|=\lceil |I_j|/2^{r(D-1)}\rceil=\lceil n/(2^{r(D-1)}m)\rceil$. Then for all $1\le j<k\le D$, the submatrix $Q^{(j,k)}[I_j',I_k']$ consists either only of zeroes or only of ones. Consider the graph $H$ on the vertex set $[D]$ where for $1\le j<k\le D$ we draw an edge if all entries of $Q^{(j,k)}[I_j',I_k']$ are one (and we don't draw an edge if all entries are zero). Then, by Ramsey's theorem (specifically, Erd\H os and Szekeres' classical bound~\cite{ES35}), this graph $H$ must have a clique or independent set $S\su [D]$ of size $|S|\geq(\log_2 D)/2$. Without loss of generality assume that $S=\{1,\ldots,|S|\}$. Let us now consider the induced subgraph of the original graph $G$ on the vertex set $I_1'\cup\cdots\cup I_{|S|}'$.

If $S=\{1,\ldots,|S|\}$ is an independent set in $H$, then for all $1\leq j<k\leq |S|$ the matrix $Q^{(j,k)}[I_j'\!\times\!I_k']$ is all-zero, so $A[I_j\!\times\! I_k]\in \{0,1\}^{I_j\!\times\! I_k}$ can contain at most $C_r \eps^{1/4} (n/m)^2$ ones (since $\|A[I_j\!\times\! I_k]-Q^{(j,k)}\|_\mr F^2\le  C_r \eps^{1/4} (n/m)^2$). In other words, for all $1\leq j<k\leq |S|$ the graph $G[I_1'\cup\cdots\cup I_{|S|}']$ has at most $C_r \eps^{1/4} (n/m)^2\leq C_r \eps^{1/4}\cdot 2^{2r(D-1)}\cdot |I_j'|\cdot |I_k'|\leq (\alpha/2)\cdot |I_j'|\cdot |I_{k}'|$ edges between $I_j'$ and $I_k'$. As $|I_1'|=\cdots=|I_{|S|}'|$, the edges within the sets $I_1,\ldots,I_k$ also contribute at most $1/|S|\leq 2/\log_2 D<\alpha/2$ to the density of $G[I_1'\cup\cdots\cup I_{|S|}']$. Thus, the graph $G[I_1'\cup\cdots\cup I_{|S|}']$ has density less than $\alpha$, but it is a $2C/(1-\delta)$-Ramsey graph since $|I_1'\cup\cdots\cup I_{|S|}'|\geq n/(2^{r(D-1)}m)\geq n^{1-\delta}/2^{r(D-1)+1}\geq n^{(1-\delta)/2}$. This is a contradiction.

Similarly, if $S=\{1,\ldots,|S|\}$ is a clique in $H$, then for all $1\leq j<k\leq |S|$ the matrix $Q^{(j,k)}[I_j',I_k']$ is an all-ones matrix, and we can perform a similar calculation for the number of non-edges in $G[I_1'\cup\cdots\cup I_{|S|}']$. We find that $G[I_1'\cup\cdots\cup I_{|S|}']$ has density greater than $1-\alpha$, which is again a contradiction.
\end{proof}

\section{Lemmas for products of Boolean slices}\label{sec:coupling-lemmas}
In this section we study products of Boolean slices (that is, we consider random vectors $\vec x\in\{-1,1\}^n$ whose index set is divided into ``buckets'', uniform among all vectors with a particular number of ``1''s in each bucket). The main outputs we will need from this section are summarized in the following lemma. Namely, for a ``well-behaved'' quadratic polynomial $f$, a Gaussian vector $\vec z$ and a vector $\vec x$ sampled from an appropriate product of slices, we can compare $f(\vec x)$ with $f(\vec z)$. Our assumptions on $f$ are certain bounds on the coefficients, and that our polynomial is in a certain sense ``balanced'' within each bucket.

\begin{lemma}\label{lem:transfer-summary}
Fix $0<\delta<1/4$. Suppose we are given a partition $[n] = I_1\cup\cdots\cup I_{m}$, with $|I_1|=\cdots=|I_m|$ and $n^\delta/2\le m\le 2n^\delta$, where $n$ is sufficiently large with respect to $\delta$. Consider a symmetric matrix $F\in\mb{R}^{n\times n}$, a vector $\vec f\in\mb{R}^{n}$ and a real number $f_0$ satisfying the following conditions:
\begin{enumerate}
    \item[(a)] $\snorm{\vec{f}}_\infty\le n^{1/2+3\delta}$.
    \item[(b)] $|F_{i,j}|\le 1$ for all $i,j\in [n]$.
    \item[(c)] For each $k=1,\ldots,m$, the sum of the entries in $\vec f_{I_k}$ is equal to zero.
    \item[(d)] For all $k,h\in [m]$, in the submatrix $F[I_k\!\times\! I_h]$ of $F$ all row and column sums are zero.
\end{enumerate}
Consider a sequence $(\ell_1,\ldots,\ell_m)\in\mb{N}^{m}$
with $|\ell_{k}-|I_k|/2|\le\sqrt{n^{1-\delta}}\log n$
for $k=1,\ldots,m$. Then, let $\vec{x}\in\{-1,1\}^{n}$ be a uniformly
random vector such that $\vec{x}_{I_{k}}$ has exactly $\ell_{k}$
ones for each $k=1,\ldots,m$, and let $\vec z\sim\mc{N}(0,1)^{\otimes n}$ be a vector of independent standard Gaussian random variables. Define $X=f_0+\vec f\cdot\vec{x}+\vec{x}^{\intercal}F\vec{x}$ and $Z=f_0+\vec f\cdot\vec{z}+\vec{z}^{\intercal}F\vec{z}$. Then the following three statements hold.
\begin{enumerate}
\item $\mb{E} X=f_0+\sum_{i=1}^{n}F_{i,i}+O(n^{3/4+4\delta})$ and $\mb{E} Z=f_0+\sum_{i=1}^{n}F_{i,i}$.
\item $\sigma(X)^2=2\|F\|_{\mr F}^2+\|\vec f\|_2^2+O(n^{7/4+7\delta})$  and $\sigma(Z)^2=2\|F\|_{\mr F}^2+\|\vec f\|_2^2$.
\item For any $\tau\in\mb{R}$ we have \[|\varphi_X(\tau)-\varphi_Z(\tau)|\lesssim |\tau|^4\cdot  n^{3+12\delta}+|\tau|\cdot n^{3/4+4\delta}.\]
\end{enumerate}
\end{lemma}

We will apply this lemma in the additively structured case of our proof of \cref{thm:short-interval}. In that proof, we will use \cref{lem:decomposition-abstract} to partition (most of) the vertices of our graph into ``buckets'', where vertices in the same bucket have similar values of $d_v$ (for the vector $\vec{d}$ defined in \cref{def:additively-structured}). This choice of buckets will ensure that (a) holds, for a conditional random variable obtained by conditioning on the number of vertices in each bucket (the resulting conditional distribution is a product of slices).

We also remark that the precise form of the right-hand side of the inequality in (3) is not important; we only need that $\int_{|\tau|\le n^{-0.99}}|\varphi_X(\tau)-\varphi_Z(\tau)|\,d\tau$ is substantially smaller than $1/\sigma(X)$ (for small $\delta$).

\cref{lem:transfer-summary} can be interpreted as a type of \emph{Gaussian invariance principle}, comparing quadratic functions of products of slices to Gaussian analogs. There are already some invariance principles available for the Boolean slice (see \cite{FM19,FKMW18}), and it would likely be possible to prove \cref{lem:transfer-summary} by repeatedly applying results from \cite{FM19,FKMW18} to the individual factors of our product of slices. However, for our specific application it will be more convenient to deduce \cref{lem:transfer-summary} from a Gaussian invariance principle for products of Rademacher random variables.

Indeed, we will first compare $X$ to its ``independent Rademacher analog'' (i.e., to the random variable $Y$ defined as $Y=f_0+\vec f\cdot\vec{y}+\vec{y}^{\intercal}F\vec{y}$, where $\vec{y}\in \{-1,1\}^n$ is uniformly random). In order to do this, we will first show that for different choices of the sequence $(\ell_1,\ldots,\ell_m)$, we can closely couple the resulting random variables $X$ (essentially, we just randomly ``flip the signs'' of an appropriate number of entries in each $I_k$). Note that the ``balancedness'' conditions (c) and (d) in \cref{lem:transfer-summary} ensure that the expected value of $X$ does not depend strongly on the choice of $(\ell_1,\ldots,\ell_m)$.

\begin{lemma}\label{lem:slice-transfer}
Fix $0<\delta<1/4$, and consider a partition $[n]=I_1\cup\cdots\cup I_m$ as in \cref{lem:transfer-summary}, as well as a symmetric matrix $F\in\mb{R}^{n\times n}$, a vector $\vec f\in\mb{R}^{n}$ and a real number $f_0$ satisfying conditions (a--d). Assume that $n$ is sufficiently large with respect to $\delta$. 

Consider sequences $(\ell_1,\ldots,\ell_m), (\ell_1',\ldots,\ell_m')\in\mb{N}^{m}$ with $|\ell_{k}-|I_k|/2|\le\sqrt{n^{1-\delta}}\log n$ and $|\ell_{k}'-|I_k|/2|\le\sqrt{n^{1-\delta}}\log n$ for $k=1,\ldots,m$. Then, let $\vec{x}\in\{-1,1\}^{n}$ be a uniformly random vector such that $\vec{x}_{I_{k}}$ has exactly $\ell_{k}$ ones for each $k=1,\ldots,m$ and let $\vec{x}'\in\{-1,1\}^{n}$ be a uniformly random vector such that $\vec{x}'_{I_{k}}$ has exactly $\ell_{k}'$ ones for each $k=1,\ldots,m$. Let $X= f_0+\vec f\cdot\vec{x}+\vec{x}^{\intercal}F\vec{x}$ and $X'=f_0+\vec f\cdot\vec{x}'+\vec{x}'^{\intercal}F\vec{x}'$. Then we can couple $\vec{x}$ and $\vec{x}'$ such that $|X-X'|\le n^{3/4+4\delta}$
with probability at least $1-\exp(-n^{\delta/2})$.
\end{lemma}
\begin{proof}
Let us couple the random vectors $\vec{x}$ and $\vec{x}'$ in the following way. First, independently for each $k=1,\ldots,m$, let us choose a uniformly random subset $R_k\su I_k$ of size $|I_k|-2\lfloor |I_k|/2 - \sqrt{n^{1-\delta}}\log n\rfloor$. Note that then $|I_k\setminus R_k|$ is even and $2\sqrt{n^{1-\delta}}\log n\le |R_k|\le 3\sqrt{n^{1-\delta}}\log n$. We also have $0\le \ell_k-|I_k\setminus R_k|/2\le |R_k|$ and $0\le \ell_k'-|I_k\setminus R_k|/2\le |R_k|$. Let us now sample $\vec{x}_{R_k}\in \{-1,1\}^{R_k}$ by taking a uniformly random vector with exactly $\ell_{k}-|I_k\setminus R_k|/2$ ones, and independently let us sample $\vec{x}_{R_k}'\in\{-1,1\}^{R_k}$ by taking a uniformly random vector with exactly $\ell_{k}'-|I_k\setminus R_k|/2$ ones. Furthermore, let us sample a random vector in $\{-1,1\}^{I_k\setminus R_k}$ with exactly $|I_k\setminus R_k|/2$ ones and define both of $\vec{x}_{I_k\setminus R_k}$ and $\vec{x}_{I_k\setminus R_k}'$ to agree with this vector. After doing this for all $k=1,\ldots,m$, we have defined $\vec{x}$ and $\vec{x}'$ with the appropriate number of ones in each index set $I_k$. For convenience, write $R=R_1\cup\cdots\cup R_k$.

We now need to check that $|X-X'|\le n^{3/4+4\delta}$
with probability at least $1-\exp(-n^{\delta/2})$. Since $\vec{x}$ and $\vec{x}'$ agree in all coordinates outside $R$, all terms that do not involve coordinates in $R$ cancel out in $X-X'$. We may therefore write $X-X'=g_R(\vec{x})-g_R(\vec{x}')$, where (using that $F$ is symmetric)
\begin{equation}\label{eq:def-of-g-R}
g_R(\vec{x}):=\sum_{i\in R}f_ix_i + \sum_{\substack{(i,j)\in [n]\\i\in R \text{ or }j\in R}} F_{i,j}x_ix_{j}=\sum_{i\in R}f_ix_i+\sum_{(i,j)\in R^2} F_{i,j}x_ix_{j}+2\sum_{i\not\in R}\sum_{j\in R} F_{i,j}x_ix_{j}.
\end{equation}
(and similarly for $g_R(\vec{x}')$). It suffices to prove that with probability at least $1-\exp(-n^{\delta/2})/2$ we have $|g_R(\vec{x})|\le n^{3/4+4\delta}/2$ (then the same holds analogously for $|g_R(\vec{x'})|$ and overall we obtain $|X-X'|=|g_R(\vec{x})-g_R(\vec{x}')|\le n^{3/4+4\delta}$ with probability at least $1-\exp(-n^{\delta/2})$).

Let us first consider the first two summands on the right-hand side of \cref{eq:def-of-g-R}. Their expectation is
\begin{equation}\label{eq:expectation-first-two-summands-g-R}
\mb E\left[\sum_{i\in R}f_ix_i+\sum_{(i,j)\in R^2} F_{i,j}x_ix_{j}\right]=\sum_{i=1}^{n}f_i\cdot\mb E[\mbm 1_{i\in R}x_i]+\sum_{i=1}^{n}\sum_{j=1}^{n} F_{i,j}\cdot\mb E[\mbm 1_{i,j\in R}x_ix_j].
\end{equation}
Now note that for each $k=1,\ldots,m$, the expectation $\mb E[\mbm 1_{i\in R}x_i]$ is the same for all indices $i\in I_k$. Since $\sum_{i\in I_k} f_i=0$ by condition (c), this means that the first summand on the right-hand side of \cref{eq:expectation-first-two-summands-g-R} is zero. For the second summand in \cref{eq:expectation-first-two-summands-g-R}, note that for any $k,h\in [m]$ the expectation $\mb E[\mbm 1_{i,j\in R}x_ix_j]$ has  the same value $E_{k,h}$ for all indices $i\in I_k$ and $j\in I_h$ with $i\neq j$. For all $i\in I_k$ and $j\in I_h$, the magnitude of this expectation is at most $\Pr[i\in R]\le 3\sqrt{n^{1-\delta}}\log n/|I_k|\leq n^{-1/2+\delta}$ (noting that $|I_k|=n/m\geq n^{1-\delta}/2$). By (d) we have $\sum_{i\in I_k}\sum_{j\in I_h} F_{i,j}=0$, and so we can conclude that
\[\left|\mb E\left[\sum_{i\in R}f_ix_i+\sum_{(i,j)\in R^2} F_{i,j}x_ix_{j}\right]\right|=\left|\sum_{k=1}^{m}\sum_{i\in I_k} F_{i,i}(\mb E[\mbm 1_{i\in R}x_i^2]-E_{k,k})\right|\leq \sum_{i=1}^{n} |F_{i,i}| \cdot 2n^{-1/2+\delta}\le 2n^{1/2+\delta},\]
where in the last step we used (b).
Furthermore, note that
\begin{equation}\label{eq:first-two-summands-g-R}
\sum_{i\in R}f_ix_i+\sum_{(i,j)\in R^2} F_{i,j}x_ix_{j}=\vec{f}\cdot \vec{x}_R+\vec{x}_R^{\intercal}F\vec{x}_R,
\end{equation}
where here by slight abuse of notation we consider $\vec{x}_R$ as a vector in $\{-1,0,1\}^n$ given by extending $\vec{x}_R\in \{-1,1\}^R$ by zeroes for the coordinates outside $R$. Note that this describes a random vector in $\{-1,0,1\}^n$ such that for each set $I_k$ for $k=1,\ldots,m$, exactly $\ell_k\leq n^{1/2}$ entries are $1$, exactly $|I_k|-2\lfloor |I_k|/2 - \sqrt{n^{1-\delta}}\log n\rfloor-\ell_k\leq 3\sqrt{n^{1-\delta}}-\ell_k\log n\le n^{1/2}-\ell_k$ entries are $-1$, and the remaining entries are $0$. Note that for any two outcomes of such a random vector differing by switching two entries, the resulting values of $\vec{f}\cdot \vec{x}_R+\vec{x}_R^{\intercal}F\vec{x}_R$ differ by at most $5n^{1/2+3\delta}$ (indeed, by (a) the linear term $\vec{f}\cdot \vec{x}_R$ differs by at most $4\snorm{f}_\infty\leq 4n^{1/2+3\delta}$, and by (b) the term $\vec{x}_R^{\intercal}F\vec{x}_R$ differs by at most $8|R|\leq n^{1/2+3\delta}$). Thus, we can apply \cref{lem:slice-concentration} and conclude that with probability at least $1-2\exp(-n^{3/2+8\delta}/(16\cdot 2m\cdot n^{1/2}\cdot 25n^{1+6\delta}))\geq 1-2\exp(-n^{\delta}/800)$ the quantity in \cref{eq:first-two-summands-g-R} differs from its expectation by at most $n^{3/4+4\delta}/4$. Given the above bound for this expectation, we can conclude that with probability at least $1-2\exp(-n^{\delta}/800)$,
\begin{equation}\label{eq:first-two-summands-g-R-conc}
\left|\sum_{i\in R}f_ix_i+\sum_{(i,j)\in R^2} F_{i,j}x_ix_{j}\right|\le n^{3/4+4\delta}/3.
\end{equation}
It remains to bound the third summand on the right-hand side of \cref{eq:def-of-g-R}.

In order to do so, we first claim that with probability at least $1-2n\exp(-n^{\delta}/256)$ for each $i=1,\ldots,n$ we have $|\sum_{j\in R}2F_{i,j} x_j|\le n^{1/4+\delta}$. Indeed, for any fixed $i$, the sum $\sum_{j\in R}2F_{i,j}x_j$ can be interpreted as a linear function (with coefficients bounded by $2$ in absolute value by (b)) of a random vector in $\{-1,0,1\}^n$ such that for each set $I_k$ for $k=1,\ldots,m$, exactly $\ell_k\leq n^{1/2}$ entries are $1$, exactly $|I_k|-2\lfloor |I_k|/2 - \sqrt{n^{1-\delta}}\log n\rfloor-\ell_k\le n^{1/2}-\ell_k$ entries are $-1$, and the remaining entries are $0$. So for each $i=1,\ldots,n$, by \cref{lem:slice-concentration} (noting that $\mb{E}[\sum_{j\in R}F_{i,j} x_j]=0$ by (d)) we have $|\sum_{j\in R}F_{i,j} x_j|\le n^{1/4+\delta}$ with probability at least $1-2\exp(-n^{1/2+2\delta}/(2m\cdot n^{1/2}\cdot 8^2))\geq 1-2\exp(-n^{\delta}/256)$. 

Let us now condition on an outcome of $R$ and $\vec{x}_R$ such that we have $|\sum_{j\in R}2F_{i,j} x_j|\le n^{1/4+\delta}$ for $i=1,\ldots,n$. Note that
\[2\sum_{i\not\in R}\sum_{j\in R} F_{i,j}x_ix_{j}=\sum_{i\not\in R}\left(\sum_{j\in R}2F_{i,j} x_j\right)x_i.\]
Subject to the randomness of the coordinates outside $R$ (which are chosen to be half $1$ and half $-1$ inside each set $I_k\setminus R_k$ for $k=1,\ldots,m$), the expectation of this quantity is $0$ (since for each individual $x_i$ with $i\in R$ we have $\mb{E}x_i=0$). Furthermore, this quantity can be interpreted as a linear function of the entries $x_i$ with $i\not\in R$, with coefficients bounded in absolute value by $n^{1/4+\delta}$. Thus, by \cref{lem:slice-concentration} we have $|2\sum_{i\not\in R}\sum_{j\in R} F_{i,j}x_ix_{j}|\leq n^{3/4+3\delta}$ with probability at least $1-2\exp(-n^{3/2+6\delta}/(2n\cdot 16n^{1/2+2\delta})\geq 1-2\exp(-n^{\delta})$.

Combining this with \cref{eq:first-two-summands-g-R-conc} and \cref{eq:def-of-g-R}, we conclude that $|g_R(\vec{x})|\leq n^{3/4+4\delta}/2$ with probability at least $1-2(n+2)\exp(-n^{\delta}/800)\geq 1-\exp(-n^{\delta/2})/2$.
\end{proof}

The following lemma gives a comparison between the random variable $X$ in \cref{lem:transfer-summary} and its ``independent Rademacher analog''. This lemma is a simple consequence of \cref{lem:slice-transfer}, since a uniformly random vector $\vec{y}\in\{-1,1\}^n$ can be interpreted as a mixture of different Boolean slices.

\begin{lemma}\label{lem:transfer-rad}
Fix $0<\delta<1/4$, and consider a partition $[n]=I_1\cup\cdots\cup I_m$ as in \cref{lem:transfer-summary}, as well as a symmetric matrix $F\in\mb{R}^{n\times n}$, a vector $\vec f\in\mb{R}^{n}$ and a real number $f_0$ satisfying conditions (a--d). Assume that $n$ is sufficiently large with respect to $\delta$. 

Consider a sequence $(\ell_1,\ldots,\ell_m)\in\mb N^{m}$
with $|\ell_{k}-|I_k|/2|\le\sqrt{n^{1-\delta}}\log n$ and 
for $k=1,\ldots,m$, and let $\vec{x}\in\{-1,1\}^{n}$ be a uniformly
random vector such that $\vec{x}_{I_{k}}$ has exactly $\ell_{k}$
ones for each $k=1,\ldots,m$. Furthermore let $\vec{y}\in\{-1,1\}^{n}$ be a uniformly random vector (with independent coordinates). Let $X=\vec f_0+f\cdot\vec{x}+\vec{x}^{\intercal}F\vec{x}$
and $Y=f_0+\vec f\cdot\vec{y}+\vec{y}^{\intercal}F\vec{y}$. Then we can couple $\vec{x}$ and $\vec{y}$ such that $|X-Y|\le n^{3/4+4\delta}$
with probability at least $1-\exp(-(\log n)^2/8)$.
\end{lemma}
\begin{proof}
For $k=1,\ldots,m$, consider independent binomial random variables $\ell_k'\sim\mr{Bin}(|I_k|,1/2)$. We can sample $\vec{y}$ by taking a random vector in $\{-1,1\}^n$ with exactly $\ell_{k}'$ ones among the entries with indices in $I_k$ for each $k=1,\ldots,m$. Note that altogether this gives precisely a uniformly random vector in $\{-1,1\}^{n}$.

We now need to define the desired coupling of $\vec{x}$ and $\vec{y}$. By the Chernoff bound (see \cref{lem:chernoff}), with probability at least $1-4n^\delta\cdot\exp(-(\log n)^2/4)\le 1-\exp(-(\log n)^2/6)$ we have $|\ell_{k}'-|I_k|/2|\le\sqrt{n^{1-\delta}}\log n$ for $k=1,\ldots,m$ (here, we used that $m\leq 2n^\delta$ and $|I_k|=n/m\le 2n^{1-\delta}$). Whenever this is the case, then by \cref{lem:slice-transfer} we can couple $\vec x$ and $\vec{y}$ in such a way that we have $|X-Y|\le n^{3/4+4\delta}$ with probability at least $1-\exp(-n^{\delta/2})$. Otherwise, let us couple $\vec x$ and $\vec{y}$ arbitrarily.

Now, the overall probability of having $|X-Y|\le n^{3/4+4\delta}$ is at least $1-\exp(-(\log n)^2/6)-\exp(-n^{\delta/2})\geq 1-\exp(-(\log n)^2/8)$, as desired.
\end{proof}

In order to obtain the comparison of the characteristic functions of $X$ and $Z$ in \cref{lem:transfer-summary}(3), we will use \cref{lem:transfer-rad} to relate $X$ to $Y$. It then remains to compare the characteristic functions of $Y$ and $Z$. To do so, we use the \emph{Gaussian invariance principle} of Mossel, O'Donnell, and Oleszkiewicz \cite{MOO10}. The version stated in \cref{thm:invariance} below is a special case of \cite[(11.29)]{O14}.

\begin{definition}\label{def:influence}
Given a multilinear polynomial $g(x_1,\ldots,x_n) = \sum_{S\subseteq [n]} a_S\prod_{i\in S}x_i$, for $t=1,\ldots,n$ the \emph{influence} of the variable $x_t$ is defined as
\[\on{Inf}_t[g] = \sum_{\substack{S\su [n]\\t\in S}} a_S^2.\]
\end{definition}

\begin{theorem}\label{thm:invariance}
Let $g$ be an $n$-variable multilinear polynomial of degree at most $k$. Let $\vec{y}\in \{-1,1\}^n$ be a uniformly random vector (i.e., a vector of independent Rademacher random variables), and let $\vec{z}\sim\mc{N}(0,1)^{\otimes n}$ be a vector of independent standard Gaussian random variables. Then for any four-times-differentiable function $\psi\colon\mb{R}\to\mb{R}$, we have \[\Big|\mb{E}[\psi(g(\vec{y}))-\psi(g(\vec{z}))]\Big|
\le \frac{9^k}{12}\cdot \snorm{\psi^{(4)}}_{\infty}\sum_{t=1}^n\on{Inf}_t[g]^{2}.\]
\end{theorem}

As a simple consequence of \cref{thm:invariance}, we obtain the following lemma.

\begin{lemma}\label{lem:invar}
Fix $0<\delta<1/4$. Consider a vector $\vec{f}\in\mb{R}^n$ with $\snorm{\vec{f}}_{\infty}\le n^{1/2+3\delta}$ and a matrix $F\in\mb{R}^{n\times n}$ with entries bounded in absolute value by 1, as well as a real number $f_0$. Assume that $n$ is sufficiently large with respect to $\delta$.

Let $\vec{y}\in\{-1,1\}^n$ be a uniformly random vector, and let $\vec{z}\sim\mc{N}(0,1)^{\otimes n}$ be a vector of independent standard Gaussian random variables. Let $Y = f_0+\vec{f}\cdot \vec{y} + \vec{y}^\intercal F\vec{y}$ and $Z = f_0+\vec{f}\cdot \vec{z} + \vec{z}^\intercal F\vec{z}$. Then for any four-times-differentiable function $\psi\colon\mb{R}\to \mb{R}$, we have
\[\Big|\mb{E}[\psi(Y)-\psi(Z)]\Big|
\lesssim  \snorm{\psi^{(4)}}_{\infty}\cdot n^{3+12\delta} + \snorm{\psi'}_{\infty}\cdot n^{1/2}.\]
\end{lemma}

\begin{proof}
Let $F'$ be obtained from $F$ by setting each diagonal entry to zero. Define the multilinear polynomial $g$ by $g(\vec{x}) = f_0+\vec{f}\cdot \vec{x} + \vec{x}^\intercal F'\vec{x} + \sum_{i}F_{i,i}$, and let $Y' = g(\vec y)$ and $Z'=g(\vec z)$. Note that $\on{Inf}_t[g]\le (n^{1/2+3\delta})^2 + n \le 2n^{1+6\delta}$ for $t=1,\ldots,n$, so $\sum_{t=1}^{n}\on{Inf}_t[g]^{2}\le 4n^{3+12\delta}$. \cref{thm:invariance} then implies that
\[\Big|\mb{E}[\psi(Y')-\psi(Z')]\Big|\le 27 \snorm{\psi^{(4)}}_{\infty}\cdot n^{3+12\delta}.\]
Furthermore, we always have $y_i^2 = 1$ for $i=1,\ldots,n$, meaning that $Y' = Y$ and in particular $\mb{E}[\psi(Y')-\psi(Y)]=0$. By the Cauchy--Schwarz inequality, we also have
\[|\mb{E}[\psi(Z')-\psi(Z)]|\le \mb{E}|\psi(Z')-\psi(Z)|\le \snorm{\psi'}_{\infty}\cdot \mb{E}|Z'-Z|\le  \snorm{\psi'}_{\infty}\cdot (\mb{E}[(Z'-Z)^2])^{1/2}\le 2\snorm{\psi'}_{\infty}n^{1/2},\]
where we used $\mb{E}[(Z'-Z)^2]=\mb{E}[(F_{1,1}(z_1^2-1)+\cdots+F_{n,n}(z_n^2-1))^2]=2|F_{1,1}|^2+\cdots+2|F_{n,n}|^2\le 2n$ in the last step. Combining these estimates gives the desired result. 
\end{proof}

Let us now prove \cref{lem:transfer-summary}.

\begin{proof}[Proof of \cref{lem:transfer-summary}]
We may assume that $n$ is sufficiently large with respect to $\delta$. Let $\vec{y}\in\{-1,1\}^n$ be a uniformly random vector and define $Y=f_0+\vec f\cdot\vec{y}+\vec{y}^{\intercal}F\vec{y}$. By \cref{lem:transfer-rad} we can couple $\vec{x}$ and $\vec{y}$ such that $|X-Y|\le n^{3/4+4\delta}$ with probability at least $1-\exp(-(\log n)^2/8)$.

We can now compute $\mb{E}Y=\mb{E}Z=f_0+\sum_{i=1}^{n}F_{i,i}$. Furthermore, since $|X-Y|\lesssim n^{2}$ always holds, we have $|\mb{E}X-\mb{E}Y|\le \mb{E}|X-Y|\lesssim n^{3/4+4\delta}+\exp(-(\log n)^2/8)\cdot n^{2}\lesssim n^{3/4+4\delta}$ and therefore $\mb{E}X=f_0+\sum_{i=1}^{n}F_{i,i}+O(n^{3/4+4\delta})$. This proves (1).

Note that $Y-\mb{E}Y=\vec{f}\cdot\vec{y}+\sum_{i<j} 2F_{i,j}y_iy_j$ (here we are using that $y_i^2=1$ and that $F$ is symmetric). Therefore \cref{eq:boolean-variance} gives $\sigma(Y)^2=\snorm{\vec{f}}_2^2+\sum_{i<j} 4F_{i,j}^2=2\snorm{F}_{\mr F}^2+\snorm{\vec{f}}_2^2-2\sum_{i=1}^{n}F_{i,i}^2=2\snorm{F}_{\mr F}^2+\snorm{\vec{f}}_2^2+O(n)$ (and so in particular $\sigma(Y)^2\lesssim n^{2+6\delta}$). Furthermore (using the Cauchy--Schwarz inequality), we have
\begin{align*}
|\sigma(X)^2-\sigma(Y)^2|&=\left|\mb{E}\left[(X-\mb{E}X)^2-(Y-\mb{E}Y)^2\right]\right|\le \mb{E}\left[|X-Y-\mb{E}X+\mb{E}Y|\cdot |X+Y-\mb{E}X-\mb{E}Y|\right]\\
&\le\left(\mb{E}\left[(|X-Y|+|\mb{E}X-\mb{E}Y|)^2\right]\right)^{1/2}\cdot \left(\mb{E}\left[(|X-\mb{E}X|+|Y-\mb{E}Y|)^2\right]\right)^{1/2}\\
&\le\left(\mb{E}\left[(|X-Y|+O(n^{3/4+4\delta}))^2\right]\right)^{1/2}\cdot \left(2\mb{E}\left[|X-\mb{E}X|^2\right]+2\mb{E}\left[|Y-\mb{E}Y|^2\right]\right)^{1/2}\\
&\lesssim \left(\mb{E}[|X-Y|^2] +\mb{E}|X-Y|\cdot O(n^{3/4+4\delta})+O(n^{3/2+8\delta})\right)^{1/2}\cdot \left(\sigma(X)^2+\sigma(Y)^2\right)^{1/2}\\
&\lesssim \left(n^{3/2+8\delta}+\exp(-(\log n)^2/8)\cdot n^{4} +O(n^{3/2+8\delta})\right)^{1/2}\cdot (\sigma(X)+\sigma(Y))\\
&\lesssim n^{3/4+4\delta}\cdot (\sigma(X)+\sigma(Y)).
\end{align*}
Hence $|\sigma(X)-\sigma(Y)|\lesssim n^{3/4+4\delta}$ and in particular $\sigma(X)\leq \sigma(Y)+O(n^{3/4+4\delta})\lesssim n^{1+3\delta}$. Thus, we obtain $|\sigma(X)^2-\sigma(Y)^2|= |\sigma(X)-\sigma(Y)| (\sigma(X)+\sigma(Y))\lesssim n^{3/4+4\delta} \cdot n^{1+3\delta}=n^{7/4+7\delta}$. This gives $\sigma(X)^2=\sigma(Y)^2+O(n^{7/4+7\delta})=2\snorm{F}_{\mr F}^2+\snorm{\vec{f}}_2^2+O(n^{7/4+7\delta})$.

To finish the proof of (2), we observe that $Z-\mb{E}Z=\vec{f}\cdot\vec{z}+\sum_{i=1}^{n} F_{i,i}(z_i^2-1)+\sum_{i<j} 2F_{i,j}z_iz_j$, so we can compute $\sigma(Z)^2=\snorm{\vec{f}}_2^2+\sum_{i=1}^{n} 2F_{i,i}^2+\sum_{i<j}(2F_{i,j})^2=2\snorm{F}_{\mr F}^2+\snorm{\vec{f}}_2^2$.

For (3), consider some $\tau\in\mb{R}$. We have
\begin{align*}
|\varphi_Y(\tau)-\varphi_Z(\tau)|&=\Big|\mb{E}[\exp(i\tau Y)-\exp(i\tau Z)]\Big|=\Big|\mb{E}[\cos(\tau Y)+i\sin(\tau Y)-\cos(\tau Z)-i\sin(\tau Z)]\Big|\\
&\le \Big|\mb{E}[\cos(\tau Y)-\cos(\tau Z)]\Big|+\Big|\mb{E}[\sin(\tau Y)-\sin(\tau Z)]\Big|\Big|\lesssim |\tau|^4\cdot n^{3+12\delta} + |\tau|\cdot n^{1/2},
\end{align*}
where in the last step we applied \cref{lem:invar} to the functions $u\mapsto\cos(\tau u)$ and $u\mapsto\sin(\tau u)$.
We furthermore have
\[ |\varphi_X(\tau)-\varphi_Y(\tau)|=\Big|\mb{E}[\exp(i\tau X)-\exp(i\tau Y)]\Big|\le \mb{E}\Big[|\exp(i\tau X)-\exp(i\tau Y)|\Big] \le |\tau|\cdot \mb{E}[|X-Y|]\lesssim |\tau| \cdot n^{3/4+4\delta},\]
using that the absolute value of the derivative of the function $u\mapsto \exp(i\tau u)$ is bounded by $|\tau|$. Combining these two bounds using the triangle inequality gives (3).
\end{proof}

\section{Short interval control in the additively structured case}\label{sec:short-interval-structured}
Recall the definition of $\gamma$-structuredness from \cref{def:additively-structured}, and recall that in \cref{sec:short-interval-unstructured} we fixed $\gamma=10^{-4}$ and proved \cref{thm:short-interval} in the case where $(G,\vec e)$ is $\gamma$-unstructured. In this section, we finally prove \cref{thm:short-interval} in the complementary case where $(G,\vec e)$ is $\gamma$-structured.

As outlined in \cref{sec:outline}, the idea is as follows. First, we apply  \cref{lem:decomposition-abstract} to the vector $\vec{d}$ in \cref{def:additively-structured} to divide the vertex
set into ``buckets'' such that the $d_v$ in each bucket have similar values.  We encode the number of vertices in each bucket as a vector $\vec \Delta$; if we condition on an outcome of $\vec \Delta$ then we can use
the machinery developed in the previous sections
to prove upper and lower bounds on the conditional small-ball
probabilities of $X$. Then, we need to average
these estimates over $\vec \Delta$. For this averaging, it is important that our conditional small-ball probabilities decay as we vary $\vec \Delta$ (this is where we need the non-uniform anticoncentration estimates in \cref{thm:gaussian-anticoncentration-technical}(1) and \cref{lem:esseen-upper}).

This section mostly consists of combining ingredients from previous sections, but there are still a few technical difficulties remaining. Chief among these is the fact that, as we vary the numbers of vertices in each bucket, the conditional expected value and variance of $X$ fluctuate fairly significantly. We need to keep track of these fluctuations and ensure that they do not correlate adversarially with each other.

\begin{proof}[Proof of \cref{thm:short-interval} in the $\gamma$-structured case]
Recall that $G$ is a $C$-Ramsey graph with $n$ vertices, $e_0\in\mb{R}$ and $\vec{e}\in\mb{R}^{V(G)}$ is a vector satisfying $0\leq e_v\leq Hn$ for all $v\in V(G)$, and that $U\su V(G)$ is a uniformly random vertex subset and $X=e(G[U])+\sum_{v\in U}e_v+e_0$. We may assume that $n$ is sufficiently large with respect to $C,H$, and $A$.

\medskip
\noindent\hypertarget{step:bucketing}{\textit{Step 1: Bucketing setup.}}
As in \cref{def:additively-structured}, define $\vec{d}\in\mb{R}^{V(G)}$ by $d_v=e_v+\deg_G(v)/2$ for all $v\in V(G)$. We are assuming that $(G,\vec e)$ is $\gamma$-structured, meaning that $\wh{D}_{L,\gamma}(\vec{d})\le n^{1/2}$, where $L=\lceil 100/\gamma\rceil=10^6$ (recall that $\gamma=10^{-4}$).

Note that $\snorm{\vec{d}}_\infty\le (H+1)n$. Furthermore, for any subset $S\su V(G)$ of size $|S|=\lceil n^{1-\gamma}\rceil$, we have $\snorm{\vec{d}_S}_2\gtrsim_H n^{3/2-3\gamma/2}$ by \cref{lem:two-norm-subsets-vec-d} and therefore in particular $\snorm{\vec{d}_S}_2\ge n^{3/2-2\gamma}$. Thus, we can apply \cref{lem:decomposition-abstract} and obtain a partition $V(G)=R\cup (I_1\cup\cdots\cup I_m)$ and real numbers $\kappa_1,\ldots,\kappa_m\geq 0$ with $|R|\le n^{1-\gamma}$ and $|I_1|=\cdots=|I_m|=\lceil n^{1-2\gamma}\rceil$ such that $|d_v-\kappa_k|\leq n^{1/2+4\gamma}$ for all $k=1,\ldots,m$ and $v\in I_k$. Let $V=I_{1}\cup\cdots\cup I_{m}=V(G)\setminus R$.

Since $|R|\le n^{1-\gamma}$, we have $2n/3\leq |V|\le n$ (i.e., $|V|$ is of order $n$) and thus furthermore $|V|^{2\gamma}/2\le n^{2\gamma}/2\le m\le 2^{1-2\gamma}n^{2\gamma}\leq 2|V|^{2\gamma}$ (which means that we can apply \cref{lem:ramsey-low-rank,lem:transfer-summary} to the partition $V=I_1\cup\cdots\cup I_m$).

In the next step of the proof, we will condition on an outcome of $U\cap R$, and from then on we will only use the randomness of $U\cap (I_1\cup\cdots\cup I_m)=U\cap V$.

\medskip
\noindent\hypertarget{step:condition-garbage}{\textit{Step 2: Conditioning on an outcome of $U\cap R$.}}
Recall that $U\su V(G)$ is a random subset obtained by including each vertex with probability $1/2$ independently. Let $x_{v}=1$ if $v\in U$ and $x_{v}=-1$ if $v\notin U$, so the $x_{v}$ are independent Rademacher random variables. Then, as in \cref{eq:fourier-walsh} and the proof of \cref{prop:lowest-t}
our random variable $X=e(G[U])+\sum_{v\in U}e_v+e_0$ can be expressed as
\begin{equation}\label{eq:X-first-expression}
\mb E X+\frac{1}{2}\sum_{v\in V(G)}\left(e_{v}+\frac{1}{2}\deg_{G}(v)\right)x_{v}+\frac{1}{4}\sum_{uv\in E(G)}x_{u}x_{v}=\mb E X+\frac{1}{2}\sum_{v\in V(G)}d_vx_{v}+\frac{1}{4}\sum_{uv\in E(G)}x_{u}x_{v}.
\end{equation}
Let us now write $\vec{x}$ for the vector $(x_{v})_{v\in V}$; we emphasize that this does not include the indices in $R$. We first
rewrite \cref{eq:X-first-expression} as a quadratic polynomial in $\vec{x}$
(where we view the random variables $x_{u}$ for $u\in R=V(G)\setminus V$
as being part of the coefficients of this quadratic polynomial). To this
end, let $M\in \{0,1\}^{V\times V}$ be the adjacency matrix of $G[V]$, and also define
\[y_{v}=d_{v}+\frac{1}{2}\sum_{\substack{u\in R\\uv\in E(G)}}x_{u}~\text{for }v\in V\quad\quad\text{and}\quad\quad E=\mb E X+\frac{1}{2}\sum_{v\in R}d_vx_v+\frac{1}{4}\sum_{uv\in E(G[R])}x_{u}x_{v}.\]
Then
\begin{equation}\label{eq:R-gone-formula}
X=E+\frac{1}{2}\vec{y}\cdot\vec{x}+\frac{1}{8}\vec{x}^{\intercal}M\vec{x}.
\end{equation}
Since $|R|\le n^{1-\gamma}$, and $0\leq d_v\leq Hn+n/2\leq (H+1)n$ for all $v\in V(G)$, \cref{thm:concentration-hypercontractivity} (concentration via hypercontractivity) in combination with \cref{eq:boolean-variance} shows that with probability at least $1-\exp(-\Omega_H(n^{\gamma/2}))$ (over the randomness of $x_u$ for $u\in R$) we have 
\[\bigg|\sum_{\substack{u\in R\\uv\in E(G)}}x_{u}\bigg|\le n^{1/2}~\text{for each }v\in V,\qquad\qquad\bigg|\sum_{uv\in E(G[R])}x_{u}x_{v}\bigg|\le n, \qquad\qquad \bigg|\sum_{v\in R}d_vx_{v}\bigg|\le n^{3/2}/2,\]
which implies that $|E-\mb E X|\le n^{3/2}$ and $|y_{v}-d_v|\le n^{1/2}$
for all $v\in V$. For the rest of the proof, we implicitly condition
on an outcome of $U\cap R$ satisfying these properties, and we treat $E$ and $\vec{y}=(y_{v})_{v\in V}$ as being non-random objects.
 
Note that $\snorm{\vec{y}}_\infty\leq Hn+n/2+n^{1/2}\leq (H+2)n$ and $\snorm{\vec{y}}_2\geq \snorm{\vec{d}_V}_2-\snorm{\vec{y}-d_V}_2\geq \snorm{\vec{d}_V}_2-n$. Furthermore, we have $\snorm{\vec{d}_V}_2\gtrsim_C n^{3/2}$ by \cref{lem:two-norm-subsets-vec-d} and therefore $\snorm{\vec{y}}_2\gtrsim_C n^{3/2}$.

\medskip
\noindent\hypertarget{step:rewriting-Delta}{\textit{Step 3: Rewriting $X$ via bucket intersection sizes.}}
Recall that we have a partition $V=I_1\cup\cdots\cup I_m$ into ``buckets'' with $|I_1|=\cdots=|I_{m}|=|V|/m$ and $|V|^{2\gamma}/2\leq m\leq 2|V|^{2\gamma}$. Let $I\in\mb{R}^{V\times V}$ be the identity matrix, and let $Q\in\mb{R}^{V\times V}$ be the symmetric matrix defined by taking $Q_{u,v}=1/|I_{k}|=m/|V|$ for $u,v$ in the same bucket $I_{k}$, and $Q_{u,v}=0$ otherwise. 
Multiplying a vector $\vec{v}\in\mb{R}^V$ by this matrix $Q$ has the effect of averaging the entries of $\vec{v}$ over each of the buckets $I_k$, and hence $(I-Q)\vec{v}$ has the property that for $k=1,\ldots,m$ the sum of the entries in $\vec{v}_{I_k}$ is zero.

Let us define $\vec{\Delta}\in\mb{R}^V$ by $\vec{\Delta}=Q\vec{x}$, so for any $k=1,\ldots,m$ and any $v\in I_k$ we have
\[\Delta_v=\frac{1}{|I_{k}|}\sum_{u\in I_{k}}x_{u}=\frac{2}{|I_{k}|}\left(|U\cap I_{k}|-\frac{|I_{k}|}{2}\right).\]
Hence, $\vec{\Delta}$ encodes the sizes of the intersections $|U\cap I_{k}|$ for $k=1,\ldots,m$. In our analysis of the random variable $X$, we will condition on an outcome of $\vec{\Delta}$ and apply \cref{lem:transfer-summary} to study $X$ conditioned on $\vec{\Delta}$. However, the vector $\vec{y}$ and the matrix $M$ appearing in \cref{eq:R-gone-formula} do not satisfy conditions (a), (c), and (d) in \cref{lem:transfer-summary}. So, we need to modify the representation of $X$ in \cref{eq:R-gone-formula}.

Define $M^{*}=\frac{1}{8}(I-Q)M(I-Q)$ and $\vec{w}^{*}_{\vec{\Delta}}=\frac{1}{2}(I-Q)(\vec{y}+\frac12M\vec{\Delta})$. Then (recalling that $Q$ is symmetric)
\begin{align}
X&=E+\frac{1}{2}\vec{y}\cdot\vec{x}+\frac{1}{8}\vec{x}^{\intercal}M\vec{x}\notag\\
&=E+\frac{1}{2}(I-Q)\vec{y}\cdot\vec{x}+\frac{1}{2}\vec{y}\cdot (Q\vec{x})+\frac{1}{8}\vec{x}^{\intercal}(I-Q)M(I-Q)\vec{x}+\frac{1}{4}\vec{x}^{\intercal}(I-Q)MQ\vec{x}+\frac{1}{8}\vec{x}^{\intercal}QMQ\vec{x}\notag\\
&=\left(E+\frac{1}{2}\vec{y}\cdot\vec{\Delta}+\frac{1}{8}\vec{\Delta}^{\intercal}M\vec{\Delta}\right)+\vec{w}^{*}_{\vec{\Delta}}\cdot\vec{x}+\vec{x}^{\intercal}M^{*}\vec{x}.\label{eq:normalized-X-formula}
\end{align}
Furthermore, $M^{*}$ has the property that for all $k,h\in [m]$, in the submatrix $M^*[I_k\!\times\! I_h]$ all row and column sums are zero, and $\vec{w}^{*}_{\vec{\Delta}}$
has the property that for each $k=1,\ldots,m$, the sum of entries in $(\vec{w}^{*}_{\vec{\Delta}})_{I_{k}}$
is equal to zero. Also note that since $M$ has entries in $\{0,1\}$, all entries of $(I-Q)MQ$ and hence all entries of $M^*$ have absolute value at most $1$. Thus, $\vec{w}^{*}_{\vec{\Delta}}$ and $M^*$ satisfy conditions (b)--(d) in \cref{lem:transfer-summary}.

Also, since $M^{*}$ is defined in terms of the adjacency matrix of a Ramsey graph, \cref{lem:ramsey-low-rank} tells us that it must have large Frobenius norm. Indeed,
\begin{equation}\label{eq:M-F}
\|M^{*}\|_{\mr F}^{2}=\frac{1}{64}\|M-(MQ+QM-QMQ)\|_{\mr F}^{2}\gtrsim_{C} n^{2}
\end{equation}
by \cref{lem:ramsey-low-rank} applied with $\delta=2\gamma=2\cdot 10^{-4}$ and $r=3$ (here we are using that $M$ is the adjacency matrix of the $(2C)$-Ramsey graph $G[V]$ of size $|V|\gtrsim n$, and we are using that the matrix $B=MQ+QM-QMQ\in\mb{R}^{V\times V}$ has the property that $\operatorname{rank}B[I_{k}\!\times\! I_{h}]\leq 3$ for all $k,h\in[m]$). 

\medskip
\noindent\hypertarget{step:condition-Delta}{\textit{Step 4: Conditioning on bucket intersection sizes.}}
By a Chernoff bound, with probability at least $1-2n^{2\gamma}\cdot n^{-\omega(1)}=1-n^{-\omega(1)}$ we have $\big||U\cap I_k|-|I_k|/2\big|\le\sqrt{|I_k|}(\log n)/2=\sqrt{|V|/m}\cdot (\log n)/2$ for $k=1,\ldots,m$, or equivalently $|\Delta_{v}|\le\sqrt{m/|V|}\log n$ for all $v\in V$.

We furthermore claim that with probability $1-n^{-\omega(1)}$ we have $\|\vec{w}^{*}_{\vec{\Delta}}\|_{\infty}\le n^{1/2+5\gamma}$. Indeed, recall that $\vec{w}^{*}_{\vec{\Delta}}=\frac{1}{2}(I-Q)(\vec{y}+\frac12M\vec{\Delta})$ and (from \hyperlink{step:condition-garbage}{Step 2}) $|y_{v}-d_v|\le n^{1/2}$ for all $v\in V$.
Recall from the choice of buckets in \hyperlink{step:bucketing}{Step 1} that for all $k=1,\ldots,m$ and $v\in I_k$, we have $|d_v-\kappa_k|\leq n^{1/2+4\gamma}$, implying that $|y_{v}-\kappa_k|\le 2n^{1/2+4\gamma}$. In particular, we obtain $|y_{v}-y_u|\le 4n^{1/2+4\gamma}$ for all $u,v\in V$ that are in the same bucket $I_k$. Hence $\snorm{(I-Q)\vec{y}}_\infty\leq 4n^{1/2+4\gamma}$. Furthermore, since all entries of $(I-Q)MQ$ have absolute value at most $1$, \cref{thm:concentration-hypercontractivity} (concentration via hypercontractivity) shows that with probability at least $1-n\cdot n^{-\omega(1)}=1-n^{-\omega(1)}$ we have $\|(I-Q)M\vec{\Delta}\|_{\infty}=\|(I-Q)MQ\vec x\|_{\infty}\le\sqrt{n}\log n$, which now implies $\|\vec{w}^{*}_{\vec{\Delta}}\|_{\infty}\le n^{1/2+5\gamma}$ as claimed.

Let us say that an outcome of $\vec{\Delta}$ is \emph{near-balanced} if $\|\vec{w}^{*}_{\vec{\Delta}}\|_{\infty}\le n^{1/2+5\gamma}$ and $|\Delta_{v}|\le\sqrt{m/|V|}\log n$ for all $v\in V$. We have just shown that $\vec{\Delta}$ is near-balanced with probability $1-n^{-\omega(1)}$. Note that for near-balanced $\vec{\Delta}$ we in particular have $\|\vec{w}^{*}_{\vec{\Delta}}\|_{\infty}\le|V|^{1/2+6\gamma}$ and $\big||U\cap I_k|-|I_k|/2\big|\le\sqrt{|V|/m}\cdot(\log n)/2\le \sqrt{|V|^{1-2\gamma}}\log|V|$ for $k=1,\ldots,m$.  If we condition on a near-balanced outcome of $\vec{\Delta}$ (which is equivalent to conditioning on the bucket intersection sizes $|U\cap I_k|$ for $k=1,\ldots,m$), then we are in a position to apply \cref{lem:transfer-summary} with $\delta=2\gamma=2\cdot 10^{-4}$. Together with the machinery in \cref{sec:fourier-analysis,sec:high-fourier,sec:ramsey-robust-rank,sec:anti-Gauss} we can then obtain upper and lower bounds for the probability that, conditioning on our outcome of $\vec\Delta$, the random variable $X$ lies in some short interval\footnote{
Our upper and lower bounds for this probability differ by a constant factor. As suggested by one of the anonymous referees, one may wonder whether in this setting it would also be possible to characterize this probability for short intervals asymptotically (up to a $1+o(1)$ factor), potentially even asymptotically characterising the conditional point probabilities of the form $\Pr[X=x|\vec \Delta]$ (proving a \emph{local limit theorem} conditional on the outcome of $\Delta$). While one might be able to asymptotically characterize conditional small-ball probabilities of the form $\Pr[|X-x|\le B|\vec \Delta]$ when $B\to \infty$ as $n\to\infty$ by adapting the arguments in this paper, characterising point probabilities (or probabilities for bounded-length intervals) would likely require significant new ideas.}

To state such upper and lower bounds, let us write $E_{\vec{\Delta}}=\mb E[X|\vec{\Delta}]$ and define $\sigma_{\vec{\Delta}}\geq 0$ to satisfy $\sigma_{\vec{\Delta}}^{2}=\on{Var}[X|\vec{\Delta}]$. By \cref{lem:transfer-summary}(2), for near-balanced $\vec{\Delta}$ we have $\sigma_{\vec{\Delta}}^{2}=2\|M^{*}\|_{\mr F}^{2}+\|\vec{w}^{*}_{\vec{\Delta}}\|_{2}^{2}+O(n^{7/4+14\gamma})$, implying that  $\sigma_{\vec{\Delta}}\ge \|M^{*}\|_{\mr F}\gtrsim_C n$ by \cref{eq:M-F}.

\begin{claim}\label{clm:conditional-point-probability}
There is a constant $B=B(C)>0$
such that the following holds for any fixed near-balanced outcome of $\vec{\Delta}$.
\begin{enumerate}
\item For any $x\in\mb Z$ we have
\[
\Pr\left[|X-x|\le B\middle|\vec{\Delta}\right]\lesssim_{C}\frac{\exp\left(-\Omega_{C}\left(|x-E_{\vec{\Delta}}|/\sigma_{\vec{\Delta}}\right)\right)+n^{-0.1}}{\sigma_{\vec{\Delta}}}
\]
\item There is a sign $s\in\{-1,1\}$, depending only on $M^*$,
such that for any fixed $A>0$ and any $x\in\mb Z$ satisfying
$3n\le s(x-E_{\vec{\Delta}})\le A\sigma_{\vec{\Delta}}$, we have
\[
\Pr\left[|X-x|\le B\middle|\vec{\Delta}\right]\gtrsim_{C,A}\frac{1}{\sigma_{\vec{\Delta}}}.
\]
\end{enumerate}
\end{claim}
We defer the proof of \cref{clm:conditional-point-probability} until
the end of the section (specifically, we will prove it in \cref{sub:proofs-of-claims}). The proof combines the machinery from \cref{sec:fourier-analysis,sec:high-fourier,sec:ramsey-robust-rank,sec:anti-Gauss,sec:coupling-lemmas}.

\medskip
\noindent\hypertarget{step:conditional-E-sigma}{\textit{Step 5: Estimating the conditional mean and variance.}}
We wish to average the estimates in \cref{clm:conditional-point-probability}
over different near-balanced outcomes of $\vec{\Delta}$. To this end, we need to understand how the conditional mean and variance $E_{\vec{\Delta}}=\mb E[X|\vec{\Delta}]$ and $\sigma_{\vec{\Delta}}^{2}=\on{Var}[X|\vec{\Delta}]$ depend on $\vec{\Delta}$ (recall that we already fixed an outcome for $U\cap R$ in \hyperlink{step:condition-garbage}{Step 2}, which in particular fixes $E$ and $\vec{y}$). Most importantly, $E_{\vec{\Delta}}$ positively correlates with the coordinates of $\vec{\Delta}$: recall that $\vec{\Delta}$ encodes the number of vertices of our random set $U$ in each bucket, so naturally if we take more vertices we are likely to increase the number of edges we end up with. However, there are also certain (lower order, nonlinear) adjustments that we need to take into account. In this subsection we will define ``shift'' random variables $E_{\mathrm{shift}(1)},E_{\mathrm{shift}(2)}$ and $\sigma_\mr{shift}$ depending on $\vec \Delta$. We then show that these shift random variables control the dependence of $E_{\vec{\Delta}}$ and $\sigma_{\vec{\Delta}}$ on $\vec \Delta$.

Let $E_{\mathrm{shift}(1)}=\frac{1}{2}\vec{y}\cdot\vec{\Delta}$
and $E_{\mathrm{shift}(2)}=\frac{1}{8}\vec{\Delta}^{\intercal}M\vec{\Delta}$. Recalling \cref{eq:normalized-X-formula}, by \cref{lem:transfer-summary}(1) (applied with $\delta=2\gamma$) we have $E_{\vec{\Delta}}=\mb E[X|\vec{\Delta}]=E+E_{\mathrm{shift}(1)}+E_{\mathrm{shift}(2)}+\sum_{v\in V} M^*_{v,v}+O(n^{3/4+8\gamma})$ if $\vec{\Delta}$ is near-balanced. Recalling $\gamma=10^{-4}$ and that all entries of $M^*$ have absolute value at most $1$, we obtain
\begin{equation}\label{eq:cond-EX}
\left|E_{\vec{\Delta}}-E-E_{\mathrm{shift}(1)}-E_{\mathrm{shift}(2)}\right|\le 2n.
\end{equation}
for all near-balanced $\vec{\Delta}$ (i.e., $E_{\vec{\Delta}}$ is ``shifted'' by about $E_{\mathrm{shift}(1)}+E_{\mathrm{shift}(2)}$ from $E$).

Recall that $\|\vec y\|_2\gtrsim_C n^{3/2}$ and $\|\vec y\|_\infty\le (H+2)n$ from the end of \hyperlink{step:condition-garbage}{Step 2}. Furthermore, we observed that $\|(I-Q)\vec y\|_\infty\le 4n^{1/2+4\gamma}$ in \hyperlink{step:condition-Delta}{Step 4}, which implies $\|(I-Q)\vec y\|_2\le 4n^{1+4\gamma}$. Thus we obtain $\|Q\vec y\|_2\ge\|\vec y\|_2-\|(I-Q)\vec y\|_2\gtrsim_C n^{3/2}$ and $\|Q\vec y\|_\infty\le (H+2)n$. Roughly speaking, this means $Q\vec y$ behaves like a vector where every entry has magnitude around $n$, and we can apply the Berry--Esseen theorem to $E_{\mathrm{shift}(1)}=\frac12\vec y\cdot \vec \Delta=\frac12 (Q\vec y)\cdot \vec x=\sum_{v\in V}(\frac12Q\vec y)_v x_v$ (the Berry--Esseen theorem is a quantitative central limit theorem for sums of independent but not necessarily identically distributed random variables; see for example \cite[Chapter~V,~Theorem~3]{Pet75}). Indeed, let $Z\sim \mc N(0,(\frac12\|Q\vec y\|_2)^2)$; the Berry--Esseen theorem shows that for any interval $[a,b]\subseteq\mb{R}$, we have
\begin{equation}\label{eq:Eshift-CLT}
|\Pr[E_{\mathrm{shift}(1)}\in[a,b]]-\Pr[Z\in[a,b]]|\lesssim_{C,H} 1/\sqrt{n}.
\end{equation}
In particular, for every interval $[a,b]\subseteq \mb{R}$ of length $b-a\geq \|M^{*}\|_{\mr F}$, we have
\begin{equation}
\Pr[E_{\mathrm{shift}(1)}\in [a,b]]\lesssim_{C,H} \frac{b-a}{n^{3/2}}\label{eq:Eshift-CLT-upper-bound}
\end{equation}
(recalling that $\|M^{*}\|_{\mr F}\gtrsim_C n$ by \cref{eq:M-F}).

Recall from \hyperlink{step:condition-Delta}{Step 4} that for near-balanced $\vec{\Delta}$ we have $\sigma_{\vec{\Delta}}^{2}=2\|M^{*}\|_{\mr F}^{2}+\|\vec{w}^{*}_{\vec{\Delta}}\|_{2}^{2}+O(n^{7/4+14\gamma})= 2\|M^{*}\|_{\mr F}^{2}+\|\frac{1}{2}(I-Q)\vec{y}+\frac14(I-Q)M\vec{\Delta}\|_{2}^{2}+O(n^{7/4+14\gamma})$ (using the definition of $\vec{w}^{*}_{\vec{\Delta}}$ in \hyperlink{step:rewriting-Delta}{Step 3}). Let us now define  $\sigma\geq 0$ to satisfy $\sigma^2=2\|M^{*}\|_{\mr F}^{2}+\|\frac{1}{2}(I-Q)\vec{y}\|_{2}^{2}$. Note that $\sigma$ does not depend on $\vec{\Delta}$ (in a moment we will define $\sigma_{\mr{shift}}$ to bound the deviation of $\sigma_{\vec\Delta}$ from $\sigma$). Also note that we have $\sigma\ge\|M^{*}\|_{\mr F}\gtrsim_C n$ (recalling \cref{eq:M-F}) and $\sigma^2\leq 2n^2+4n^{2+8\gamma}\le n^{2.1}$, meaning that $\sigma\le n^{1.05}$.

Finally, let us define $\sigma_{\mathrm{shift}}=\|\frac{1}{4}(I-Q)M\vec{\Delta}\|_{2}$. Using the inequality $\snorm{\vec{v}+\vec{w}}_2^2\le 2\snorm{\vec{v}}_2^2+2\snorm{\vec{w}}_2^2$ for any vectors $\vec{v},\vec{w}\in\mb{R}^V$, as well as \cref{eq:M-F} (recalling that $\gamma=10^{-4}$), for any near-balanced $\vec{\Delta}$ we have
\[\sigma_{\vec{\Delta}}^{2}\leq 4\|M^{*}\|_{\mr F}^{2}+2\Big\|\frac{1}{2}(I-Q)\vec{y}\Big\|_{2}^{2}+2\Big\|\frac{1}{4}(I-Q)M\vec{\Delta}\Big\|_{2}^2=2\sigma^2+2\sigma_{\mathrm{shift}}^2.\]
Similarly (using $\snorm{\vec{v}-\vec{w}}_2^2\ge \frac12\snorm{\vec{v}}_2^2-\snorm{\vec{w}}_2^2$),
\[\sigma_{\vec{\Delta}}^{2}\geq \|M^{*}\|_{\mr F}^{2}+\frac{1}{2}\Big\|\frac{1}{2}(I-Q)\vec{y}\Big\|_{2}^{2}-\Big\|\frac{1}{4}(I-Q)M\vec{\Delta}\Big\|_{2}^2=\frac{1}{2}\sigma^2-\sigma_{\mathrm{shift}}^2.\]
Therefore, for every near-balanced $\vec{\Delta}$, we must have $\sigma_{\vec{\Delta}}\le 2\sigma_{\mathrm{shift}}$ or $\sigma/2\le \sigma_{\vec{\Delta}}\le 2\sigma$ (indeed, if $\sigma_{\mathrm{shift}}^2\leq \sigma_{\vec{\Delta}}^2/4$, then $\sigma_{\vec{\Delta}}^{2}/2\leq 2\sigma^2$ and $(5/4)\sigma_{\vec{\Delta}}^{2}\geq\sigma^2/2$).

\medskip
\noindent\textit{Step 6: Controlling correlations of the shifts.}
In order to average the estimates in \cref{clm:conditional-point-probability} over the different outcomes of $\vec{\Delta}$, we need to ensure that the ``shifts'' $\sigma_{\mathrm{shift}},E_{\mathrm{shift}(1)},E_{\mathrm{shift}(2)}$
(each of which are determined by $\vec{\Delta}$) do not correlate adversarially with each other. More specifically, we need that the
quantities $\sigma_{\mathrm{shift}},E_{\mathrm{shift}(2)}$ do not
correlate very strongly with $E_{\mathrm{shift}(1)}$, as shown in the following claim.

\begin{claim}\label{clm:crucial-claim}
Let $[a,b]\subseteq \mb{R}$ be an interval of length $b-a\ge \|M^{*}\|_{\mr F}$. Then
\[\mb E\left[(E_{\mathrm{shift}(2)}^{2}+\sigma_{\mathrm{shift}}^{2})\mbm 1_{E_{\mathrm{shift}(1)}\in [a,b]}\right]\lesssim_{C,H} n^{1/2}(b-a).\]
\end{claim}

In order to prove \cref{clm:crucial-claim}, we will use a similar Fourier-analytic argument as in the proof of \cref{lem:esseen-upper-crude} to estimate expressions of the form $\mb E[x_{v_1}\cdots x_{v_\ell}\mbm 1_{E_{\mathrm{shift}(1)}\in [a,b]}]$, and deduce the desired bounds by linearity of expectation. We defer the details of this proof to the end of the section (specifically, we will prove it in \cref{sub:proofs-of-claims}).

After all this setup, we are now ready to prove the desired bounds in the statement of \cref{thm:short-interval}. Let $B=B(C)>0$ be as in \cref{clm:conditional-point-probability}. Consider $x\in\mb{Z}$, and write $x'=x-E$.
Let $\mc{E}$ be the event that $|X-x|\le B$. We wish to prove
the upper bound $\Pr[\mc{E}]\lesssim_{C,H} n^{-3/2}$, and
if $|x'|\le(A+1)n^{3/2}$ for some fixed $A>0$ we wish to prove the lower bound $\Pr[\mc{E}]\gtrsim_{C, H, A}n^{-3/2}$
(recall that $|E-\mb E X|\le n^{3/2}$ from \hyperlink{step:condition-garbage}{Step 2}, so we have $|x'|=|x-E|\le(A+1)n^{3/2}$ whenever $|x-\mb{E}X|\le An^{3/2}$).

\medskip
\noindent\textit{Step 7: Proof of the upper bound.}
First, recall from \hyperlink{step:condition-Delta}{Step 4} that $\vec\Delta$ is near-balanced with probability $1-n^{-\omega(1)}$. Also, for $\mc{E}$ to have an appreciable chance of occurring, $E_{\mr{shift}(1)}$ must be quite close to $x'$. Indeed, note that if $\mc{E}$ occurs, $\vec{\Delta}$ is near-balanced, and $|E_{\mathrm{shift}(1)}-x'|\ge\sigma(\log n)^2$, then we have
\[|X-E-E_{\mathrm{shift}(1)}|\geq |E_{\mathrm{shift}(1)}+E-x|-B=|E_{\mathrm{shift}(1)}-x'|-B\geq \sigma(\log n)^2/2\]
(recalling that $\sigma\geq \snorm{M^*}_{\mr F}\gtrsim_C n$ from \hyperlink{step:conditional-E-sigma}{Step 5}). On the other hand by \cref{eq:R-gone-formula} we have (recalling that $E_{\mathrm{shift}(1)}=\frac{1}{2}\vec{y}\cdot \vec{\Delta}=\frac{1}{2}(Q\vec{y})\cdot \vec{x}$)
\[X-E-E_{\mathrm{shift}(1)}=\frac12 \vec{y}\cdot \vec{x}+\frac{1}{8}\vec{x}^{\intercal}M\vec{x} - \frac{1}{2}(Q\vec{y})\cdot \vec{x}=\frac12 ((I-Q)\vec{y})\cdot \vec{x}+\frac{1}{8}\vec{x}^{\intercal}M\vec{x}.\]
Hence (as $M$ is a symmetric matrix with zeroes on the diagonal), we have $\mb{E}[X-E-E_{\mathrm{shift}(1)}]=0$ and $\sigma(X-E-E_{\mathrm{shift}(1)})^2=\frac{1}{32}\|M\|_{\mr F}^{2}+\|\frac{1}{2}(I-Q)\vec{y}\|_{2}^{2}\le n^2+\sigma^2\lesssim_C \sigma^2$ by \cref{eq:boolean-variance} and the definition of $\sigma$ in \hyperlink{step:conditional-E-sigma}{Step 5}.
Thus, accounting for the probability that $\vec \Delta$ is not near-balanced, we have
\begin{equation}\label{eq:prob-E-shift-far-away}
\Pr[\mathcal E\cap \{|E_{\mr{shift}(1)}-x'|\ge \sigma(\log n)^2\}]\le \Pr[|X-E-E_{\mathrm{shift}(1)}|\geq \sigma(\log n)^2/2]+n^{-\omega(1)}\leq n^{-\omega_C(1)}\le n^{-3/2}
\end{equation}
by \cref{thm:concentration-hypercontractivity}  (concentration via hypercontractivity).

So, it suffices to restrict our attention to $\vec\Delta$ which are near-balanced and satisfy $|E_{\mathrm{shift}(1)}-x'|\le \sigma(\log n)^2$. The plan is to apply \cref{clm:conditional-point-probability}(1) to upper-bound $\Pr[\mc E|\vec\Delta]$ for all such $\vec \Delta$, and then to average over $\vec \Delta$. When we apply \cref{clm:conditional-point-probability}(1) we need estimates on $\sigma_{\vec \Delta}$ and $|x-E_{\vec \Delta}|$; we obtain these estimates in different ways depending on properties of $E_{\mr{shift}(1)},E_{\mr{shift}(2)},\sigma_\mr{shift}$.

First, the exponential decay in the bound in \cref{clm:conditional-point-probability}(1) is in terms of $|x-E_{\vec \Delta}|$. From \cref{eq:cond-EX} one can deduce that $|x-E_{\vec \Delta}|$ is at least roughly as large as $|x'-E_{\mr{shift}(1)}|$, unless $E_{\mathrm{shift}(2)}$ is atypically large (at the end of this step we will upper-bound the contribution from such atypical $\vec\Delta$).
Let $\mc{H}$ be the event that $\vec\Delta$ is near-balanced and satisfies $|E_{\mathrm{shift}(1)}-x'|\le \sigma(\log n)^2$ and $|x-E_{\vec{\Delta}}|\geq |E_{\mathrm{shift}(1)}-x'|/2-2n$; we start by upper-bounding $\Pr[\mc E\cap \mc H]$.

For any outcome  of $\vec{\Delta}$ such that $\mc{H}$ holds, by \cref{clm:conditional-point-probability}(1) we have 
\begin{align}
\Pr\left[|X-x|\le B\middle|\vec{\Delta}\right]&\lesssim_C \frac{\exp\left(-\Omega_{C}\left(|x-E_{\vec{\Delta}}|/\sigma_{\vec{\Delta}}\right)\right)+n^{-0.1}}{\sigma_{\vec{\Delta}}}\notag\\
&\lesssim_{C}\frac{\exp\left(-\Omega_{C}\left(|E_{\mathrm{shift}(1)}-x'|/\sigma_{\vec{\Delta}}\right)\right)}{\sigma_{\vec{\Delta}}}+n^{-1.1}\label{eq:upper-bound-first-claim}
\end{align}
(recalling from \hyperlink{step:condition-Delta}{Step 4} that $\sigma_{\vec{\Delta}}\ge \|M^{*}\|_{\mr F}\gtrsim_C n$). Also note that by \cref{eq:Eshift-CLT-upper-bound}, we have
\[\Pr[\mathcal{H}]\le \Pr[|E_{\mathrm{shift}(1)}-x'|\le \sigma(\log n)^2]\lesssim_{C,H} \frac{\sigma(\log n)^2}{n^{3/2}}\leq n^{-0.45}(\log n)^2\]
(recalling that $\sigma\ge \|M^{*}\|_{\mr F}$ and $\sigma\le n^{1.05}$ from \hyperlink{step:conditional-E-sigma}{Step 5}).

Recall from the end of \hyperlink{step:conditional-E-sigma}{Step 5} that we always have $\sigma_{\vec{\Delta}}\le 2\sigma_{\mathrm{shift}}$ or $\sigma/2\le \sigma_{\vec{\Delta}}\le 2\sigma$. First, we bound
\begin{align*}
&\Pr[\mc{E}\cap \mc{H}\cap \{\sigma/2\le \sigma_{\vec{\Delta}}\le 2\sigma\}] \\
&\quad = \sum_{j=0}^{\infty} \Pr\bigg[\mc{E}\cap\mc{H}\cap\{\sigma/2\le \sigma_{\vec{\Delta}}\le 2\sigma\}\cap\bigg\{j\le \frac{|E_{\mathrm{shift}(1)}-x'|}{\sigma}<j+1\bigg\}\bigg] \\
&\quad \lesssim_C \sum_{j=0}^{\infty} \Pr\bigg[\mc{H}\cap (\sigma/2\le \sigma_{\vec{\Delta}}\le 2\sigma)\cap  \bigg\{j\le \frac{|E_{\mathrm{shift}(1)}-x'|}{\sigma}<j+1\bigg\}\bigg] \cdot \left(\frac{\exp\left(-\Omega_{C}(j)\right)}{\sigma}+n^{-1.1}\right)\\
&\quad \le  \Pr[\mathcal{H}]\cdot n^{-1.1}+\sum_{j=0}^{\infty} \Pr\bigg[j\le \frac{|E_{\mathrm{shift}(1)}-x'|}{\sigma}<j+1\bigg] \cdot \frac{\exp\left(-\Omega_{C}(j)\right)}{\sigma}\\
&\quad \lesssim_{C,H} n^{-0.45}(\log n)^2\cdot n^{-1.1}+\sum_{j=0}^{\infty} \frac{\sigma}{n^{3/2}}\cdot \frac{\exp\left(-\Omega_{C}(j)\right)}{\sigma}\lesssim_C n^{-3/2},
\end{align*}
where in the first inequality we used \cref{eq:upper-bound-first-claim} and in the final inequality we used \cref{eq:Eshift-CLT-upper-bound} (recalling that $\sigma\geq \|M^{*}\|_{\mr F}$).

Next, let us bound $\Pr[\mc{E}\cap \mc{H}\cap \{\sigma_{\vec{\Delta}}\le 2\sigma_{\mathrm{shift}}\}]$. Note that \cref{clm:crucial-claim} implies
\begin{equation}\label{eq:upper-bound-second-claim}
\mb E\left[\sigma_{\vec{\Delta}}^{2}\mbm 1_{E_{\mathrm{shift}(1)}\in [a,b]}\mbm 1_{\sigma_{\vec{\Delta}}\le  2\sigma_{\mathrm{shift}}}\right]\le 4 \cdot \mb E\left[\sigma_{\mathrm{shift}}^{2}\mbm 1_{E_{\mathrm{shift}(1)}\in [a,b]}\right]\lesssim_{C,H} n^{1/2}(b-a)
\end{equation}
for any interval $[a,b]\subseteq \mb{R}$ of length $b-a\ge \|M^{*}\|_{\mr F}$. Hence, recalling from \hyperlink{step:condition-Delta}{Step 4} that $\sigma_{\vec{\Delta}}\ge \|M^{*}\|_{\mr F}\gtrsim_C n$ for every near-balanced $\vec{\Delta}$, we obtain
\begin{align*}
&\Pr[\mc{E}\cap \mc{H}\cap \{\sigma_{\vec{\Delta}}\le 2\sigma_{\mathrm{shift}}\}] \\
&\quad =  \sum_{i,j=0}^{\infty} \Pr\bigg[\mc{E}\cap \mc{H}\cap \{\sigma_{\vec{\Delta}}\le 2\sigma_{\mathrm{shift}}\}\cap \bigg\{2^i\le \frac{\sigma_{\vec{\Delta}}}{\|M^{*}\|_{\mr F}}< 2^{i+1}\bigg\}\cap \bigg\{j\le \frac{|E_{\mathrm{shift}(1)}-x'|}{2^i\|M^{*}\|_{\mr F}}<j+1\bigg\}\bigg] \\
&\quad \lesssim_C \sum_{i,j=0}^{\infty} \Pr\bigg[\mc{H}\cap \{\sigma_{\vec{\Delta}}\le 2\sigma_{\mathrm{shift}}\}\cap\bigg\{2^i\le \frac{\sigma_{\vec{\Delta}}}{\|M^{*}\|_{\mr F}}< 2^{i+1}\bigg\}\cap \bigg\{j\le \frac{|E_{\mathrm{shift}(1)}-x'|}{2^i\|M^{*}\|_{\mr F}}<j+1\bigg\}\bigg] \\
&\qquad\qquad\qquad\qquad\qquad\qquad\qquad\qquad\qquad\qquad\qquad\qquad\qquad\qquad\qquad\qquad\cdot \left(\frac{\exp\left(-\Omega_{C}(j)\right)}{2^i\|M^{*}\|_{\mr F}}+n^{-1.1}\right)\\
&\quad \le \frac{\Pr[\mathcal{H}]}{n^{1.1}}+\sum_{i,j=0}^{\infty} \Pr\bigg[\{\sigma_{\vec{\Delta}}\le 2\sigma_{\mathrm{shift}}\}\cap \bigg\{2^i\le \frac{\sigma_{\vec{\Delta}}}{\|M^{*}\|_{\mr F}}\bigg\}\cap \bigg\{j\le \frac{|E_{\mathrm{shift}(1)}-x'|}{2^i\|M^{*}\|_{\mr F}}<j+1\bigg\}\bigg]\cdot \frac{\exp\left(-\Omega_{C}(j)\right)}{2^i\|M^{*}\|_{\mr F}}\\
&\quad \lesssim_{C,H} n^{-0.45}(\log n)^2\cdot n^{-1.1}+\sum_{i,j=0}^{\infty}  \frac{n^{1/2}2^i\|M^{*}\|_{\mr F}}{(2^i\|M^{*}\|_{\mr F})^2}\cdot \frac{\exp\left(-\Omega_{C}(j)\right)}{2^i\|M^{*}\|_{\mr F}}\\
&\quad =n^{-1.55}(\log n)^2+\sum_{i,j=0}^{\infty} \frac{n^{1/2}}{2^{2i}\|M^{*}\|_{\mr F}^2}\cdot \exp\left(-\Omega_{C}(j)\right)\lesssim_C n^{-3/2}+\frac{n^{1/2}}{\|M^{*}\|_{\mr F}^2}\lesssim_C n^{-3/2},
\end{align*}
(The first inequality is by \cref{eq:upper-bound-first-claim} and in the third inequality we used \cref{eq:upper-bound-second-claim} with Markov's inequality.)

We have now proved that $\Pr[\mc{E}\cap \mc{H}]\lesssim_{C,H} n^{-3/2}$. Recalling the definition of $\mc H$ and \cref{eq:prob-E-shift-far-away}, it now suffices to upper-bound the probability that $\mc E$ holds, $\vec \Delta$ is near-balanced, and $|x-E_{\vec{\Delta}}|\leq |E_{\mathrm{shift}(1)}-x'|/2-2n$.

If $\vec{\Delta}$ is near-balanced and $|x-E_{\vec{\Delta}}|\leq |E_{\mathrm{shift}(1)}-x'|/2-2n$, then  $|E_{\mathrm{shift}(1)}-x'|\geq 4n$ and, using $x'=x-E$ and \cref{eq:cond-EX}, furthermore $|E_{\mathrm{shift}(2)}|\geq |E_{\mathrm{shift}(1)}+E-x|-|E_{\vec{\Delta}}-x|-2n\geq |E_{\mathrm{shift}(1)}-x'|/2$. Hence (using \cref{clm:crucial-claim} noting that $\snorm{M^*}_{\mr F}\le n$, and Markov's inequality)
\begin{align*}
    &\Pr[|x-E_{\vec{\Delta}}|\leq |E_{\mathrm{shift}(1)}-x'|/2-2n \text{ and }\vec{\Delta}\text{ is near-balanced}]\\
    &\quad\leq \sum_{i=2}^{\infty}\Pr[(2^i n\le |E_{\mathrm{shift}(1)}-x'|<2^{i+1} n)\cap (|E_{\mathrm{shift}(2)}|\geq 2^{i-1}n)]\lesssim_{C,H} \sum_{i=2}^{\infty}\frac{n^{1/2}\cdot 2^i n}{2^{2(i-1)}n^2}\lesssim n^{-1/2}.
\end{align*}
For every near-balanced outcome of $\vec{\Delta}$, by \cref{clm:conditional-point-probability}(1) we have $\Pr[\mathcal{E}|\vec{\Delta}]\lesssim_C 1/\sigma_{\vec{\Delta}}\lesssim_C 1/n$ (recalling from \hyperlink{step:condition-Delta}{Step 4} that $\sigma_{\vec{\Delta}}\ge \|M^{*}\|_{\mr F}\gtrsim_C n$). Hence the probability that $\mathcal{E}$ holds, $\vec{\Delta}$ is near-balanced, and $|x-E_{\vec{\Delta}}|\leq |E_{\mathrm{shift}(1)}-x'|/2-2n$ is bounded by $O_{C,H}(n^{-3/2})$, completing the proof of the upper bound.

\medskip
\noindent\textit{Step 8: Proof of the lower bound.}
Fix $A>0$, and assume that $|x-E|=|x'|\le (A+1)n^{3/2}$. We need to show that $\Pr[\mathcal{E}]\gtrsim_{C,H,A} n^{-3/2}$. To do so, we define an event $\mc{F}$ such that we can conveniently apply \cref{clm:conditional-point-probability}(2) after conditioning on this event (roughly speaking, we need $E_{\mr{shift}(1)}$ to take ``about the right value'', and we need $E_{\mr{shift}(2)}$ and $\sigma_{\mr{shift}}$ ``not to be too large''). We study the probability of $\mc{F}$ by applying \cref{eq:Eshift-CLT} (Gaussian approximation for $E_{\mr{shift}(1)}$) as well as \cref{clm:crucial-claim} together with Markov's inequality (as in the upper bound proof in the previous step).

Let $s\in\{-1,1\}$ be as in \cref{clm:conditional-point-probability}(2). For any $0<K<n^{3/2}/(2\sigma)$, we can consider the event that $K\sigma \le s(x'-E_{\mathrm{shift}(1)})\le 2K\sigma$, which can be interpreted as the event that $E_{\mathrm{shift}(1)}$ lies in a certain interval of length $K\sigma$ whose endpoints both have absolute value at most $|x'|+2K\sigma\le (A+2)n^{3/2}$.
Using \cref{eq:Eshift-CLT}, we can compare the probability for this event to the probability that a normal random variable with distribution $\mc N(0,(\frac12\|Q\vec y\|_2)^2)$ lies in this interval. In this way, we see that the probability of the event $K\sigma \le s(x'-E_{\mathrm{shift}(1)})\le 2K\sigma$ is at least
\begin{equation}\label{eq:lower-bound-for-definition-K}
K\sigma\cdot \frac{\exp(-(A+2)^2n^3/(\frac12\snorm{Q\vec{y}}_2^2))}{\sqrt{2\pi}\cdot \frac{1}{2}\snorm{Q\vec{y}}_2}-O_{C,H}(1/\sqrt{n})\ge K\sigma\cdot \frac{\exp(-O_{C,A}(1))}{O_H(n^{3/2})}-O_{C,H}(1/\sqrt{n}),
\end{equation}
where we used that $\snorm{Q\vec{y}}_2\gtrsim_C n^{3/2}$ and $\snorm{Q\vec{y}}_\infty\le (H+2)n$ (which implies that $\snorm{Q\vec{y}}_2\lesssim_H n^{3/2}$), as discussed in \hyperlink{step:conditional-E-sigma}{Step 5}.

Now, recalling that $n^{1.05}\ge\sigma\ge\snorm{M^*}_{\mr F}\gtrsim_C n$ from \hyperlink{step:conditional-E-sigma}{Step 5}, we can take $K=K(C,H,A)\ge 10^4$ to be a sufficiently large constant such that the right-hand-side of \cref{eq:lower-bound-for-definition-K} is at least $\sigma/n^{3/2}$, such that $\snorm{M^*}_{\mr F}\ge K^{-1/4}\cdot n$, and such that the hidden constant in the $\lesssim_{C,H}$ notation in the statement of \cref{clm:crucial-claim} is at most $K^{1/4}$. By the choice of $K$, we have
\[\Pr[K\sigma \le s(x'-E_{\mathrm{shift}(1)})\le 2K\sigma]\ge \frac{\sigma}{n^{3/2}}.\]
Furthermore, using \cref{clm:crucial-claim} and Markov's inequality we have
\[\Pr[(E_{\mathrm{shift}(2)}^2+\sigma_{\mathrm{shift}}^2\ge 2K^{5/4}n^2)\cap (K\sigma \le s(x'-E_{\mathrm{shift}(1)})\le 2K\sigma)]\le K^{1/4}\cdot \frac{n^{1/2}\cdot \sigma K}{2K^{5/4}n^2}=\frac{\sigma}{2n^{3/2}}.\]
Thus, with probability at least $\sigma/(2n^{3/2})$, we have $E_{\mathrm{shift}(2)}^2+\sigma_{\mathrm{shift}}^2\le 2K^{5/4}n^2$ and $K\sigma \le s(x'-E_{\mathrm{shift}(1)})\le 2K\sigma$. Let $\mathcal{F}$ be the event that these two conditions are satisfied and $\vec{\Delta}$ is near-balanced (and note that $\mathcal{F}$ only depends on the randomness of $\vec{\Delta}$). Recalling from \hyperlink{step:condition-Delta}{Step 4} that $\vec{\Delta}$ is near-balanced with probability $1-n^{-\omega(1)}$, we see that $\Pr[\mc F]\geq \sigma/(4n^{3/2})$.

We claim that whenever $\mathcal{F}$ holds, we have $\sigma/K^2\le \sigma_{\vec{\Delta}}\le K^2\sigma$ and $3n\le s(x-E_{\vec{\Delta}})\le 3K^3\sigma_{\vec{\Delta}}$. For the first claim, note that if $\mathcal{F}$ holds, then $\sigma_{\mathrm{shift}}^2\le 2K^{5/4}n^2\le K^2n^2/4$ and hence $\sigma_{\vec{\Delta}}^2\geq \sigma^2/2-\sigma_{\mathrm{shift}}^2\geq \sigma^2/2-K^2n^2/4$. So, if $\sigma\geq Kn$, we obtain the desired lower bound $\sigma_{\vec{\Delta}}\geq \sigma/2\geq \sigma/K^2$. If $\sigma\leq Kn$, then we instead obtain the desired lower bound on $\sigma_{\vec{\Delta}}$ by observing that $\sigma\leq Kn\le K^2 \snorm{M^*}_{\mr F}\le K^2 \sigma_{\vec{\Delta}}$ (using that $\vec{\Delta}$ is near-balanced). For the upper bound on $\sigma_{\vec{\Delta}}$, recall from the end of \hyperlink{step:conditional-E-sigma}{Step 5} that we have $\sigma_{\vec{\Delta}}\le 2\sigma \le K^2\sigma$ or $\sigma_{\vec{\Delta}}\le 2\sigma_{\mathrm{shift}}$. In the latter case, we obtain $\sigma_{\vec{\Delta}}\le 2\sigma_{\mathrm{shift}}\le Kn\le K^{2}\snorm{M^*}_{\mr F}\le K^2\sigma$. Altogether, we have proved that $\sigma/K^2\le \sigma_{\vec{\Delta}}\le K^2\sigma$ whenever $\mathcal{F}$ holds, as claimed.

For the second of our two claims, note that whenever $\mathcal{F}$ holds, we have $E_{\mathrm{shift}(2)}^2\le 2K^{5/4}n^2\le 2K^{7/4}\snorm{M^*}_{\mr F}^2\le K^2\sigma^2/4$, so $|E_{\mathrm{shift}(2)}|\le K\sigma/2$ and hence $K\sigma/2 \le s(x'-E_{\mathrm{shift}(1)}-E_{\mathrm{shift}(2)})\le 2.5K\sigma$. Recalling \cref{eq:cond-EX} and $x'=x-E$, this implies the desired claim
\[3n\le K\sigma/2 -2n \le s(x-E_{\vec{\Delta}})\le 2.5K\sigma+2n\le 3K\sigma \le 3K^3\sigma_{\vec{\Delta}},\]
where in the first and fourth inequalities we used that $n\leq K^{1/4}\snorm{M^*}_{\mr F}\le K^{1/4}\sigma$, and in the last inequality we used the first claim.

Now, having established the above claims for all outcomes of $\vec{\Delta}$ satisfying $\mathcal{F}$, \cref{clm:conditional-point-probability}(2) implies that $\Pr[\mc E| \mc F]\gtrsim_{C,H,A} 1/(K^2\sigma)$. Thus, $\Pr[\mc E]\ge\Pr[\mc F]\cdot \Pr[\mc E|\mc F]\gtrsim_{C,H,A} \sigma/(4n^{3/2})\cdot 1/(K^2\sigma)\gtrsim_{C,H,A}n^{-3/2}$, completing the proof of the lower bound.
\end{proof}

\subsection{Proofs of claims\label{sub:proofs-of-claims}}
In order to finish the proof of \cref{thm:short-interval} in the $\gamma$-structured case, it remains to prove \cref{clm:conditional-point-probability,clm:crucial-claim}.

\begin{proof}[Proof of \cref{clm:conditional-point-probability}]
Recall that in the statement of \cref{clm:conditional-point-probability} we fixed a near-balanced outcome of $\vec{\Delta}$ and the desired conclusions are conditional on this outcome of $\vec{\Delta}$. Throughout this proof, let us therefore always condition on the fixed outcome of $\vec{\Delta}$, which we now view as being non-random, and for notational simplicity we omit all ``$|\vec \Delta$'' notation.

Recall that we have $\sigma_{\vec{\Delta}}^{2}=2\|M^{*}\|_{\mr F}^{2}+\|\vec{w}^{*}_{\vec{\Delta}}\|_{2}^{2}+O(n^{7/4+14\gamma})$ and $\|\vec{w}^{*}_{\vec{\Delta}}\|_{\infty}\le n^{1/2+5\gamma}$ (since $\vec{\Delta}$ is near-balanced). Also recalling that all entries of $M^*$ have absolute value at most $1$, this implies $\sigma_{\vec \Delta}^2\le n^2+n\cdot n^{1+10\gamma}+O(n^{7/4+14\gamma})\le  n^{2.2}$ (as $\gamma=10^{-4}$). Thus, $\sigma_{\vec \Delta}\le n^{1.1}$.

For the upper bound in (1) we will use \cref{lem:esseen-upper} and for the lower bound in (2) we will use \cref{lem:esseen-lower}. 
Recalling \cref{eq:normalized-X-formula}, let $Z$ be the ``Gaussian
analog'' of $X$: let $\vec{z}\sim\mathcal{N}(0,1)^{\otimes n}$
be a standard $n$-variate Gaussian random vector and let \[Z=\left(E+\frac{1}{2}\vec{y}\cdot\vec{\Delta}+\frac{1}{8}\vec{\Delta}^{\intercal}M\vec{\Delta}\right)+\vec{w}^{*}_{\vec{\Delta}}\cdot\vec{z}+\vec{z}^{\intercal}M^{*}\vec{z}.\]

Let $\nu=\nu(2C,0.001)>0$ be as in \cref{thm:quadratic-decoupling} and let $\varepsilon=2/\nu$. Let $s \in \{-1, 1\}$ be the sign of the eigenvalue of $M^*$ with the largest magnitude. We collect several estimates.
\begin{enumerate}
\item[(A)] $\sigma(Z)\asymp_{C}\sigma_{\vec \Delta}\gtrsim_{C} n$ and $|\mb E Z-E_{\vec\Delta}|\le 2n$.
\item[(B)] For all $x\in\mb{R}$, 
\[\Pr[|Z-x|\le \eps]\lesssim_{C}\frac{\eps}{\sigma(Z)}\exp\left(-\Omega_{C}\left(\frac{|x-\mb E Z|}{\sigma(Z)}\right)\right)\le \frac\varepsilon{\sigma(Z)}.\]
\item[(C)] $\int_{-2/\eps}^{2/\eps}|\varphi_{X}(\tau)-\varphi_{Z}(\tau)|\,d\tau\le n^{-1.2}$.
\item[(D)] For any fixed $A'\in\mb{R}_{\ge0}$, assuming that $n$ is sufficiently large with respect to $A'$, we have $p_{Z}(y_{1})/p_{Z}(y_{2})\le 2$ for all $y_1,y_2\in\mb{R}$ with $0\le s (y_{1}-\mb E Z)\le A'\sigma(Z)$ and $|y_{1}-y_{2}|\le 2n^{1/4}\varepsilon$.
\item[(E)] For any fixed $A'>0$ and any
$x\in\mb Z$ satisfying $0\le s(x-\mb E Z)\le A'\sigma(Z)$,
\[\Pr[|Z-x|\le \eps]\gtrsim_{C,A'}\frac{1}{\sigma(Z)}\qquad\text{and}\qquad p_Z(x)\gtrsim_{C,A'}\frac1{\sigma(Z)}.\]
\end{enumerate}
We will prove (A--E) using the results from \cref{sec:coupling-lemmas,sec:high-fourier,sec:ramsey-robust-rank,sec:anti-Gauss}; before explaining how to do this, we deduce the desired upper and lower bounds in (1) and (2). Let $B=B(C)=10^4\cdot 2\varepsilon$. First, using that by (A) we have $\eps\le\sigma(Z)$ for sufficiently large $n$, and using (B), we can apply \cref{lem:esseen-upper} to $X-\mb EZ$ and $Z-\mb E Z$ and $\sigma(Z)$. Hence for all $x\in\mb Z$ we have
\begin{align*}
\Pr[|X-x|\leq B]&\le 2\cdot 10^4\sup_{\substack{y\in\mb{R}\\|x-y|\le B}}\Pr[|X-y|\leq \eps]\\
&\lesssim_{C} \frac{\eps^2}{\sigma(Z)^2}+ \frac{\eps}{\sigma(Z)}\exp\left(-\Omega_{C}\left(\frac{|x-\mb E Z|}{\sigma(Z)}\right)\right)+\eps \int_{-2/\varepsilon}^{2/\varepsilon}|\varphi_X(\tau)-\varphi_Z(\tau)|\,d\tau.
\end{align*}
The bound in (1) then follows from (A) and (C). Second, by (A) and (E), if $x\in\mb{Z}$ satisfies $3n\le s(x-E_{\vec \Delta})\le A\sigma_{\vec \Delta}$ then $\Pr[|Z-x|\le \varepsilon]\gtrsim_{C,A} 1/\sigma_{\vec \Delta}$. Furthermore, for all $y_1,y_2\in [x-n^{1/4}\eps,x+n^{1/4}\eps]$ by (A) we have $0\le 3n-|\mb E Z-E_{\vec \Delta}|-n^{1/4}\eps\le  s (y_{1}-\mb E Z)\le A'\sigma(Z)$ for some $A'=A'(C,A)$, and therefore $p_{Z}(y_{1})/p_{Z}(y_{2})\le 2$ by (D). Let $K=2$ and $R=n^{1/4}$, so by \cref{lem:esseen-lower} we have (recalling that $B=10^4\cdot 2\eps=10^4K\eps$)
\begin{align*}
\Pr[|X-x|\leq B]&\ge \Omega_{C,A}(1/\sigma_{\vec \Delta})-C_{\ref{lem:esseen-lower}}\left(R^{-1}\mc L(Z,\varepsilon)+\eps\int_{-2/\varepsilon}^{2/\varepsilon}|\varphi_X(\tau)-\varphi_Z(\tau)|\,d\tau\right).
\end{align*}
The bound in (2) then follows from (A--C).

Now we prove (A--E). First, note that for any matrix $\widetilde M\in\mb{R}^{V\times V}$ with rank at most say 400, we
have $\|M^{*}-\widetilde{M}\|_{\mathrm{F}}^{2}=\frac{1}{64}\|M-(MQ+QM-QMQ+64\widetilde{M})\|_{\mathrm{F}}^{2}\gtrsim_{C} n^{2}\ge\|M^{*}\|_{\mathrm{F}}^{2}$
by \cref{lem:ramsey-low-rank}. Also note that $M^*$ and $\vec{w}^{*}_{\vec{\Delta}}$ satisfy conditions (a)--(d) in \cref{lem:transfer-summary} for $\delta=2\gamma=2\cdot 10^{-4}$, as discussed at the end of \hyperlink{step:rewriting-Delta}{Step 3} and the start of \hyperlink{step:condition-Delta}{Step 4} above.

Then, the two parts of (A) follow from parts (1) and (2) of \cref{lem:transfer-summary} (applied with $\delta=2\gamma=2\cdot 10^{-4}$), recalling $\sigma_{\vec{\Delta}}\gtrsim_C n$ from the end of \hyperlink{step:condition-Delta}{Step 4}. Furthermore, (B) and (E) follow from \cref{thm:gaussian-anticoncentration-technical}(1--2) (for the second part of (E), we use \cref{thm:gaussian-anticoncentration-technical}(2) with $\eps\to 0$).

Now, consider $y_1,y_2$ as in (D), so in particular $|y_{1}-y_{2}|\le 2n^{1/4}\eps$. By the inversion formula \cref{eq:inversion} and \cref{lem:gaussian-high-t} (with $r=8$), and (A), we have
 \begin{align*}
|p_{Z}(y_{1})-p_{Z}(y_{2})| & =\left|\frac{1}{2\pi}\int_{-\infty}^{\infty}(e^{-i\tau y_{1}}-e^{-i\tau y_{2}})\mb E e^{i\tau Z}\,d\tau\right|\lesssim\int_{-\infty}^{\infty}\min\{|\tau(y_{1}-y_{2})|,1\}\cdot |\mb E e^{i\tau Z}|\,d\tau\\
 & \lesssim_{C}\int_{-\infty}^{\infty}\min\{n^{1/4}|\tau|,1\}\cdot(1+\tau^{2}n^{2})^{-2}\,d\tau\lesssim n^{-7/4}=o(1/\sigma(Z)),
\end{align*}
from which we may deduce (D) using the second part of (E). It remains to prove (C), i.e., to bound the integral $\int_{-2/\eps}^{2/\eps}|\varphi_{X}(\tau)-\varphi_{Z}(\tau)|\,d\tau$ by $n^{-1.2}$. If $|\tau|\le n^{-0.99}$, then by \cref{lem:transfer-summary}(3) (with $\delta=2\gamma$) we have $|\varphi_{X}(\tau)-\varphi_{Z}(\tau)|\lesssim |\tau|^4\cdot  n^{3+24\gamma}+|\tau|\cdot n^{3/4+8\gamma}\lesssim |\tau|\cdot n^{3/4+8\gamma}$. Thus, the contribution of the range $|\tau|\le n^{-0.99}$ to the integral $
\int_{-2/\eps}^{2/\eps}|\varphi_{X}(\tau)-\varphi_{Z}(\tau)|\,d\tau$ is $O((n^{-0.99})^2 \cdot n^{3/4+8\gamma})=O(n^{-1.23+8\gamma})$, which is smaller than $n^{-1.2}/2$ (recalling that $\gamma=10^{-4}$).

For $n^{-0.99}\le|\tau|\le 2/\eps$ we bound $|\varphi_X(\tau)|$ and $|\varphi_Z(\tau)|$ separately. By \cref{lem:gaussian-high-t} (with $r=400$) we have $|\varphi_Z(\tau)|\lesssim_C (1+\tau^2n^2)^{-100}\le (n^{0.02})^{-100}=n^{-2}$. To bound $|\varphi_X(\tau)|$ we use \cref{thm:quadratic-decoupling}, after conditioning on any outcome of $U\cap (I_2\cup\cdots\cup I_m)$. After this conditioning, the remaining randomness is just within the first bucket $I_1$, and conditionally $X$ is of the form required to apply \cref{thm:quadratic-decoupling} with respect to the $(2C)$-Ramsey graph $G[I_1]$ of size $|I_1|\geq n^{1-2\gamma}$, and we obtain $|\varphi_X(\tau)|\lesssim n^{-(1-2\gamma)5}\le n^{-4}$ since $|\tau|\ge n^{-0.99}\ge|I_1|^{-0.999}$. Thus, in the range $n^{-0.99}\le|\tau|\le 2/\eps$ we have $|\varphi_{X}(\tau)-\varphi_{Z}(\tau)|\le |\varphi_X(\tau)|+|\varphi_Z(\tau)|\lesssim n^{-2}$, and so the contribution of this range to the integral $\int_{-2/\eps}^{2/\eps}|\varphi_{X}(\tau)-\varphi_{Z}(\tau)|\,d\tau$ is also smaller than $n^{-1.2}/2$.
\end{proof}

We will deduce \cref{clm:crucial-claim} from the following auxiliary estimate, applied with $k=1$ and with $k=2$ (recall that the functions $\psi$ and $f$ already appeared in the proof of \cref{lem:esseen-upper-crude}).

\begin{claim}\label{clm:prep-crucial-claim}
Fix $k\in\mb{N}$. Let us define the function $\psi\colon\mb{R}\to\mb{R}$ as the convolution $\psi=\mbm 1_{[-1,1]}*\mbm 1_{[-1,1]}$ (where $\mbm 1_{[-1,1]}$ is the indicator function of the interval $[-1,1]$) and let $f=\hat{\psi}$ be the Fourier transform of $\psi$. Consider a matrix $A\in\mb{R}^{V\times V}$ whose entries
have absolute value at most $1$, and a vector $\vec{\beta}\in\mb{R}^{V}$ with $\|\vec{\beta}\|_{\infty}\le \pi/4$. Then for any $t\in\mb{R}$ we have $|\mb E[(\vec{x}^{\intercal}A\vec{x})^{k}f(\vec{\beta}\cdot\vec{x}-t)]|\lesssim_{k}(\sqrt{n}/\snorm{\vec{\beta}}_2)^{2k+1}\cdot n^{k-1/2}$.
\end{claim}
\begin{proof}
Observing that $x_{v}^{2}=1$, we can express $(\vec{x}^{\intercal}A\vec{x})^{k}$ as a multilinear polynomial of degree at most $2k$ in the $|V|\le n$ variables $x_v$ for $v\in V$. For each $\ell\leq 2k$ this polynomial has at most $O(n^\ell)$ terms of degree $\ell$, and for each such term the corresponding coefficient has absolute value at most $O_k(n^{(2k-\ell)/2})$.

It suffices to prove that $|\mb E[x_{v_{1}}\cdots x_{v_{\ell}}f(\vec{\beta}\cdot\vec{x}-t)]|\lesssim_{\ell}\snorm{\beta}_2^{-(\ell+1)}$ for any $\ell\le 2k$ and any distinct $v_{1},\ldots,v_{\ell}\in V$. Indeed, this does imply $|\mb E[(\vec{x}^{\intercal}A\vec{x})^{k}f(\vec{\beta}\cdot\vec{x}-t)]|\lesssim_k \sum_{\ell=0}^{2k} n^\ell\cdot n^{(2k-\ell)/2}\cdot \snorm{\beta}_2^{-(\ell+1)}\lesssim_k (\sqrt{n}/\snorm{\vec{\beta}}_2)^{2k+1}\cdot n^{k-1/2}$ using that $\snorm{\vec{\beta}}_2\le \sqrt{n}$ since $|V|\le n$ and $\|\vec{\beta}\|_{\infty}\le \pi/4\le 1$.

Note that the support of the function $\psi$ is inside the interval $[-2,2]$ and we furthermore have $0\le\psi(\theta)\le2$ for all $\theta\in\mb{R}$. Therefore we can write
\[|\mb E[x_{v_{1}}\cdots x_{v_{\ell}}f(\vec{\beta}\cdot\vec{x}-t)]|  =\left|\mb E\left[\int_{-\infty}^{\infty}x_{v_{1}}\cdots x_{v_{\ell}}\psi(\theta)e^{-i\theta(\vec{\beta}\cdot\vec{x}-t)}\,d\theta\right]\right|\le2\int_{-2}^{2}|\mb E[x_{v_{1}}\cdots x_{v_{\ell}}e^{-i\theta(\vec{\beta}\cdot\vec{x})}]|\,d\theta.\]
By \cref{eq:cos}, for $-\pi/2\le\lambda\le\pi/2$ and
$v\in V$ we have $|\mb E[e^{i\lambda x_{v}}]|=|\cos\lambda|\le\exp(-\lambda^{2}/\pi^2)$,
and 
\[|\mb E[x_{v}e^{i\lambda x_{v}}]|=\left|\frac{1}{2}\exp(i\lambda)-\frac{1}{2}\exp(-i\lambda)\right|=|\sin \lambda|\le|\lambda|.\]
Since $|\theta\beta_{v}|\le \pi/2$ for all $v\in V$ and $-2\le \theta \le 2$, we can deduce (also using that $|\beta_{v}|\le 1$ for all $v\in V$)
\begin{align*}
|\mb E[x_{v_{1}}\cdots x_{v_{\ell}}f(\vec{\beta}\cdot\vec{x}-t)]|  &\le2\int_{-2}^{2}\prod_{j=1}^{\ell}|\theta \beta_{v_j}| \prod_{v\in V\setminus\{v_{1},\ldots,v_{\ell}\}}e^{-(\theta^2/\pi^2)\beta_v^2}\le 
2 \int_{-2}^{2}|\theta|^{\ell}e^{-(\theta^{2}/\pi^2)(\snorm{\vec{\beta}}_2^2-\ell)}d\theta\\
&\lesssim_\ell\int_{-2}^{2} |\theta|^{\ell}e^{-\theta^2\snorm{\vec{\beta}}_2^2/\pi^2}d\theta=\frac{\pi^{\ell+1}}{\snorm{\vec{\beta}}_2^{\ell+1}} \int_{-2\snorm{\vec{\beta}}_2/\pi}^{2\snorm{\vec{\beta}}_2/\pi} |z|^{\ell}e^{-z^2}dz\lesssim_\ell\snorm{\vec{\beta}}_2^{-(\ell+1)},
\end{align*}
as desired (where in the last step we used that the integral $\int_{-\infty}^{\infty}|z|^{m}e^{-z^2}dz$ is finite).
\end{proof}

Finally, let us deduce \cref{clm:crucial-claim}.

\begin{proof}[Proof of \cref{clm:crucial-claim}]
First, note that it suffices to consider the case where the interval $[a,b]$ has length exactly $(2H+4)n$. Indeed, in the
general case we can cover $[a,b]$ with $\lceil (b-a)/((2H+4)n)\rceil\lesssim_{C.H} (b-a)/((2H+4)n)$ intervals of length exactly $(2H+4)n$ (here, we used that $b-a\ge \snorm{M^*}_{\rm F}\gtrsim_C n$ by \cref{eq:M-F}). So assume that $b-a=(2H+4)n$ and let $s=(a+b)/2$, then $[a,b]=[s-(H+2)n,s+(H+2)n]$.

Using that $Q$ and $M$ are symmetric, recall from \hyperlink{step:conditional-E-sigma}{Step 5} that
\[
E_{\mathrm{shift}(1)}=\frac{1}{2}\vec{y}\cdot\vec{\Delta}=\frac{1}{2}(Q\vec{y})\cdot\vec{x},\quad\quad E_{\mathrm{shift}(2)}=\frac{1}{8}\vec{\Delta}^{\intercal}M\vec{\Delta}=\frac{1}{8}\vec{x}^{\intercal}(QMQ)\vec{x},\]
\[\sigma_{\mathrm{shift}}^{2}=\frac{1}{16}\snorm{(I-Q)M\vec{\Delta}}_2^2=\frac{1}{16}\snorm{(I-Q)MQ\vec{x}}_2^2=\frac{1}{16}\vec{x}^{\intercal}QM(I-Q)^2MQ\vec{x}=\frac{n}{16}\vec{x}^{\intercal}\frac{QM(I-Q)^2MQ}{n}\vec{x}.
\]
Recall that $M$ has entries in $\{0,1\}$, and recall the definition of $Q$ in \hyperlink{step:rewriting-Delta}{Step 3} (and the fact that multiplying with $Q$ has the effect of averaging values over buckets). This shows that in $QMQ$ and also in $(I-Q)MQ$ (and consequently in $(1/n)QM(I-Q)^2MQ$) all entries have absolute value at most $1$.

Furthermore recall from \hyperlink{step:condition-Delta}{Step 4} that $\|Q\vec{y}\|_{\infty}\le (H+2)n$ and $\|Q\vec{y}\|_{2}\gtrsim_C n^{3/2}$. Consider $\psi$ and $f$ as in the statement of \cref{clm:prep-crucial-claim}, and recall from the proof of \cref{lem:esseen-upper-crude} that $f(t)\ge \mbm 1_{[-1,1]}(t)$ for all $t\in\mb{R}$ (more specifically, the function $f$ is given by $f(t)=(2(\sin t)/t)^2$ for $t\neq 0$ and $f(0)=2^2$). Also note that $E_{\mathrm{shift}(2)}^{2}$ and $\sigma_{\mathrm{shift}}^{2}$
are both nonnegative.

Now, let $\vec{\beta}\in\mb{R}^V$ be given by $((H+2)n)^{-1}\cdot \frac{1}{2}Q\vec{y}$, and note that then $\snorm{\vec{\beta}}_\infty\le 1/2<\pi/4$ and $\snorm{\vec{\beta}}_2\gtrsim_{C,H} n^{1/2}$. Furthermore, let $t=((H+2)n)^{-1}s$, so (recalling that $E_{\mathrm{shift}(1)}=\frac{1}{2}(Q\vec{y})\cdot\vec{x}$ and $[a,b]=[s-(H+2)n,s+(H+2)n]$) we have $E_{\mathrm{shift}(1)}\in [a,b]$ if and only if $\vec{\beta}\cdot \vec{x}-t\in [-1,1]$. Hence
\[\mb{E}[E_{\mathrm{shift}(2)}^2\mbm{1}_{E_{\mathrm{shift}(1)}\in [a,b]}]=\mb{E}[E_{\mathrm{shift}(2)}^2\mbm{1}_{\vec{\beta}\cdot \vec{x}-t\in [-1,1]}]\le \mb{E}[E_{\mathrm{shift}(2)}^2f(\vec{\beta}\cdot \vec{x}-t)]=\frac{\mb{E}[(\vec{x}^{\intercal}(QMQ)\vec{x})^2f(\vec{\beta}\cdot \vec{x}-t)]}{64}\]
and therefore by \cref{clm:prep-crucial-claim} applied with $A=QMQ$ and $k=2$,
\[\mb{E}[E_{\mathrm{shift}(2)}^2\mbm{1}_{E_{\mathrm{shift}(1)}\in [a,b]}]\lesssim (\sqrt{n}/\snorm{\vec{\beta}}_2)^5\cdot n^{3/2}\lesssim_{C,H} n^{3/2}.\]
Similarly, writing $A=(1/n)QM(I-Q)^2MQ$ and applying \cref{clm:prep-crucial-claim} wih $k=1$, we have
\[\mb{E}[\sigma_{\mathrm{shift}}^2\mbm{1}_{E_{\mathrm{shift}(1)}\in [a,b]}]\le \mb{E}[\sigma_{\mathrm{shift}}^2f(\vec{\beta}\cdot \vec{x}-t)]=\frac{n}{16}\cdot \mb{E}[(\vec{x}^{\intercal}A\vec{x})f(\vec{\beta}\cdot \vec{x}-t)]\lesssim n\cdot (\sqrt{n}/\snorm{\vec{\beta}}_2)^3\cdot n^{1/2}\lesssim_{C,H} n^{3/2}.\]
Summing these two estimates and recalling that $b-a=(2H+4)n$ now gives the desired result
\[\mb{E}[(E_{\mathrm{shift}(2)}^2+\sigma_{\mathrm{shift}}^2)\mbm{1}_{E_{\mathrm{shift}(1)}\in [a,b]}]\lesssim_{C,H} n^{1/2}(b-a).\qedhere\]
\end{proof}

\section{Switchings for pointwise probability estimates}\label{sec:switch}
So far (in \cref{thm:short-interval}), we have obtained near-optimal estimates on probabilities of events of the form $|X-x|\le B$, for some large constant $B$. However, in order to prove  \cref{thm:point-control}, we need to control the probability that $X$ is \emph{exactly} equal to $x$ (assuming that $e_0$ and the entries of the vector $\vec{e}$ are integers). Of course, an upper bound on $\Pr[|X-x|\le B]$ as in \cref{thm:short-interval} implies an upper bound on $\Pr[X=x]$. So it only remains to prove the lower bound in \cref{thm:point-control}.

In order to deduce the lower bound in \cref{thm:point-control} from \cref{thm:short-interval}, it suffices to show that $\Pr[X=x]$ does not differ too much from $\Pr[X=x']$ for $x'\in [x-B,x+B]$. In order to show this, we use the \emph{switching} method, by which we study the effect of small perturbations to $U$. For example, in the setting of \cref{thm:point-control} one can show that for a typical outcome of $U$ there are many
pairs of vertices $(y,z)$ such that $y\in U$, $z\notin U$ and $|N(z)\cap (U\setminus \{y\})|-|N(y)\cap (U\setminus \{z\})|+e_z-e_y=\ell$. For such a pair $(y,z)$, modifying $U$ by removing $y$ and adding
$z$ (a ``switch'' of $y$ and $z$) changes $X$ by exactly
$\ell$.

As discussed in \cref{subsec:switching}, we introduce
an \emph{averaged} version of the switching method. Roughly speaking,
we define random variables that measure the number of ways to switch
between two classes, and study certain moments of these random variables.
We can then make our desired probabilistic conclusions with the Cauchy--Schwarz inequality.

First, we need a lemma providing us with a special set of vertices which we will use for switching operations (the properties in the lemma make it tractable to compute the relevant moments).

For vertices $v_{1},\ldots,v_{s}$ in a graph $G$, let us define \[\overline{N}(v_{1},\ldots,v_{s})=V(G)\setminus\big(\{v_1,\ldots,v_s\}\cup N(v_1)\cup\cdots\cup N(v_s)\big)\]
to be the set of vertices in $V(G)\setminus\{v_1,\ldots,v_s\}$ that are not adjacent to any of the vertices $v_1,\ldots,v_s$.

\begin{lemma}\label{lem:switching-setup}
For any fixed $C,H>0$ and $D\in\mb N$, there exist $\rho=\rho(C,D)$ with $0<\rho<1$ and $\delta=\delta(C,D)>0$ with $\delta<\rho^3/3^{D+1}$ such that the following holds for all sufficiently large $n$. For every $C$-Ramsey graph $G$ on $n$ vertices and every vector $\vec{e}\in\mb{Z}^{V(G)}$ with $0\le e_v\le Hn$ for all $v\in V(G)$, there exist subsets $S\su S_0\su V(G)$ with $|S|\geq n^{0.48}$ and $|S_0|\geq \delta^{1/\rho}\cdot n$ such that the following properties hold.
\begin{enumerate}
\item The induced subgraph $G[S_0]$ is $(\delta,\rho)$-rich (see \cref{def:rich}).
\item For any vertices $v_{1},\ldots,v_{s}\in S$ with $s\leq D$, we have $|\overline{N}(v_{1},\ldots,v_{s})\cap S_0|\ge \delta |S_0|$.
\item For any vertices $v,w\in S$, we have $|\deg_G(v)/2+e_v-\deg_G(w)/2-e_w|\leq \sqrt{n}$.
\end{enumerate}
\end{lemma}

\begin{remark}\label{rem:dependence-B}
We will apply \cref{lem:switching-setup} with $D=8B+4$, where $B=B(C)$ is as in \cref{thm:short-interval}. So the size of $S_0$ depends on $B$. Eventually, we will apply \cref{thm:short-interval} to a Ramsey graph $G[\overline N]$, for a certain subset $\overline N\subseteq S_0$ (with $U\cap \overline N$ as our random vertex set, conditioning on an outcome of $U\setminus \overline N$). Since the proportion of $G$ that $\overline N\subseteq S_0$ occupies depends on $D$, we will have to apply \cref{thm:short-interval} with $A,H$ depending on $D$ (and therefore on $B$). So, it is crucial that in \cref{thm:short-interval}, $B$ does not depend on $A,H$.
\end{remark}

To prove \cref{lem:switching-setup} (specifically, property (2)), we will need a \emph{dependent random choice} lemma: the following simple yet powerful lemma appears as \cite[Lemma 2.1]{FS11}.

\begin{lemma}\label{lem:dependent-random-choice}
Let $F$ be a graph on $n$ vertices with average degree $d$. Suppose that $a,s,r\in\mb N$ satisfy
\[
\sup_{t\in\mb N}\left(\frac{d^{t}}{n^{t-1}}-\binom{n}{r}\cdot \left(\frac{s}{n}\right)^{t}\right)\ge a.
\]
Then, $F$ has a subset $W$ of at least $a$ vertices such that every
$r$ vertices in $W$ have at least $s$ common neighbors in $F$.
\end{lemma}

\begin{proof}[Proof of \cref{lem:switching-setup}] Let $\varepsilon=\varepsilon(2C)$ be as in \cref{thm:dense-ramsey}, so for sufficiently large $m$ every
$2C$-Ramsey graph on $m$ vertices has average degree at least $\eps m$. Let $\rho=\rho(C,1/5)>0$ be as in \cref{lem:rich-subset}.
Let $\delta=\delta(C,D)>0$ be sufficiently small such that $\delta<\rho^3/3^{D+1}$ and for all sufficiently large $m$ (in terms of $C$ and $D$) we have
\[
\sup_{t\in\mb N}\left(\varepsilon^{t}m-\binom{m}{D}\delta^{t}\right)\ge m^{0.99}.
\]
To see that this is possible, consider $t=\eta\log m$ for
some small $\eta$ (in terms of $\eps$), and let $\delta$ be small in terms of $\eta$ and $D$.

By \cref{lem:rich-subset}, we can find a $(\delta,\rho)$-rich induced subgraph $G[S_{0}]$ of size $|S_0|\ge\delta^{1/\rho}\cdot n$. 

Since $|S_0|\geq \delta^{1/\rho}\cdot n\ge\sqrt{n}$, the graph $G[S_0]$ is $2C$-Ramsey. Let $\overline G[S_0]$ be the complement of this graph, so that $\overline G[S_0]$ is also a $2C$-Ramsey
graph and therefore has average degree at least $\varepsilon|S_{0}|$. By \cref{lem:dependent-random-choice}
and the choice of $\delta$, the graph $\overline G[S_0]$ contains a set $S'$ of $|S'|\geq |S_{0}|^{0.99}\ge 2(H+1)n^{0.98}$
vertices such that every $D$ vertices in $S'$ have at least
$\delta|S_{0}|$ common neighbors in $\overline G[S_0]$. 
This means that for any $s\le D$ and any $v_{1},\ldots,v_{s}\in S'$, we have $|\overline{N}(v_{1},\ldots,v_{s})\cap S_0|\ge \delta |S_{0}|$,
so (2) holds for any subset $S\su S'$.

Finally, note that $\deg_G(v)/2+e_v\in [0,(H+1)n]$ for all $v\in S'$, and consider a partition of the interval $[0,(H+1)n]$ into $\lfloor 2(H+1) \sqrt{n}\rfloor$ sub-intervals of length  $ (H+1)n/\lfloor 2(H+1) \sqrt{n}\rfloor\leq \sqrt{n}$. By the pigeonhole principle, there exists a set $S\su S'$ of at least $ 2(H+1)n^{0.98}/\lfloor 2(H+1) \sqrt{n}\rfloor\ge n^{0.48}$ vertices $v$ whose associated values $\deg_G(v)/2+e_v$ lie in the same sub-interval. Then (3) holds.
\end{proof}

As foreshadowed earlier, the next lemma estimates
moments of certain random variables that measure the number of ways
to switch between certain choices of the set $U$. The proof of this lemma relies on \cref{thm:short-interval}.

\begin{lemma}\label{lem:switching-moment-estimate} Fix $C,H,A>0$, let $B=B(2C)$ be as in \cref{thm:short-interval} and define $D=D(C)=8B+4$. Consider a $C$-Ramsey graph $G$ on $n$ vertices and a vector vector $\vec{e}\in\mb{Z}^{V(G)}$ with $0\le e_v\le Hn$ for all $v\in V(G)$. Let $S\su S_0\su V(G)$, $\rho=\rho(C,D)>0$ and $\delta=\delta(C,D)>0$ be as in \cref{lem:switching-setup}, and define
\[T=\big\{(y,z)\in S^{2} \,:\, |(N(z)\setminus N(y))\cap S_0|\geq \rho^2 |S_0| \text{ and } |(N(y)\setminus N(z))\cap S_0|\geq \rho^2 |S_0| \big\}.\]
Consider a random vertex subset $U\subseteq V(G)$ obtained
by including each vertex with probability $1/2$ independently, and let $X=e(G[U])+\sum_{u\in U}e_u$. For $\ell=-B,\ldots,B$, let $Y_{\ell}$ be the number of vertex pairs $(y,z)\in T$ with $y\in U$ and $z\notin U$ such that $(|N(z)\cap (U\setminus \{y\})|+e_z)-(|N(y)\cap (U\setminus \{z\})|+e_y)=\ell$.
For $x\in\mb Z$, let $Z_{x-B,x+B}\in\{0,1\}$ be the indicator random variable for the event that $x-B\le X\le x+B$.

Then, for any $x\in\mb Z$ satisfying $|x-\mb EX|\le An^{3/2}$, and
any $a_{-B},\ldots,a_{B}\in\{0,1,2\}$, we have 
\[
\mb E[Y_{-B}^{a_{-B}}\cdots Y_{B}^{a_{B}}Z_{x-B,x+B}]\asymp_{C,H,A}\frac{(|T|/\sqrt{n})^{a_{-B}+\cdots+a_B}}{n^{3/2}}.
\]
\end{lemma}

We defer the proof of \cref{lem:switching-moment-estimate} (using \cref{thm:short-interval}) until the end
of the section, first showing how it can be used to prove \cref{thm:point-control}. This argument requires the set $T$ in \cref{lem:switching-moment-estimate} to be non-empty, which is implied by the following lemma.

\begin{lemma}\label{lem:switching-T-large}
The set $T$ defined in \cref{lem:switching-moment-estimate} has size $|T|\geq |S|^2/2\geq n^{0.96}/2$.
\end{lemma}
\begin{proof}
Recall that the set $S\su S_0$ has size $|S|\geq n^{0.48}$ and that $G[S_0]$ is $(\delta,\rho)$-rich, where $\delta<\rho^3/3^{D+1}< \rho$ is as in \cref{lem:switching-setup}. We first claim that at least $(3/4)\cdot |S|^2$ pairs $(y,z)\in S^{2}$ satisfy the first condition $|(N(z)\setminus N(y))\cap S_0|\geq \rho^2 |S_0|$ in the definition of $T$. Indeed, by \cref{def:rich}, all but at most $n^{1/5}$ vertices $z\in S_0$ satisfy $|N(z)\cap S_0|\geq\rho|S_0|$. Hence, $|N(z)\cap S_0|\geq\rho|S_0|$ for at least $|S|-n^{1/5}$ vertices $z\in S$. Furthermore, for each such $z\in S$ we have $|(N(z)\setminus N(y))\cap S_0|=|(N(z)\cap S_0)\sm N(y)|\geq \rho\cdot |N(z)\cap S_0|\geq\rho^2|S_0|$ for all but at most $n^{1/5}$ vertices $y\in S_0$ and in particular for at least $|S|-n^{1/5}$ vertices $y\in S$. Thus, there are at least $(|S|-n^{1/5})^2\geq (3/4)\cdot |S|^2$ pairs $(y,z)\in S^2$ satisfying $|(N(z)\setminus N(y))\cap S_0|\geq \rho^2 |S_0|$.
Analogously, at least $(3/4)\cdot |S|^2$ pairs $(y,z)\in S^{2}$  satisfy the second condition $|(N(y)\setminus N(z))\cap S_0|\geq \rho^2 |S_0|$ in the definition of $T$. This means that the number of pairs $(y,z)\in S^2$ satisfying both conditions is at least $|S|^2-2(|S|^2-(3/4)\cdot |S|^2)=|S|^2/2$ and hence $|T|\geq |S|^2/2\geq n^{0.96}/2$.
\end{proof}

Now we are ready to deduce \cref{thm:point-control} from \cref{lem:switching-moment-estimate}.

\begin{proof}[Proof of \cref{thm:point-control}]
Consider a $C$-Ramsey graph $G$, a random subset $U\su V(G)$ and $X=e(G[U])+\sum_{v\in U}e_v+e_0$ as in \cref{thm:point-control}, and
consider the setup of \cref{lem:switching-moment-estimate}. Note that the upper bound in \cref{thm:point-control} follows immediately from the upper bound in \cref{thm:short-interval}, so it only remains to prove the lower bound.

For $x\in\mb Z$ let $Z_{x}$ be the indicator random variable for the event that
$X=x$. Note that for all $x\in\mb Z$ and $\ell=-B,\ldots,B$ we have $\mb E[Y_{-\ell}Z_{x+\ell}]=\mb E[Y_{\ell}Z_{x}]$.
Indeed, if $X=e(G[U])+\sum_{u\in U}e_u+e_0=x+\ell$, then $Y_{-\ell}$ is the number of ways
to perform a ``switch'' of two vertices $y\in U$, $z\notin U$ with $(y,z)\in T$, to obtain a vertex subset $U'=(U\setminus \{y\})\cup \{z\}$ with $e(G[U'])+\sum_{v\in U'}e_v+e_0=x$. Conversely, if $X=e(G[U])+\sum_{v\in U}e_v+e_0=x$, then $Y_{\ell}$ is the number of ways to perform such
a switch ``in reverse'' to obtain a vertex subset $U'$ with $e(G[U'])+\sum_{v\in U'}e_v+e_0=x+\ell$. So, $2^{n}\mb E[Y_{-\ell}Z_{x+\ell}]$ and $2^{n}\mb E[Y_{\ell}Z_{x}]$ both describe the total number of ways to switch in this way between an outcome of $U$ with $X=x+\ell$ and an outcome with $X=x$.

Now, for every $x\in\mb Z$ with $|x-\mb E X|\le An^{3/2}$ there is some $\ell\in\{-B,\ldots,B\}$ such that
\begin{align*}
\mb E[Y_{-B}\cdots Y_{B}Z_{x+\ell}]&\ge\frac{1}{2B+1}\sum_{\ell'=-B}^B\mb E[Y_{-B}\cdots Y_{B}Z_{x+\ell'}]\\
&=\frac{1}{2B+1} \mb E[Y_{-B}\cdots Y_{B}Z_{x-B,x+B}]\gtrsim_{C,H,A}\frac{(|T|/\sqrt{n})^{2B+1}}{n^{3/2}},\end{align*}
where the last step is by \cref{lem:switching-moment-estimate}. For this $\ell$, the Cauchy--Schwarz inequality, together with \cref{lem:switching-moment-estimate} and the fact that $Z_{x+\ell}\le Z_{x-B,x+B}$, implies that
\[\mb E[Y_{\ell}Z_{x}]=\mb E[Y_{-\ell}Z_{x+\ell}]\ge\frac{(\mb E[Y_{-B}\cdots Y_{B}Z_{x+\ell}])^{2}}{\mb E[Y_{-B}^{2}\cdots Y_{-\ell-1}^{2}Y_{-\ell}Y_{-\ell+1}^{2}\cdots Y_{B}^{2}Z_{x+\ell}]}\gtrsim_{C,H,A}\frac{(|T|/\sqrt{n})^{4B+2}/n^{3}}{(|T|/\sqrt{n})^{4B+1}/n^{3/2}}=\frac{|T|/\sqrt{n}}{n^{3/2}}.
\]

Finally, we use the Cauchy--Schwarz inequality and \cref{lem:switching-moment-estimate} once more (noting that $Z_{x}\le Z_{x-B,x+B}$) to conclude that
\[
\Pr[X=x]=\mb E Z_{x}\ge\frac{(\mb E[Y_{\ell}Z_{x}])^{2}}{\mb E[Y_{\ell}^{2}Z_{x}]}\gtrsim_{C,H,A} \frac{(|T|/\sqrt{n})^{2}/n^{3}}{(|T|/\sqrt{n})^{2}/n^{3/2}}=\frac{1}{n^{3/2}}.\tag*{\qedhere}
\]
\end{proof}

It now remains to prove the moment estimates in \cref{lem:switching-moment-estimate}. We will write the desired moments as a combinatorial sum of probabilities; for various tuples of pairs of vertices $(y,z)$, we then need to control the joint probability that $X=e(G[U])+\sum_{u\in U}e_u$ lies in a certain interval and that $U$ contains a specified number of vertices from the neighborhoods of the various $y$ and $z$. The next lemma gives a lower bound for certain probabilities of this form. Slightly more precisely, it allows us to specify the intersection sizes of $U$ in with given disjoint vertex subsets $W_{1},\ldots,W_{s}$. When applying this lemma in the proof of \cref{lem:switching-moment-estimate}, we will take $s=a_{-B}+\cdots+a_B$, and given $s$ pairs of vertices $(y_1,z_1),\ldots,(y_s,z_s)\in T$, we will take $W_1,\ldots,W_s$ to be certain regions of the Venn diagram given by the neighborhoods of $y_1,z_1,\ldots,y_s,z_s$. We can then use the intersection sizes of $U$ with $W_1,\ldots,W_s$ to control the events that the $s$-tuple of pairs $(y_1,z_1),\ldots,(y_s,z_s)$ contributes to $Y_{-B}^{a_{-B}}\cdots Y_{B}^{a_{B}}Z_{x-B,x+B}$. For this argument, we will, however, need to condition on the outcome of $U$ outside these special regions of the Venn diagram. This conditioning affects the linear terms and constant terms in our random variable $X$, so we use the variables $f_v$ and $f_0$ in the lemma statement below (when applying the lemma, we take $f_v$ and $f_0$ to be the terms obtained from $e_v$ and $e_0$ after accounting for this conditioning).


\begin{lemma}\label{lem:switching-robust-conditioning}
Let $\delta'>0$ and $R\geq 1$, and consider an $n$-vertex graph $G$, a real number $f_0$, and a sequence $\vec{f}\in\mb{R}^{V(G)}$ with $|f_{v}|\le R n$ for each $v\in V(G)$. Let $U\subseteq V(G)$ be a vertex subset obtained by including each vertex with probability $1/2$ independently, and let $X=e(G[U])+\sum_{v\in U}f_{v}+f_0$. Then the following hold.
\begin{enumerate}
    \item $\on{Var}[X]\leq R^2n^3$.
    \item For any $s\leq R$ and any disjoint subsets $W_{1},\ldots,W_{s}\su V(G)$, each of size at least $\delta' n$, and any $w_{1},\ldots,w_{s}\in\mb Z$ satisfying $\big|w_{i}-|W_{i}|/2\big|\le R\sqrt{n}$ for $i=1,\ldots,s$, we have
\[\Pr\left[|X-\mb E X|\leq 6R^2n^{3/2}\text{ and }|U\cap W_i|=w_i\text{ for }i=1,\ldots,s\right]\gtrsim_{\delta', R} n^{-s/2}.\]
\end{enumerate}
\end{lemma}

\begin{proof}
For (1), the expression for $X$ in \cref{eq:fourier-walsh} and the formula in \cref{eq:boolean-variance} show that
\[\on{Var}[X]=\frac{1}{4}\sum_{v\in V(G)}\left(f_v+\frac{1}{2}\deg(v)\right)^2+\frac{1}{16}e(G)\leq R^2n^3.\]

Let $E=\mb{E}X$ and note that for each $i=1,\ldots,s$ we have \[\Pr[|U\cap W_i|=w_i]=\binom{|W_i|}{w_i}^{-1}\asymp_{\delta',R} n^{-1/2}.\]
and these events are independent for all $i$. Thus, in order to establish (2), it suffices to show that when conditioning on $|U\cap W_i|=w_i$ for $i=1,\ldots,s$, we have $|X-E|\leq 6R^2n^{3/2}$ with probability at least $1/2$.

Also note that the value of $X$ changes by at most $(R+1)n$ when adding or deleting a vertex of $U$. We can sample a uniformly random subset $U\su V(G)$ conditioned on $|U\cap W_i|=w_i$ for $i=1,\ldots,s$ by the following procedure. First, sample a uniformly random subset $U'\su V(G)$, and then construct $U$ from $U'$ by deleting $|U'\cap W_i|-w_i$ uniformly randomly chosen vertices from $U'\cap W_i$ (if $|U'\cap W_i|\ge w_i$) or adding $w_i-|U'\cap W_i|$ randomly chosen vertices from $W_i\setminus U'$ to $U'$ (if $|U'\cap W_i|<w_i$) for each $i=1,\ldots,s$. With probability at least $1/2$ the value $X'=e(G[U'])+\sum_{v\in U'}f_{v}+f_0$ satisfies $|X'-E|\leq 2Rn^{3/2}$ and we have $||U'\cap W_i|-|W_i|/2|\leq s\sqrt{n}$ for $i=1,\ldots,n$ (by Chebyshev's inequality using $\on{Var}[X']\leq R^2n^3$ and $\on{Var}[|U'\cap W_i|]\leq n/4$). Whenever this is the case, we have $\big||U'\cap W_i|-w_i\big|\leq 2R\sqrt{n}$ for $i=1,\ldots,s$, implying $|X-X'|\leq 4R^2n^{3/2}$ and thus $|X-E|\leq 4R^2n^{3/2}+2Rn^{3/2}\leq 6R^2n^{3/2}$, as desired.
\end{proof}

The proof of \cref{lem:switching-moment-estimate} involves the consideration of tuples $((y_1,z_1),\ldots,(y_s,z_s))\in T^s$ and studies the probability that each $(y_i,z_i)$ contributes to some specified $Y_{\ell_i}$. So, we will need to establish various properties of the tuples $((y_1,z_1),\ldots,(y_s,z_s))\in T^s$. In particular, the properties in the following definition will be used in our proof of the upper bound in \cref{lem:switching-moment-estimate}. In this definition, and for the rest of this section, we write $\vec 1_A$ for the characteristic vector of a set $A$ (with $(\vec 1_A)_i=1$ if $i\in A$, and $(\vec 1_A)_i=0$ otherwise)\footnote{
In this section, we will not use the notation $\vec{x}_A$ for the restriction of a vector $\vec{x}$ to a set of indices $A$.}.

\begin{definition}\label{def:degeneracy}
Fix $C>0$ and let $\rho=\rho(C)>0$ and $\delta=\delta(C)>0$ be as in \cref{lem:switching-moment-estimate}. For a $C$-Ramsey graph $G$ on $n$ vertices and vertex pairs $(y_1,z_1),\ldots,(y_s,z_s)\in V(G)^2$, let us define $M(y_1,z_1,\ldots,y_s,z_s)$ to be the $s\times n$ matrix (with rows indexed by $1,\ldots,s$ and columns indexed by $V(G)$) with entries in $\{-1,0,1\}$ such that for $i=1,\ldots,s$ the $i$-th row of $M(y_1,z_1,\ldots,y_s,z_s)$ is the difference of characteristic vectors $\vec{1}_{N(z_i)\setminus\{y_i\}}-\vec{1}_{N(y_i)\setminus\{z_i\}}\in\mb{R}^{V(G)}$. We say that $((y_1,z_1),\ldots,(y_s,z_s))$ is \emph{$k$-degenerate} for some $k\in\{0,\ldots,s\}$ if it is possible to delete at most $\delta^{3/\rho}\cdot n$ columns from the matrix $M(y_1,z_1,\ldots,y_s,z_s)$ and obtain a matrix of rank at most $s-k$. We furthermore define the \emph{degeneracy} of $((y_1,z_1),\ldots,(y_s,z_s))$ to be the maximum $k$ such that $((y_1,z_1),\ldots,(y_s,z_s))$ is $k$-degenerate.
\end{definition}

Note that $(y_1,z_1,\ldots,y_s,z_s)$ is always $0$-degenerate (so the definition of degeneracy is well-defined).

The significance of the matrix $M(y_1,z_1,\ldots,y_s,z_s)$ is as follows. For any subset $U\su V(G)$ the entries of the product $M(y_1,z_1,\ldots,y_s,z_s)\vec{1}_U$ (which is a vector with $s$ entries) are precisely $|N(z_i)\cap (U\setminus \{y_i\})|-|N(y_i)\cap (U\setminus \{z_i\})|$ for $i=1,\ldots,s$ (these quantities occur in the definition of $Y_\ell$ in \cref{lem:switching-moment-estimate}). We can obtain a bound on the joint anticoncentration of these quantities from the following version of a theorem of Hal\'asz \cite{Hal77} (which can be viewed as a multi-dimensional version of the Erd\H{o}s--Littlewood--Offord theorem~\cite{Erd45}). This version follows via a fairly short deduction from the standard version of Hal\'asz' theorem \cite[Theorem 1]{Hal77} (for the case $r=s$, see also \cite[Exercise 7.2.3]{TV10}), but it is slightly more convenient to instead make our deduction from a version of Hal\'asz' theorem due to Ferber, Jain and Zhao~\cite{FJZ22}.

\begin{theorem}\label{thm:halasz-discrete}
Fix integers $s\geq r\geq 0$ and $\lambda>0$ and consider a matrix $M\in\mb{R}^{s\times n}$. Suppose that whenever we delete at most $\lambda n$ columns of $M$, the resulting matrix still has rank at least $r$. Then for a uniformly random vector $\vec{\xi}\in \{0,1\}^{n}$ we have $\Pr[M\vec{\xi}=\vec{\lambda}]\lesssim_{s,\lambda} n^{-r/2}$ for any vector $\vec{\lambda}\in\mb{R}^s$.
\end{theorem}

\begin{proof}
The assumption on $M$ implies that the set of columns of $M$ contains  $\lceil\lambda n/r\rceil$ disjoint linearly independent subsets of size $r$ (indeed, consider a maximal collection of such subsets, and note that upon deleting the corresponding columns from $M$ the resulting matrix has rank less than $r$). Hence the columns of $M$ can be partitioned into $\lceil\lambda n/r\rceil$ subsets, such that the span of each of these subsets has dimension at least $r$. By \cite[Theorem 1.10]{FJZ22} this implies that $\Pr[M\vec{\xi}=\vec{\lambda}]\lesssim_{s} (\lceil\lambda n/r\rceil)^{-r/2} \lesssim_{s,\lambda} n^{-r/2}$.
\end{proof}

Applying this theorem to the matrix-vector product $M(y_1,z_1,\ldots,y_s,z_s)\vec{1}_U$
yields bounds that get weaker as the degeneracy of $((y_1,z_1),\ldots,(y_s,z_s))$ increases. We therefore need to show that there are only few $s$-tuples $((y_1,z_1),\ldots,(y_s,z_s))\in T^s$ with high degeneracy (see part (b) of \cref{lem:s-tuples-in-T-to-consider} below), and we will use the following technical lemma to do this.

\begin{lemma}\label{lem:switching-technical-lemma-degenerate-tuples}
For a $C$-Ramsey graph $G$ on $n$ vertices (where $n$ is sufficiently large with respect to $C$), let $S\su S_0\su V(G)$, $T\su V(G)^2$, $D = D(C)$, $\rho=\rho(C)>0$ and $\delta=\delta(C)>0$ be defined as in \cref{lem:switching-moment-estimate}. Let $((y_1,z_1),\ldots,(y_s,z_s))\in T^s$ be a $k$-degenerate $s$-tuple for some $0\leq s\leq D/2$ and $k\in \{0,\ldots,s\}$. Then there exist indices $1\leq i_1<\cdots<i_{s-k}\leq s$ such that the following holds. For every vector $\vec{t}\in \{-1,0,1\}^{s-k}$, let $W_{\vec{t}}\su V(G)$ be the set of vertices such that the corresponding column of the $(s-k)\times n$ matrix $M(y_{i_1},z_{i_1},\ldots,y_{i_{s-k}},z_{i_{s-k}})$ (as in \cref{def:degeneracy}) equals $\vec{t}$. Then for each $j\in [s]\setminus \{i_1,\ldots,i_{s-k}\}$ one can find a vector $\vec{t}\in \{-1,0,1\}^{s-k}$ such that the set $W_{\vec{t}}$ fulfills the following three conditions:
\begin{itemize}
    \item[(i)] $|W_{\vec{t}}\cap S_0|\geq \delta\cdot |S_0|$.
    \item[(ii)] $|N(y_j)\cap W_{\vec{t}}\cap S_0|\leq  \rho\cdot |W_{\vec{t}}\cap S_0|$.
    \item[(iii)] $|N(z_j)\cap W_{\vec{t}}\cap S_0|\geq (1-\rho)\cdot |W_{\vec{t}}\cap S_0|$.
\end{itemize}
\end{lemma}

\begin{proof}
Since $((y_1,z_1),\ldots,(y_s,z_s))\in T^s$ is $k$-degenerate, there is a way to delete at most $\delta^{3/\rho}\cdot n$ columns from the $s\times n$ matrix  $M(y_1,z_1,\ldots,y_s,z_s)$ and obtain a matrix $M'$ of rank at most $s-k$. Let $Q\su V(G)$ be the set of vertices corresponding to the deleted columns. We have the bound $|Q|+2\leq \delta^{3/\rho}\cdot n+2\leq \delta^{2/\rho}\cdot |S_0|+2\leq \delta \cdot|S_0|\leq (\rho^2/2)\cdot |S_0|$ (recall from \cref{lem:switching-setup} that $|S_0|\ge\delta^{1/\rho}\cdot n$ and $\delta<\rho^3/3^{D+1}$).

Since $M'$ has rank at most $s-k$, we can choose indices $1\leq i_1<\cdots<i_{s-k}\leq s$ such that every row of $M'$ can be written as a linear combination of the rows with indices $i_1,\ldots,i_{s-k}$. We will show that this choice of indices satisfies the desired statement.

The rows of $M'$ with indices $i_1,\ldots,i_{s-k}$ form precisely the matrix $M(y_{i_1},z_{i_1},\ldots,y_{i_{s-k}},z_{i_{s-k}})$ with the columns corresponding to vertices in $Q$ deleted. Note that for each vector $\vec{t}\in \{-1,0,1\}^{s-k}$ and each $h=1,\ldots,s-k$, the entries in the $i_h$-th row of $M'$ in the columns with indices in $W_{\vec{t}}\setminus Q$ all have the same value, namely $t_h$. In  other words, writing
\[\vec M'_j=\vec{1}_{N(z_j)\setminus(\{y_j\}\cup Q)}-\vec{1}_{N(y_j)\setminus(\{z_j\}\cup Q)}\in \{-1,0,1\}^{V(G)\sm Q}\]
for the $j$-th row of $M'$ for $j=1,\ldots,s$, each of the row vectors $\vec M'_{i_1},\ldots,\vec M'_{i_{s-k}}$ are constant on each of the column sets $W_{\vec{t}}\setminus Q$, for $\vec{t}\in \{-1,0,1\}^{s-k}$. Since every row $\vec M'_j$ is a linear combination of these vectors, it follows that in fact each row $\vec M'_j$ is constant on each of the column sets $W_{\vec{t}}\setminus Q$.

Now, let us fix some $j\in [s]\setminus \{i_1,\ldots,i_{s-k}\}$. We need to show that we can find some $\vec{t}\in\{-1,0,1\}^{s-k}$ satisfying conditions (i)--(iii) in the lemma. Since $(y_j,z_j)\in T$, the definition of $T$ (see the statement of \cref{lem:switching-moment-estimate}) implies $|(N(z_j)\setminus N(y_j))\cap S_0|\geq \rho^2 \cdot|S_0|$, and so $|(N(z_j)\setminus N(y_j))\cap (S_0\setminus (Q\cup\{y_j,z_j\}))|\geq \rho^2 \cdot|S_0|-|Q|-2\geq (\rho^2/2)\cdot|S_0|$. This means that $\vec M'_j$ has at least $(\rho^2/2)|S_0|$ entries corresponding to vertices in $S_0\setminus (Q\cup\{y_j,z_j\})$ with value $1-0=1$. Hence, by the pigeonhole principle there must be some $\vec{t}\in \{-1,0,1\}^{s-k}$ for which there are at least $\rho^2\cdot |S_0|/(2\cdot 3^{s-k})\geq (\rho^2/3^{D+1})\cdot |S_0|$ vertices in $(W_{\vec{t}}\cap S_0)\setminus (Q\cup\{y_j,z_j\})$ such that the corresponding entry in $\vec M'_j$ is 1.

For this $\vec t$ we have $|W_{\vec{t}}\cap S_0|\geq (\rho^2/3^{D+1})\cdot |S_0|\geq(\delta/\rho)\cdot |S_0|$, so $\vec{t}$ satisfies (i) (recall from \cref{lem:switching-setup} that $0<\rho<1$). Furthermore recall that $\vec M'_j$ is constant on the index set $W_{\vec{t}}\setminus Q$, so this constant value must be 1. This means that for all vertices $v\in W_{\vec{t}}\setminus (Q\cup\{y_j,z_j\})$ we must have $v\in N(z_i)$ and $v\not\in N(y_i)$. Hence $|N(y_j)\cap W_{\vec{t}}\cap S_0|\leq |Q\cup\{y_j,z_j\}|\leq |Q|+2\leq \delta\cdot |S_0|\le\rho\cdot |W_{\vec{t}}\cap S_0|$, establishing (ii). Furthermore, we similarly have $|N(z_j)\cap W_{\vec{t}}\cap S_0|\geq |W_{\vec{t}}\cap S_0|-|Q\cup\{y_j,z_j\}|\geq (1-\rho)\cdot |W_{\vec{t}}\cap S_0|$ as required in (iii).
\end{proof}

Given a graph $G$ and vertex pairs $(y_1,z_1),\ldots,(y_s,z_s)\in V(G)^2$, for each $i=1,\ldots,s$ define
\[N_i(y_1,z_1,\ldots,y_s,z_s)=N(z_i)\cap\overline{N}(y_1,z_1,\ldots,y_{i-1},z_{i-1},y_i,y_{i+1},z_{i+1},\ldots,y_s,z_s)\]
to be the set of vertices in $V(G)\setminus \{y_1,z_1,\ldots,y_s,z_s\}$ that are adjacent to $z_i$ but not to any of the other vertices among $y_1,z_1,\ldots,y_s,z_s$. For the lower bound in \cref{lem:switching-moment-estimate}, we will consider tuples $((y_1,z_1),\ldots,(y_s,z_s))\in T^s$ such that $|N_i(y_1,z_1,\ldots,y_s,z_s)\cap S_0|\geq \rho\delta\cdot |S_0|$ for all $i=1,\ldots,s$.

\begin{lemma}\label{lem:s-tuples-in-T-to-consider}
For a $C$-Ramsey graph $G$ on $n$ vertices (where $n$ is sufficiently large with respect to $C$), let $S\su S_0\su V(G)$, $T\su V(G)^2$, $D = D(C)$, $\rho=\rho(C)>0$ and $\delta=\delta(C)>0$ be defined as in \cref{lem:switching-moment-estimate}. Then for each $s=0,1,\ldots,D/2$ the following statements hold.
\begin{itemize}
    \item[(a)] At least $|T|^s/2$ different $s$-tuples $((y_1,z_1),\ldots,(y_s,z_s))\in T^s$ with distinct $y_1,z_1,\ldots,y_s,z_s$ satisfy $|N_i(y_1,z_1,\ldots,y_s,z_s)\cap S_0|\geq \rho\delta\cdot |S_0|$ for all $i=1,\ldots,s$.
    \item[(b)] For each $k=0,\ldots,s$, the number of $k$-degenerate $s$-tuples $((y_1,z_1),\ldots,(y_s,z_s))\in T^s$ is at most $|T|^s/\sqrt{n}^k$.
\end{itemize}
\end{lemma}

\begin{proof}
For (a), we first claim that for each fixed $i=1,\ldots,s$ there are at most $|T|^s/(4D)$ different $s$-tuples $((y_1,z_1),\ldots,(y_s,z_s))\in T^s$ with  $|N_i(y_1,z_1,\ldots,y_s,z_s)\cap S_0|< \rho\delta\cdot |S_0|$. Indeed, without loss of generality assume $i=s$ and note that there are $|T|^{s-1}$ choices for the pairs $(y_1,z_1),\ldots,(y_{s-1},z_{s-1})$ and $|S|$ choices for $y_s$. Fixing these choices determines the set $\overline{N}(y_1,z_1,\ldots,y_{s-1},z_{s-1},y_s)$ and by property (2) of \cref{lem:switching-setup} this set satisfies
\[|\overline{N}(y_1,z_1,\ldots,y_{s-1},z_{s-1},y_s)\cap S_0|\geq \delta\cdot |S_0|.\]Hence, since the graph $G[S_0]$ is $(\delta,\rho)$-rich (by property (1) of \cref{lem:switching-setup}), there are at most $n^{1/5}$ choices for the remaining vertex $z_s$ such that the set
\[N_s(y_1,z_1,\ldots,y_s,z_s)\cap S_0=N(z_s)\cap \overline{N}(y_1,z_1,\ldots,y_{s-1},z_{s-1},y_s)\cap S_0\]
has size at most $\rho \cdot |\overline{N}(y_1,z_1,\ldots,y_{s-1},z_{-1},y_s)\cap S_0|$. In particular, there are at most $n^{1/5}$ choices for $z_s$ with $|N_s(y_1,z_1,\ldots,y_s,z_s)\cap S_0|< \rho\delta\cdot |S_0|$.

This indeed shows that for each $i=1,\ldots,s$ there are at most $|T|^{s-1}\cdot |S|\cdot n^{1/5}\leq |T|^{s}/(4D)$ different $s$-tuples $((y_1,z_1),\ldots,(y_s,z_s))\in T^s$ with $|N_i(y_1,z_1,\ldots,y_s,z_s)\cap S_0|< \rho\delta\cdot |S_0|$ (recall from \cref{lem:switching-T-large} that $|T|\geq |S|^2/2\geq |S|\cdot n^{0.48}/2$). Hence there are at least $(3/4)\cdot |T|^s$ different $s$-tuples $((y_1,z_1),\ldots,(y_s,z_s))\in T^s$ with $|N_i(y_1,z_1,\ldots,y_s,z_s)\cap S_0|\geq \rho\delta\cdot |S_0|$ for all $i=1,\ldots,s$. Now, at most $O_s(|T|^{s-1}\cdot |S|)\leq |T|^s/4$ of these $s$-tuples can have a repetition among the vertices $y_1,z_1,\ldots,y_s,z_s$. This proves (a).

For (b), fix some $k\in\{0,\ldots,s\}$. For each $k$-degenerate $s$-tuple $((y_1,z_1),\ldots,(y_s,z_s))\in T^s$ we can find indices $1\leq i_1<\cdots<i_{s-k}\leq s$ with the property in \cref{lem:switching-technical-lemma-degenerate-tuples}. It suffices to show that for any fixed $1\leq i_1<\cdots<i_{s-k}\leq s$, there are at most $|T|^s/(\sqrt{n}^k\cdot \binom{s}{k})$ different $s$-tuples $((y_1,z_1),\ldots,(y_s,z_s))\in T^s$ with the property in \cref{lem:switching-technical-lemma-degenerate-tuples}. To show this, first note that there are $|T|^{s-k}$ choices for $(y_{i_1},z_{i_1}),\ldots,(y_{i_{s-k}},z_{i_{s-k}})\in T$. After fixing these choices, we claim that for each $j\in [s]\setminus \{i_1,\ldots,i_{s-k}\}$ there are at most $3^{s-k}\cdot n^{2/5}$ possibilities for the vertices $y_j$ and $z_j$. Indeed, for every such $j$ there must be a vector $\vec{t}\in \{-1,0,1\}^{s-k}$ such that conditions (i) to (iii) in \cref{lem:switching-technical-lemma-degenerate-tuples} hold. There are at most $3^{s-k}$ possibilities for $\vec{t}$ satisfying (i), and whenever (i) holds there are at most $n^{1/5}$ choices for $y_j$ satisfying (ii) and at most $n^{1/5}$ choices for $z_j$ satisfying (iii), since the graph $G[S_0]$ is $(\delta,\rho)$-rich.
So overall, for fixed indices $1\leq i_1<\cdots<i_{s-k}\leq s$, there are indeed at most $|T|^{s-k}\cdot (3^{s-k}n^{2/5})^{k}\leq 3^{Dk}\cdot |T|^{s-k} \cdot (n^{0.4})^k\leq |T|^s/(\sqrt{n}^k\cdot \binom{s}{k})$ different $s$-tuples $((y_1,z_1),\ldots,(y_s,z_s))\in T^s$ satisfying the property in \cref{lem:switching-technical-lemma-degenerate-tuples} for $n$ sufficiently large (recalling that $|T|\geq n^{0.96}/2$ by \cref{lem:switching-T-large}).
\end{proof}

Now we prove \cref{lem:switching-moment-estimate}.

\begin{proof}[Proof of \cref{lem:switching-moment-estimate}]
We may assume that $n$ is sufficiently large with respect to $C$ and $A$. Let $X=e(G[U])+\sum_{u\in U}e_u$ and let us define $E=\mb{E}X$. Consider $x\in\mb Z$ such that $|x-E|\leq An^{3/2}$, and fix $a_{-B},\ldots,a_B\in \{0,1,2\}$. Let $s=a_{-B}+\cdots+a_B\leq 4B+2$ and fix a list $(\ell_{1},\ldots,\ell_{s})$ containing $a_{\ell}$ copies of each $\ell=-B,\ldots,B$. For $(y,z)\in T$, let $\mathcal{E}_i(y,z)$ be the event that $(y,z)$ contributes to $Y_{\ell_i}$; i.e., the event that we have $y\in U$ and $z\notin U$ and $(|N(z)\cap(U\setminus \{y\})|+e_z)-(|N(y)\cap (U\setminus \{x\})|+e_y)=\ell_i$. Now,
\begin{equation}\label{eq:moment-sum}
\mb E[Y_{-D}^{a_{-D}}\cdots Y_{D}^{a_{D}}Z_{x-B,x+B}]  =\sum \Pr\big[|X-x|\leq B\text{ and }\mathcal{E}_i(y_{i},z_{i})\text{ holds for }i=1,\ldots,s\big],
\end{equation}
where the sum is over all $s$-tuples $((y_1,z_1),\ldots,(y_s,z_s))\in T^s$. To prove the lemma, we separately establish lower and upper bounds
on this quantity. Note that for $s=0$, we already know that $\Pr[|X-x|\leq B] \asymp_{C,H,A} n^{-3/2}$ by \cref{thm:short-interval}, so we may assume that $s\geq 1$.

\medskip
\noindent\textit{Step 1: the lower bound. }
For the lower bound, we will only consider the contribution to \cref{eq:moment-sum} from $s$-tuples in $T^s$ satisfying \cref{lem:s-tuples-in-T-to-consider}(a). There are at least $|T|^s/2$ such $s$-tuples. So in order to establish the desired lower bound $\Omega_{C,H,A}((|T|/\sqrt{n})^s\cdot n^{-3/2})$ for the sum in \cref{eq:moment-sum}, it suffices to prove that each such $s$-tuple contributes at least $\Omega_{C,H,A}(n^{-(s+3)/2})$ to the sum. In other words, it suffices to show that
\begin{equation}\label{eq:moment-sum-lower-bound}
    \Pr\big[|X-x|\leq B\text{ and }\mathcal{E}_i(y_{i},z_{i})\text{ holds for }i=1,\ldots,s\big]\gtrsim_{C,H,A} n^{-s/2}\cdot n^{-3/2}
\end{equation}
for any $s$-tuple $((y_1,z_1),\ldots,(y_s,z_s))\in T^s$ with $|N_i(y_1,z_1,\ldots,y_s,z_s)\cap S_0|\geq \rho\delta|S_0|$ for all $i=1,\ldots,s$ and such that the vertices $y_1,z_1,\ldots,y_s,z_s$ are distinct. So let $((y_1,z_1),\ldots,(y_s,z_s))\in T^s$ be such an $s$-tuple. For simplicity of notation we write $\overline{N}=\overline{N}(y_1,z_1,\ldots,y_s,z_s)\cap S_0$ and  $N_i=N_i(y_1,z_1,\ldots,y_s,z_s)\cap S_0$ for $i=1,\ldots,s$. Then  $|N_i|\geq \rho\delta|S_0|\ge \rho\delta^{1+1/\rho}\cdot n$ for $i=1,\ldots,s$, and also $|\overline{N}|\geq \delta|S_0|\geq\delta^{1+1/\rho}\cdot n$ by property (2) of \cref{lem:switching-setup} (as $2s\leq 8B+4\leq D$). Note that  $N_1,\ldots,N_s$ and $\overline{N}$ are disjoint subsets of $S_0\sm\{y_1,z_1,\ldots,y_s,z_s\}$. Let us write $W=V(G)\setminus (N_1\cup\cdots\cup N_s\cup\overline{N})$, and note that $N(y_i)\su W$ and $N(z_i)\su W\cup N_i$ for $i=1,\ldots,s$.

We will now expose the random subset $U\su V(G)$ in several steps. First, we expose $U\cap W$ and consider the conditional expectation $\mb E[X\,|\, U\cap W]$ (which is a function of the random outcome of $U\cap W$). Note that this random variable is of the form in \cref{lem:switching-robust-conditioning} applied to the graph $G[W]$ with the random set $U\cap W\subseteq W$, with $f_w=e_w+\deg_{V(G)\setminus W}(w)$ for all $w\in W$, with $f_0 = e(V(G)\setminus W)+\sum_{v\in V(G)\setminus W}e_v$, and with $R=(H+1)n/|W|$. By \cref{lem:switching-robust-conditioning}(1), its variance is at most $((H+1)n/|W|)^2\cdot |W|^3\le (H+1)^2n^3$, and trivially its expectation is exactly $E=\mb E[X]$. Now, we claim that with probability at least $2^{-2s-2}=\Omega_C(1)$ the random outcome of $U\cap W$ satisfies the following three properties:
\begin{enumerate}
    \item[(A)] $y_1,\ldots,y_s\in U$ and $z_1,\ldots,z_s\notin U$, and
    \item[(B)] $|\mb E[X\,|\, U\cap W]-E|\leq 2^{s+1}(H+1)n^{3/2}$, and
    \item[(C)] for all $i=1,\ldots,s$, the quantity $|U\cap W\cap (N(z_i)\setminus \{y_i\})|=|U\cap (N(z_i)\setminus (\{y_i\}\cup N_i))|$ differs from $|N(z_i)\setminus (\{y_i\}\cup N_i)|/2$ by at most $2^{s+1}s\sqrt{n}$ and similarly $|U\cap W\cap (N(y_i)\setminus \{z_i\})|=|U\cap (N(y_i)\setminus \{z_i\})|$ differs from $|N(y_i)\setminus \{z_i\}|/2$ by at most $2^{s+1}s\sqrt{n}$.
\end{enumerate}
Indeed, (A) holds with probability exactly $2^{-2s}$, and by Chebyshev's inequality, (B) and (C) fail with probability at most $2^{-2s-2}$ and $2s\cdot 2^{-2s-2}/s^2$, respectively.

From now on we condition on an outcome of $U\cap W$ satisfying (A--C). Next we expose $U\cap (N_1\cup\cdots\cup N_s)$, which then determines all of $U\setminus \overline{N}$ and in particular determines whether the events $\mathcal{E}_i(y_{i},z_{i})$ for $i=1,\ldots,s$ hold. More precisely, after fixing the outcome of $U\cap W$, for each $i=1,\ldots,s$ the event $\mathcal{E}_i(y_{i},z_{i})$ is now determined by $U\cap N_i$ and holds if and only if 
\begin{equation}\label{eq:switch-lower-bound-intersection-size}
|U\cap N_i|=-|U\cap (N(z_i)\setminus (\{y_i\}\cup N_i))|-e_{z_i}+|U\cap (N(y_i)\setminus \{z_i\})|+e_{y_i}+\ell_i.
\end{equation}
In particular, the quantity on the right-hand side is determined given the information $U\cap W$. By (C), this quantity differs by at most $2^{s+2}s\sqrt{n}\leq 2^{D+2}D\sqrt{n}$ from
\begin{align*}
    &-|N(z_i)\setminus (\{y_i\}\cup N_i)|/2-e_{z_i}+|N(y_i)\setminus \{z_i\}|/2+e_{y_i}+\ell_i\\
    &\qquad \qquad =|N_i|/2-|N(z_i)\setminus \{y_i\}|/2-e_{z_i}+|N(y_i)\setminus \{z_i\}|/2+e_{y_i}+\ell_i\\
    &\qquad \qquad =|N_i|/2+(\deg(y_i)/2+e_{y_i})-(\deg(z_i)/2+e_{z_i})+\ell_i
\end{align*}
Recalling that $|(\deg(y_i)/2+e_{y_i})-(\deg(z_i)/2+e_{z_i})|\leq \sqrt{n}$ by property (3) of \cref{lem:switching-setup}, this means that the quantity on the right-hand side of \cref{eq:switch-lower-bound-intersection-size} differs from $|N_i|/2$ by at most $(2^{D+2}D+1)\sqrt{n}+B\leq 2^{D+3}D\sqrt{n}$.
Now note that, conditioning on our fixed outcome of $U\cap W$, the random variable $\mb E[X\,|\,U\setminus \overline{N}]$ is of the form in \cref{lem:switching-robust-conditioning} with the graph $G[N_1\cup\cdots\cup N_S]$ (of size at least $\rho\delta\cdot\delta^{1/\rho}n$) and with $R=R(C,H)=\max\{2^{D+3}D,(H+1)/(\rho\delta^{1+1/\rho})\}$. This random variable has expected value $\mb E [X\,|\,U\cap W]$, which differs from $E$ by at most $2^{s+1}(H+1)n^{3/2}$ by (B). So, by \cref{lem:switching-robust-conditioning}(2), 
with probability at least  $\Omega_{C,H}(n^{-s/2})$ the outcome of $U\setminus \overline{N}$ satisfies both
\begin{equation}\label{eq:switching-expectation-shift}
\big|\mb E[X\,|\,U\setminus \overline{N}]-E\big|\le (2^{s+1}(H+1)+6R^2) \cdot n^{3/2}
\end{equation}
and \cref{eq:switch-lower-bound-intersection-size} for all $i=1,\ldots,s$ (which implies that $\mathcal E_i (y_i,z_i)$ holds for all $i=1,\ldots,s$). From now on, we condition on such an outcome of $U\setminus \overline{N}$.

Finally, consider the randomness of $U\cap \overline{N}$ (having conditioned on our outcome of $U\setminus \overline{N}$). Note that $G[\overline{N}]$ is a $(2C)$-Ramsey graph (as $|\overline{N}|\geq \delta^{1+1/\rho}\cdot n\geq \sqrt{n}$), and that (in our conditional probability space) $X$ has the form in \cref{thm:short-interval}, with expectation $\mb E[X\,|\, U\setminus \overline{N}]$. Now, recalling \cref{eq:switching-expectation-shift} and the fact that $|x-E|\leq An^{3/2}$, note that $x$ differs from $\mb E[X\,|\, U\setminus \overline{N}]$ by at most $(A+2^{s+1}(H+1)+6R^2) \cdot n^{3/2}$. Therefore \cref{thm:short-interval} (plugging in $(H+1)/\delta^{1+1/\rho}$ for the ``$H$'' and $(A+2^{s+1}(H+1)+6R^2)/(\delta^{1+1/\rho})^{3/2}$ for the ``$A$'' in \cref{thm:short-interval}) implies that (conditioned on our fixed outcome of $U\setminus \overline{N}$ and subject only to the randomness of $U\cap \overline{N}$) we have $\Pr[|X-x|\leq B]\gtrsim_{C,H,A} n^{-3/2}$. This proves \cref{eq:moment-sum-lower-bound} and thereby gives the desired lower bound for the sum in \cref{eq:moment-sum}.

\medskip
\noindent \textit{Step 2: the upper bound.} To establish the desired upper bound $O_{C,H,A}((|T|/\sqrt{n})^s\cdot n^{-3/2})$ for the sum in \cref{eq:moment-sum}, for each $k=0,\ldots,s$, we separately consider the contribution of  $s$-tuples $((y_1,z_1),\ldots,(y_s,z_s))\in T^s$  of degeneracy $k$ (see \cref{def:degeneracy}). By \cref{lem:s-tuples-in-T-to-consider}, for each $k=0,\ldots,s$ there are at most $|T|^s/\sqrt{n}^k$ different such $s$-tuples of degeneracy $k$. Thus, it suffices to prove that for every $s$-tuple $((y_1,z_1),\ldots,(y_s,z_s))\in T^s$  of degeneracy $k$ we have
\begin{equation}\label{eq:moment-sum-upper-bound}
\Pr\big[|X-x|\leq B\text{ and }\mathcal{E}_i(y_{i},z_{i})\text{ holds for }i=1,\ldots,s\big]\lesssim_{C,H} n^{-(s-k)/2}\cdot n^{-3/2}.
\end{equation}
Recall the definition of the $s\times n$ matrix $M(y_1,z_1,\ldots,y_s,z_s)$ in \cref{def:degeneracy}. For every outcome of $U\su V(G)$, the entries of the vector $M(y_1,z_1,\ldots,y_s,z_s)\vec{1}_U$ are precisely $|N(z_i)\cap (U\setminus \{y_i\})|-|N(y_i)\cap (U\setminus \{z_i\})|$ for $i=1,\ldots,s$, since
\begin{align*}
\vec{1}_{N(z_i)\setminus \{y_i\}}\cdot \vec{1}_U-\vec{1}_{N(y_i)\setminus \{z_i\}}\cdot \vec{1}_U&=|(N(z_i)\setminus \{y_i\})\cap U|-|(N(y_i)\setminus \{z_i\})\cap U|\\
&=|N(z_i)\cap (U\setminus \{y_i\})|-|N(y_i)\cap (U\setminus \{z_i\})|.
\end{align*}
So if the events $\mathcal{E}_i(y_{i},z_{i})$ for $i=1,\ldots,s$  hold, we must have $M(y_1,z_1,\ldots,y_s,z_s)\vec{1}_U=(e_{y_i}-e_{z_i}+\ell_i)_{i=1}^s$.
Since $((y_1,z_1),\ldots,(y_s,z_s))$ is not $(k+1)$-degenerate, whenever we delete $\delta^{3/\gamma}\cdot n$ columns of the matrix $M(y_1,z_1,\ldots,y_s,z_s)$ the resulting matrix still has rank at least $s-k$. So applying \cref{thm:halasz-discrete} (with $\lambda=\delta^{3/\rho}$ and $r=s-k$) yields:
\[\Pr\big[\mathcal{E}_i(y_{i},z_{i})\text{ holds for }i=1,\ldots,s\big]\leq \Pr\big[M(y_1,z_1,\ldots,y_s,z_s)\vec{1}_U=(e_{y_i}-e_{z_i}+\ell_i)_{i=1}^s \big]\lesssim_{C}n^{-(s-k)/2}.\]
Thus in order to show \cref{eq:moment-sum-upper-bound}, it now suffices to prove the conditional probability bound
\begin{equation}\label{eq:switching-upper-bound-conditional}
\Pr\big[|X-x|\leq B\,\big|\,\mathcal{E}_i(y_{i},z_{i})\text{ for }i=1,\ldots,s\big]\lesssim_{C,H} n^{-3/2}.
\end{equation}
Note that the events $\mathcal{E}_i(y_{i},z_{i})$ for $i=1,\ldots,s$ only depend on $U\cap (V(G)\setminus \overline{N}(y_1,z_1,\ldots,y_s,z_s))$. So, condition on any outcome of $U\cap (V(G)\setminus \overline{N}(y_1,z_1,\ldots,y_s,z_s))$ such that $\mathcal{E}_i(y_{i},z_{i})$ holds for $i=1,\ldots,s$. Subject to the randomness of $U\cap \overline{N}(y_1,z_1,\ldots,y_s,z_s)$, our random variable $X$ has the form in \cref{thm:short-interval}, with the graph $G[\overline{N}(y_1,z_1,\ldots,y_s,z_s)]$ (which is a $(2C)$-Ramsey graph, since $|\overline{N}(y_1,z_1,\ldots,y_s,z_s)|\geq \delta|S_0|\geq  \delta^{1+1/\rho}\cdot n\geq \sqrt{n}$ by property (2) of \cref{lem:switching-setup}). Thus, in our conditional probability space, \cref{thm:short-interval} (plugging in  $(H+1)\delta^{-1-1/\rho}$ for the ``$H$'' in \cref{thm:short-interval}) yields
\[\Pr\big[|X-x|\leq B\,\big|\,U\cap (V(G)\setminus \overline{N}(y_1,z_1,\ldots,y_s,z_s))\big]\lesssim_{C,H}n^{-3/2}.\]
This proves \cref{eq:switching-upper-bound-conditional} and therefore establishes \cref{eq:moment-sum-upper-bound}, as desired.
\end{proof}

\bibliographystyle{amsplain0.bst}
\bibliography{main.bib}

\providecommand{\bysame}{\leavevmode\hbox to3em{\hrulefill}\thinspace}
\providecommand{\MR}{\relax\ifhmode\unskip\space\fi MR }
\providecommand{\MRhref}[2]{%
  \href{http://www.ams.org/mathscinet-getitem?mr=#1}{#2}
}
\providecommand{\href}[2]{#2}
\begin{thebibliography}{10}

\bibitem{Abb72}
H.~L. Abbott, \emph{Lower bounds for some {R}amsey numbers}, Discrete Math.
  \textbf{2} (1972), 289--293.

\bibitem{AH91}
N.~Alon and A.~Hajnal, \emph{Ramsey graphs contain many distinct induced
  subgraphs}, Graphs Combin. \textbf{7} (1991), 1--6.

\bibitem{ABKS09}
Noga Alon, J\'ozsef Balogh, Alexandr Kostochka, and Wojciech Samotij,
  \emph{Sizes of induced subgraphs of {R}amsey graphs}, Combin. Probab. Comput.
  \textbf{18} (2009), 459--476.

\bibitem{AB89}
Noga Alon and B\'{e}la Bollob\'{a}s, \emph{Graphs with a small number of
  distinct induced subgraphs}, Discrete Math. \textbf{75} (1989), 23--30, Graph
  theory and combinatorics (Cambridge, 1988).

\bibitem{AGS04}
Noga Alon, Gregory Gutin, and Michael Krivelevich, \emph{Algorithms with large
  domination ratio}, J. Algorithms \textbf{50} (2004), 118--131.

\bibitem{AHKT20}
Noga Alon, Dan Hefetz, Michael Krivelevich, and Mykhaylo Tyomkyn,
  \emph{Edge-statistics on large graphs}, Combin. Probab. Comput. \textbf{29}
  (2020), 163--189.

\bibitem{AK09}
Noga Alon and A.~V. Kostochka, \emph{Induced subgraphs with distinct sizes},
  Random Structures Algorithms \textbf{34} (2009), 45--53.

\bibitem{AKS03}
Noga Alon, Michael Krivelevich, and Benny Sudakov, \emph{Induced subgraphs of
  prescribed size}, J. Graph Theory \textbf{43} (2003), 239--251.

\bibitem{AO95}
Noga Alon and Alon Orlitsky, \emph{Repeated communication and {R}amsey graphs},
  IEEE Trans. Inform. Theory \textbf{41} (1995), 1276--1289.

\bibitem{AS16}
Noga Alon and Joel~H. Spencer, \emph{The probabilistic method}, fourth ed.,
  Wiley Series in Discrete Mathematics and Optimization, John Wiley \& Sons,
  Inc., Hoboken, NJ, 2016.

\bibitem{BRSW12}
Boaz Barak, Anup Rao, Ronen Shaltiel, and Avi Wigderson, \emph{2-source
  dispersers for {$n^{o(1)}$} entropy, and {R}amsey graphs beating the
  {F}rankl-{W}ilson construction}, Ann. of Math. (2) \textbf{176} (2012),
  1483--1543.

\bibitem{Ber18}
Ross Berkowitz, \emph{A local limit theorem for cliques in ${G} (n, p)$},
  arXiv:1811.03527.

\bibitem{Ber16}
Ross Berkowitz, \emph{A quantitative local limit theorem for triangles in
  random graphs}, arXiv:1610.01281.

\bibitem{BDMM}
Bhaswar~B Bhattacharya, Sayan Das, Somabha Mukherjee, and Sumit Mukherjee,
  \emph{Asymptotic distribution of random quadratic forms}, arXiv:2203.02850.

\bibitem{BMM}
Bhaswar~B. Bhattacharya, Somabha Mukherjee, and Sumit Mukherjee,
  \emph{Asymptotic distribution of {B}ernoulli quadratic forms}, Ann. Appl.
  Probab. \textbf{31} (2021), 1548--1597.

\bibitem{BS07}
Boris Bukh and Benny Sudakov, \emph{Induced subgraphs of {R}amsey graphs with
  many distinct degrees}, J. Combin. Theory Ser. B \textbf{97} (2007),
  612--619.

\bibitem{CFM92}
Neil Calkin, Alan Frieze, and Brendan~D. McKay, \emph{On subgraph sizes in
  random graphs}, Combin. Probab. Comput. \textbf{1} (1992), 123--134.

\bibitem{CGMS}
Marcelo Campos, Simon Griffiths, Robert Morris, and Julian Sahasrabudhe,
  \emph{An exponential improvement for diagonal {R}amsey}, arXiv:2303.09521.

\bibitem{CW01}
Anthony Carbery and James Wright, \emph{Distributional and {$L^q$} norm
  inequalities for polynomials over convex bodies in {$\mathbb R^n$}}, Math.
  Res. Lett. \textbf{8} (2001), 233--248.

\bibitem{CZ16}
Eshan Chattopadhyay and David Zuckerman, \emph{Explicit two-source extractors
  and resilient functions}, S{TOC}'16---{P}roceedings of the 48th {A}nnual
  {ACM} {SIGACT} {S}ymposium on {T}heory of {C}omputing, ACM, New York, 2016,
  pp.~670--683.

\bibitem{Chu81}
F.~R.~K. Chung, \emph{A note on constructive methods for {R}amsey numbers}, J.
  Graph Theory \textbf{5} (1981), 109--113.

\bibitem{Chu97}
F.~R.~K. Chung, \emph{Open problems of {P}aul {E}rd{\H{o}}s in graph theory},
  J. Graph Theory \textbf{25} (1997), 3--36.

\bibitem{CG98}
\noop{0}Fan Chung and Ron Graham, \emph{Erd{\H{o}}s on graphs}, A K Peters,
  Ltd., Wellesley, MA, 1998, His legacy of unsolved problems.

\bibitem{Coh16}
Gil Cohen, \emph{Two-source dispersers for polylogarithmic entropy and improved
  {R}amsey graphs}, S{TOC}'16---{P}roceedings of the 48th {A}nnual {ACM}
  {SIGACT} {S}ymposium on {T}heory of {C}omputing, ACM, New York, 2016,
  pp.~278--284.

\bibitem{Con13}
David Conlon and Jacob Fox, \emph{Graph removal lemmas}, Surveys in
  combinatorics 2013, London Math. Soc. Lecture Note Ser., vol. 409, Cambridge
  Univ. Press, Cambridge, 2013, pp.~1--49.

\bibitem{Cos13}
Kevin~P. Costello, \emph{Bilinear and quadratic variants on the
  {L}ittlewood-{O}fford problem}, Israel J. Math. \textbf{194} (2013),
  359--394.

\bibitem{CTV06}
Kevin~P. Costello, Terence Tao, and Van Vu, \emph{Random symmetric matrices are
  almost surely nonsingular}, Duke Math. J. \textbf{135} (2006), 395--413.

\bibitem{PV96}
Victor~H. de~la Pe\~{n}a, \emph{From dependence to complete independence: the
  decoupling approach}, Fourth {S}ymposium on {P}robability {T}heory and
  {S}tochastic {P}rocesses ({S}panish) ({G}uanajuato, 1996), Aportaciones Mat.
  Notas Investigaci\'{o}n, vol.~12, Soc. Mat. Mexicana, M\'{e}xico, 1996,
  pp.~37--48.

\bibitem{DSW21}
Daniel Di~Benedetto, J\'{o}zsef Solymosi, and Ethan~P. White, \emph{On the
  directions determined by a {C}artesian product in an affine {G}alois plane},
  Combinatorica \textbf{41} (2021), 755--763.

\bibitem{Dur19}
Rick Durrett, \emph{Probability---theory and examples}, fifth ed., Cambridge
  Series in Statistical and Probabilistic Mathematics, vol.~49, Cambridge
  University Press, Cambridge, 2019.

\bibitem{EY36}
Carl Eckart and Gale Young, \emph{The approximation of one matrix by another of
  lower rank}, Psychometrika \textbf{1} (1936), 211--218.

\bibitem{Erd45}
\noop{0}P. Erd{\H{o}}s, \emph{On a lemma of {L}ittlewood and {O}fford}, Bull.
  Amer. Math. Soc. \textbf{51} (1945), 898--902.

\bibitem{Erd47}
\noop{1}P. Erd{\H{o}}s, \emph{Some remarks on the theory of graphs}, Bull.
  Amer. Math. Soc. \textbf{53} (1947), 292--294.

\bibitem{Erd92}
\noop{2}Paul Erd{\H{o}}s, \emph{Some of my favourite problems in various
  branches of combinatorics. \textup{Combinatorics 92 (Catania, 1992)}},
  Matematiche (Catania) \textbf{47} (1992), 231--240 (1993).

\bibitem{Erd95}
\noop{3}Paul Erd{\H{o}}s, \emph{Some of my favourite problems in number theory,
  combinatorics, and geometry}, Resenhas \textbf{2} (1995), 165--186,
  Combinatorics Week (Portuguese) (S\~{a}o Paulo, 1994).

\bibitem{Erd97}
\noop{4}Paul Erd{\H{o}}s, \emph{Some recent problems and results in graph
  theory. \textup{The Second {K}rakow Conference on Graph Theory
  ({Z}gorzelisko, 1994)}}, Discrete Math. \textbf{164} (1997), 81--85.

\bibitem{EH77}
P.~Erd{\H{o}}s and A.~Hajnal, \emph{On spanned subgraphs of graphs},
  Contributions to graph theory and its applications ({I}nternat. {C}olloq.,
  {O}berhof, 1977), Tech. Hochschule Ilmenau, Ilmenau, 1977, pp.~80--96.

\bibitem{ES35}
P.~Erd{\H{o}}s and G.~Szekeres, \emph{A combinatorial problem in geometry},
  Compositio Math. \textbf{2} (1935), 463--470.

\bibitem{ES72}
P.~Erd{\H{o}}s and A.~Szemer\'{e}di, \emph{On a {R}amsey type theorem}, Period.
  Math. Hungar. \textbf{2} (1972), 295--299.

\bibitem{FM07}
Veerle Fack and Brendan~D. McKay, \emph{A generalized switching method for
  combinatorial estimation}, Australas. J. Combin. \textbf{39} (2007),
  141--154.

\bibitem{FJZ22}
Asaf Ferber, Vishesh Jain, and Yufei Zhao, \emph{On the number of {H}adamard
  matrices via anti-concentration}, Combin. Probab. Comput. \textbf{31} (2022),
  455--477.

\bibitem{FKMW18}
Yuval Filmus, Guy Kindler, Elchanan Mossel, and Karl Wimmer, \emph{Invariance
  principle on the slice}, ACM Trans. Comput. Theory \textbf{10} (2018), Art.
  11, 37.

\bibitem{FM19}
Yuval Filmus and Elchanan Mossel, \emph{Harmonicity and invariance on slices of
  the {B}oolean cube}, Probab. Theory Related Fields \textbf{175} (2019),
  721--782.

\bibitem{FKS21}
Jacob Fox, Matthew Kwan, and Lisa Sauermann, \emph{Combinatorial
  anti-concentration inequalities, with applications}, Math. Proc. Cambridge
  Philos. Soc. \textbf{171} (2021), 227--248.

\bibitem{FS20}
Jacob Fox and Lisa Sauermann, \emph{A completion of the proof of the
  edge-statistics conjecture}, Adv. Comb. (2020), Paper No. 4, 52.

\bibitem{FS11}
Jacob Fox and Benny Sudakov, \emph{Dependent random choice}, Random Structures
  Algorithms \textbf{38} (2011), 68--99.

\bibitem{FW81}
P.~Frankl and R.~M. Wilson, \emph{Intersection theorems with geometric
  consequences}, Combinatorica \textbf{1} (1981), 357--368.

\bibitem{Fra77}
P\'{e}ter Frankl, \emph{A constructive lower bound for some {R}amsey numbers},
  Ars Combin. \textbf{3} (1977), 297--302.

\bibitem{GS91}
Roger~G. Ghanem and Pol~D. Spanos, \emph{Stochastic finite elements: a spectral
  approach}, Springer-Verlag, New York, 1991.

\bibitem{GK16}
Justin Gilmer and Swastik Kopparty, \emph{A local central limit theorem for
  triangles in a random graph}, Random Structures Algorithms \textbf{48}
  (2016), 732--750.

\bibitem{Gne48}
Boris~Vladimirovich Gnedenko, \emph{On a local limit theorem of the theory of
  probability}, Uspekhi Matematicheskikh Nauk \textbf{3} (1948), 187--194.

\bibitem{Gop14}
Parikshit Gopalan, \emph{Constructing {R}amsey graphs from {B}oolean function
  representations}, Combinatorica \textbf{34} (2014), 173--206.

\bibitem{GKM17}
Catherine Greenhill, Mikhail Isaev, Matthew Kwan, and Brendan~D. McKay,
  \emph{The average number of spanning trees in sparse graphs with given
  degrees}, European J. Combin. \textbf{63} (2017), 6--25.

\bibitem{Gut22}
Larry Guth, \emph{Decoupling estimates in {F}ourier analysis},
  arXiv:2207.00652.

\bibitem{Hal77}
G.~Hal\'{a}sz, \emph{Estimates for the concentration function of combinatorial
  number theory and probability}, Period. Math. Hungar. \textbf{8} (1977),
  197--211.

\bibitem{HP21}
Brandon Hanson and Giorgis Petridis, \emph{Refined estimates concerning sumsets
  contained in the roots of unity}, Proc. Lond. Math. Soc. (3) \textbf{122}
  (2021), 353--358.

\bibitem{HW71}
D.~L. Hanson and F.~T. Wright, \emph{A bound on tail probabilities for
  quadratic forms in independent random variables}, Ann. Math. Statist.
  \textbf{42} (1971), 1079--1083.

\bibitem{HM10}
Mahdieh Hasheminezhad and Brendan~D. McKay, \emph{Combinatorial estimates by
  the switching method}, Combinatorics and graphs, Contemp. Math., vol. 531,
  Amer. Math. Soc., Providence, RI, 2010, pp.~209--221.

\bibitem{JLR00}
Svante Janson, Tomasz {\L}uczak, and Andrzej Rucinski, \emph{Random graphs},
  Wiley-Interscience Series in Discrete Mathematics and Optimization,
  Wiley-Interscience, New York, 2000.

\bibitem{JKLY20}
Matthew Jenssen, Peter Keevash, Eoin Long, and Liana Yepremyan, \emph{Distinct
  degrees in induced subgraphs}, Proc. Amer. Math. Soc. \textbf{148} (2020),
  3835--3846.

\bibitem{Kan17}
Daniel Kane, \emph{A structure theorem for poorly anticoncentrated polynomials
  of {G}aussians and applications to the study of polynomial threshold
  functions}, Ann. Probab. \textbf{45} (2017), 1612--1679.

\bibitem{KV00}
Jeong~Han Kim and Van~H. Vu, \emph{Concentration of multivariate polynomials
  and its applications}, Combinatorica \textbf{20} (2000), 417--434.

\bibitem{KS06}
M.~Krivelevich and B.~Sudakov, \emph{Pseudo-random graphs}, More sets, graphs
  and numbers, Bolyai Soc. Math. Stud., vol.~15, Springer, Berlin, 2006,
  pp.~199--262.

\bibitem{KLP17}
Greg Kuperberg, Shachar Lovett, and Ron Peled, \emph{Probabilistic existence of
  regular combinatorial structures}, Geom. Funct. Anal. \textbf{27} (2017),
  919--972.

\bibitem{KS20b}
Matthew Kwan and Lisa Sauermann, \emph{An algebraic inverse theorem for the
  quadratic {L}ittlewood-{O}fford problem, and an application to {R}amsey
  graphs}, Discrete Anal. (2020), Paper No. 12, 34.

\bibitem{KS19}
Matthew Kwan and Benny Sudakov, \emph{Proof of a conjecture on induced
  subgraphs of {R}amsey graphs}, Trans. Amer. Math. Soc. \textbf{372} (2019),
  5571--5594.

\bibitem{KS20}
Matthew Kwan and Benny Sudakov, \emph{Ramsey graphs induce subgraphs of
  quadratically many sizes}, Int. Math. Res. Not. IMRN (2020), 1621--1638.

\bibitem{KST19}
Matthew Kwan, Benny Sudakov, and Tuan Tran, \emph{Anticoncentration for
  subgraph statistics}, J. Lond. Math. Soc. (2) \textbf{99} (2019), 757--777.

\bibitem{LCM06}
S.~N. Lahiri, A.~Chatterjee, and T.~Maiti, \emph{A sub-{G}aussian
  {B}erry-{E}sseen theorem for the hypergeometric distribution}, arXiv:0602276.

\bibitem{LPRT17}
Massimo Lauria, Pavel Pudl\'{a}k, Vojt\v{e}ch R\"{o}dl, and Neil Thapen,
  \emph{The complexity of proving that a graph is {R}amsey}, Combinatorica
  \textbf{37} (2017), 253--268.

\bibitem{Li18}
Xin Li, \emph{Non-malleable extractors and non-malleable codes: partially
  optimal constructions}, 34th {C}omputational {C}omplexity {C}onference,
  LIPIcs. Leibniz Int. Proc. Inform., vol. 137, Schloss Dagstuhl. Leibniz-Zent.
  Inform., Wadern, 2019, pp.~Art. No. 28, 49.

\bibitem{LP}
Eoin Long and Lauren\c{t}iu Ploscaru, \emph{A bipartite version of the
  {E}rd{\H{o}}s-{M}c{K}ay conjecture}, 2023, pp.~465--477.

\bibitem{MMNT19}
Anders Martinsson, Frank Mousset, Andreas Noever, and Milo\v{s} Truji\'{c},
  \emph{The edge-statistics conjecture for {$\ell\ll k^{6/5}$}}, Israel J.
  Math. \textbf{234} (2019), 677--690.

\bibitem{MOO10}
Elchanan Mossel, Ryan O'Donnell, and Krzysztof Oleszkiewicz, \emph{Noise
  stability of functions with low influences: invariance and optimality}, Ann.
  of Math. (2) \textbf{171} (2010), 295--341.

\bibitem{Nag72}
Zsigmond Nagy, \emph{A certain constructive estimate of the {R}amsey number},
  Mat. Lapok \textbf{23} (1972), 301--302 (1974).

\bibitem{NST16}
Bhargav Narayanan, Julian Sahasrabudhe, and Istv\'{a}n Tomon, \emph{Ramsey
  graphs induce subgraphs of many different sizes}, Combinatorica \textbf{39}
  (2019), 215--237.

\bibitem{Ngu12}
\noop{0}Hoi~H. Nguyen, \emph{Inverse {L}ittlewood-{O}fford problems and the
  singularity of random symmetric matrices}, Duke Math. J. \textbf{161} (2012),
  545--586.

\bibitem{NV11}
\noop{1}Hoi Nguyen and Van Vu, \emph{Optimal inverse {L}ittlewood-{O}fford
  theorems}, Adv. Math. \textbf{226} (2011), 5298--5319.

\bibitem{NV13}
\noop{2}Hoi~H. Nguyen and Van~H. Vu, \emph{Small ball probability, inverse
  theorems, and applications}, Erd{\H{o}}s centennial, Bolyai Soc. Math. Stud.,
  vol.~25, J\'{a}nos Bolyai Math. Soc., Budapest, 2013, pp.~409--463.

\bibitem{NW94}
Noam Nisan and Avi Wigderson, \emph{On rank vs. communication complexity}, 35th
  {A}nnual {S}ymposium on {F}oundations of {C}omputer {S}cience ({S}anta {F}e,
  {NM}, 1994), IEEE Comput. Soc. Press, Los Alamitos, CA, 1994, pp.~831--836.

\bibitem{O14}
Ryan O'Donnell, \emph{Analysis of {B}oolean functions}, Cambridge University
  Press, New York, 2014.

\bibitem{Pat49}
P.~B. Patnaik, \emph{The non-central $\chi^2$- and $f$-distribution and their
  applications}, Biometrika \textbf{36} (1949), 202--232.

\bibitem{Pet75}
V.~V. Petrov, \emph{Sums of independent random variables}, Ergebnisse der
  Mathematik und ihrer Grenzgebiete, Band 82, Springer-Verlag, New
  York-Heidelberg, 1975, Translated from the Russian by A. A. Brown.

\bibitem{PR99}
Hans~J\"urgen Pr\"omel and Vojt{\v{e}}ch R\"odl, \emph{Non-{R}amsey graphs are
  {$c\log n$}-universal}, J. Combin. Theory Ser. A \textbf{88} (1999),
  379--384.

\bibitem{Roo}
Bero Roos, \emph{New inequalities for permanents and hafnians and some
  generalizations}, arXiv:1906.06176.

\bibitem{Rud14}
Mark Rudelson, \emph{Recent developments in non-asymptotic theory of random
  matrices}, Modern aspects of random matrix theory, Proc. Sympos. Appl. Math.,
  vol.~72, Amer. Math. Soc., Providence, RI, 2014, pp.~83--120.

\bibitem{RV08}
Mark Rudelson and Roman Vershynin, \emph{The {L}ittlewood-{O}fford problem and
  invertibility of random matrices}, Adv. Math. \textbf{218} (2008), 600--633.

\bibitem{SS20}
Ashwin Sah and Mehtaab Sawhney, \emph{Local limit theorems for subgraph
  counts}, 2022, pp.~950--1011.

\bibitem{Sha11}
Ronen Shaltiel, \emph{An introduction to randomness extractors}, Automata,
  languages and programming. {P}art {II}, Lecture Notes in Comput. Sci., vol.
  6756, Springer, Heidelberg, 2011, pp.~21--41.

\bibitem{She98}
Saharon Shelah, \emph{Erd{\H{o}}s and {R}\'enyi conjecture}, J. Combin. Theory
  Ser. A \textbf{82} (1998), 179--185.

\bibitem{TV10b}
Terence Tao and Van Vu, \emph{A sharp inverse {L}ittlewood-{O}fford theorem},
  Random Structures Algorithms \textbf{37} (2010), 525--539.

\bibitem{TV09}
Terence Tao and Van~H. Vu, \emph{Inverse {L}ittlewood-{O}fford theorems and the
  condition number of random discrete matrices}, Ann. of Math. (2) \textbf{169}
  (2009), 595--632.

\bibitem{TV10}
Terence Tao and Van~H. Vu, \emph{Additive combinatorics}, Cambridge Studies in
  Advanced Mathematics, vol. 105, Cambridge University Press, Cambridge, 2010.

\bibitem{Var15}
P\'{e}ter~P\'{a}l Varj\'{u}, \emph{Random walks in {E}uclidean space}, Ann. of
  Math. (2) \textbf{181} (2015), 243--301.

\bibitem{Ver14}
Roman Vershynin, \emph{Invertibility of symmetric random matrices}, Random
  Structures Algorithms \textbf{44} (2014), 135--182.

\end{thebibliography}

\end{document}